\documentclass[reqno,11pt]{amsart}
\usepackage{amsfonts,amsmath,amssymb}
\usepackage{graphicx}

\usepackage[margin=1.30in]{geometry}
\numberwithin{equation}{section}

\def\div{\mathrm{div}}

\def\eps{\epsilon}

\def\ze{\zeta}

\def\varep{\varepsilon}

\def\D{\mathcal{D}}

\def\phii{\widetilde{\varphi}}

\def\div{{\,\mathrm{div}\,}}
\def\curl{{\,\mbox{curl}\,}}

\def\hat{\widehat}

\def\bar{\overline}
\def\R{{\mathbb R}}
\def\e{{\varepsilon}}

\def\M{{\mathcal M}}

\newtheorem{theorem}{Theorem}[section]
\newtheorem{lemma}[theorem]{Lemma}
\newtheorem{corollary}[theorem]{Corollary}
\newtheorem{proposition}[theorem]{Proposition}
\newtheorem{definition}[theorem]{Definition}
\newtheorem{remark}[theorem]{Remark}

\begin{document}

\title[Gravity-capillary water waves in 3D]{Global solutions of the gravity-capillary water-wave system in three dimensions}

\author{Y. Deng}
\address{Princeton University}
\email{yudeng@math.princeton.edu}

\author{A. D. Ionescu}
\address{Princeton University}
\email{aionescu@math.princeton.edu}

\author{B. Pausader}
\address{Brown University}
\email{benoit.pausader@math.brown.edu}

\author{F. Pusateri}
\address{Princeton University}
\email{fabiop@math.princeton.edu}

\thanks{Y. Deng was supported in part by a Jacobus Fellowship from Princeton University. A. D. Ionescu was supported in part by NSF grant DMS-1265818 and NSF-FRG grant DMS-1463753. 
B. Pausader was supported in part by NSF grant DMS-1362940, and a Sloan fellowship. F. Pusateri was supported in part by NSF grant DMS-1265875.}

\begin{abstract}

{\small
In this paper we prove global regularity for the full water waves system in 3 dimensions for small data, under the influence of both gravity and surface tension. This problem presents essential difficulties which were absent in all of the earlier global regularity results for other water wave models.

To construct global solutions we use a combination of energy estimates and matching dispersive estimates. There is a significant new difficulty in proving energy estimates in our problem, namely the combination of slow pointwise decay of solutions (no better than $|t|^{-5/6}$) and the presence of a large, codimension 1, set of quadratic time-resonances. To deal with such a situation we propose here a new mechanism, which exploits a non-degeneracy property of the time-resonant hypersurfaces and some special structure of the quadratic part of the nonlinearity, connected to the conserved energy of the system.

The dispersive estimates rely on analysis of the Duhamel formula in the Fourier space. The main contributions come from the set of space-time resonances, which is a large set of dimension 1. To control the corresponding bilinear interactions we use Harmonic Analysis techniques, such as orthogonality arguments in the Fourier space and atomic decompositions of functions. Most importantly, we construct and use a refined norm which is well adapted to the geometry of the problem.}
\end{abstract}

\maketitle

\setcounter{tocdepth}{1}
\tableofcontents
\setcounter{secnumdepth}{3}

\section{Introduction}\label{Intro}

The study of the motion of water waves, such as those on the surface of the ocean, is a classical question, 
and one of the main problems in fluid dynamics. 
The origins of water waves theory can be traced back\footnote{We refer to the review paper of Craik \cite{Craik}, 
and references therein,
for more details about these early studies.}
at least to the work of Laplace and Lagrange, Cauchy \cite{CauchyMemoirs} and Poisson, 
and then Russel, Green and Airy, among others.
Classical studies include those by Stokes \cite{Stokes}, Levi-Civita \cite{LeviC}
and Struik \cite{Struik} on progressing waves, 
the instability analysis of Taylor \cite{Taylor},
the works on solitary waves by Friedrichs and Hyers \cite{FriedHy}, 
and on steady waves by Gerber \cite{Gerber}.

The main questions one can ask about water waves are the typical ones for any physical evolution problem:
the local in time well-posedness of the Cauchy problem, the regularity of solutions and the formation of singularities,
the existence of special solutions (such as solitary waves) and their stability, and the global existence and long-time behavior of solutions.
There is a vast body of literature dedicated to all of these aspects.
As it would be impossible to give exhaustive references, we will mostly mention works that are connected to our results,
and refer to various books and review papers for others.

Our main interest here is the existence of global solutions for the initial value problem.
In particular, we will consider the full irrotational water waves problem
for a three dimensional fluid occupying a region of infinite depth and infinite extent below the graph of a function.
This is a model for the motion of waves on the surface of the deep ocean.
We will consider such dynamics under the influence of the gravitational force and surface tension acting on particles at the interface.
Our main result is the existence of global classical solutions
for this problem, for sufficiently small initial data.


\subsection{Free boundary Euler equations and water waves}\label{secWW}
The evolution of an inviscid perfect fluid that occupies a domain $\Omega_t \subset \R^n$, for $n \geq 2$, at time $t \in \R$,
is described by the free boundary incompressible Euler equations.
If $v$ and $p$ denote respectively the velocity and the pressure of the fluid (with constant density equal to $1$)
at time $t$ and position $x \in \Omega_t$, these equations are
\begin{equation}\label{E}
(\partial_t + v \cdot \nabla) v = - \nabla p - g e_n,\qquad \nabla \cdot v = 0,\qquad x \in \Omega_t,
\end{equation}
where $g$ is the gravitational constant. The first equation in \eqref{E} is the conservation of momentum equation, while the second is the incompressibility condition.
The free surface $S_t := \partial \Omega_t$ moves with the normal component of the velocity according to the kinematic boundary condition
 \begin{equation}\label{BC1}
\partial_t + v \cdot \nabla  \,\, \mbox{is tangent to} \,\, {\bigcup}_t S_t \subset \R^{n+1}_{x,t}.
\end{equation}
The pressure on the interface is given by
\begin{equation}
\label{BC2}
p (x,t) = \sigma \kappa(x,t),  \qquad x \in S_t,
\end{equation}
where $\kappa$ is the mean-curvature of $S_t$ and $\sigma \geq 0$ is the surface tension coefficient.
At liquid-air interfaces, the surface tension force results from the greater attraction
of water molecules to each other than to the molecules in the air.

One can also consider the free boundary Euler equations \eqref{E}-\eqref{BC2} in various types of domains $\Omega_t$ (bounded, periodic, unbounded)
and study flows with different characteristics (rotational/irrotational, with gravity and/or surface tension),
or even more complicated scenarios where the moving interface separates two fluids.

In the case of irrotational flows, $\rm{\curl} v = 0$, one can reduce \eqref{E}-\eqref{BC2} to a system on the boundary.
Indeed, assume also that $\Omega_t \subset \R^n$ is the region below the graph of a function $h : \R^{n-1}_x \times I_t \rightarrow \R$,
that is
\begin{align*}
\Omega_t = \{ (x,y) \in \R^{n-1} \times \R \, : y \leq h(x,t) \} \quad \mbox{and} \quad S_t = \{ (x,y) : y = h(x,t) \}.
\end{align*}
Let $\Phi$ denote the velocity potential, $\nabla_{x,y} \Phi(x,y,t) = v (x,y,t)$, for $(x,y) \in \Omega_t$.
If $\phi(x,t) := \Phi (x, h(x,t),t)$ is the restriction of $\Phi$ to the boundary $S_t$,
the equations of motion reduce to the following system for the unknowns $h, \phi : \R^{n-1}_x \times I_t \rightarrow \R$:
\begin{equation}
\label{WWE}
\left\{
\begin{array}{l}
\partial_t h = G(h) \phi,
\\
\partial_t \phi = -g h  + \sigma \div \Big[ \dfrac{\nabla h}{ (1+|\nabla h|^2)^{1/2} } \Big]
  - \dfrac{1}{2} {|\nabla \phi|}^2 + \dfrac{{\left( G(h)\phi + \nabla h \cdot \nabla \phi \right)}^2}{2(1+{|\nabla h|}^2)}.
\end{array}
\right.
\end{equation}
Here
\begin{equation}
\label{defG0}
G(h) := \sqrt{1+{|\nabla h|}^2} \mathcal{N}(h),
\end{equation}
and $\mathcal{N}(h)$ is the Dirichlet-Neumann map associated to the domain $\Omega_t$.
Roughly speaking, one can think of $G(h)$ as a first order, non-local, linear operator that depends nonlinearly on the domain.
We refer to  \cite[chap. 11]{SulemBook} or the book of Lannes \cite{LannesBook} for the derivation of \eqref{WWE}.
For sufficiently small smooth solutions, this system admits the conserved energy
\begin{equation}\label{CPWHam}
\begin{split}
\mathcal{H}(h,\phi) &:= \frac{1}{2} \int_{\R^{n-1}} G(h)\phi \cdot \phi \, dx + \frac{g}{2} \int_{\R^{n-1}} h^2 \,dx
  + \sigma\int_{\R^{n-1}} \frac{{|\nabla h|}^2}{1 + \sqrt{1+|\nabla h|^2} } \, dx
  \\
  &\approx {\big\| |\nabla |^{1/2} \phi \big\|}_{L^2}^2 + {\big\| (g-\sigma\Delta)^{1/2}h \big\|}_{L^2}^2,
\end{split}
\end{equation}
which is the sum of the kinetic energy corresponding to the $L^2$ norm of the velocity field and the potential energy due to gravity and surface tension.
It was first observed by Zakharov \cite{Zak0} that \eqref{WWE} is the Hamiltonian flow associated to \eqref{CPWHam}.


One generally refers to the system \eqref{WWE} as the gravity water waves system when $g>0,\sigma=0$,
as the capillary water waves system when $g=0,\sigma>0$, and as the gravity-capillary water waves system when $g>0,\sigma>0$.

\subsection{The main theorem}\label{MainResult}

Our results in this paper concern the gravity-capillary water waves system \eqref{WWE}, in the case $n=3$.
In this case $h$ and $\phi$ are real-valued functions defined on $\mathbb{R}^2\times I$.

To state our main theorem we first introduce some notation. The rotation vector-field
\begin{equation}\label{Ama10}
\Omega:=x_1\partial_{x_2}-x_2\partial_{x_1}
\end{equation}
commutes with the linearized system. For $N\geq 0$ let $H^N$ denote the standard Sobolev spaces on $\mathbb{R}^2$. More generally, for $N,N'\geq 0$ and $b\in[-1/2,1/2]$, $b\leq N$, we define the norms
\begin{equation}\label{normH}
{\|f\|}_{H^{N',N}_{\Omega}} := \sum_{j\leq N'}\|\Omega^jf\|_{H^N},\qquad {\|f\|}_{\dot{H}^{N,b}} := {\big\| (|\nabla|^N + |\nabla|^b)f \big\|}_{L^2}.
\end{equation}
For simplicity of notation, we sometimes let $H^{N'}_\Omega:=H^{N',0}_{\Omega}$. Our main theorem is the following:

\begin{theorem}[Global Regularity] \label{MainTheo}
Assume that $g,\sigma>0$, $\delta>0$ is sufficiently small, and $N_0,N_1,N_3,N_4$ are sufficiently large\footnote{The
values of $N_0$ and $N_1$, the total number of derivatives we assume under control, can certainly be decreased by reworking parts of the argument.
We prefer, however, to simplify the argument wherever possible instead of aiming for such improvements. For convenience, we arrange that $N_1-N_4=(N_0-N_3)/2-N_4=1/\delta$.}
(for example $\delta=1/2000$, $N_0:=4170$, $N_1:=2070$, $N_3:=30$, $N_4:=70$, compare with Definition \ref{MainZDef}). Assume that the data $(h_0,\phi_0)$ satisfies
\begin{equation}\label{h0p0}
\begin{split}
&\|\mathcal{U}_0\|_{H^{N_0}\cap H_\Omega^{N_1,N_3}} + \sup_{2m+|\alpha|\leq N_1+N_4}{\|(1+|x|)^{1-50\delta}D^\alpha \Omega^m\mathcal{U}_0\|}_{L^2}= \e_0 \leq \bar{\e_0},\\
&\mathcal{U}_0:=(g-\sigma\Delta)^{1/2}h_0+i|\nabla|^{1/2}\phi_0,
\end{split}
\end{equation}
where $\bar{\e_0}$ is a sufficiently small constant and $D^\alpha=\partial_1^{\alpha^1}\partial_2^{\alpha^2}$, $\alpha=(\alpha^1,\alpha^2)$. Then, there is a unique global solution $(h,\phi)\in C\big([0,\infty) : H^{N_0+1}\times \dot{H}^{N_0+1/2,1/2}\big)$ of the system \eqref{WWE}, with $(h(0),\phi(0))=(h_0,\phi_0)$.
In addition
\begin{align}\label{mainconcl1}
(1+t)^{-\delta^2} {\|\mathcal{U}(t)\|}_{H^{N_0}\cap H_\Omega^{N_1,N_3}}\lesssim \e_0,\qquad (1+t)^{5/6-3\delta^2} {\|\mathcal{U}(t)\|}_{L^\infty}\lesssim \e_0,
\end{align}
for any $t\in[0,\infty)$, where $\mathcal{U}:=(g-\sigma\Delta)^{1/2}h+i|\nabla|^{1/2}\phi$.
\end{theorem}

\begin{remark}\label{MainRemark} (i) One can derive additional information about the global solution $(h,\phi)$. Indeed, by rescaling we may assume that $g=1$ and $\sigma=1$. Let
\begin{equation}\label{Ama30}
\mathcal{U}(t):=(1-\Delta)^{1/2}h+i|\nabla|^{1/2}\phi,\qquad \mathcal{V}(t):=e^{it\Lambda}\mathcal{U}(t),\qquad \Lambda(\xi):=\sqrt{|\xi|+|\xi|^3}.
\end{equation}
Here $\Lambda$ is the linear dispersion relation, and $\mathcal{V}$ is the profile of the solution $\mathcal{U}$.
The proof of the theorem gives the strong uniform bound
\begin{align}\label{Ama30.1}
\sup_{t\in [0,\infty)}{\|\mathcal{V}(t)\|}_{Z}\lesssim \e_0,
\end{align}
see Definition \ref{MainZDef}.
The pointwise decay bound in \eqref{mainconcl1} follows from this and the linear estimates in Lemma \ref{LinEstLem} below.

(ii) The global solution $\mathcal{U}$ scatters in the $Z$ norm as $t\to\infty$, i.e. there is $\mathcal{V}_\infty \in Z$ such that
\begin{align*}
\lim_{t \rightarrow \infty} {\| e^{it\Lambda}\mathcal{U}(t) - \mathcal{V}_\infty \|}_Z = 0 .
\end{align*}
However, the asymptotic behavior is somewhat nontrivial since $|\widehat{\mathcal{U}}(\xi,t)|\gtrsim \log t\to\infty$ for frequencies $\xi$ on a circle in $\mathbb{R}^2$ (the set of space-time resonance outputs) and for some data. This unusual behavior is due to the presence of a large set of space-time resonances.

(iii) The function $\mathcal{U}:=(g-\sigma\Delta)^{1/2}h+i|\nabla|^{1/2}\phi$ is called the ``Hamiltonian variable'', due to its connection to the Hamiltonian \eqref{CPWHam}. This variable is important in order to keep track correctly of the relative Sobolev norms of the functions $h$ and $\phi$ during the proof.
\end{remark}

\subsection{Background}

We now discuss some background on the water waves system and review some of the history and previous work on this problem.

\subsubsection{The equations and the local well-posedness theory}
The free boundary Euler equations \eqref{E}-\eqref{BC2} are a time reversible system of evolution equations
which preserve the total (kinetic plus potential) energy.
Under the Rayleigh-Taylor sign condition \cite{Taylor}
\begin{align}
 \label{RT}
- \nabla_{n(x,t)} p(x,t) < 0, \qquad  x \in S_t,
\end{align}
where $n$ is the outward pointing unit normal to $\Omega_t$, the system has a (degenerate) hyperbolic structure.
This structure is somewhat hard to capture because of the moving domain and the quasilinear nature of the problem. 
Historically, this has made the task of establishing local wellposedness
(existence and uniqueness of smooth solutions for the Cauchy problem) non-trivial.

Early results on the local wellposedness of the system  include those by Nalimov \cite{Nalimov}, 
Yosihara \cite{Yosi}, and Craig \cite{CraigLim}; these results deal with small perturbations of a flat interface for which \eqref{RT} always holds.
It was first observed by Wu \cite{Wu2} that in the irrotational case the Rayleigh-Taylor sign condition holds without smallness assumptions,
and that local-in-time solutions can be constructed with initial data of arbitrary size in Sobolev spaces \cite{Wu1,Wu2}.

Following the breakthrough of Wu, in recent years
the question of local wellposedness of the water waves and free boundary Euler equations has been addressed by several authors.
Christodoulou--Lindblad \cite{CL} and Lindblad \cite{Lindblad} considered the gravity problem with vorticity,
Beyer--Gunther \cite{BG} took into account the effects of surface tension,
and Lannes \cite{Lannes} treated the case of non-trivial bottom topography.
Subsequent works by Coutand-Shkoller \cite{CS2} and Shatah-Zeng \cite{ShZ1,ShZ3}
extended these results to more general scenarios with vorticity and surface tension, including two-fluids systems \cite{CCS1,ShZ3} where surface tension is necessary for wellposedness.
Some recent papers that include surface tension and/or low regularity analysis are those by Ambrose-Masmoudi \cite{AM},
Christianson-Hur-Staffilani \cite{CHS}, and Alazard-Burq-Zuily \cite{ABZ1,ABZ2}.

Thanks to all the contributions mentioned above 
the local well-posedness theory is presently well-understood in a variety of different scenarios.
In short, one can say that for sufficiently nice initial configurations, it is possible to find classical smooth solutions on a small time interval, which depends on the smoothness of the initial data.

\subsubsection{Asymptotic models} We note that 
many simplified models have been derived and studied in special regimes, with the goal of understanding the complex dynamics of the water wave system.
These include the KdV equation, the Benjamin--Ono equation, the Boussinesq and the KP equations, as well as the nonlinear Schr\"odinger equation.
We refer to \cite{CraigLim,SWgsigma,ASL,TotzWuNLS} and to the book \cite{LannesBook} and references therein for more
about approximate models.

\subsubsection{Previous work on long-time existence}
The problem of long time existence of solutions is more challenging, and fewer results have been obtained so far.
As in all quasilinear problems, the long-time regularity has been studied in a perturbative (and dispersive) setting,
that is in the regime of small and localized perturbations of a flat interface. 
Large perturbations can lead to breakdown in finite time, see for example the papers on ``splash'' singularities \cite{CCFGG,CSSplash}.

The first long-time result for the water waves system \eqref{WWE} is due to Wu \cite{WuAG}
who showed almost global existence for the gravity problem ($g>0$, $\sigma=0$) in two dimensions ($1$d interfaces).
Subsequently, Germain-Masmoudi-Shatah \cite{GMS2} and Wu \cite{Wu3DWW} proved global existence of gravity waves in three dimensions ($2$d interfaces).
Global regularity in $3$d was also proved for the capillary problem ($g=0$, $\sigma>0$) by Germain-Masmoudi-Shatah \cite{GMSC}.
See also the recent work of Wang \cite{Wa2,Wa3} on the gravity problem in $3$d over a finite flat bottom.

Global regularity for the gravity water waves system in $2$d (the harder case) has been proved by two of the authors in \cite{IoPu2} and, independently, by Alazard-Delort \cite{ADa,ADb}.
A different proof of Wu's $2$d almost global existence result was later given by Hunter-Ifrim-Tataru \cite{HIT},
and then complemented to a proof of global regularity in \cite{IT}.
Finally, Wang \cite{Wa1} proved global regularity for a more general class of small data of infinite energy,
thus removing the momentum condition on the velocity field that was present in all the previous $2$d results.
For the capillary problem in $2$d, global regularity 
was proved by two of the authors in \cite{IoPu4} and, independently,
by Ifrim-Tataru \cite{IT2} in the case of data satisfying an additional momentum condition. 

We remark that all the global regularity results that have been proved so far require 3 basic assumptions: small data (small perturbations of the rest solution), trivial vorticity inside the fluid, and flat Euclidean geometry. Additional properties are also important, such as the Hamiltonian structure of the equations, the rate of decay of the linearized waves, and the resonance structure of the bilinear wave interactions.

\subsection{Main ideas} The classical mechanism to establish global regularity for quasilinear equations has two main components:

\setlength{\leftmargini}{1.8em}
\begin{itemize}
  \item[(1)] Propagate control of high frequencies (high order Sobolev norms);
\smallskip
  \item[(2)] Prove dispersion/decay of the solution over time.
\end{itemize}

The interplay of these two aspects has been present since the seminal work of Klainerman \cite{K0,K1} on nonlinear wave equations and vector-fields,
Shatah \cite{shatahKGE} on Klein-Gordon and normal forms, and Christodoulou-Klainerman \cite{CK} on the stability of Minkowski space, and Delort \cite{DelortKGE} on $1$d Klein-Gordon equations.
We remark that even in the weakly nonlinear regime (small perturbations of trivial solutions) smooth and localized initial
data can lead to blow-up in finite time, see John \cite{John} on quasilinear wave equations and Sideris \cite{Sideris} on compressible Euler.

In the last few years new methods have emerged in the study of global solutions of quasilinear evolutions, inspired by the advances in semilinear theory.
The basic idea is to combine the classical energy and vector-fields methods with refined analysis of the Duhamel formula, using the Fourier transform.
This is the essence of the ``method of space-time resonances'' of Germain-Masmoudi-Shatah \cite{GMS2, GMSC,GM}, see also Gustafson-Nakanishi-Tsai \cite{GNT1},
and of the refinements in \cite{IP1,IP2,GIP,IoPu1,IoPu2,IoPu3,IoPu4,DIP,Deng}, using atomic decompositions and more sophisticated norms.

The situation we consider in this paper is substantially more difficult, due to the combination of the following factors:

\begin{itemize}
\item Strictly less than $|t|^{-1}$ pointwise decay of solutions. In our case, the dispersion relation is $\Lambda(\xi)=\sqrt{g|\xi|+\sigma|\xi|^3}$ and the 
best possible pointwise decay, even for solutions of the linearized equation corresponding to Schwartz data, is $|t|^{-5/6}$ (see Fig. \ref{DispRelFigure} below). 

\item Large set of time resonances. In certain cases one can overcome the slow pointwise decay using the method of normal forms of Shatah \cite{shatahKGE}. The critical ingredient needed is the absence of time resonances (or at least a 
suitable ``null structure'' of the quadratic nonlinearity matching the set of time resonances). Our system, however, has a full (codimension 1) set of time 
resonances (see Fig. \ref{ResonantSet} below) and no meaningful null structures.
\end{itemize}

We remark that all the previous work on long term solutions of water waves models was under the assumption that either $g=0$ or $\sigma=0$. 
This is not coincidental: in these cases {\it{the combination of slow decay and full set of time resonances was not present}}. More precisely, in all the previous global 
results in 3 dimensions in \cite{GMS2, Wu3DWW, GMSC, Wa2, Wa3} it was possible to prove $1/t$ pointwise 
decay of the nonlinear solutions and combine this with high order energy estimates with slow growth.

On the other hand, in all the two-dimensional models analyzed in 
\cite{WuAG, IoPu2, ADa, ADb, HIT, IT, IoPu4, IT2, Wa1} there were no significant time 
resonances for the quadratic terms.\footnote{More precisely, 
the only time resonances are at the $0$ frequency, but they are canceled by a suitable null structure. 
Some additional ideas are needed in the case of capillary waves \cite{IoPu4} where certain singularities arise.
Moreover, new ideas, which exploit the Hamiltonian structure of the system as in \cite{IoPu2},
are needed to prove global (as opposed to almost-global) regularity.} As a result, in all of these papers it was possible to prove a {\it{quartic energy inequality}} of the form
\begin{equation*}
\big|\mathcal{E}_N(t)-\mathcal{E}_N(0)\big|\lesssim \int_0^t\mathcal{E}_N(s)\cdot\|U(s)\|^2_{W^{N/2+4,\infty}}\,ds,
\end{equation*}
for a suitable functional $\mathcal{E}_N(t)$ satisfying $\mathcal{E}_N(t)\approx \|U(t)\|_{H^N}^2$. The point is to get two factors of $\|U(s)\|_{W^{N/2+4,\infty}}$ in the right-hand side, in order to have suitable decay, and simultaneously avoid loss of derivatives.
A quartic energy inequality of this form cannot hold in our case due to the presence of large resonant sets.

To address these issues, in this paper we use a combination of improved energy estimates and Fourier analysis. The main components of our analysis are:

\begin{itemize}
\item The energy estimates, which are used to control high Sobolev norms and weighted norms (corresponding to the rotation vector-field). They rely on several new ingredients, most importantly on a {\it{strongly semilinear structure}} of the space-time integrals that control the increment of energy, and on a {\it{restricted nondegeneracy condition}} (see \eqref{IntroUpsilon}) of the time resonant hypersurfaces. The strongly semilinear structure is due to an algebraic correlation (see \eqref{corre}) between the size of the multipliers of the space-time integrals and the size of the modulation, and is related to the Hamiltonian structure of the original system. 

\item The dispersive estimates, which lead to decay and rely on a partial bootstrap argument in a suitable {\it{$Z$ norm}}. We analyze carefully the Duhamel formula, in particular the quadratic interactions related to the slowly decaying frequencies and to the {\it{set of space-time resonances}}. The choice of the $Z$ norm in this argument is very important; we use an atomic norm, based on a space-frequency decomposition of the profile of the solution, which depends in a significant way on the location and the shape of the space-time resonant set, thus on the quadratic part of the nonlinearity.  
\end{itemize}  

We hope that such ideas can be used in other quasilinear problems in two and three dimensions (such as other fluid and plasma models) that involve large resonant sets and slowly decaying solutions. We illustrate some of these main ideas in a simplified model below.

\subsection{A simplified model}\label{SimpleModel} To illustrate these ideas, consider the initial-value problem
\begin{equation}\label{ModelWW}
\begin{split}
&(\partial_t+i\Lambda)U=\nabla V\cdot\nabla U+(1/2)\Delta V\cdot U,\qquad U(0)=U_0,\\
&\Lambda(\xi):=\sqrt{|\xi|+|\xi|^3},\qquad V:=P_{[-10,10]}\Re U.
\end{split}
\end{equation} 

Compared to the full equation, this model has the same linear part and a quadratic nonlinearity leading to similar resonant sets. It is important that $V$ is real-valued, such that solutions of \eqref{ModelWW} satisfy the $L^2$ conservation law
\begin{equation}\label{qaz2}
\|U(t)\|_{L^2}=\|U_0\|_{L^2},\quad t\in[0,\infty).
\end{equation}

The model \eqref{ModelWW} carries many 
of the difficulties of the real problem and has the advantage that it is much more transparent algebraically. There are, however, significant additional 
issues when dealing with the full problem, see subsection \ref{Comparison} below for a short discussion.

The specific dispersion relation $\Lambda(\xi)=\sqrt{|\xi|+|\xi|^3}$ in \eqref{ModelWW} is important. It is radial and has stationary points 
when $|\xi|=\gamma_0:=(2/\sqrt{3}-1)^{1/2}\approx 0.393$ (see Figure \ref{DispRelFigure} below). As a result, linear solutions can only have $|t|^{-5/6}$ pointwise decay, i.e.
\begin{equation*}
\|e^{it\Lambda}\phi\|_{L^\infty}\approx |t|^{-5/6},
\end{equation*}
even for Schwartz functions $\phi$ whose Fourier transforms do not vanish on the sphere $\{|\xi|=\gamma_0\}$.

\begin{figure}[ht!]
\centering
\includegraphics[width=0.45\linewidth]{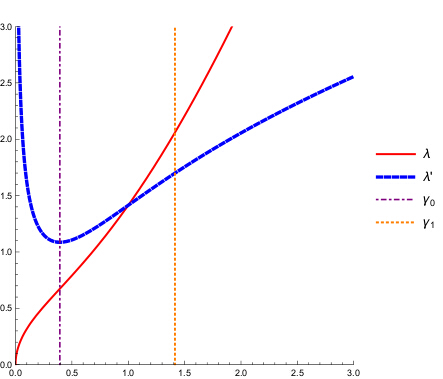}
\\
\caption{\small The curves represent the dispersion relation $\lambda(r)=\sqrt{r^3+r}$ and the group velocity $\lambda'$, for $g = 1 = \sigma$. Notice that $\lambda''(r)$ vanishes at $r=\gamma_0\approx 0.393$. The frequency $\gamma_1=\sqrt{2}$ corresponds to the sphere of space-time resonant outputs. 
Notice that while the slower decay at $\gamma_0$ is due to some degeneracy in the linear problem, $\gamma_1$ is unremarkable from the point of view of
the linear dispersion.}
\label{DispRelFigure}
\end{figure}

\subsubsection{Energy estimates}\label{SimModEnergy} We would like to control the increment of both high order Sobolev norms and weighted norms for 
solutions of \eqref{ModelWW}. It is convenient to do all the estimates in the Fourier space, using a quasilinear I-method as 
in some of our earlier work. This has similarities with the well-known I-method of Colliander--Keel--Staffilani--Takaoka--Tao 
\cite{CKSTT1,CKSTT2} used in semilinear problems, and to the energy methods of \cite{GM,ADa,ADb,HIT}. Our main estimate is the 
following partial bootstrap bound:
\begin{equation}\label{qaz0}
 \text{ if }\sup_{t\in[0,T]}\big[(1+t)^{-\delta^2}\mathcal{E}(t)^{1/2}+\|e^{it\Lambda}U(t)\|_{Z}\big]\leq \varep_1\,\,\text{ then }
\,\,\sup_{t\in[0,T]}(1+t)^{-\delta^2}\mathcal{E}(t)^{1/2}\lesssim \varep_0+\varep_1^{3/2},
\end{equation}
where $U$ is a solution on $[0,T]$ of \eqref{ModelWW}, $$\mathcal{E}(t)=\|U(t)\|_{H^N}^2+\|U(t)\|_{H^{N'}_\Omega}^2,$$ and the initial data has small 
size $\sqrt{\mathcal{E}(0)}+\|U(0)\|_{Z}\leq \varep_0$. The $Z$ norm is important and will be discussed in detail in the next subsection. 
For simplicity, we focus on the high order Sobolev norms, and divide the argument into four steps.

{\bf{Step 1.}} For $N$ sufficiently large, let
\begin{equation}\label{qaz3}
W:=W_N:=\langle\nabla\rangle^N U,\qquad E_N(t):=\int_{\mathbb{R}^2}|\widehat{W}(\xi,t)|^2\,d\xi.
\end{equation}
A simple calculation, using the equation and the fact that $V$ is real, shows that
\begin{equation}\label{qaz4}
\frac{d}{dt}E_N=\int_{\mathbb{R}^2\times\mathbb{R}^2}m(\xi,\eta)\widehat{W}(\eta)\widehat{\overline{W}}(-\xi)\widehat{V}(\xi-\eta)\,d\xi d\eta,
\end{equation}
where
\begin{equation}\label{qaz5}
m(\xi,\eta)=\frac{(\xi-\eta)\cdot (\xi+\eta)}{2}\frac{(1+|\eta|^2)^N-(1+|\xi|^2)^N}{(1+|\eta|^2)^{N/2}(1+|\xi|^2)^{N/2}}.
\end{equation}
Notice that $|\xi-\eta|\in[2^{-11},2^{11}]$ in the support of the integral, due to the Littlewood-Paley operator in the definition of $V$. We notice that $m$ satisfies
\begin{equation}\label{qaz6}
m(\xi,\eta)=\mathfrak{d}(\xi,\eta)m'(\xi,\eta),\quad\text{ where }\quad\mathfrak{d}(\xi,\eta):=\frac{[(\xi-\eta)\cdot(\xi+\eta)]^2}{1+|\xi+\eta|^2},\quad m'\approx 1.
\end{equation}
The {\it{depletion factor}} $\mathfrak{d}$ is important in establishing energy estimates, due to its correlation with the 
modulation function $\Phi$ (see \eqref{corre} below). The presence of this factor is related to the exact conservation law \eqref{qaz2}.

{\bf{Step 2.}} We would like to estimate now the increment of $E_N(t)$. We use \eqref{qaz4} and consider only the main case, when $|\xi|,|\eta|\approx 2^k\gg 1$, and $|\xi-\eta|$ is close to the slowly decaying frequency $\gamma_0$. So we need to bound space-time integrals of the form
\begin{equation*}
I:=\int_0^t\int_{\mathbb{R}^2\times\mathbb{R}^2}m(\xi,\eta)\widehat{P_kW}(\eta,s)\widehat{P_k\overline{W}}(-\xi,s)\widehat{U}(\xi-\eta,s)\chi_{\gamma_0}(\xi-\eta)\,d\xi d\eta ds,
\end{equation*} 
where $\chi_{\gamma_0}$ is a smooth cutoff function supported in the set $\{\xi:||\xi|-\gamma_0|\ll 1\}$, and we replaced $V$ by $U$ (replacing $V$ by $\overline{U}$ leads to a similar calculation). Notice that it is not possible to estimate $|I|$ by moving the absolute value inside the time integral, due to the slow decay of $U$ in $L^\infty$. So we need to integrate by parts in time; for this define the profiles
\begin{equation}\label{qaz6.5}
u(t):=e^{it\Lambda}U(t),\qquad w(t):=e^{it\Lambda}W(t).
\end{equation}
Then decompose the integral in dyadic pieces over the size of the modulation and over the size of the time variable. In terms of the profiles $u,w$, we need to consider the space-time integrals
\begin{equation}\label{qaz7}
\begin{split}
I_{k,m,p}:=\int_{\mathbb{R}}q_m(s)\int_{\mathbb{R}^2\times\mathbb{R}^2}e^{is\Phi(\xi,\eta)}&m(\xi,\eta)\widehat{P_kw}(\eta,s)\widehat{P_k\overline{w}}(-\xi,s)\\
&\times\widehat{u}(\xi-\eta,s)\chi_{\gamma_0}(\xi-\eta)\varphi_p(\Phi(\xi,\eta))\,d\xi d\eta ds,
\end{split}
\end{equation} 
where $\Phi(\xi,\eta):=\Lambda(\xi)-\Lambda(\eta)-\Lambda(\xi-\eta)$ is the associated modulation, $q_m$ is smooth and supported in 
the set $s\approx 2^m$ and $\varphi_p$ is supported in the set $\{x:|x|\approx 2^p\}$.

\begin{figure}[h!]
\includegraphics[width=0.36\linewidth]{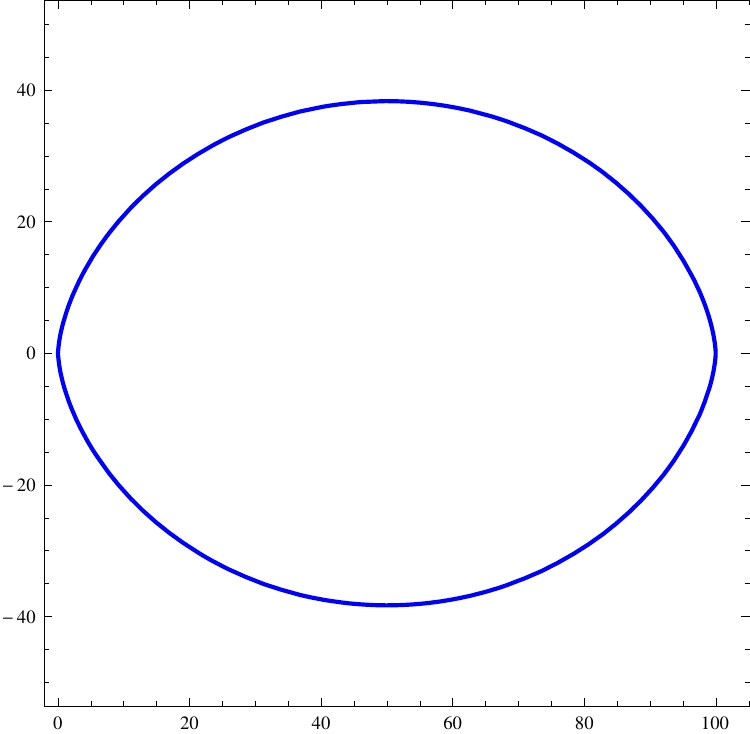}
\hskip20pt \includegraphics[width=0.403\linewidth]{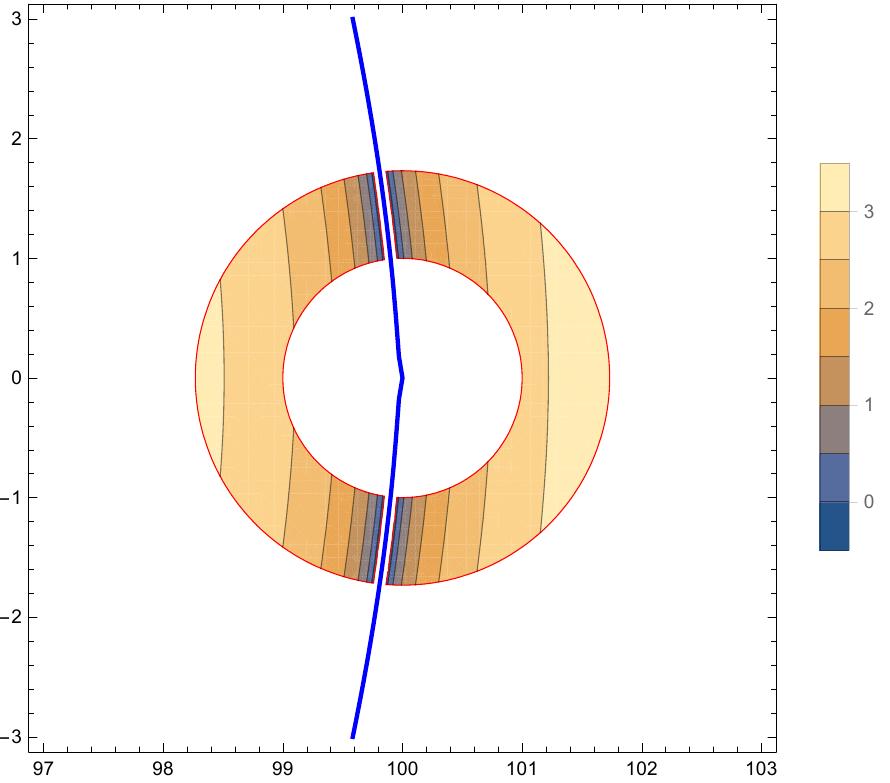} 
\\
\caption{\small The first picture illustrates the resonant set $\{\eta:0=\Phi(\xi,\eta)=\Lambda(\xi)-\Lambda(\eta)-\Lambda(\xi-\eta)\}$ for a fixed large frequency $\xi$ (in the picture $\xi=(100,0)$). 
The second picture illustrates the intersection of a neighborhood of this resonant set with the set where $|\xi-\eta|$ is close to $\gamma_0$. 
Note in particular that near the resonant set $\xi-\eta$ is almost perpendicular to $\xi$ (see \eqref{qaz6}, \eqref{corre}). Finally, the colors show the 
level sets of $\log |\Phi|$.} 
\label{ResonantSet}
\end{figure}

{\bf{Step 3.}} To estimate the integrals $I_{k,m,p}$ we consider several cases depending on the relative size of $k,m,p$. Assume that $k,m$ are large, i.e. $2^k\gg 1, 2^m\gg 1$, which is the harder case. To deal with the case of small modulation, when one cannot integrate by parts in time, we need an $L^2$ bound on the Fourier integral operator
\begin{equation*}
T_{k,m,p}(f)(\xi) := \int_{\R^2} e^{is\Phi(\xi,\eta)} \varphi_k(\xi) \varphi_{\leq p}(\Phi(\xi,\eta))\chi_{\gamma_0}(\xi-\eta)
  f(\eta) \, d\eta,
\end{equation*}
where $s\approx 2^m$ is fixed. The critical bound we prove in Lemma \ref{L2Prop0} (``the main $L^2$ lemma'') is
\begin{equation}\label{qaz9}
\|T_{k,m,p}(f)\|_{L^2}\lesssim_\eps 2^{\eps m}(2^{(3/2)(p-k/2)}+2^{p-k/2-m/3})\|f\|_{L^2},\qquad \eps>0,
\end{equation}
provided that $p-k/2\in[-0.99 m,-0.01m]$. The main gain here is the factor $3/2$ in $2^{(3/2)(p-k/2)}$ in the right-hand side (Schur's test would only give a factor of $1$). 

The proof of \eqref{qaz9} uses a $TT^\ast$ argument, which is a standard tool to prove $L^2$ bounds for Fourier integral operators. This argument depends on a key nondegeneracy property of the function $\Phi$, more precisely on what we call the {\it{restricted nondegeneracy condition}}  
\begin{align}\label{IntroUpsilon}
 \Upsilon(\xi,\eta) = \nabla_{\xi,\eta}^2\Phi(\xi,\eta)[\nabla_\xi^\perp\Phi(\xi,\eta), \nabla^\perp_\eta \Phi(\xi,\eta)]\neq 0\qquad\text{ if }\qquad\Phi(\xi,\eta)=0.
\end{align}
This condition, which appears to be new, 
 can be verified explicitly in our case, when 
$||\xi-\eta|-\gamma_0|\ll 1$. The function $\Upsilon$ does in fact vanish at two points on the resonant set $\{\eta:\Phi(\xi,\eta)=0\}$ 
(where $||\xi-\eta|-\gamma_0|\approx 2^{-k}$), but our argument can tolerate vanishing up to order $1$.

The nondegeneracy condition \eqref{IntroUpsilon} can be interpreted geometrically: the nondegeneracy of the mixed Hessian of $\Phi$ is a standard condition 
that leads to optimal $L^2$ bounds on Fourier integral operators. In our case, however, we have the additional cutoff function $\varphi_{\leq p}(\Phi(\xi,\eta))$, so we
can only integrate by parts in the directions tangent to the level sets of $\Phi$. This explains the additional restriction to these directions in the 
definition of $\Upsilon$ in \eqref{IntroUpsilon}.

Given the bound \eqref{qaz9}, we can control the contribution of small modulations, i.e.
\begin{equation}\label{qaz10}
 p-k/2\leq -2m/3-\epsilon m.
\end{equation}

{\bf{Step 4.}} In the high modulation case we integrate by parts in time in the formula \eqref{qaz7}. The main contribution is when the time derivative 
hits the high frequency terms, so we focus on estimating the resulting integral
\begin{equation}\label{qaz12}
\begin{split}
I'_{k,m,p}:=\int_{\mathbb{R}}q_m(s)\int_{\mathbb{R}^2\times\mathbb{R}^2}e^{is\Phi(\xi,\eta)}&m(\xi,\eta)\frac{d}{ds}
\big[\widehat{P_kw}(\eta,s)\widehat{P_k\overline{w}}(-\xi,s)\big]\\
&\times\widehat{u}(\xi-\eta,s)\chi_{\gamma_0}(\xi-\eta)\frac{\varphi_p(\Phi(\xi,\eta))}{\Phi(\xi,\eta)}\,d\xi d\eta ds.
\end{split}
\end{equation} 

Notice that $\partial_tw$ satisfies the equation
\begin{align}\label{qaz13}
\partial_t w =\langle\nabla\rangle^Ne^{it\Lambda}\big[\nabla V\cdot\nabla U+(1/2)\Delta V\cdot U\big].
\end{align}
The right-hand side of \eqref{qaz13} is quadratic. We thus see that replacing $w$ by $\partial_t w$ essentially gains a unit of decay 
(which is $|t|^{-5/6+}$), but loses a derivative. This causes a problem in some range of parameters, for example when $2^{p}\approx 2^{k/2-2m/3}$, 
$1\ll 2^k\ll 2^m$, compare with \eqref{qaz10}.

We then consider two cases: if the modulation is sufficiently small then 
we can use the depletion factor $\mathfrak{d}$ in the multiplier $m$, see \eqref{qaz6}, and the following key algebraic correlation
\begin{equation}\label{corre}
\text{ if }\qquad |\Phi(\xi,\eta)|\lesssim 1\qquad\text{ then }\qquad |m(\xi,\eta)| \lesssim 2^{-k}.
\end{equation}
See Fig. \ref{ResonantSet}. As a result, we gain one derivative in the integral $I'_{k,m,p}$, which compensates for the loss of one derivative in \eqref{qaz13}, and the integral can be estimated again using \eqref{qaz9}.

On the other hand, if the modulation is not small, $2^p\geq 1$, then the denominator $\Phi(\xi,\eta)$ becomes a favorable factor, and one can use the formula \eqref{qaz13} and reiterate the symmetrization procedure implicit in the energy estimates. This symmetrization avoids the loss of one derivative and gives suitable estimates on $|I'_{k,m,p}|$ in this case. The proof of \eqref{qaz0} follows. 

\subsubsection{Dispersive analysis}\label{SimModDisp} It remains to prove a partial bootstrap estimate for the $Z$ norm, i.e.
\begin{equation}\label{waz0}
 \text{ if }\sup_{t\in[0,T]}\big[(1+t)^{-\delta^2}\mathcal{E}(t)^{1/2}+\|e^{it\Lambda}U(t)\|_{Z}\big]\leq \varep_1\quad\text{ then }
\quad\sup_{t\in[0,T]}\|e^{it\Lambda}U(t)\|_{Z}\lesssim \varep_0+\varep_1^2.
\end{equation}
This complements the energy bootstrap estimate \eqref{qaz0}, and closes the full bootstrap argument. 

The first main issue is to define an effective $Z$ norm. We use the Duhamel formula, written in terms of the profile $u$ (recall the equation \eqref{ModelWW})
\begin{equation}\label{waz1}
\widehat{u}(\xi,t)=\widehat{u}(\xi,0)-\frac{1}{2}\int_0^t\int_{\mathbb{R}^2}e^{is\Lambda(\xi)}(|\xi|^2-|\eta|^2)\widehat{V}(\xi-\eta,s)e^{-is\Lambda(\eta)}\widehat{u}(\eta,s)\,d\eta ds.
\end{equation}
For simplicity, consider one of the terms, corresponding to the component $U$ of $V$ (the contribution of $\overline{U}$ is similar). So we are looking to understand bilinear expressions of the form
\begin{equation}\label{waz2}
\begin{split}
&\widehat{h}(\xi,t):=\int_0^t\int_{\mathbb{R}^2}e^{is\Phi(\xi,\eta)}n(\xi,\eta)\widehat{u}(\xi-\eta,s)\widehat{u}(\eta,s)\,d\eta ds,\\
&n(\xi,\eta):=(|\xi|^2-|\eta|^2)\varphi_{[-10,10]}(\xi-\eta),\qquad \Phi(\xi,\eta)=\Lambda(\xi)-\Lambda(\eta)-\Lambda(\xi-\eta).
\end{split}
\end{equation}
The idea is to estimate the function $\widehat{h}$ by integrating by parts either in $s$ or in $\eta$. 
This is the method of space-time resonances of Germain--Masmoudi--Shatah \cite{GMS2}. The main contribution is expected to come from the {\it{set of space-time resonances}} (the stationary points of the integral)
\begin{equation}\label{waz3}
\mathcal{SR}:=\{(\xi,\eta):\,\Phi(\xi,\eta)=0,\,(\nabla_\eta\Phi)(\xi,\eta)=0\}.
\end{equation}

To illustrate how this analysis works in our problem, we consider the contribution of the integral over $s\approx 2^m\gg 1$ in \eqref{waz2}, and assume that 
the frequencies are $\approx 1$. 

{\bf{Case 1.}} Start with the contribution of small modulations,
\begin{equation}\label{waz4}
\widehat{h_{m,l}}(\xi):=\int_{\mathbb{R}}q_m(s)\int_{\mathbb{R}^2}\varphi_{\leq l}(\Phi(\xi,\eta))e^{is\Phi(\xi,\eta)}n(\xi,\eta)\widehat{u}(\xi-\eta,s)\widehat{u}(\eta,s)\,d\eta ds,
\end{equation}
where $l=-m+\delta m$ ($\delta$ is a small constant) and $q_m(s)$ restricts the time integral to $s\approx 2^m$. Assume that $u(.,s)$ is a Schwartz function supported at frequency $\approx 1$, independent of $s$ (this is the situation at the first iteration). Integration by parts in $\eta$ (using the formula \eqref{loca2} to avoid taking $\eta$ derivatives of the factor $\varphi_{\leq l}(\Phi(\xi,\eta))$) shows that the main contribution comes from a small neighborhood of the stationary points where $|\nabla_{\eta}\Phi(\xi,\eta)|\leq 2^{-m/2+\delta m}$, up to negligible errors. Thus, the main contribution comes from space-time resonant points as in \eqref{waz3}. 

In our case, the space-time resonant set is
\begin{equation}\label{waz5}
\{(\xi,\eta)\in\mathbb{R}^2\times\mathbb{R}^2:|\xi|=\gamma_1=\sqrt{2},\,\eta=\xi/2\}.
\end{equation}
Moreover, the space-time resonant points are {\it{nondegenerate}} (according to the terminology introduced in \cite{IP2}), in the sense that the Hessian of the matrix $\nabla_{\eta\eta}^2\Phi(\xi,\eta)$ is non-singular at these points. A simple calculation shows that
\begin{equation*}
\widehat{h_{m,l}}(\xi)\approx c(\xi)\varphi_{\leq -m}(|\xi|-\gamma_1),
\end{equation*}
up to smaller contributions, where we have also ignored factors of $2^{\delta m}$, and $c$ is smooth. 

We are now ready to describe more precisely the $Z$ space. This space should include all Schwartz functions. It also has to include functions like $\widehat{u}(\xi)=\varphi_{\leq -m}(|\xi|-\gamma_1)$, due to the calculation above, for any $m$ large. It should measure localization in both space and frequency, and be strong enough, at least, to recover the $t^{-5/6+}$ pointwise decay. 

We use the 
framework introduced by two of the authors in \cite{IP1}, which was later refined by some of the authors in \cite{IP2,GIP,DIP}. 
The idea is to decompose the profile as a superposition of atoms, using localization in both space and frequency,
\begin{equation*}
f={\sum}_{j,k}Q_{jk}f,\qquad Q_{jk}f=\varphi_j(x)\cdot P_kf(x).
\end{equation*} 
The $Z$ norm is then defined by measuring suitably every atom. We define first 
\begin{equation}\label{waz6}
\|f\|_{Z_1}=\sup_{j,k} 2^j\cdot\|||\xi|-\gamma_1|^{1/2}\widehat{Q_{jk}f}(\xi)\|_{L^2_\xi}
\end{equation} 
up to small corrections (see Definition \ref{MainZDef} for the precise formula, including the small but important $\delta$-corrections), and then we define the $Z$ norm by applying a suitable number of vector-fields $D$ and $\Omega$.

These considerations and \eqref{waz1} can also be used to justify the approximate formula
\begin{equation}\label{waz7}
(\partial_t\widehat{u})(\xi,t)\approx (1/t){\sum}_j \,\, g_j(\xi)e^{it\Phi(\xi,\eta_j(\xi))}+\,\text{lower order terms},
\end{equation}
as $t\to\infty$, where $\eta_j(\xi)$ denote the stationary points where $\nabla_\eta\Phi(\xi,\eta_j(\xi))=0$. This approximate formula, which holds 
at least as long as the stationary points are nondegenerate, is consistent with the asymptotic behavior of the solution described in Remark \ref{MainRemark} (ii). 
Indeed, at space-time resonances $\Phi(\xi,\eta_j(\xi))=0$, which leads to logarithmic growth for $\widehat{u}(\xi,t)$, while away from 
these space-time resonances the oscillation of $e^{it\Phi(\xi,\eta_j(\xi))}$ leads to convergence.

{\bf{Case 2.}} Consider now the case of higher modulations, $l\geq -m+\delta m$. We start from a formula similar to \eqref{waz4} and integrate by parts in $s$. The main case is when $d/ds$ hits one of the profiles $u$. Using again the equation (see \eqref{waz1}), we have to estimate cubic expressions of the form 
\begin{equation}\label{waz8}
\begin{split}
\widehat{h'_{m,l}}(\xi):=\int_{\mathbb{R}}q_m(s)&\int_{\mathbb{R}^2\times\mathbb{R}^2}\frac{\varphi_{l}(\Phi(\xi,\eta))}{\Phi(\xi,\eta)}e^{is\Phi(\xi,\eta)}n(\xi,\eta)\widehat{u}(\xi-\eta,s)\\
&\times e^{is\Phi'(\eta,\sigma)}n(\eta,\sigma)\widehat{\overline{u}}(\eta-\sigma,s)\widehat{u}(\sigma,s)\,d\eta d\sigma ds,
\end{split}
\end{equation}
where $\Phi'(\eta,\sigma)=\Lambda(\eta)+\Lambda(\eta-\sigma)-\Lambda(\sigma)$. Assume again that the three functions $u$ are Schwartz functions supported at frequency $\approx 1$. We combine $\Phi$ and $\Phi'$ into a combined phase,
\begin{equation*}
\widetilde{\Phi}(\xi,\eta,\sigma):=\Phi(\xi,\eta)+\Phi'(\eta,\sigma)=\Lambda(\xi)-\Lambda(\xi-\eta)+\Lambda(\eta-\sigma)-\Lambda(\sigma).
\end{equation*}
We need to estimate $h'_{m,l}$ according to the $Z_1$ norm. Integration by parts in $\xi$ (approximate finite speed of propagation) shows that the main contribution in $Q_{jk}h'_{m,l}$ is when $2^j\lesssim 2^m$. 

We have two main cases: if $l$ is not too small, say $l\geq -m/14$, then we use first multilinear H\"{o}lder-type estimates, placing two of the factors $e^{is\Lambda}u$ in $L^\infty$ and one in $L^2$, together with analysis of the stationary points of $\widetilde{\Phi}$ in $\eta$ and $\sigma$. This suffices is most cases, except when all the variables are close to $\gamma_0$. In this case we need a key algebraic property, of the form
\begin{equation}\label{waz9}
\text{ if }\quad\nabla_{\eta,\sigma}\widetilde{\Phi}(\xi,\eta,\sigma)=0\quad\text{ and }\quad\widetilde{\Phi}(\xi,\eta,\sigma)=0\quad\text{ then }\quad\nabla_{\xi}\widetilde{\Phi}(\xi,\eta,\sigma)=0,
\end{equation}
if $|\xi-\eta|,|\eta-\sigma|,|\sigma|$ are all close to $\gamma_0$.

On the other and, if $l$ is very small, $l\leq -m/14$, then the denominator $\Phi(\xi,\eta)$ in \eqref{waz8} is dangerous. However, we can restrict to small neighborhoods of the stationary points of $\widetilde{\Phi}$ in $\eta$ and $\sigma$, thus to space-time resonances. This is the most difficult case in the dispersive analysis. We need to rely on one more algebraic property, of the form
\begin{equation}\label{waz9.1}
\text{ if }\quad\nabla_{\eta,\sigma}\widetilde{\Phi}(\xi,\eta,\sigma)=0\quad\text{ and }\quad|\Phi(\xi,\eta)|+|\Phi'(\eta,\sigma)|\ll 1\quad\text{ then }\quad\nabla_{\xi}\widetilde{\Phi}(\xi,\eta,\sigma)=0.
\end{equation} 
See Lemma \ref{cubicphase} for the precise quantitative claims for both \eqref{waz9} and \eqref{waz9.1}.
 
The point of both \eqref{waz9} and \eqref{waz9.1} is that in the resonant region for the cubic integral we have that $\nabla_{\xi}\widetilde{\Phi}(\xi,\eta,\sigma)=0$. We call them {\it{slow propagation of iterated resonances}} properties; as a consequence the resulting function is essentially supported when $|x|\ll 2^m$, using the approximate finite speed of propagation. This gain is reflected in the factor $2^j$ in \eqref{waz6}.

We remark that the analogous property for quadratic resonances
\begin{equation*}
\text{ if }\quad\nabla_{\eta}\Phi(\xi,\eta)=0\quad\text{ and }\quad\Phi(\xi,\eta)=0\quad\text{ then }\quad\nabla_{\xi}\Phi(\xi,\eta)=0
\end{equation*} 
fails. In fact, in our case $|\nabla_{\xi}{\Phi}(\xi,\eta)|\approx 1$ on the space-time resonant set.

In proving \eqref{waz0}, there are, of course, many cases to consider. The full proof covers sections \ref{partialt} and \ref{Sec:Z1Norm}. 
The type of arguments presented above are typical in the proof: we decompose our profiles in space and frequency, localize to small sets in the frequency space, 
keeping track in particular of the special frequencies of size $\gamma_0,\gamma_1,\gamma_1/2,2\gamma_0$, use integration by parts in $\xi$ to control the location of the output, and use multilinear H\"{o}lder-type estimates to bound $L^2$ norms. We remark that the dispersive analysis in the $Z$ norm is much more involved in this paper than in the 
earlier papers mentioned above.

\subsubsection{The special quadratic structure of the full water-wave system}\label{Comparison} The model \eqref{ModelWW} is useful in understanding the full problem. There are, however, additional difficulties to keep in mind. 

 In this paper we use {\it{Eulerian coordinates}}. The local well-posedness theory, which is nontrivial because of 
the quasilinear nature of the equations and the hidden hyperbolic structure, then relies on the so-called ``good unknown'' of 
Alinhac \cite{Alinhac,AlMet1,ABZ1,ADb}.

In our problem, however, this is not enough. Alinhac's good unknown $\omega$ is suitable for the local theory, in the sense that it prevents loss of 
derivatives in energy estimates. However, for the global theory, we need to adjust the main complex variable $U$ which diagonalizes the system, 
using a quadratic correction of the form $T_{m'}\omega$ (see \eqref{defscalarunk}). This way we can identify certain {\it{special quadratic structure}}, 
somewhat similar to the structure in the nonlinearity of \eqref{ModelWW}. This structure, 
which appears to be new, is ultimately responsible for the favorable multipliers of the space-time integrals (similar to \eqref{qaz6}), and leads to 
global energy bounds.

Identifying this structure is, unfortunately, technically involved. Our main result is in Proposition \ref{proeqU}, but its proof depends on 
paradifferential calculus using the Weyl quantization (see Appendix \ref{SecParaOp}) and on a suitable paralinearization of the Dirichlet--Neumann operator. We include all the details of this paralinearization in Appendix \ref{DNOpe}, mostly because its exact form has to be properly adapted to our norms and suitable for global analysis. For this we need some auxiliary spaces: (1) the $\mathcal{O}_{m,p}$ hierarchy, which measures functions, keeping track of both multiplicity (the index $m$) and smoothness (the index $p$), and (2) the $\mathcal{M}^{l,m}_r$ hierarchy, which measures the symbols of the paradifferential operators, keeping track also of the order $l$. 

\subsubsection{Additional remarks} We list below some other issues one needs to keep in mind in the proof of the main theorem.

\setlength{\leftmargini}{1.8em}
\begin{itemize}
\item[(1)] Another significant difficulty of the full water wave system, which is not present in \eqref{ModelWW}, is that the ``linear'' part of the equation 
is given by a more complicated paradifferential operator $T_{\Sigma}$, not by the simple operator $\Lambda$. The operator $T_\Sigma$ 
includes nonlinear cubic terms that lose $3/2$ derivatives, and an additional smoothing effect is needed. 
\smallskip
\item[(2)] The very low frequencies $|\xi|\ll 1$ play an important role in all the global results for water wave systems. These frequencies are not captured 
in the model \eqref{ModelWW}. In our case, there is a suitable {\it{null structure}}: the multipliers of the quadratic terms are bounded by $|\xi|\min(|\eta|,|\xi-\eta|)^{1/2}$, 
see \eqref{Assumptions2}, which is an important ingredient in the dispersive part of the argument.  
\smallskip
\item[(3)] It is important to propagate energy control of both high Sobolev norms and weighted norms using many copies of the rotation vector-field. 
Because of this control, we can pretend that all the profiles in the dispersive part of the argument are almost radial and located at 
frequencies $\lesssim 1$. The linear estimates (in Lemma \ref{LinEstLem}) and many of the bilinear estimates are much stronger because of 
this almost radiality property of the profiles.     
\smallskip
\item[(4)] At many stages it is important that the four spheres, the sphere of slow decay $\{|\xi|=\gamma_0\}$, the sphere 
of space-time resonant outputs $\{|\xi|=\gamma_1\}$, and the sphere of space-time resonant inputs $\{|\xi|=\gamma_1/2\}$, and 
the sphere $\{|\xi|=2\gamma_0\}$ are all separated from each other. Such separation conditions played an important role also in other papers, such as \cite{GM,GIP,DIP}. 
\end{itemize}

\subsection{Organization}\label{Orga} The rest of the paper is organized as follows: in section \ref{MPO} we state the main propositions and summarize the main definitions and notation in the paper.

In sections \ref{Eqs}--\ref{L2proof} we prove Proposition \ref{MainBootstrapEn}, which is the main improved energy estimate. The key components of the proof are Proposition \ref{proeqU} (derivation of the main quasilinear scalar equation, identifying the special quadratic structure), Proposition \ref{ProBulk} (the first energy estimate, including the strongly semilinear structure), Proposition \ref{EEMainProp} (reduction to a space-time integral bound), Lemma \ref{L2Prop0} (the main $L^2$ bound on a localized Fourier integral operator), and Lemma \ref{EELemmaMain} (the main interactions in Proposition \ref{EEMainProp}). The proof of Proposition \ref{MainBootstrapEn} uses also the material presented in the appendices, in particular the paralinearization of the Dirichlet--Neumann operator in Proposition \ref{DNmainpro}.

In sections \ref{NotationsF}--\ref{Sec:Z1Norm} we prove Proposition \ref{MainBootstrapDisp}, which is the main improved dispersive estimate. The key components of the proof are the reduction to Proposition \ref{bootstrap}, the precise analysis of the time derivative of the profile in Lemmas \ref{dtfLem1}--\ref{dtfLem2}, and the analysis of the Duhamel formula, divided in several cases, in Lemmas \ref{FSPLem}--\ref{lsmallBound}.

In sections \ref{phacolle}-\ref{upsilon} we collect estimates on the dispersion relation and the phase functions. The main results are Proposition \ref{spaceres11} (structure of the resonance sets), Proposition \ref{volume} (bounds on sublevel sets), Lemma \ref{cubicphase} (slow propagation of iterated resonances), and Lemmas \ref{Geomgamma0}--\ref{lemmaD2} (the restricted nondegeneracy property of the resonant hypersurfaces).  

\subsection{Acknowledgements}\label{Ack} We would like to thank Thomas Alazard for very useful discussions and for sharing an unpublished note on 
paralinearization, and Javier G\'{o}mez-Serrano for discussions on numerical simulations. The third author would like to thank Vladimir Georgescu for inspiring discussions on the Weyl quantization. The last author would also like to thank Jalal Shatah 
for generously sharing his expertise on water waves on many occasions.

\section{The main propositions}\label{MPO}

Recall the water-wave system with gravity and surface tension,
\begin{equation}\label{WW0}
\left\{
\begin{array}{l}
\partial_t h = G(h) \phi,\\
\partial_t \phi = -g h  + \sigma \div \Big[ \dfrac{\nabla h}{ (1+|\nabla h|^2)^{1/2} } \Big]
  - \dfrac{1}{2} {|\nabla \phi|}^2 + \dfrac{{\left( G(h)\phi + \nabla h \cdot \nabla \phi \right)}^2}{2(1+{|\nabla h|}^2)},
\end{array}
\right.
\end{equation}
where $G(h)\phi$ denotes the Dirichlet-Neumann operator associated to the water domain. Theorem \ref{MainTheo} is a consequence of 
Propositions \ref{LocalEx}, \ref{MainBootstrapEn}, and \ref{MainBootstrapDisp} below.

\begin{proposition}\label{LocalEx} (Local existence and continuity)
(i) Assume that $N\geq 10$. There is $\overline{\varep}>0$ such that if
\begin{equation}
\label{Ama20}
\|h_0\|_{H^{N+1}}+\|\phi_0\|_{\dot{H}^{N+1/2,1/2}}\leq\overline{\varep}
\end{equation}
then there is a unique solution $(h,\phi)\in C([0,1]:H^{N+1}\times \dot{H}^{N+1/2,1/2})$ of the system \eqref{WW0} with $g=1$ and $\sigma=1$, 
with initial data $(h_0,\phi_0)$.

(ii) Assume $T_0\geq 1$, $N =N_1 + N_3$, and $(h,\phi)\in C([0,T_0]:H^{N+1}\times \dot{H}^{N+1/2,1/2})$ is a solution of the system \eqref{WW0} 
with $g=1$ and $\sigma=1$. With the $Z$ norm as in Definition \ref{MainZDef} below and the profile $\mathcal{V}$ defined as 
in \eqref{Ama30}, assume that for some $t_0\in[0,T_0]$
\begin{equation}\label{Ama21}
\mathcal{V}(t_0)\in H^{N_0}\cap H^{N_1,N_3}_{\Omega}\cap Z, \qquad {\| \mathcal{V}(t_0) \|}_{H^N} \leq 2\bar{\e}.
\end{equation}
Then there is $\tau=\tau(\|\mathcal{V}(t_0)\|_{H^{N_0}\cap H^{N_1,N_3}\cap Z})$ such that the mapping 
$t \to {\|\mathcal{V}(t) \|}_{H^{N_0}\cap H^{N_1,N_3}_{\Omega}\cap Z} $ is continuous on $[0,T_0]\cap [t_0,t_0+\tau]$, and
\begin{equation}\label{Ama22}
\sup_{t\in[0,T_0]\cap [t_0,t_0+\tau]}\|\mathcal{V}(t)\|_{H^{N_0}\cap H^{N_1,N_3}_{\Omega}\cap Z}\leq 2\|\mathcal{V}(t_0)\|_{H^{N_0}\cap H^{N_1,N_3}_{\Omega}\cap Z}.
\end{equation}
\end{proposition}

Proposition \ref{LocalEx} is a local existence result for the water waves system. 
We will not provide the details of its proof in the paper, but only briefly discuss it. Part (i) is a standard wellposedness statement 
in a sufficiently regular Sobolev space, see for example \cite{Wu1,ABZ1}.

Part (ii) is a continuity statement for the Sobolev norm $H^{N_0}$ as well as for the 
$H_{\Omega}^{N_1,N_3}$ and $Z$ norms\footnote{Notice 
that we may assume uniform in time smallness of the high Sobolev norm $H^N$ with $N=N_1 + N_3$, 
thanks to the uniform control on the $Z$ norm, see Proposition \ref{MainBootstrapEn}, and Definition \ref{MainZDef}.}. Continuity for the $H^{N_0}$ 
norm is standard. A formal proof of continuity for the $H_{\Omega}^{N_1,N_3}$ and $Z$ norms and of \eqref{Ama22} requires some adjustments of the 
arguments given in the paper, due to the quasilinear and non-local nature of the equations.

More precisely, we can define $\eps$-truncations of the rotational vector-field $\Omega$, i.e. $\Omega_\eps := (1+\eps^2 |x|^2)^{-1/2} \Omega$, and the associated 
spaces $H_{\Omega_\eps}^{N_1,N_3}$, with the obvious adaptation of the norm in \eqref{normH}. Then we notice that
\begin{equation*}
\Omega_\eps T_a b = T_{\Omega_\e a} b + T_{a} \Omega_\e b + R
\end{equation*}
where $R$ is a suitable remainder bounded uniformly in $\e$. Because of this we can adapt the arguments in Proposition \ref{ProBulk} and in 
appendices \ref{SecParaOp} and \ref{DNOpe} to prove energy estimates in the $\eps$-truncated spaces $H_{\Omega_\eps}^{N_1,N_3}$. 
For the $Z$ norm one can proceed similarly using an $\eps$-truncated version $Z_\eps$ (see the proof of Proposition 2.4 in \cite{IP2} for a similar argument) 
and the formal expansion of the Dirichlet--Neumann operator in section \ref{ProofYu}. 
The conclusion follows from the uniform estimates by letting $\eps \rightarrow 0$.

The following two propositions summarize our main bootstrap argument.

\begin{proposition}\label{MainBootstrapEn}
(Improved energy control) Assume that $T\geq 1$ and $(h,\phi)\in C([0,T]:H^{N_0+1}\times \dot{H}^{N_0+1/2,1/2})$
is a solution of the system \eqref{WW0} with $g=1$ and $\sigma=1$, with initial data $(h_0,\phi_0)$. Assume that, with $\mathcal{U}$ and $\mathcal{V}$ defined as in \eqref{Ama30},
\begin{equation}\label{Ama31}
\|\mathcal{U}(0)\|_{H^{N_0}\cap H^{N_1,N_3}_{\Omega}}+\|\mathcal{V}(0)\|_{Z}\leq \varep_0\ll 1,
\end{equation}
and, for any $t\in[0,T]$,
\begin{equation}\label{Ama32}
(1+t)^{-\delta^2}\|\mathcal{U}(t)\|_{H^{N_0}\cap H^{N_1,N_3}_{\Omega}}+\|\mathcal{V}(t)\|_{Z}\leq \varep_1\ll 1,
\end{equation}
where the $Z$ norm is as in Definition \ref{MainZDef}. Then, for any $t\in[0,T]$,
\begin{equation}\label{Ama33}
(1+t)^{-\delta^2}\|\mathcal{U}(t)\|_{H^{N_0}\cap H^{N_1,N_3}_{\Omega}}\lesssim \varep_0+\varep_1^{3/2}.
\end{equation}
\end{proposition}

\begin{proposition}\label{MainBootstrapDisp}
(Improved dispersive control) With the same assumptions as in Proposition \ref{MainBootstrapEn} above, in particular \eqref{Ama31}--\eqref{Ama32}, we have, for any $t\in[0,T]$,
\begin{equation}\label{Ama34}
\|\mathcal{V}(t)\|_Z\lesssim \varep_0+\varep_1^2.
\end{equation}
\end{proposition}

It is easy to see that Theorem \ref{MainTheo} follows from Propositions \ref{LocalEx}, \ref{MainBootstrapEn}, and \ref{MainBootstrapDisp} by a standard continuity argument and \eqref{LinftyBd3.5} (for the $L^\infty$ bound on $\mathcal{U}$ in \eqref{mainconcl1}). The rest of the paper is concerned with the proofs of Propositions \ref{MainBootstrapEn} and \ref{MainBootstrapDisp}.

\subsection{Definitions and notation}\label{secdefs}

We summarize in this subsection some of the main definitions we use in the paper.

\subsubsection{General notation}\label{defznorm} We start by defining several multipliers that allow us to localize in the Fourier space. We fix $\varphi:\mathbb{R}\to[0,1]$ an even smooth 
function supported in $[-8/5,8/5]$ and equal to $1$ in $[-5/4,5/4]$. For simplicity of notation, we also 
let $\varphi:\mathbb{R}^2\to[0,1]$ denote the corresponding radial function on $\mathbb{R}^2$. Let
\begin{equation*}
\begin{split}
&\varphi_k(x):=\varphi(|x|/2^k)-\varphi(|x|/2^{k-1})\,\,\text{ for any }\,\,k\in\mathbb{Z},\qquad\varphi_I:=\sum_{m\in I\cap\mathbb{Z}}\varphi_m\text{ for any }I\subseteq\mathbb{R},\\
&\varphi_{\leq B}:=\varphi_{(-\infty,B]},\quad\varphi_{\geq B}:=\varphi_{[B,\infty)},\quad\varphi_{<B}:=\varphi_{(-\infty,B)},\quad \varphi_{>B}:=\varphi_{(B,\infty)}.
\end{split}
\end{equation*}
For any $a<b\in\mathbb{Z}$ and $j\in[a,b]\cap\mathbb{Z}$ let
\begin{equation}\label{Alx80}
\varphi^{[a,b]}_j:=
\begin{cases}
\varphi_{j}\quad&\text{ if }a<j<b,\\
\varphi_{\leq a}\quad&\text{ if }j=a,\\
\varphi_{\geq b}\quad&\text{ if }j=b.
\end{cases}
\end{equation}

For any $x\in\mathbb{R}$ let $x^{+}=\max(x,0)$ and $x^-:=\min(x,0)$. Let
\begin{equation*}
\mathcal{J}:=\{(k,j)\in\mathbb{Z}\times\mathbb{Z}_+:\,k+j\geq 0\}.
\end{equation*}
For any $(k,j)\in\mathcal{J}$ let
\begin{equation*}
\phii^{(k)}_j(x):=
\begin{cases}
\varphi_{\leq -k}(x)\quad&\text{ if }k+j=0\text{ and }k\leq 0,\\
\varphi_{\leq 0}(x)\quad&\text{ if }j=0\text{ and }k\geq 0,\\
\varphi_j(x)\quad&\text{ if }k+j\geq 1\text{ and }j\geq 1,
\end{cases}
\end{equation*}
and notice that, for any $k\in\mathbb{Z}$ fixed, $\sum_{j\geq-\min(k,0)}\phii^{(k)}_j=1$.

Let $P_k$, $k\in\mathbb{Z}$, denote the Littlewood--Paley projection operators defined by the Fourier multipliers $\xi\to \varphi_k(\xi)$. 
Let $P_{\leq B}$ (respectively $P_{>B}$) denote the operators defined by the Fourier 
multipliers $\xi\to \varphi_{\leq B}(\xi)$ (respectively $\xi\to \varphi_{>B}(\xi)$). For $(k,j)\in\mathcal{J}$ let $Q_{jk}$ denote the operator
\begin{equation}\label{qjk}
(Q_{jk}f)(x):=\phii^{(k)}_j(x)\cdot P_kf(x).
\end{equation}
In view of the uncertainty principle the operators $Q_{jk}$ are relevant only when $2^j2^k\gtrsim 1$, which explains the definitions above. 

We will use two sufficiently large constants $\mathcal{D}\gg\mathcal{D}_1\gg 1$ ($\mathcal{D}_1$ is only used in sections \ref{phacolle}--\ref{upsilon} to prove properties of the phase functions). For $k,k_1,k_2\in\mathbb{Z}$ let
\begin{equation}\label{Rset}
\mathcal{D}_{k,k_1,k_2}:=\{(\xi,\eta)\in(\mathbb{R}^2)^2:\,|\xi|\in[2^{k-4},2^{k+4}],\,|\eta|\in[2^{k_2-4},2^{k_2+4}],\,|\xi-\eta|\in[2^{k_1-4},2^{k_1+4}]\}.
\end{equation}

Let $\lambda(r)=\sqrt{|r|+|r|^3}$, $\Lambda(\xi)=\sqrt{|\xi|+|\xi|^3}=\lambda(|\xi|)$, $\Lambda:\mathbb{R}^2\to [0,\infty)$. Let
\begin{equation}\label{notation}
\mathcal{U}_+:=\mathcal{U},\quad \mathcal{U}_-:=\overline{\mathcal{U}},\quad \mathcal{V}(t)=\mathcal{V}_+(t):=e^{it\Lambda}\mathcal{U}(t),\quad \mathcal{V}_{-}(t):=e^{-it\Lambda}\mathcal{U}_{-}(t).
\end{equation} 
Let $\Lambda_+=\Lambda$ and $\Lambda_{-}:=-\Lambda$. For $\sigma,\mu,\nu\in \{+,-\}$, we define the associated phase functions
\begin{equation}\label{phasedef}
\begin{split}
&\Phi_{\sigma\mu\nu}(\xi,\eta):=\Lambda_{\sigma}(\xi)-\Lambda_{\mu}(\xi-\eta)-\Lambda_{\nu}(\eta),\\
&\widetilde{\Phi}_{\sigma\mu\nu\beta}(\xi,\eta,\sigma):=\Lambda_{\sigma}(\xi)-\Lambda_{\mu}(\xi-\eta)-\Lambda_{\nu}(\eta-\sigma)-\Lambda_{\beta}(\sigma).
\end{split}
\end{equation}

\subsubsection{The spaces $O_{m,p}$} We will need several spaces of functions, in order to properly measure linear, quadratic, cubic, and quartic and higher order terms.
In addition, we also need to track the Sobolev smoothness and angular derivatives. Assume that $N_2=40\geq N_3+10$ and $N_0$ (the maximum number of Sobolev derivatives) and $N_1$ (the maximum number of angular derivatives) and $N_3$ (additional Sobolev regularity) are as before.

\begin{definition}\label{DefOHierarchy}
Assume $T\geq 1$ and let $p \in[-N_3,10]$. For $m\geq 1$ we define $O_{m,p}$ as the space of functions $f\in C([0,T]:L^2)$ satisfying
\begin{align}\label{O_1}
\begin{split}
\|f\|_{O_{m,p}}:=\sup_{t\in[0,T]}(1+t)^{(m-1)(5/6-20\delta^2)-\delta^2}\big[&\| f(t) \|_{H^{N_0+p}}+\|f(t) \|_{H_{\Omega}^{N_1,N_3+p}}\\
& + {(1+t)}^{5/6-2\delta^2} \|f(t) \|_{\widetilde{W}_\Omega^{N_1/2,N_2+p}}\big]<\infty,
\end{split}
\end{align}
where, with $P_k$ denoting standard Littlewood-Paley projection operators,
\begin{equation*}
\|g\|_{\widetilde{W}^N}:=\sum_{k\in\mathbb{Z}}2^{Nk^+}{\|P_kg\|}_{L^\infty},\qquad {\|g\|}_{\widetilde{W}^{N',N}_{\Omega}} := \sum_{j\leq N'}\|\Omega^jg\|_{\widetilde{W}^N}.
\end{equation*}
The spaces $\widetilde{W}^N$ are used in this paper as substitutes of the standard $L^\infty$ based Sobolev spaces,
which have the advantage of being closed under the action of singular integrals.
\end{definition}

Note that the parameter $p$ in $O_{m,p}$ corresponds to a gain at high frequencies and does not affect the low frequencies. We observe that, see Lemma \ref{lemmaaux1},
\begin{equation}\label{Alu7}
\begin{split}
O_{m,p}\subseteq O_{n,p}\,\text{ if }\,1\leq n\leq m,\qquad O_{m,p}O_{n,p}\subseteq O_{m+n,p}\,\text{ if }\,1\leq m,n.
\end{split}
\end{equation}
Moreover, by our assumptions \eqref{Ama32} and Lemma \ref{LinEstLem}, the main variables satisfy
\begin{equation}\label{bin1}
\|(1-\Delta)^{1/2}h\|_{O_{1,0}}+\|\,|\nabla|^{1/2}\phi\|_{O_{1,0}}\lesssim \varep_1.
\end{equation}

The spaces $O_{m,p}$ are used mostly in the energy estimates in section \ref{Eqs} and in the (elliptic) analysis of the Dirichlet--Neumann operator in appendix \ref{DNOpe}. However, they are not precise enough for the dispersive analysis of our evolution equation. For this we need the more precise $Z$-norm defined below, which is better adapted to our equation.  

\subsubsection{The $Z$ norm} Let $\gamma_0:=\sqrt{\frac{2\sqrt{3}-3}{3}}$ denote the radius of the sphere of slow decay and $\gamma_1:=\sqrt{2}$ denote the radius of the space-time resonant sphere. For $n\in\mathbb{Z}$, $I\subseteq\mathbb{R}$, and $\gamma\in (0,\infty)$ we define
\begin{equation}\label{aop}
\begin{split}
&\widehat{A_{n,\gamma}f}(\xi):=\varphi_{-n}(2^{100}||\xi|-\gamma|)\cdot\widehat{f}(\xi),\\
&A_{I,\gamma}:=\sum_{n\in I}A_{n,\gamma},\quad A_{\leq B,\gamma}:=A_{(-\infty,B],\gamma},\quad A_{\geq B,\gamma}:=A_{[B,\infty),\gamma}.
\end{split}
\end{equation}
Given an integer $j\geq 0$ we define the operators $A^{(j)}_{n,\gamma}$, $n\in\{0,\ldots,j+1\}$, $\gamma\geq 2^{-50}$, by
\begin{equation}\label{aop2}
\qquad A^{(j)}_{j+1,\gamma}:=\sum _{n'\geq j+1}A_{n',\gamma},\qquad A^{(j)}_{0,\gamma}:=\sum _{n'\leq 0}A_{n',\gamma},\qquad A^{(j)}_{n,\gamma}:=A_{n,\gamma}\,\,\text{ if }\,\,1\leq n\leq j.
\end{equation}
These operators localize to thin anuli of width $2^{-n}$ around the circle of radius $\gamma$. Most of the times, 
for us $\gamma=\gamma_0$ or $\gamma=\gamma_1$. We are now ready to define the main $Z$ norm.

\begin{definition}\label{MainZDef}
Assume that $\delta$, $N_0,N_1,N_4$ are as in Theorem \ref{MainTheo}. We define
\begin{equation}\label{sec5}
Z_1:=\{f\in L^2(\mathbb{R}^2):\,\|f\|_{Z_1}:=\sup_{(k,j)\in\mathcal{J}}\|Q_{jk}f\|_{B_{j}}<\infty\},
\end{equation}
where
\begin{equation}\label{znorm2}
\|g\|_{B_j}:=2^{(1-50\delta)j}\sup_{0\leq n\leq j+1}2^{-(1/2-49\delta)n}\|A^{(j)}_{n,\gamma_1}g\|_{L^{2}}.
\end{equation}
Then we define, with $D^\alpha:=\partial_1^{\alpha^1}\partial_2^{\alpha^2}$, $\alpha=(\alpha^1,\alpha^2)$,
\begin{equation}\label{znorm}
Z:=\big\{f\in L^2(\mathbb{R}^2):\|f\|_{Z}:=\sup_{2m+|\alpha|\leq N_1+N_4,\,m\leq N_1/2+20}\|D^{\alpha}\Omega^mf\|_{Z_1}<\infty\big\}.
\end{equation}
\end{definition}

We remark that the $Z$ norm is used to estimate  the {\it{linear profile}} of the solution, which is $\mathcal{V}(t):=e^{it\Lambda}\mathcal{U}(t)$, not the solution itself.

\subsubsection{Paradifferential calculus} We need some elements of paradifferential calculus in order to be able to describe the Dirichlet--Neumann operator $G(h)\phi$ in \eqref{WW0}. Our paralinearization relies on the {\it{Weyl quantization}}. More precisely, given a symbol $a=a(x,\zeta)$, and a function $f \in L^2$, we define the paradifferential operator $T_a f$
according to
\begin{align}
\label{Tsigmaf}
\mathcal{F} \big( T_{a}f \big) (\xi) = \frac{1}{4\pi^2}\int_{\R^2}\chi\Big(\frac{\vert\xi-\eta\vert}{\vert\xi+\eta\vert}\Big)
  \widetilde{a}(\xi-\eta,(\xi+\eta)/2) \widehat{f}(\eta)d\eta,
\end{align}
where $\widetilde{a}$ denotes the Fourier transform of $a$ in the first coordinate and
$\chi=\varphi_{\leq -20}$.
In Appendix \ref{SecParaOp} we prove several important lemmas related to the paradifferential calculus.

\section{The ``improved good variable'' and strongly semilinear structures}\label{Eqs}

\subsection{Renormalization}

In this section we assume $(h,\phi):\mathbb{R}^2\times[0,T]\to\mathbb{R}\times\mathbb{R}$ is a solution of \eqref{WW0} satisfying the hypothesis of Proposition \ref{MainBootstrapEn}; in particular, see \eqref{bin1},
\begin{equation}\label{bin2}
\|\langle \nabla\rangle h\|_{O_{1,0}}+\|\,|\nabla|^{1/2}\phi\|_{O_{1,0}}\lesssim \varep_1.
\end{equation}
Our goal in this section is to write the system \eqref{WW0} as a scalar equation for a suitably constructed complex-valued function (the ``improved good variable''). The main result is the following:

\begin{proposition}\label{proeqU}
Assume that \eqref{bin2} holds and let $\lambda_{DN}$ be the symbol of the 
Dirichlet-Neumann operator defined in \eqref{deflambda}, let $\Lambda:=\sqrt{g|\nabla|+\sigma|\nabla|^3}$, and let
\begin{align}
\label{defell}
\begin{split}
\ell(x,\zeta) & := L_{ij}(x) \zeta_i \zeta_j - \Lambda^2 h,  
  \qquad L_{ij} := \frac{\sigma }{\sqrt{1+|\nabla h|^2}} \Big(\delta_{ij} - \frac{\partial_i h \partial_j h}{1 + |\nabla h|^2} \Big),
\end{split}
\end{align}
be the mean curvature operator coming from the surface tension.
Define the symbol 
\begin{align}
\label{defSigma}
\Sigma & := \sqrt{\lambda_{DN}(g+\ell)} 
\end{align}
and the complex-valued unknown
\begin{align}
\label{defscalarunk}
\begin{split}
&U :=  T_{\sqrt{g+\ell}}h + iT_{\Sigma}T_{1/\sqrt{g+\ell}}\omega+iT_{m^\prime}\omega,\qquad
m^\prime := \frac{i}{2}\frac{\div V}{\sqrt{g+\ell}}\in\varep_1\mathcal{M}^{-1,1}_{N_3-2},
\end{split}
\end{align}
where $B,V$ and (the ``good variable'') $\omega=\phi-T_Bh$ are defined in \eqref{DefBV}. Then 
\begin{align}
\label{Uhorel} 
U = \sqrt{g + \sigma|\nabla|^2} h + i |\nabla|^{1/2} \omega + \varep_1^2O_{2,0},
\end{align}
and $U$ satisfies the equation
\begin{align}
\label{eqU}
& (\partial_t + i T_{\Sigma} + iT_{V\cdot\zeta})U = \mathcal{N}_U + \mathcal{Q}_S + \mathcal{C}_U,
\end{align}
where 

\setlength{\leftmargini}{2em}
\begin{itemize}

\item  The quadratic term $\mathcal{N}_U$ has the special structure
\begin{equation}
\label{eqUN}
\mathcal{N}_U := T_{\gamma}(c_1U+c_2\overline{U})
\end{equation}
for some constants $c_1,c_2 \in \mathbb{C}$, where
\begin{equation}\label{DefGamma}
\gamma(x,\zeta) := \frac{\zeta_i\zeta_j}{\vert\zeta\vert^2}\vert\nabla\vert^{-1/2}\partial_i\partial_j(\Im U)(x).
\end{equation}

\item the quadratic terms $\mathcal{Q}_S$ have a gain of one derivative, i.e. they are of the form
\begin{align}
\label{eqUquad}
& \mathcal{Q}_S = A_{++}(U,U) + A_{+-}(U,\bar{U}) + A_{--}(\bar{U},\bar{U})\in \varep_1^2O_{2,1},
\end{align}
with symbols $a_{\eps_1\eps_2}$ satisfying, for all $k,k_1,k_2 \in \mathbb{Z}$, and $(\eps_1\eps_2) \in \{ (++), (+-), (--)\}$,
\begin{align}
\label{symbolseqU}
\begin{split}
{\| a_{\eps_1\eps_2}^{k,k_1,k_2} \|}_{S^\infty_\Omega} & \lesssim 2^{-\max(k_1,k_2,0)}(1+2^{3\min(k_1,k_2)});
\end{split}
\end{align}

\item $\mathcal{C}_U$ is a cubic term, $\mathcal{C}_U\in\varep_1^3O_{3,0}$.
\end{itemize}
\end{proposition}

Let us comment on the structure of the main equation \eqref{eqU}. In the left-hand side we have the ``quasilinear'' part $(\partial_t + i T_{\Sigma} + iT_{V\cdot\zeta})U$. In the right-hand side we have three types of terms: 

(1) a quadratic term $\mathcal{N}_U$ with special structure;

(2) a strongly semilinear quadratic term $\mathcal{Q}_S$, given by symbols of order -1; 

(3) a semilinear cubic term $\mathcal{C}_U\in \varep_1^3O_{3,0}$, whose contribution is easy to estimate.

The key point is the special structure of the quadratic terms, which allows us to obtain favorable energy estimates in Proposition \ref{ProBulk}. This special structure is due to the definition of the variable $U$, in particular the choice of the symbol $m'$ in \eqref{defscalarunk}. We observe that
\begin{align*}
\widetilde{\gamma}(\eta,\zeta) = -\frac{\zeta_i\zeta_j}{|\zeta|^2} \frac{\eta_i \eta_j}{|\eta|^{1/2}} \widehat{\Im U}(\eta)
\end{align*}
and we remark that the angle $\zeta\cdot\eta$ in this expression gives us the strongly semilinear structure we will use later (see also the factor $\mathfrak{d}$ in \eqref{bin3.7}). For comparison, the use of the standard ``good unknown'' of Alinhac leads to generic quadratic terms that do not lose derivatives. This would suffice to prove local regularity of the system, but would not be suitable for global analysis.

This proposition is the starting point of our energy analysis. Its proof is technically involved, as it requires the material in 
appendices \ref{SecParaOp} and \ref{DNOpe}. One can start by understanding the definition \ref{DefSym} of the decorated spaces of symbols
$\mathcal{M}^{l,m}_r$, the simple properties \eqref{OimpliesS}--\eqref{auxcomm2}, and the statement of Proposition \ref{DNmainpro} (the proof is not needed). 
The spaces of symbols $\mathcal{M}^{l,m}_r$ are analogous to the spaces of functions $\mathcal{O}_{m,p}$; for symbols, however, the order $l$ is important 
(for example a symbol of order $2$ counts as two derivatives), but its exact differentiability (measured by the parameter $r$) is less important. 

In Proposition \ref{proeqU} we keep the parameters $g$ and $\sigma$ due to their physical significance.

\begin{remark}\label{RmkQuasilinearTerms}
(i) The symbols defined in this proposition can be estimated in terms of the decorated norms introduced in Definition \ref{DefSym}. More precisely, using the hypothesis \eqref{bin2}, the basic bounds \eqref{OimpliesS} and \eqref{SymbMult}, and the definition \eqref{deflambda}, it is easy to verify that
\begin{equation}\label{bin3}
\begin{split}
&(g+\ell)=\frac{(g+\sigma|\zeta|^2)}{\sqrt{1+|\nabla h|^2}}\Big(1-\frac{\sigma(\zeta\cdot\nabla h)^2}{(g+\sigma|\zeta|^2)}-\frac{\Lambda^2h}{(g+\sigma|\zeta|^2)}+\varep_1^4\mathcal{M}_{N_3-2}^{0,4}+\varep_1^2\mathcal{M}_{N_3-2}^{-2,2}\Big),\\
&\lambda_{DN}=|\zeta|\Big(1+\frac{|\zeta|^2|\nabla h|^2-(\zeta\cdot\nabla h)^2}{2|\zeta|^2}+\frac{|\zeta|^2\Delta h-\zeta_j\zeta_k\partial_j\partial_kh}{2|\zeta|^3}\varphi_{\geq 0}(\zeta)+\varep_1^4\mathcal{M}_{N_3-2}^{0,4}+\varep_1^3\mathcal{M}_{N_3-2}^{-1,3}\Big),
\end{split}
\end{equation}
uniformly for every $t\in[0,T]$. Thus, with $\Lambda=\sqrt{g|\nabla|+\sigma|\nabla|^3}$, we derive an expansion for $\Sigma$,
\begin{equation}\label{ExpSigma}
\begin{split}
&\Sigma=\Lambda+\Sigma_1+\Sigma_{\geq 2},\\
&\Sigma_1:=\frac{1}{4}\frac{\Lambda(\zeta)}{\vert\zeta\vert}\left[\Delta h-\frac{\zeta_i\zeta_j}{\vert\zeta\vert^2}\partial_{ij}h\right]\varphi_{\ge0}(\zeta)-\frac{1}{2}\frac{\vert\zeta\vert}{\Lambda(\zeta)}\Lambda^2h\in\varep_1\mathcal{M}^{1/2,1}_{N_3-2},\qquad \Sigma_{\ge 2}\in\varep_1^2\mathcal{M}^{3/2,2}_{N_3-2}.
\end{split}
\end{equation}
The formulas are slightly simpler if we disregard quadratic terms, i.e.  
\begin{equation}\label{SymApprox0}
\begin{split}
\lambda_{DN}^p &=|\zeta|^p(1+p\lambda_1^{(0)}(x,\zeta)/\vert\zeta\vert
  + \varep_1^2\mathcal{M}^{0,2}_{N_3-2}),\\
(g+\ell)^p& = (g+\sigma\vert\zeta\vert^2)^p (1 -  p\Lambda^2h/(g+\sigma\vert\zeta\vert^2)+\varep_1^2\mathcal{M}^{0,2}_{N_3-2}),\\
	\Sigma&=\Lambda(1+\Sigma_1(x,\zeta)/\Lambda+\varep_1^2\mathcal{M}^{0,2}_{N_3-2}),
\end{split}
\end{equation}
for $p\in[-2,2]$, where $\lambda_1^{(0)}(x,\zeta)=\frac{|\zeta|^2\Delta h-\zeta_j\zeta_k\partial_j\partial_kh}{2|\zeta|^2}\varphi_{\geq 0}(\zeta)$ as in Remark \ref{RmkLambda}. 

In addition, the identity $\partial_th=G(h)\phi=|\nabla|\omega+\varep_1^2\mathcal{O}_{2,-1/2}$ shows that
\begin{equation}
\label{SymApprox2}
\begin{split}
\partial_t\sqrt{g+\ell}&=(g+\sigma|\zeta|^2)^{-1/2}[\Delta(g-\sigma\Delta)\omega/2]+\varep_1^2\mathcal{M}^{1,2}_{N_3-4}
  \in \varep_1\mathcal{M}^{-1,1}_{N_3-4}+\varep_1^2\mathcal{M}^{1,2}_{N_3-4},\\
\partial_t\sqrt{\lambda_{DN}} & = \frac{1}{2\sqrt{\vert\zeta\vert}}\partial_t\lambda_1^{(0)}+\varep_1^2\mathcal{M}^{1/2,2}_{N_3-4}
  \in \varep_1\mathcal{M}^{-1/2,1}_{N_3-4}+\varep_1^2\mathcal{M}^{1/2,2}_{N_3-4},\\
\partial_t\Sigma &=\partial_t\Sigma_1+\varep_1^2\mathcal{M}^{3/2,2}_{N_3-4}\in\varep_1\mathcal{M}^{1/2,1}_{N_3-4}+\varep_1^2\mathcal{M}^{3/2,2}_{N_3-4}.
\end{split}
\end{equation}

(ii) It follows from Proposition \ref{DNmainpro} that $V\in\varep_1\mathcal{O}_{1,-1/2}$. Therefore $m'\in\varep_1\mathcal{M}^{-1,1}_{N_3-2}$ and the identity \eqref{Uhorel} follows using also Lemma \ref{lemmaaux2}. Moreover, using Proposition \ref{DNmainpro} again,
\begin{equation}\label{ExpV}
\begin{split}
V=V_1+V_2,\qquad V_1:=\vert\nabla\vert^{-1/2}\nabla\Im U,\qquad V_2\in\varep_1^2\mathcal{O}_{2,-1/2},\\
m'=m'_1+\varep_1^2\mathcal{M}^{-1,2}_{N_3-2},\qquad m'_1(x,\zeta):=-\frac{i}{2}\frac{|\nabla|^{3/2}\Im U(x)}{\sqrt{g+\sigma|\zeta|^2}}
\end{split}
\end{equation}
\end{remark}

\subsection{Symmetrization and special quadratic structure}\label{secprosym}

In this subsection we prove Proposition \ref{proeqU}. We first write \eqref{WW0} as a system for $h$ and $\omega$, and then symmetrize it.
We start by combining Proposition \ref{DNmainpro} on the Dirichlet-Neumann operator with a paralinearization of the equation
for $\partial_t \phi$, to obtain the following:

\begin{lemma}\label{bin9}[Paralinearization of the system]
With the notation of Proposition \ref{DNmainpro} and Proposition \ref{proeqU}, we can rewrite the system \eqref{WW0} as
\begin{equation}\label{WWpara}
\left\{
\begin{array}{l}
\partial_t h = T_{\lambda_{DN}} \omega - \div( T_V h) + G_2 + \varep_1^3O_{3,1},
\\
\partial_t \omega = - g h  - T_\ell h - T_V \nabla\omega + \Omega_2 + \varep_1^3O_{3,1},
\end{array}
\right.
\end{equation}
where $\ell$ is given in \eqref{defell} and
\begin{align}
\label{d_tomegaquadratic}
& \Omega_2 := \frac{1}{2} \mathcal{H}(|\nabla|\omega,|\nabla|\omega) - \frac{1}{2} \mathcal{H}(\nabla\omega,\nabla\omega)  \in\varep_1^2 O_{2,2}.
\end{align}
\end{lemma}

\begin{proof}
First, we see directly from \eqref{WW0} and Proposition \ref{DNmainpro} that, for any $t\in[0,T]$,
\begin{equation}
\label{LinearBV}
\begin{split}
&G(h)\phi,\,B,\,V,\,\partial_th\in\varep_1\mathcal{O}_{1,-1/2},\qquad\partial_t\phi\in\varep_1\mathcal{O}_{1,-1},\\
&B =|\nabla|\omega+\varep_1^2\mathcal{O}_{2,-1/2},\qquad V =\nabla\omega+\varep_1^2\mathcal{O}_{2,-1/2}.
\end{split}
\end{equation}
The first equation in \eqref{WWpara} comes directly from Propostion \ref{DNmainpro}.
To obtain the second equation, we use Lemma \ref{PropHHSym} (ii) with $F_l(x) = x_l / \sqrt{1+|x|^2}$ to see that
\begin{equation*}
\begin{split}
F_l(\nabla h)& = T_{\partial_kF_l(\nabla h)} \partial_kh +  \varep_1^3\mathcal{O}_{3,3},\qquad\hbox{hence}\quad  \sigma \div \Big[ \dfrac{\nabla h}{ \sqrt{1+|\nabla h|^2} } \Big] = -T_{L_{jk}\zeta_j\zeta_k}h + \varep_1^3\mathcal{O}_{3,1}.
 \end{split}
\end{equation*}
Next we paralinearize the other nonlinear terms in the second equation in \eqref{WW0}. Recall the definition of $V$, $B$ in \eqref{DefBV}. We first write
\begin{align*}
- \frac{1}{2} {|\nabla \phi|}^2 + \dfrac{{\left( G(h)\phi + \nabla h \cdot \nabla \phi \right)}^2}{2(1+{|\nabla h|}^2)}&= 
- \frac{|V + B \nabla h|^2}{2}  + \frac{(1+ |\nabla h|^2) B^2}{2}   = \frac{B^2-2 B V \cdot \nabla h - | V |^2}{2}.
\end{align*}
Using \eqref{WW0}, we calculate $\partial_t h = G(h)\phi =B - V \cdot \nabla h$, and
\begin{align*}
\partial_t \omega & = \partial_t \phi - T_{\partial_t B} h - T_B \partial_t h
\\
& = -g h -  T_{L_{jk}\zeta_j\zeta_k} h + \frac{1}{2}\big( B^2 - 2 B V \cdot \nabla h - | V |^2 \big) - T_{\partial_t B} h - T_B B + T_B (V \cdot \nabla h)+\varep_1^3\mathcal{O}_{3,1}.
\end{align*}
Then, since $V = \nabla \phi - B \nabla h$, we have
\begin{align*}
T_V \nabla \omega = T_V \nabla\phi - T_V (\nabla T_B h) = T_V V + T_V (B \nabla h) - T_V (\nabla T_B h) ,
\end{align*}
and we can write
\begin{align*}
\partial_t \omega & = -g h - T_{L_{jk}\zeta_j\zeta_k+\partial_tB} h  - T_V \nabla \omega + I + II ,
\\
 I &:= \frac{1}{2} B^2 - T_B B - \frac{1}{2} \vert V\vert^2 + T_V V= \frac{1}{2}\mathcal{H}(B,B) - \frac{1}{2}\mathcal{H}(V,V) =  \Omega_2+\varep_1^3\mathcal{O}_{3,1},
\\
II &:= -B V \cdot \nabla h + T_B (V\cdot\nabla h) + T_V (B\nabla h) - T_V (\nabla T_Bh)+\varep_1^3\mathcal{O}_{3,1}.
\end{align*}
Using \eqref{LinearBV}, \eqref{DefBV}, \eqref{WW0}, and Corollary \ref{CorFixedPoint} (ii) we easily see that 
\begin{align*}
L_{jk}\zeta_j\zeta_k+\partial_tB=L_{jk}\zeta_j\zeta_k+\vert\nabla\vert\partial_t\phi+\varep_1^2\mathcal{O}_{2,-2}= \ell+\varep_1^2\mathcal{O}_{2,-2}.
\end{align*}
Moreover we can verify that $II$ is an acceptable cubic remainder term:
\begin{align*}
II & = - T_{V \cdot \nabla h} B + \mathcal{H}(B,V\cdot\nabla h) + T_V (B\nabla h) - T_V T_B \nabla h - T_V T_{\nabla B} h+\varep_1^3\mathcal{O}_{3,1}
\\
& = - T_{V \cdot \nabla h} B + T_V T_{\nabla h} B + T_V \mathcal{H}(B,\nabla h)   - T_V T_{\nabla B} h +\varep_1^3\mathcal{O}_{3,1}\\
& = \varep_1^3\mathcal{O}_{3,1},
\end{align*}
and the desired conclusion follows.
\end{proof}

Since our purpose will be to identify quadratic terms as in \eqref{eqUquad}-\eqref{symbolseqU}, we need a more precise notion of strongly semilinear quadratic errors.

\begin{definition}
\label{defSS}
Given $t\in[0,T]$ we define $\varep_1^2\mathcal{O}^\ast_{2,1}$ to be the set of finite linear combinations of terms of the form $S[T_1,T_2]$ where $ T_1,T_2\in\{U(t),\overline{U}(t)\}$, and $S$ satisfies
\begin{equation}\label{QuadSSEqHPsi}
\begin{split}
& \mathcal{F}(S[f,g])(\xi):=\frac{1}{4\pi^2}\int_{\mathbb{R}^2}s(\xi,\eta)\widehat{f}(\xi-\eta)\widehat{g}(\eta)d\eta,
\\
& \Vert s^{k, k_1, k_2} \Vert_{S^\infty_\Omega} \lesssim 
  2^{-\max(k_1,k_2,0)}(1+2^{3\min(k_1,k_2)}).
\end{split}
\end{equation}
These correspond precisely to the acceptable quadratic error terms according to \eqref{symbolseqU}.
\end{definition}

We remark that if $S$ is defined by a symbol as in \eqref{QuadSSEqHPsi} and $p\in[-5,5]$ then
\begin{equation}\label{bin6}
S[\mathcal{O}_{m,p},\mathcal{O}_{n,p}]\subseteq \mathcal{O}_{m+n,p+1}.
\end{equation}
This follows by an argument similar to the argument used in Lemma \ref{lemmaaux1}. As a consequence, given the assumptions \eqref{bin2} and with $U$ defined as in \eqref{defscalarunk}, we have that $\mathcal{O}^\ast_{2,1}\subseteq \mathcal{O}_{2,1}$.

\begin{proof}[Proof of Proposition \ref{proeqU}] {\bf{Step 1.}} To diagonalize the principal part of the system \eqref{WWpara} we define the symmetrizing variables $(H,\Psi)$ by
\begin{equation}\label{NewUnknowns}
\begin{split}
H & := T_{\sqrt{g+\ell}}h, \qquad \Psi := T_{\Sigma}T_{1/\sqrt{g+\ell}}\omega+T_{m^\prime}\omega,
\end{split}
\end{equation}
where $m'$ is as in \eqref{defscalarunk}. Using \eqref{SymApprox0} and Lemma \ref{lemmaaux2}, we see that
\begin{align}
 \label{unkrelations}
\begin{split}
&H = \Re(U)+\varep_1^2\mathcal{O}_{2,0}, \qquad \sqrt{g+\sigma\vert\nabla\vert^2} h=\Re(U)+\varep_1^2\mathcal{O}_{2,0},\\
&\Psi =\Im(U)+ \varep_1^2\mathcal{O}_{2,0},\,\qquad |\nabla|^{1/2}\omega=\Im(U)+ \varep_1^2\mathcal{O}_{2,0}.
\end{split}
\end{align}
As a consequence, if $T_1,T_2\in\{U,\overline{U},H,\Psi,(g-\sigma\Delta)^{1/2}h,\vert\nabla\vert^{1/2}\omega\}$, 
and $S$ is as in \eqref{QuadSSEqHPsi}, then
\begin{equation}\label{Acceptable}
S[T_1,T_2]\in\varep_1^2\mathcal{O}^\ast_{2,1}+\varep_1^3\mathcal{O}_{3,0}.
\end{equation}

We will show that
\begin{equation}\label{SymSys}
\left\{
\begin{array}{l}
\partial_t H - T_{\Sigma} \Psi+iT_{V\cdot\zeta}H = -T_\gamma H -\dfrac{1}{2}T_{\sqrt{g+\ell} \, \mathrm{div} V}h -T_{ m^\prime\Sigma}\omega + \varep_1^2\mathcal{O}^\ast_{2,1}+ \varep_1^3\mathcal{O}_{3,0},
\\
\\
\partial_t \Psi + T_{\Sigma } H + iT_{V\cdot\zeta}\Psi 
  = -\dfrac{1}{2}T_\gamma\Psi- T_{m^\prime(g+\ell)}h +\dfrac{1}{2}T_{\sqrt{\lambda_{DN}} \, \div V}\omega+ \varep_1^2\mathcal{O}^\ast_{2,1} + \varep_1^3\mathcal{O}_{3,0}.
\end{array}
\right.
\end{equation}

{\bf{Step 2.}} We examine now the first equation in \eqref{SymSys}. The first equation in \eqref{WWpara} and the identity $\div T_Vh=(1/2)T_{\div V}h+iT_{V\cdot\ze}h$ show that
\begin{equation}\label{EqdtH}
\begin{split}
\partial_tH-&T_\Sigma\Psi+iT_{V\cdot\zeta}H+T_\gamma H
 + \frac{1}{2}T_{\sqrt{g+\ell} \, \div V}h+T_{ m^\prime\Sigma}\omega
\\
&= (T_{\sqrt{g+\ell}}T_{\lambda_{DN}}-T_{\Sigma}T_{\Sigma}T_{1/\sqrt{g+\ell}})\omega-(T_\Sigma T_{m^\prime}-T_{ m^\prime\Sigma})\omega
\\
& +i(T_{V\cdot\zeta}H-T_{\sqrt{g+\ell}}T_{V\cdot\zeta}h-iT_\gamma T_{\sqrt{g+\ell}}h)\\
& +T_{\partial_t\sqrt{g+\ell}}h-\frac{1}{2}(T_{\sqrt{g+\ell}}T_{\div V}-T_{\sqrt{g+\ell} \, \div V})h+T_{\sqrt{g+\ell}}G_2+\varep_1^3T_{\sqrt{g+\ell}}\mathcal{O}_{3,1}.
\end{split}
\end{equation}
We will treat each line separately. For the first line, we notice that the contribution of low frequencies $P_{\leq-9}\omega$ is acceptable. For the high frequencies we use Proposition \ref{PropCompSym} to write
\begin{align}
\nonumber
\big( T_{\sqrt{g+\ell}} &T_{\lambda_{DN}} - T_{\Sigma}T_{\Sigma}  T_{1/\sqrt{g+\ell}} \big) P_{\geq -8}\omega
\\
\label{prosym5}
& = \Big( T_{ \lambda_{DN} \sqrt{g+\ell}} + \frac{i}{2} T_{ \{ \sqrt{g+\ell}, \lambda_{DN} \} }
  -\big( T_{\Sigma^2/\sqrt{g+\ell}} + \frac{i}{2} T_{ \{ \Sigma^2, 1/\sqrt{g+\ell} \} } \big)
  \Big) P_{\geq -8}\omega 
\\
\label{prosym6}
& +[E(\sqrt{g+\ell},\lambda_{DN})-E(\Sigma,\Sigma)T_{1/\sqrt{g+\ell}}-E(\Sigma^2,1/\sqrt{g+\ell})]P_{\geq -8}\omega.
\end{align}
Since
\begin{equation*}
\begin{split}
\lambda_{DN}\sqrt{g+\ell}=\Sigma^2/\sqrt{g+\ell},\qquad\{\sqrt{g+\ell},\lambda_{DN}\}=\{\Sigma^2,1/\sqrt{g+\ell}\}
\end{split}
\end{equation*}
we observe that the expression in \eqref{prosym5} vanishes. Using \eqref{SymApprox0} and Lemma \ref{lemmaaux3}, we see that, up to acceptable cubic terms, we can rewrite the second line of \eqref{EqdtH} as
\begin{equation*}
\begin{split}
&\Big[ E(\sqrt{g+\sigma\vert\zeta\vert^2},\lambda_1^{(0)}) + E(-\frac{\Lambda^2 h}{2\sqrt{g+\sigma\vert\zeta\vert^2}},\vert\zeta\vert) 
  - (E(\Lambda,\Sigma_1)+E(\Sigma_1,\Lambda))(g-\sigma\Delta)^{-1/2}
\\
& - E(\Lambda^2, \frac{\Lambda^2 h}{2(g+\sigma\vert\zeta\vert^2)^{3/2}}) - 2E(\Lambda\Sigma_1,\frac{1}{\sqrt{g+\sigma\vert\zeta\vert^2}}) 
  - \frac{i}{2}T_{\{\Lambda,m^\prime_1\}}-E(\Lambda,m^\prime_1) \Big]P_{\geq -8}\omega + \varep_1^3\mathcal{O}_{3,0}.
\end{split}
\end{equation*}
Using \eqref{ExEab} these terms are easily seen to be acceptable $\varep_1^2\mathcal{O}^\ast_{2,1}$ quadratic terms. 

To control the terms in the second line of the right-hand side of \eqref{EqdtH}, we observe that
\begin{equation}\label{Errors2}
\begin{split}
&T_{V\cdot\zeta}T_{\sqrt{g+\ell}}h - T_{\sqrt{g+\ell}}T_{V\cdot\zeta}h-iT_\gamma T_{\sqrt{g+\ell}}
  =i(T_{\{V\cdot\zeta,\sqrt{g+\ell}\}}-T_{\gamma\sqrt{g+\ell}})h\\
&+E(V\cdot\zeta,\sqrt{g+\ell})h-E(\sqrt{g+\ell},V\cdot\zeta)h+\frac{1}{2}T_{\{\gamma,\sqrt{g+\ell}\}}h-iE(\gamma,\sqrt{g+\ell})h.
\end{split}
\end{equation}
Using \eqref{SymApprox0} and \eqref{ExpV}, we notice that
\begin{equation*}
{\{V\cdot\zeta,\sqrt{g+\ell}\}}=\frac{\zeta_j\partial_kV_j(x)\cdot \sigma\zeta_k}{\sqrt{g+\sigma|\zeta|^2}}+\varep_1^2\mathcal{M}^{1,2}_{N_3-2}=\frac{\sigma\zeta_j\zeta_k\cdot \partial_j\partial_k|\nabla|^{-1/2}\Im(U)(x)}{\sqrt{g+\sigma|\zeta|^2}}+\varep_1^2\mathcal{M}^{1,2}_{N_3-2}.
\end{equation*}
Using the definition \eqref{DefGamma} it follows that $T_{\{V\cdot\zeta,\sqrt{g+\ell}\}-\gamma\sqrt{g+\ell}}h\in\varep_1^2\mathcal{O}^\ast_{2,1}+ \varep_1^3\mathcal{O}_{3,0}$ as desired. The terms in the second line of \eqref{Errors2} are also acceptable contributions, as one can see easily by extracting the quadratic parts and  using  \eqref{ExEab}.

Finally, for the third line, using \eqref{SymApprox2}, \eqref{ExpV}, and Lemmas \ref{lemmaaux2} and \ref{lemmaaux3}, we observe that
\begin{equation*}
\begin{split}
T_{\partial_t\sqrt{g+\ell}}h & = T_{\frac{1}{2}\frac{\Delta(g-\sigma\Delta)\omega}{\sqrt{g+\sigma\vert\zeta\vert^2}}}h + \varep_1^3\mathcal{O}_{3,0},
\\
\big(T_{\sqrt{g+\ell}}T_{\div V}-T_{\div V \cdot \sqrt{g+\ell}} \big)h & = \big(iT_{\{\sqrt{g+\sigma\vert\zeta\vert^2}, \div V_1\}}
  + E(\sqrt{g+\sigma\vert\zeta\vert^2},\div V_1) \big)h + \varep_1^3\mathcal{O}_{3,0},
\\
T_{\sqrt{g+\ell}}G_2 +\varep_1^3T_{\sqrt{g+\ell}}\mathcal{O}_{3,1}& = T_{\sqrt{g+\sigma\vert\zeta\vert^2}}G_2+\varep_1^3\mathcal{O}_{3,0}.
\end{split}
\end{equation*}
Using \eqref{ExEab}, the bounds for $G_2$ in \eqref{DNquadratic}-\eqref{DNquadraticsym}, 
and collecting all the estimates above, we obtain the identity in the first line in \eqref{SymSys}.

{\bf{Step 3.}} To prove the second identity in \eqref{SymSys} we first use \eqref{NewUnknowns} and \eqref{WWpara} to compute
\begin{equation}\label{EqdtPsi}
\begin{split}
\partial_t\Psi + T_{\Sigma}H + &iT_{V\cdot\zeta}\Psi+\frac{1}{2}T_\gamma\Psi+T_{m^\prime(g+\ell)}h-\frac{1}{2}
T_{\sqrt{\lambda_{DN}}\mathrm{div}V}\omega
\\
&= (T_{\Sigma}T_{\sqrt{g+\ell}}-T_{\Sigma}T_{1/\sqrt{g+\ell}}T_{g+\ell})h+(T_{m^\prime(g+\ell)}-T_{m^\prime}T_{g+\ell})h
\\
& + i(T_{V\cdot\zeta}\Psi-\frac{i}{2}T_\gamma\Psi-(T_\Sigma T_{1/\sqrt{g+\ell}} + T_{m^\prime})T_{V\cdot\zeta}\omega)
\\
& + \frac{1}{2}(T_{\Sigma}T_{1/\sqrt{g+\ell}}T_{\div V}-T_{\sqrt{\lambda_{DN}}\, \div V })\omega + \frac{1}{2}T_{m^\prime}T_{\div V}\omega
\\
& + [\partial_t,T_{\Sigma}T_{1/\sqrt{g+\ell}} + T_{m^\prime}]\omega+(T_{\Sigma}T_{1/\sqrt{g+\ell}} + T_{m^\prime})(\Omega_2+\varep_1^3\mathcal{O}_{3,1}).
\end{split}
\end{equation}
Again, we verify that all lines after the equality sign give acceptable remainders. 

For the terms in the first line, using Proposition \ref{PropCompSym}, \eqref{SymApprox0}, and Lemma \ref{lemmaaux3},
\begin{align*}
\big(T_{\Sigma}  T_{\sqrt{g+\ell}} &-T_{\Sigma}T_{1/\sqrt{g+\ell}} T_{g+\ell} \big)h=-T_\Sigma E(1/\sqrt{g+\ell},g+\ell)h
\\
& = \Lambda\big[E(\frac{\Lambda^2h}{2(g+\sigma\vert\zeta\vert^2)^{3/2}},g+\sigma\vert\zeta\vert^2)-E(1/\sqrt{g+\sigma\vert\zeta\vert^2},\Lambda^2h)\big]h + \varep_1^3\mathcal{O}_{3,0}.
\end{align*}
Using also \eqref{ExEab}, this gives acceptable contributions. In addition,
\begin{equation*}
\begin{split}
(T_{m^\prime}T_{g+\ell}-T_{m^\prime(g+\ell)})h&=\frac{i}{2}T_{\{m^\prime,g+\ell\}}h+E(m^\prime,g+\ell)h=i\sigma T_{\zeta\cdot\nabla_x m^\prime}h+E(m^\prime,\sigma\vert\zeta\vert^2)h+\varep_1^3\mathcal{O}_{3,0}.
\end{split}
\end{equation*}
This gives acceptable contributions, in view of \eqref{NewUnknowns} and \eqref{ExEab}. 

For the terms in the second line of the right-hand side of \eqref{EqdtPsi} we observe that
\begin{equation*}
\begin{split}
T_{V\cdot\zeta}\Psi-\frac{i}{2}&T_\gamma\Psi-(T_\Sigma T_{1/\sqrt{g+\ell}}+T_{m^\prime})T_{V\cdot\zeta}\omega
\\
&=[T_{V\cdot\zeta}T_\Sigma T_{1/\sqrt{g+\ell}}-T_\Sigma T_{1/\sqrt{g+\ell}}T_{V\cdot\zeta}]\omega-\frac{i}{2}T_{\gamma}T_\Sigma T_{1/\sqrt{g+\ell}}\omega+\varep_1^3\mathcal{O}_{3,0}\\
&=[T_{V_1\cdot\zeta}T_{|\zeta|^{1/2}}-T_{|\zeta|^{1/2}}T_{V_1\cdot\zeta}]\omega-\frac{i}{2}T_{\gamma}T_{|\zeta|^{1/2}}\omega+\varep_1^3\mathcal{O}_{3,0}.
\end{split}
\end{equation*}
Using the definitions \eqref{ExpV} and \eqref{DefGamma}, we notice that for $p\in[0,2]$
\begin{equation}\label{gamprop}
\{V_1\cdot\zeta,|\zeta|^{p}\}=\gamma\cdot p|\zeta|^{p}\qquad\text{ on }\qquad\R^2\times\R^2.
\end{equation}
Thus the terms in the second line of the right-hand side of \eqref{EqdtPsi} are acceptable $\varep_1^2\mathcal{O}^\ast_{2,1} + \varep_1^3\mathcal{O}_{3,0}$ contributions.

It is easy to see, using Lemma \ref{lemmaaux3} and the definitions, that the terms the third line in the right-hand side of \eqref{EqdtPsi} are acceptable. Finally, for the last line in \eqref{EqdtPsi}, we observe that
\begin{equation*}
\begin{split}
[\partial_t,T_\Sigma T_{1/\sqrt{g+\ell}}&+T_{m'}]\omega=T_{\partial_t\Sigma}T_{1/\sqrt{g+\ell}}\omega 
  + T_{\Sigma}T_{\partial_t(1/\sqrt{g+\ell})}\omega +T_{\partial_tm'}\omega\\
	&=T_{\partial_t\Sigma_1}T_{(g+\sigma|\zeta|^2)^{-1/2}}\omega 
  -\Lambda T_{\frac{\Delta(g-\sigma\Delta)\omega}{2(g+\sigma|\zeta|^2)^{3/2}}}\omega +\frac{i}{2}T_{\partial_t(\mathrm{div}V)(g+\sigma|\zeta|^2)^{-1/2}}\omega+\varep_1^3\mathcal{O}_{3,0},
\end{split}
\end{equation*}
where we used \eqref{SymApprox0} and \eqref{SymApprox2}. Since $\partial_th=|\nabla|\omega+\varep_1^2\mathcal{O}_{2,-1/2}$ and $\partial_tV=-\nabla(g+\sigma|\nabla|^2)h+\varep_1^2\mathcal{O}_{2,-2}$ (see Lemma \ref{bin9} and Proposition \ref{DNmainpro}), it follows that the terms in the formula above are acceptable. Finally, using the relations in Lemma \ref{bin9},
\begin{equation*}
\begin{split}
(T_{\Sigma}T_{1/\sqrt{g+\ell}} + T_{m^\prime})(\Omega_2)=\varep_1^3\mathcal{O}_{3,0}+\varep_1^2\mathcal{O}^\ast_{2,1},\qquad (T_{\Sigma}T_{1/\sqrt{g+\ell}} + T_{m^\prime})(\varep_1^3\mathcal{O}_{3,1})&=\varep_1^3\mathcal{O}_{3,0}.
\end{split}
\end{equation*}
Therefore, all the terms in the right-hand side of \eqref{EqdtPsi} are acceptable, which completes the proof of \eqref{SymSys}.

{\bf{Step 4.}} Starting from the system \eqref{SymSys} we now want to write a scalar equation for the complex unknown $U = H + i\Psi$ defined in \eqref{defscalarunk}. Using \eqref{SymSys}, we readily see that
\begin{align*}
\begin{split}
& \partial_t U + i T_{\Sigma} U +iT_{V\cdot\zeta}U= Q_U+ \mathcal{N}_U +\varep_1^2\mathcal{O}^\ast_{2,1}+\varep_1^3\mathcal{O}_{3,0},
\\
&Q_U :=(-\frac{1}{2}T_{\sqrt{g+\ell}\, \div V}-iT_{m^\prime (g+\ell)})h + (-T_{m^\prime\Sigma}+\frac{i}{2}T_{\sqrt{\lambda_{DN}}\, \div V})\omega = 0,
\\
&\mathcal{N}_U :=-T_\gamma H-\frac{i}{2}T_\gamma \Psi=-\frac{1}{4} T_\gamma\big(3U +\overline{U} \big)+\varep_1^3\mathcal{O}_{3,0},
\end{split}
\end{align*}
where $Q_U$ vanishes in view of our choice of $m^\prime$, and $\mathcal{N}_U$ has the special structure as claimed.
\end{proof}

\subsection{High order derivatives}\label{secproeqW}

To derive higher order Sobolev and weighted estimates for $U$, and hence for $h$ and $|\nabla|^{1/2}\omega$,
we need to apply (a suitable notion of) derivatives to the equation \eqref{eqU}.
We will then consider quantities of the form 
\begin{align}
 \label{defW_n}
W_n := (T_\Sigma)^n U, \quad n\in[0,2N_0/3],\qquad Y_{m,p}:=\Omega^p(T_\Sigma)^m U,\quad p\in[0,N_1],\,m\in[0,2N_3/3],
\end{align}
for $U$ as in \eqref{defscalarunk} and $\Sigma$ as in \eqref{defSigma}. We have the following consequence of Proposition \ref{proeqU}:

\begin{proposition}\label{proeqW}
With the notation above and $\gamma$ as in \eqref{DefGamma}, we have
\begin{align}
\label{eqW}
\partial_t W_n + i T_\Sigma W_n+iT_{V\cdot\zeta}W_n =  T_\gamma (c_n W_n + d_n\overline{W_n} ) + \mathcal{B}_{W_n} + \mathcal{C}_{W_n} ,
\end{align}
and
\begin{align}
\label{eqUR}
\partial_t Y_{m,p} + i T_\Sigma Y_{m,p} + iT_{V\cdot\zeta}Y_{m,p} = T_\gamma(c_{m}Y_{m,p} + d_{m}\overline{Y_{m,p}}) + \mathcal{B}_{Y_{m,p}} + \mathcal{C}_{Y_{m,p}},
\end{align}
for some complex numbers $c_n,d_n$. The cubic terms $\mathcal{C}_{W_n}$ and $\mathcal{C}_{Y_{m,p}}$ satisfy the bounds
\begin{align}\label{eqWrem}
\| \mathcal{C}_{W_n} \|_{L^2} +\|\mathcal{C}_{Y_{m,p}}\|_{L^2}\lesssim \e_1^3 {(1+t)}^{-3/2}.
\end{align} 

The quadratic strongly semilinear terms $\mathcal{B}_{W_n}$ have the form
\begin{align}
\label{eqWSsemi}
\mathcal{B}_{W_n} = \sum_{\iota_1\iota_2\in\{+,-\}}F_{\iota_1\iota_2}^{n}[U_{\iota_1},U_{\iota_2}] ,
\end{align}
where $U_+:=U$, $U_-=\overline{U}$, and the symbols $f=f_{\iota_1\iota_2}^n$ of the bilinear operators $F_{\iota_1\iota_2}^{n}$ satisfy
\begin{align}
\label{eqWsymbols}
{\| f^{k, k_1, k_2} \|}_{S^\infty} 
  \lesssim 2^{(3n/2-1)\max(k_1,k_2,0)}(1+2^{3\min(k_1,k_2)}).
\end{align}
The quadratic strongly semilinear terms $\mathcal{B}_{Y_{m,p}}$ have the form
\begin{align}
\label{eqWSsemi2}
\mathcal{B}_{Y_{m,p}} = \sum_{\iota_1\iota_2\in\{+,-\}}\big\{G_{\iota_1\iota_2}^{m,p}[U_{\iota_1},\Omega^pU_{\iota_2}]+\sum_{p_1+p_2\leq p,\max(p_1,p_2)\leq p-1}H_{\iota_1\iota_2}^{m,p,p_1,p_2}[\Omega^{p_1}U_{\iota_1},\Omega^{p_2}U_{\iota_2}]\big\},
\end{align}
where the symbols $g=g_{\iota_1\iota_2}^{m,p}$ and $h=h_{\iota_1\iota_2}^{m,p,p_1,p_2}$ of the operators $G_{\iota_1\iota_2}^{m,p}$ and $H_{\iota_1\iota_2}^{m,p,p_1,p_2}$ satisfy
\begin{align}
\label{eqWsymbols2}
\begin{split}
&{\| g^{k, k_1, k_2} \|}_{S^\infty} 
  \lesssim 2^{(3m/2-1)\max(k_1,k_2,0)}(1+2^{3\min(k_1,k_2)}),\\
&{\| h^{k, k_1, k_2} \|}_{S^\infty} 
  \lesssim 2^{(3m/2+1)\max(k_1,k_2,0)}(1+2^{\min(k_1,k_2)}).
	\end{split}
\end{align}
\end{proposition}

We remark that we have slightly worse information on the quadratic terms $\mathcal{B}_{Y_{m,p}}$ than on the quadratic terms $\mathcal{B}_{W_n}$. This is due mainly to the commutator of the operators $\Omega^p$ and $T_{V\cdot\zeta}$, which leads to the additional terms in \eqref{eqWSsemi2}. These terms can still be regarded as strongly semilinear because they do not contain the maximum number of $\Omega$ derivatives (they do contain, however, $2$ extra Sobolev derivatives, but this is acceptable due to our choice of $N_0$ and $N_1$).

\begin{proof}[Proof of Proposition \ref{proeqW}] In this proof we need to expand the definition of our main spaces $\mathcal{O}_{m,p}$ to exponents $p<-N_3$. More precisely, we define, for any $t\in[0,T]$,
\begin{equation}\label{bin2.4}
\begin{split}
&\|f\|_{\mathcal{O}'_{m,p}}:=\|f\|_{\mathcal{O}_{m,p}}\quad\text{ if }p\geq -N_3,\\
&\|f\|_{\mathcal{O}'_{m,p}}:=\langle t\rangle^{(m-1)(5/6-20\delta^2)-\delta^2}\big[\| f\|_{H^{N_0+p}}+\langle t\rangle^{5/6-2\delta^2}\|f\|_{\widetilde{W}^{N_2+p}}\big]\quad\text{ if }p<-N_3,
\end{split}
\end{equation}
compare with \eqref{fw1}. As in Lemmas \ref{lemmaaux2} and \ref{lemmaaux3}, we have the basic imbeddings
\begin{equation}\label{bin2.5}
T_a\mathcal{O}'_{m,p}\subseteq \mathcal{O}'_{m+m_1,p-l_1},\quad (T_aT_b-T_{ab})\mathcal{O}'_{m,p}\subseteq \mathcal{O}'_{m+m_1+m_2,p-l_1-l_2+1},
\end{equation}
if $a\in\mathcal{M}^{l_1,m_1}_{20}$ and $b\in\mathcal{M}^{l_2,m_2}_{20}$. In particular, recalling that, see \eqref{ExpSigma},
\begin{align}
\label{proeqW1}
\begin{split}
\Sigma - \Lambda \in \varep_1\mathcal{M}^{3/2,1}_{N_3-2}\qquad 
  \Sigma - \Lambda-\Sigma_1 \in \varep_1^2\mathcal{M}^{3/2,2}_{N_3-2},
\end{split}
\end{align}
it follows from \eqref{bin2.5} that, for any $n\in[0,2N_0/3]$,
\begin{equation}\label{Sigma-Lambda}
\begin{split}
T_\Sigma^nU\in\varep_1\mathcal{O}'_{1,-3n/2},\qquad T_\Sigma^nU-\Lambda^nU=\sum_{l=0}^{n-1}\Lambda^{n-1-l}(T_{\Sigma-\Lambda})T_\Sigma^lU\in\varep_1^2\mathcal{O}'_{2,-3n/2}.
\end{split}
\end{equation}

{\bf{Step 1.}} For $n\in[0,2N_0/3]$, we prove first that the function $W_n=(T_\Sigma)^nU$ satisfies
\begin{equation}\label{InducWp}
\begin{split}
(\partial_t+iT_\Sigma +iT_{V\cdot\zeta})W_{n} & = T_{\gamma}(c_{n}W_{n} + d_{n}\overline{W_{n}}) + \mathcal{N}_{S,n} + \varep_1^3\mathcal{O}'_{3,-3n/2},\\
\mathcal{N}_{S,n}&=\sum_{\iota_1,\iota_2\in\{+,-\}}B_{\iota_1\iota_2}^{n}[U_{\iota_1},U_{\iota_2}]\in\varep_1^2\mathcal{O}'_{2,-3n/2+1},\\
\Vert (b_{\iota_1\iota_2}^{n})^{k,k_1,k_2}\Vert_{S^\infty_\Omega}
&\lesssim (1+2^{3\min(k_1,k_2)})\cdot (1+2^{\max(k_1,k_2)})^{3n/2-1}.
\end{split}
\end{equation}
Indeed, the case $n=0$ follows from Proposition \ref{proeqU}. Assuming that this is true for some $n<2N_0/3$ and applying $T_\Sigma$, we find that
\begin{equation*}
\begin{split}
(\partial_t+iT_\Sigma +iT_{V\cdot\zeta})W_{n+1}&=T_{\gamma}(c_{n}W_{n+1} + d_{n}\overline{W_{n+1}})+i[T_{V\cdot\zeta},T_{\Sigma}]W_{n}+ [\partial_t,T_{\Sigma}]W_n\\
&+[T_\Sigma,T_\gamma](c_nW_n+d_n\overline{W_n})+T_\Sigma\mathcal{N}_{S,n}+\varep_1^3T_\Sigma \mathcal{O}'_{3,-3n/2}.
\end{split}
\end{equation*}
Using \eqref{bin2.5}--\eqref{Sigma-Lambda} and \eqref{SymApprox2} it follows that
\begin{equation*}
[\partial_t,T_\Sigma]W_n=T_{\partial_t\Sigma_1}\Lambda^nU+\varep_1^3\mathcal{O}'_{3,-3(n+1)/2},\qquad T_\Sigma\mathcal{N}_{S,n}=\Lambda\mathcal{N}_{S,n}+\varep_1^3\mathcal{O}'_{3,-3(n+1)/2},
\end{equation*}
and, using also \eqref{gamprop},
\begin{equation*}
\begin{split}
&[T_\Sigma,T_\gamma](c_nW_n+d_n\overline{W_n})=[T_\Lambda,T_\gamma](c_n\Lambda^nU+d_n\Lambda^n\overline{U})+\varep_1^3\mathcal{O}'_{3,-3(n+1)/2},\\
&[T_{V\cdot\zeta},T_{\Sigma}]W_{n}=[T_{V_1\cdot\zeta},T_{\Lambda}]W_{n}+\varep_1^3\mathcal{O}'_{3,-3(n+1)/2}=\frac{3i}{2}T_\gamma W_{n+1}+\mathcal{N}'(\Im U,\Lambda^nU)+\varep_1^3\mathcal{O}'_{3,-3(n+1)/2},
\end{split}
\end{equation*}
where $\mathcal{N}'(\Im U,\Lambda^nU)$ is an acceptable strongly semilinear quadratic term as in \eqref{InducWp}. Since $\partial_th=|\nabla|\omega+\varep_1^2\mathcal{O}_{2,-1/2}$, and recalling the formulas \eqref{ExpSigma} and \eqref{unkrelations}, it is easy to see that all the remaining quadratic terms are of the strongly semilinear type described in \eqref{InducWp}. This completes the induction step.

{\bf{Step 2.}} We can now prove the proposition. The claims for $W_n$ follow directly from \eqref{InducWp}. It remains to prove the claims for the functions $Y_{m,p}$. Assume $m\in[0,2N_3/3]$ is fixed. We start from the identity \eqref{InducWp} with $n=m$, and apply the rotation vector field $\Omega$. Clearly
\begin{equation*}
\begin{split}
(\partial_t+iT_\Sigma +iT_{V\cdot\zeta})Y_{m,p}&=T_{\gamma}(c_{m}Y_{m,p} + d_{m}\overline{Y_{m,p}}) + \Omega^p\mathcal{N}_{S,m} + \varep_1^3\Omega^p\mathcal{O}_{3,-3m/2}\\
&-i[\Omega^p,T_{\Sigma}]W_m-i[\Omega^p,T_{V\cdot\zeta}]W_m+[\Omega^p,T_{\gamma}](c_{m}W_m + d_{m}\overline{W_m}).
\end{split}
\end{equation*}
The terms in the first line of the right-hand side are clearly acceptable. It remains to show that the commutators in the second line can also be written as strongly semilinear quadratic terms and cubic terms. Indeed, for $\sigma\in\{\Sigma,V\cdot\zeta,\gamma\}$ and $W\in\{W_m,\overline{W_m}\}$,
\begin{equation}\label{bin2.6}
[\Omega^p,T_{\sigma}]W=\sum_{p'=0}^{p-1}c_{p,p'}T_{\Omega_{x,\zeta}^{p-p'}\sigma}\Omega^{p'}W,
\end{equation}
as a consequence of \eqref{Alu3}. In view of \eqref{Sigma-Lambda},
\begin{equation}\label{bin2.7}
\begin{split}
&\|\Omega^{N_1}W_m\|_{L^2}+\|\langle\nabla\rangle^{N_0-N_3}W_m\|_{L^2}\lesssim \varep_1\langle t\rangle^{\delta^2},\\
&\|\Omega^{N_1}(W_m-\Lambda^mU)\|_{L^2}+\|\langle\nabla\rangle^{N_0-N_3}(W_m-\Lambda^m)U\|_{L^2}\lesssim \varep_1^2\langle t\rangle^{21\delta^2-5/6}.
\end{split}
\end{equation}
and, for $q\in[0,N_1/2]$
\begin{equation}\label{bin2.8}
\|\Omega^{q}W_m\|_{\widetilde{W}^3}\lesssim \varep_1\langle t\rangle^{3\delta^2-5/6},
\qquad \|\Omega^{q}(W_m-\Lambda^mU)\|_{\widetilde{W}^3}\lesssim \varep_1^2\langle t\rangle^{23\delta^2-5/3}.
\end{equation}
By interpolation, and using the fact that $N_0-N_3\geq 3N_1/2$, it follows from \eqref{bin2.7} that
\begin{equation}\label{bin2.9}
\|\Omega^{q}\langle\nabla\rangle^{3/2}W_m\|_{L^2}\lesssim \varep_1\langle t\rangle^{\delta^2},
\qquad\|\Omega^{q}\langle\nabla\rangle^{3/2}(W_m-\Lambda^mU)\|_{L^2}\lesssim \varep_1^2\langle t\rangle^{21\delta^2-5/6}
\end{equation}
for $q\in[0,N_1-1]$. Moreover, for $\sigma\in\{\Sigma,V\cdot\zeta,\gamma\}$ and $q\in[1,N_1]$, we have
\begin{equation}\label{bin2.10}
\|\langle\zeta\rangle^{-3/2}\Omega_{x,\zeta}^{q}\sigma\|_{\mathcal{M}_{20,2}}\lesssim \varep_1\langle t\rangle^{2\delta^2},
\qquad\|\langle\zeta\rangle^{-3/2}\Omega_{x,\zeta}^{q}(\sigma-\sigma_1)\|_{\mathcal{M}_{20,2}}\lesssim \varep_1^2\langle t\rangle^{22\delta^2-5/6},
\end{equation}
while for $q\in[1,N_1/2]$ we also have
\begin{equation}\label{bin2.11}
\|\langle\zeta\rangle^{-3/2}\Omega_{x,\zeta}^{q}\sigma\|_{\mathcal{M}_{20,\infty}}\lesssim \varep_1\langle t\rangle^{4\delta^2-5/6},
\qquad\|\langle\zeta\rangle^{-3/2}\Omega_{x,\zeta}^{q}(\sigma-\sigma_1)\|_{\mathcal{M}_{20,\infty}}\lesssim \varep_1^2\langle t\rangle^{24\delta^2-5/3}.
\end{equation}
See \eqref{SymNorm1} for the definition of the norms ${\mathcal{M}_{20,q}}$. In these estimates $\sigma_1$ denotes the linear part of $\sigma$, i.e. $\sigma_1\in\{\Sigma_1,V_1\cdot\zeta,\gamma\}$. Therefore, using Lemma \ref{lemmaaux2} and \eqref{bin2.8}--\eqref{bin2.11},
\begin{equation*}
T_{\Omega_{x,\zeta}^{p-p'}\sigma}\Omega^{p'}W=T_{\Omega_{x,\zeta}^{p-p'}\sigma}\Omega^{p'}\Lambda^mU_{\pm}+\varep_1^3\langle t\rangle^{-8/5}L^2=T_{\Omega_{x,\zeta}^{p-p'}\sigma_1}\Omega^{p'}\Lambda^mU_{\pm}+\varep_1^3\langle t\rangle^{-8/5}L^2,
\end{equation*} 
for $p'\in[0,p-1]$ and $W\in\{W_m,\overline{W_m}\}$. Notice that $T_{\Omega_{x,\zeta}^{p_1}\sigma_1}\Omega^{p_2}\Lambda^mU_{\pm}$ can be written as $H^{m,p,p_1,p_2}_{\iota_1\iota_2}[\Omega^{p_1}U_{\iota_1},\Omega^{p_2}U_{\iota_2}]$, with symbols as in \eqref{eqWsymbols2}, up to acceptable cubic terms (the loss of $1$ high derivative comes from the case $\sigma_1=V_1\cdot\zeta$). The conclusion of the proposition follows. 
\end{proof}

\section{Energy estimates, I: setup and the main $L^2$ lemma}\label{secEE1}

In this section we set up the proof of Proposition \ref{MainBootstrapEn} and collect some of the main ingredients needed in the proof. From now on we set $g=1$ and $\sigma=1$. With $W_n$ and $Y_{m,p}$ as in \eqref{defW_n}, we define our main energy functional
\begin{equation}\label{bin3.1}
\mathcal{E}_{tot}:=\frac{1}{2}\sum_{0\leq n \leq 2N_0/3}\Vert W_n\Vert_{L^2}^2+\frac{1}{2}\sum_{0 \leq m\leq 2N_3/3}\sum_{0 \leq p \leq N_1}\Vert Y_{m,p}\Vert_{L^2}^2.
\end{equation}

We start with a proposition:

\begin{proposition}\label{ProBulk}
Assume that \eqref{bin2} holds. Then 
\begin{equation}\label{bin2.1}
\Vert \mathcal{U}(t)\Vert_{H^{N_0}\cap H^{N_1,N_3}_\Omega}^2 \lesssim \mathcal{E}_{tot}(t)+\varep_1^3, \qquad \mathcal{E}_{tot}(t)\lesssim \|\mathcal{U}(t)\Vert_{H^{N_0}\cap H^{N_1,N_3}_\Omega}^2 +\varep_1^3,
\end{equation}
where $\mathcal{U}(t)=\langle\nabla\rangle h(t)+i|\nabla|^{1/2}\phi(t)$ as in Proposition \ref{MainBootstrapEn}. Moreover
\begin{equation}\label{bin2.15}
\frac{d}{dt}\mathcal{E}_{tot}=\mathcal{B}_{0} + \mathcal{B}_1 + \mathcal{B}_E, \qquad\vert \mathcal{B}_E(t)\vert\lesssim \varep_1^3(1+t)^{-4/3}.
\end{equation}
The (bulk) terms $\mathcal{B}_0$ and $\mathcal{B}_1$ are finite sums of the form
\begin{equation}\label{bin3.5}
\mathcal{B}_l(t):= \sum_{G\in\mathcal{G},\,W,W'\in\mathcal{W}_l}\iint_{\mathbb{R}^2\times\mathbb{R}^2}\mu_l(\xi,\eta)\widehat{G}(\xi-\eta)
  \widehat{W}(\eta)\widehat{W'}(-\xi)\,d\xi d\eta,
\end{equation}
where $U$ and $\Sigma$ are defined as in Proposition \ref{proeqU}, $U_+:=U$, $U_-:=\overline{U}$, and
\begin{equation}\label{bin3.6}
\begin{split}
\mathcal{G}&:=\{\Omega^a\langle\nabla\rangle^bU_{\pm}:\, a\leq N_1/2\text{ and }b\leq N_3+2\},\\
\mathcal{W}_0&:=\{\Omega^aT_\Sigma^mU_{\pm}:\,\text{ either }(a=0\text{ and }m\leq 2N_0/3)\,\,\text{ or }\,\,(a\leq N_1\text{ and }m\leq 2N_3/3)\},\\
\mathcal{W}_1&:=\mathcal{W}_0\cup \{(1-\Delta)\Omega^aT_\Sigma^mU_{\pm}:\,a\leq N_1-1\text{ and }m\leq 2N_3/3)\}.
\end{split}
\end{equation}
The symbols $\mu_l=\mu_{l;(G,W,W')}$, $l\in\{0,1\}$, satisfy
\begin{equation}\label{bin3.7}
\begin{split}
&\mu_0(\xi,\eta)=c|\xi-\eta|^{3/2}\mathfrak{d}(\xi,\eta),\qquad \mathfrak{d}(\xi,\eta)
  := \chi\Big(\frac{|\xi-\eta|}{|\xi+\eta|}\Big)\Big(\frac{\xi-\eta}{\vert\xi-\eta\vert}\cdot\frac{\xi+\eta}{\vert\xi+\eta\vert}\Big)^2,\,c\in\mathbb{C},\\
&\Vert \mu_1^{k,k_1,k_2}\Vert_{S^\infty} \lesssim 2^{-\max(k_1,k_2,0)}2^{3k_1^+},
\end{split}
\end{equation}
for any $k,k_1,k_2\in\mathbb{Z}$, see definitions \eqref{defclassA}--\eqref{mloc}.
\end{proposition}

Notice that the a priori energy estimates we prove here are stronger than 
standard energy estimates. The terms $\mathcal{B}_0,\mathcal{B}_1$ are strongly semilinear terms, in the sense that they either gain one derivative or contain
the depletion factor $\mathfrak{d}$  (which effectively gains one derivative when the modulation is small, compare with \eqref{corre}).

\begin{proof}[Proof of Proposition \ref{ProBulk}]
The bound \eqref{bin2.1} follows from \eqref{Uhorel} and \eqref{Sigma-Lambda},
\begin{equation*}
\Vert \langle\nabla\rangle h(t)\Vert_{H^{N_0}\cap H^{N_1,N_3}_\Omega}^2 + \Vert\vert\nabla\vert^{1/2}\phi(t)\Vert_{H^{N_0}\cap H^{N_1,N_3}_\Omega}^2\lesssim \Vert  U(t)\Vert_{H^{N_0}\cap H^{N_1,N_3}_\Omega}^2+\varep_1^3\lesssim \mathcal{E}_{tot}(t)+\varep_1^3.
\end{equation*}

To prove the remaining claims we start from \eqref{eqW} and \eqref{eqUR}. For the terms $W_n$ we have
\begin{equation}\label{bin3.3}
\frac{d}{dt}\frac{1}{2}\Vert W_n\Vert_{L^2}^2=\Re\langle T_\gamma (c_n W_n + d_n\overline{W_n} ),W_n\rangle+\Re \langle\mathcal{B}_{W_n},W_n\rangle + \Re\langle \mathcal{C}_{W_n},W_n\rangle,
\end{equation}
since, as a consequence of Lemma \ref{PropSym} (ii), 
\begin{equation*}
\Re\langle i T_\Sigma W_n+iT_{V\cdot\zeta}W_n,W_n\rangle=0.
\end{equation*}
Clearly, $|\langle \mathcal{C}_{W_n},W_n\rangle|\lesssim \varep_1^3\langle t\rangle^{-3/2+2\delta^2}$, so the last term can be placed in $\mathcal{B}_E(t)$. Moreover, using \eqref{DefGamma} and the definitions, $\langle T_\gamma (c_n W_n + d_n\overline{W_n} ),W_n\rangle$ can be written in the Fourier space as part of the term $\mathcal{B}_0(t)$ in \eqref{bin3.5}. 

Finally, $\langle\mathcal{B}_{W_n},W_n\rangle $ can be written in the Fourier space as part of the term $\mathcal{B}_1(t)$ in \eqref{bin3.5} plus acceptable errors. Indeed, given a symbol $f$ as in \eqref{eqWsymbols}, one can write 
\begin{equation*}
f(\xi,\eta)=\mu_1(\xi,\eta)\cdot [(1+\Lambda(\xi-\eta)^n)+(1+\Lambda(\eta)^n)],\qquad \mu_1(\xi,\eta):=\frac{f(\xi,\eta)}{2+\Lambda(\xi-\eta)^n+\Lambda(\eta)^n}.
\end{equation*}
The symbol $\mu_1$ satisfies the required estimate in \eqref{bin3.7}. The factors $1+\Lambda(\xi-\eta)^n$ and $1+\Lambda(\eta)^n$ can be combined with 
the functions $\widehat{U_{\iota_1}}(\xi-\eta)$ and $\widehat{U_{\iota_2}}(\eta)$ respectively. Recalling that 
$\Lambda^nU-W_n\in\varep_1^2\mathcal{O}'_{2,-3n/2}$, see \eqref{Sigma-Lambda}, the desired representation \eqref{bin3.5} follows, up to acceptable errors.

The analysis of the terms $Y_{m,p}$ is similar, using \eqref{eqWSsemi2}-\eqref{eqWsymbols2}. This completes the proof.
\end{proof}

In view of \eqref{bin2.1}, to prove Proposition \ref{MainBootstrapEn} it suffices to prove that $|\mathcal{E}_{tot}(t)-\mathcal{E}_{tot}(0)|\lesssim \varep_1^3\langle t\rangle^{2\delta^2}$ for any $t\in[0,T]$. In view of \eqref{bin2.15} it suffices to prove that, for $l\in\{0,1\}$,
\begin{equation*}
\Big|\int_0^t\mathcal{B}_l(s)\,ds\Big|\lesssim \varep_1^3(1+t)^{2\delta^2},
\end{equation*}
for any $t\in[0,T]$. Given $t\in[0,T]$, we fix a suitable decomposition of the function $\mathbf{1}_{[0,t]}$, i.e. we fix functions $q_0,\ldots,q_{L+1}:\mathbb{R}\to[0,1]$, $|L-\log_2(2+t)|\leq 2$, with the properties
\begin{equation}\label{nh2}
\begin{split}
&\mathrm{supp}\,q_0\subseteq [0,2], \quad \mathrm{supp}\,q_{L+1}\subseteq [t-2,t],\quad\mathrm{supp}\,q_m\subseteq [2^{m-1},2^{m+1}]\text{ for } m\in\{1,\ldots,L\},\\
&\sum_{m=0}^{L+1}q_m(s)=\mathbf{1}_{[0,t]}(s),\qquad q_m\in C^1(\mathbb{R})\text{ and }\int_0^t|q'_m(s)|\,ds\lesssim 1\text{ for }m\in \{1,\ldots,L\}.
\end{split}
\end{equation}
It remains to prove that for $l\in\{0,1\}$ and $m\in\{0,\ldots,L+1\}$,
\begin{equation}\label{njk5}
\Big|\int_{\mathbb{R}}\mathcal{B}_l(s)q_m(s)\,ds\Big|\lesssim \varep_1^32^{2\delta^2m}.
\end{equation}

In order to be able to use the hypothesis $\|\mathcal{V}(s)\|_Z\leq\varep_1$ (see \eqref{Ama32}) we need to modify slightly the functions $G$ that appear in the terms $\mathcal{B}_l$. More precisely, we define
\begin{equation}\label{njk8}
\mathcal{G}':=\{\Omega^a\langle\nabla\rangle^b\mathcal{U}_{\iota}:\, \iota\in\{+,-\},\,a\leq N_1/2\text{ and }b\leq N_3+2\},
\end{equation} 
where $\mathcal{U}=\langle\nabla\rangle h+i|\nabla|^{1/2}\phi$, $\mathcal{U}_+=\mathcal{U}$ and $\mathcal{U}_-=\overline{\mathcal{U}}$. Then we define the modified bilinear terms
\begin{equation}\label{njk9}
\mathcal{B}'_l(t):=\sum_{G\in\mathcal{G}',\,W,W'\in\mathcal{W}_l}\iint_{\mathbb{R}^2\times\mathbb{R}^2}\mu_l(\xi,\eta)\widehat{G}(\xi-\eta,t)\widehat{W}(\eta,t)\widehat{W'}(-\xi,t)\,d\xi d\eta,
\end{equation}
where the sets $\mathcal{W}_0,\mathcal{W}_1$ are as in \eqref{bin3.6}, and the symbols $\mu_0$ and $\mu_1$ are as in \eqref{bin3.7}. In view of \eqref{Uhorel}, $U(t)-\mathcal{U}(t)\in\varep_1^2\mathcal{O}_{2,0}$. Therefore, simple estimates as in the proof of Lemma \ref{lemmaaux1} show that
\begin{equation*}
|\mathcal{B}_l(s)|\lesssim\varep_1^3(1+s)^{-4/5},\qquad |\mathcal{B}_l(s)-\mathcal{B}'_l(s)|\lesssim\varep_1^3(1+s)^{-8/5}.
\end{equation*}
As a result of these reductions, for Proposition \ref{MainBootstrapEn} it suffices to prove the following:

\begin{proposition}\label{EEMainProp}
Assume that $(h,\phi)$ is a solution of the system \eqref{WW0} with $g=1$, $\sigma=1$ on $[0,T]$, and let $\mathcal{U}=\langle\nabla\rangle h+i|\nabla|^{1/2}\phi$, $\mathcal{V}(t)=e^{it\Lambda}\mathcal{U}(t)$. Assume that 
\begin{align}
\label{EEMainPropapriori}
\langle t\rangle^{-\delta^2}\|\mathcal{U}(t)\|_{H^{N_0}\cap H^{N_1,N_3}_\Omega}+\|\mathcal{V}(t)\|_{Z}\leq\varep_1,
\end{align}
for any $t\in[0,T]$, see \eqref{Ama32}. Then, for any $m\in[\mathcal{D}^2,L]$ and $l\in\{0,1\}$, 
\begin{equation}\label{EEMainPropbound}
\Big| \int_{\R}\iint_{\mathbb{R}^2\times\mathbb{R}^2} q_m(s) \mu_l(\xi,\eta)\widehat{G}(\xi-\eta,s)\widehat{W}(\eta,s)\widehat{W'}(-\xi,s)\,d\xi d\eta ds\Big| \lesssim \varep_1^32^{2\delta^2m},
\end{equation}
where $G\in\mathcal{G}'$ (see \eqref{njk8}), and $W,W'\in\mathcal{W}':=\mathcal{W}_1$ (see \eqref{bin3.6}), and $q_m$ are as in \eqref{nh2}. The symbols 
$\mu_0,\mu_1$ satisfy the bounds (compare with \eqref{bin3.7}) 
\begin{equation}\label{njk14}
\begin{split}
&\mu_0(\xi,\eta)=|\xi-\eta|^{3/2}\mathfrak{d}(\xi,\eta),\qquad \mathfrak{d}(\xi,\eta)
  := \chi\Big(\frac{|\xi-\eta|}{|\xi+\eta|}\Big)\Big(\frac{\xi-\eta}{\vert\xi-\eta\vert}\cdot\frac{\xi+\eta}{\vert\xi+\eta\vert}\Big)^2,\\
&\Vert \mu_1^{k,k_1,k_2}\Vert_{S^\infty} \lesssim 2^{-\max(k_1,k_2,0)}2^{3k_1^+}.
\end{split}
\end{equation}
\end{proposition}

The proof of this proposition will be done in several steps. We remark that both the symbols $\mu_0$ and $\mu_1$ introduce certain {\it{strongly semilinear}} structures. The symbols $\mu_0$ contain the depletion factor $\mathfrak{d}$, which counts essentially as a gain of one high derivative in resonant situations. The symbols $\mu_1$ clearly contain a gain of one high derivative.  

We will need to further subdivide the expression in \eqref{EEMainPropbound} into the contributions of ``good frequencies'' with optimal decay and the ``bad frequencies'' with slower decay. Let
\begin{align}\label{njk20}
\chi_{ba}(x) &:= \varphi(2^{\mathcal{D}}(\vert x\vert-\gamma_0)) + \varphi(2^{\mathcal{D}}(\vert x\vert-\gamma_1)),
  \qquad \chi_{go}(x) := 1-\chi_{ba}(x),
\end{align}
where $\gamma_0=\sqrt{\frac{2\sqrt{3}-3}{3}}$ is the radius of the sphere of degenerate frequencies, and $\gamma_1=\sqrt{2}$ is the radius of the sphere of space-time resonances. 
We then define for $l\in\{0,1\}$ and $Y\in\{go,ba\}$,
\begin{align}\label{DecBbasic}
\mathcal{A}^l_{Y}[F;H_1,H_2] := \iint_{\R^2 \times \R^2}\mu_l(\xi,\eta)
  \chi_Y(\xi-\eta)\widehat{F}(\xi-\eta)\widehat{H_1}(\eta)\widehat{H_2}(-\xi)\,d\xi d\eta.
\end{align}

In the proof of \eqref{EEMainPropbound} we will need to distinguish between functions $G$ and $W$ that originate from $U=U_+$ and functions that originate from $\overline{U}=U_-$. For this we define, for $\iota\in\{+,-\}$,
\begin{equation}\label{njk21}
\mathcal{G}'_\iota:=\{\Omega^a\langle\nabla\rangle^b\mathcal{U}_{\iota}:\,a\leq N_1/2\text{ and }b\leq N_3+2\},
\end{equation}
and
\begin{equation}\label{njk22}
\begin{split}
\mathcal{W}'_\iota:=\{\langle\nabla\rangle^c&\Omega^aT_{\Sigma}^mU_{\iota}:\,\text{ either }(a=c=0\text{ and }m\leq 2N_0/3)\\
&\text{ or }(c\in\{0,2\},\,c/2+a\leq N_1,\text{ and }m\leq 2N_3/3)\}.
\end{split}
\end{equation}

\subsection{Some lemmas}\label{secAuxenergy} In this subsection we collect some lemmas that are used often in the proofs in the next section. We will often use the Schur's test:

\begin{lemma}[Schur's lemma]\label{ShurLem}
Consider the operator $T$ given by
\begin{equation*}
Tf(\xi)=\int_{\mathbb{R}^2}K(\xi,\eta)f(\eta)d\eta.
\end{equation*}
Assume that
\begin{equation*}
\sup_\xi\int_{\mathbb{R}^2}\vert K(\xi,\eta)\vert d\eta\le K_1,\qquad \sup_\eta\int_{\mathbb{R}^2}\vert K(\xi,\eta)\vert d\xi\le K_2.
\end{equation*}
Then
\begin{equation*}
\Vert Tf\Vert_{L^2}\lesssim \sqrt{K_1K_2}\Vert f\Vert_{L^2}.
\end{equation*}
\end{lemma}

We will also use a lemma about functions in $\mathcal{G}'_+$ and $\mathcal{W}'_+$.

\begin{lemma}\label{GWprop}
(i) Assume $G\in\mathcal{G}'_+$, see \eqref{njk21}. Then
\begin{equation}\label{njk30}
\sup_{|\alpha|+2a\leq 30}\|D^\alpha\Omega^a[e^{it\Lambda}G(t)]\|_{Z_1}\lesssim\varep_1,\qquad \|G(t)\|_{H^{N_1-2}\cap H_\Omega^{N_1/2-1,0}}\lesssim\varep_1\langle t\rangle^{\delta^2},
\end{equation}
for any $t\in[0,T]$. Moreover, $G$ satisfies the equation
\begin{equation}\label{njk31}
(\partial_t+i\Lambda)G=\mathcal{N}_G,\qquad \|\mathcal{N}_G(t)\|_{H^{N_1-4}\cap H_\Omega^{N_1/2-2,0}}\lesssim\varep_1^2\langle t\rangle^{-5/6+\delta}.
\end{equation}

(ii) Assume $W\in\mathcal{W}'_+$, see \eqref{njk22}. Then
\begin{equation}\label{njk33}
\|W(t)\|_{L^2}\lesssim\varep_1\langle t\rangle^{\delta^2},
\end{equation}
for any $t\in[0,T]$. Moreover, $W$ satisfies the equation
\begin{equation}\label{njk34}
(\partial_t+i\Lambda)W=\mathcal{Q}_W+\mathcal{E}_W,
\end{equation}
where, with $\Sigma_{\geq 2}:=\Sigma-\Lambda-\Sigma_1\in\varep_1^2\mathcal{M}^{3/2,2}_{N_3-2}$ as in \eqref{ExpSigma},
\begin{equation}\label{njk34.1}
\mathcal{Q}_W=-iT_{\Sigma_{\geq 2}}W-iT_{V\cdot\zeta}W,\qquad\|\langle\nabla\rangle^{-1/2}\mathcal{E}_W\|_{L^2}\lesssim \varep_1^2\langle t\rangle^{-5/6+\delta}.
\end{equation}
Using Lemma \ref{PropSym} we see that for all $k\in\mathbb{Z}$ and $t\in[0,T]$
\begin{align}\label{BoundsQuasiTerms}
\begin{split}
\| (P_kT_{V\cdot\zeta}W)(t)\Vert_{L^2}&\lesssim \varep_12^{k^+} \langle t\rangle^{-5/6+\delta}\|P_{[k-2,k+2]}W(t)\|_{L^2},\\
\| (P_{k}T_{\Sigma_{\geq 2}}W)(t)\Vert_{L^2} &\lesssim \varep_1^22^{3k^+/2}\langle t\rangle^{-5/3+\delta}\|P_{[k-2,k+2]}W(t)\|_{L^2}.
\end{split}
\end{align}
\end{lemma}

\begin{proof} The claims in \eqref{njk30} follow from Definition \ref{MainZDef}, the assumptions \eqref{EEMainPropapriori}, and interpolation 
(recall that $N_0-N_3=2N_1$). 
The identities \eqref{njk31} follow from \eqref{defscalarunk}--\eqref{eqU}, since $(\partial_t+i\Lambda)\mathcal{U}\in\varep_1^2\mathcal{O}_{2,-2}$. 
The inequalities \eqref{njk33} follow from \eqref{Sigma-Lambda}. The identities \eqref{njk34}--\eqref{njk34.1} follow from Proposition \ref{proeqW}, 
since all quadratic terms that lose up to $1/2$ derivatives can be placed into $\mathcal{E}_W$. Finally, the bounds \eqref{BoundsQuasiTerms} 
follow from \eqref{LqBdTa} and \eqref{CommFreqP'}. 
\end{proof}

Next we summarize some properties of the linear profiles of the functions in $\mathcal{G}'_+$.

\begin{lemma}\label{LinEstLemNew}
Assume $G\in\mathcal{G}'_+$ as before and let $f=e^{it\Lambda}G$. Recall the operators $Q_{jk}$ and $A_{n,\gamma}, A_{n,\gamma}^{(j)}$ 
defined in \eqref{qjk}--\eqref{aop2}. For $(k,j)\in\mathcal{J}$ and $n\in\{0,\ldots,j+1\}$ let 
\begin{equation*}
f_{j,k}:=P_{[k-2,k+2]}Q_{jk}f,\qquad f_{j,k,n}:=A_{n,\gamma_1}^{(j)}f_{j,k}.
\end{equation*}
Then, if $m \geq 0$, for all $t \in  [2^m-1, 2^{m+1}]$ we have
\begin{align}
\label{Gbound0}
\begin{split}
& \sup_{|\alpha|+2a\leq 30}\|D^\alpha\Omega^a f(t)\|_{Z_1}\lesssim\varep_1,\qquad \|f(t)\|_{H^{N_1-2}\cap H_\Omega^{N_1/2-1,0}}\lesssim\varep_1 2^{\delta^2m},\\
& {\| P_k \partial_t f(t) \|}_{L^2} \lesssim \e_1^2 2^{-8k^+}2^{-5m/6+\delta m},\qquad {\| P_k e^{-it\Lambda} \partial_t f(t) \|}_{L^\infty} \lesssim \e_1^2 2^{-5m/3+\delta m}.
\end{split}
\end{align}
Also, the following $L^\infty$ bounds hold, for any $k\in\mathbb{Z}$ and $s\in\mathbb{R}$ with $|s-t|\leq 2^{m-\delta m}$,
\begin{align}
\label{Gdecay}
\begin{split}
& {\| e^{-is\Lambda} A_{\leq 2\D, \gamma_0} P_k f(t)\|}_{L^\infty} \lesssim \e_1 \min(2^{k/2}, 2^{-4k}) 2^{-m} 2^{52\delta m},
\\
& {\| e^{-is\Lambda} A_{\geq 2\D+1, \gamma_0} P_k f(t) \|}_{L^\infty} \lesssim \e_1 2^{-5m/6 + 3\delta^2m}.
\end{split}
\end{align}
Moreover, we have
\begin{equation}\label{Gbound1}
\begin{split}
&{\| e^{-is\Lambda} f_{j,k}(t) \|}_{L^\infty} \lesssim \varep_1\min(2^k,2^{-4k}) 2^{-j + 50\delta j},\\
&{\| e^{-is\Lambda} f_{j,k}(t) \|}_{L^\infty} \lesssim \varep_1\min(2^{3k/2},2^{-4k}) 2^{-m+ 50\delta j}\quad\text{ if }\quad |k|\geq 10.
\end{split}
\end{equation}
Away from the bad frequencies, we have the stronger bound
\begin{align}
\label{Gdecaygood}
\begin{split}
& {\| e^{-is\Lambda} A_{\leq 2\D, \gamma_0} A_{\leq 2\D, \gamma_1}f_{j,k}(t) \|}_{L^\infty} \lesssim \e_1 2^{-m} \min(2^k, 2^{-4k}) 2^{-j/4},
\end{split}
\end{align}
provided that $j \leq (1-\delta^2)m + |k|/2$ and $|k| + D \leq m/2$.

Finally, for all $n \in \{0,1,\dots, j\}$ we have
\begin{align}
\label{Gbound2}
\begin{split}
& \big\| \widehat{f_{j,k,n}} \big\|_{L^\infty} \lesssim \e_1 2^{2\delta^2m}2^{-4k^+}2^{3 \delta n} \cdot 2^{-(1/2-55\delta)(j-n)},
\\
& \big\| \sup_{\theta \in \mathbb{S}^1} | \widehat{f_{j,k,n}}(r\theta) | \big\|_{L^2(rdr)}
\lesssim \e_1 2^{2\delta^2m}2^{-4k^+}2^{n/2} 2^{-j+55\delta j}.
\end{split}
\end{align}
\end{lemma}

\begin{proof} The estimates in the first line of \eqref{Gbound0} follow from \eqref{njk30}. The estimates \eqref{Gdecay}, \eqref{Gbound1}, 
and \eqref{Gbound2} then follow from Lemma \ref{LinEstLem}, while the estimate \eqref{Gdecaygood} follows from \eqref{NewLinfty}. 
Finally, the estimate on $\partial_t f$ in \eqref{Gbound0} follows from the bound \eqref{Brc2}.
\end{proof}

We prove now a lemma that is useful when estimating multilinear expression containing a localization in the modulation $\Phi$.

\begin{lemma}\label{EstimAKP}
Assume that $k,k_1,k_2\in\mathbb{Z}$, $m \geq \D$, $\overline{k}:=\max(k,k_1,k_2)$, $|k|\leq m/2$, $p \geq -m$. Assume that $\underline{\mu}_0$ and $\underline{\mu}_1$ are symbols supported in the set $\mathcal{D}_{k_2,k,k_1}$ and satisfying
\begin{equation}\label{njk4}
\begin{split}
&\underline{\mu}_0(\xi,\eta) = \mu_0(\xi,\eta)n(\xi,\eta),\qquad\underline{\mu}_1(\xi,\eta) = \mu_1(\xi,\eta)n(\xi,\eta),\qquad \|n\|_{S^\infty}\lesssim 1,\\
&\mu_0(\xi,\eta)=|\xi-\eta|^{3/2}\mathfrak{d}(\xi,\eta),\qquad \|\mu_1(\xi,\eta)\|_{S^\infty}\lesssim 2^{3k^+-\overline{k}^+},
\end{split}
\end{equation}
compare with \eqref{njk14}. For $l\in\{0,1\}$ and $\Phi=\Phi_{\sigma\mu\nu}$ as in \eqref{ph1} let
\begin{equation*}
\mathcal{I}^l_{p}[F;H_1,H_2] = \iint_{(\R^2)^2} \underline{\mu}_l(\xi,\eta) \psi_p(\Phi(\xi,\eta))
  \widehat{P_kF}(\xi-\eta)\widehat{P_{k_1}H_1}(\eta) \widehat{P_{k_2}H_2}(-\xi) d\xi d\eta,
\end{equation*}
where $\psi\in C^\infty_0(-1,1)$ and $\psi_p(x):=\psi(2^{-p}x)$. Then
\begin{equation}\label{SimpleBdAKP2}
\begin{split}
\big| \mathcal{I}^0_{p}[F;H_1,H_2] \big| & \lesssim 2^{3k/2}\min(1,2^{-\overline{k}^+}2^{\max(2p,3k^+)}2^{-2k}) N(P_kF) \cdot\| P_{k_1}H_1 \|_{L^2} \| P_{k_2}H_2 \|_{L^2},
\\
\big| \mathcal{I}^1_{p}[F;H_1,H_2] \big| & \lesssim 2^{3k^+-\overline{k}^+} N(P_kF) \cdot \| P_{k_1}H_1 \|_{L^2} \| P_{k_2}H_2 \|_{L^2},
\end{split}
\end{equation}
where
\begin{equation}\label{No}
N(P_kF):=\sup_{|\rho|\leq 2^{-p} 2^{\delta m}}
  \| e^{i\rho\Lambda}P_kF \|_{L^\infty} + 2^{-10m} \| P_kF \|_{L^2}.
\end{equation}
In particular, if $2^k\approx 1$ then 
\begin{equation}\label{SimpleBdAKP}
\begin{split}
\big| \mathcal{I}^0_{p}[F;H_1,H_2] \big| & \lesssim \min(1,2^{2p^+-\overline{k}^+}) N(P_kF) \cdot\| P_{k_1}H_1 \|_{L^2} \| P_{k_2}H_2 \|_{L^2},
\\
\big| \mathcal{I}^1_{p}[F;H_1,H_2] \big| & \lesssim 2^{-\overline{k}^+} N(P_kF) \cdot \| P_{k_1}H_1 \|_{L^2} \| P_{k_2}H_2 \|_{L^2}.
\end{split}
\end{equation}
\end{lemma}

\begin{proof} The proof when $l=1$ is easy. We start from the formula
\begin{equation}\label{njk4.5}
\psi_p(\Phi(\xi,\eta)) = C\int_{R} \widehat{\psi}(s)e^{is2^{-p}\Phi(\xi,\eta)}\,ds.
\end{equation}
Therefore
\begin{equation*}
\mathcal{I}^1_{p}[F;H_1,H_2]= C\int_{\mathbb{R}}\widehat{\psi}(s)\iint_{(\R^2)^2} 
  e^{is2^{-p}\Phi(\xi,\eta)}\underline{\mu}_1(\xi,\eta)\widehat{P_{k}F}(\xi-\eta)\widehat{P_{k_1}H_1}(\eta) \widehat{P_kH_2}(-\xi) d\xi d\eta.
\end{equation*}
Using Lemma \ref{touse} (i) and \eqref{njk4}, it follows that
\begin{equation*}
|\mathcal{I}^1_{p}[F;H_1,H_2]|\lesssim \int_{\mathbb{R}}|\widehat{\psi}(s)|2^{3k^+-\overline{k}^+}\|e^{-is2^{-p}\Lambda_\mu}P_{k}F\|_{L^\infty}\|P_{k_1}H_1\|_{L^2}\|P_{k_2}H_2\|_{L^2}\,ds.
\end{equation*}
The bound for $l=1$ in \eqref{SimpleBdAKP} follows.

In the case $l=0$, the desired bound follows in the same way unless
\begin{equation}\label{njk4.6}
\overline{k}^++2k\geq \max(2p,3k^+)+\D.
\end{equation}
On the other hand, if \eqref{njk4.6} holds then we need to take advantage of the depletion factor $\mathfrak{d}$. The main point is that if \eqref{njk4.6} holds and
\begin{equation}\label{njk4.7}
\text{ if }\quad|\Phi(\xi,\eta)|\lesssim 2^{p}\quad\text{ then }\quad \mathfrak{d}(\xi,\eta)\lesssim \frac{2^{-\overline{k}}(2^{2p}+2^{3k^+})}{2^{2k}}.
\end{equation}
Indeed, if \eqref{njk4.6} holds then $\overline{k}\geq \D$ and $p\leq 3\overline{k}/2-\D/4$, and we estimate
\begin{equation*}
\mathfrak{d}(\xi,\eta)\lesssim \Big(\frac{|\xi|-|\eta|}{|\xi-\eta|}\Big)^2\lesssim \Big(\frac{2^{-\overline{k}/2}|\lambda(|\xi|)-\lambda(|\eta|)|}{2^{k}}\Big)^2\lesssim \frac{2^{-\overline{k}}(|\Phi(\xi,\eta)|+\lambda(|\xi-\eta|))^2}{2^{2k}}
\end{equation*}
in the support of the function $\mathfrak{d}$, which gives \eqref{njk4.7}. 

To continue the proof, we fix a function $\theta\in C^\infty_0(\mathbb{R}^2)$ supported in the ball of radius $2^{k^++1}$ with the property that $\sum_{v\in(2^{k^+}\mathbb{Z})^2}\theta(x-v)=1$ for any $x\in\mathbb{R}^2$. For any $v\in(2^{k^+}\mathbb{Z})^2$, consider the operator $Q_v$ defined by
\begin{equation*}
\widehat{Q_v f}(\xi)=\theta(\xi-v)\widehat{f}(\xi).
\end{equation*}
In view of the localization in $(\xi-\eta)$, we have
\begin{equation}\label{njk1}
\mathcal{I}^0_{p}[F;H_1,H_2]=\sum_{|v_1+v_2|\lesssim 2^{k^+}}\mathcal{I}^0_{p;v_1,v_2},\qquad \mathcal{I}^0_{p;v_1,v_2}:=\mathcal{I}^0_{p}[F;Q_{v_1}H_1,Q_{v_2}H_2].
\end{equation}
Moreover, using \eqref{njk4.7} we can insert a factor of $\varphi_{\leq\D}(2^{-X}(\xi-\eta)\cdot v_1)$ in the integral defining $\mathcal{I}^l_{p}[F;Q_{v_1}H_1,Q_{v_2}H_2]$ without changing the integral, where $2^X\approx (2^{p}+2^{3k^+/2})2^{\overline{k}/2}$. Let
\begin{equation*}
m_{v_1}(\xi,\eta):=\underline{\mu}_0(\xi,\eta) \cdot\varphi_{[k_2-2,k_2+2]}(\xi)\varphi_{[k-2,k+2]}(\xi-\eta)\varphi_{\leq k^++2}(\eta-v_1)\varphi_{\leq\D}(2^{-X}(\xi-\eta)\cdot v_1).
\end{equation*}

We will show below that for any $v_1\in\mathbb{R}^2$ with $|v_1|\approx 2^{\overline{k}}$
\begin{equation}\label{njk2}
\|\mathcal{F}^{-1}(m_{v_1})\|_{L^1(\mathbb{R}^2\times\mathbb{R}^2)}\lesssim 2^{3k/2}\cdot 2^{2X}2^{-2k}2^{-2\overline{k}}.
\end{equation}
Assuming this, the desired bound follows as in the case $l=1$ treated earlier. To prove \eqref{njk2} we recall that $\|\mathcal{F}^{-1}(ab)\|_{L^1}\lesssim \|\mathcal{F}^{-1}(a)\|_{L^1}\|\mathcal{F}^{-1}(b)\|_{L^1}$. Then we write
\begin{equation*}
(\xi-\eta)\cdot(\xi+\eta)=2(\xi-\eta)\cdot v_1+|\xi-\eta|^2+2(\xi-\eta)\cdot (\eta-v_1).
\end{equation*} 
The bound \eqref{njk2} follows by analyzing the contributions of the 3 terms in this formula.
\end{proof}

Our next lemma concerns a linear $L^2$ estimate on certain localized Fourier integral operators.  

\begin{lemma}\label{L2Prop0}
Assume that $k\ge -100$, $m\geq \D^2$,
\begin{equation}\label{AssL2Prop}
-(1-\delta)m\le p-k/2\le -\delta m,\qquad  2^{m-2}\le \vert s\vert\le 2^{m+2}.
\end{equation}
Given $\chi\in C^\infty_0(\mathbb{R})$ supported in $[-1,1]$, introduce the operator $L_{p,k}$ defined by
\begin{equation}\label{LMainDef}
L_{p,k}f(\xi):=\varphi_{\ge-100}(\xi)\int_{\mathbb{R}^2}e^{is\Phi(\xi,\eta)}\chi(2^{-p}\Phi(\xi,\eta))\varphi_k(\eta)a(\xi,\eta)f(\eta)d\eta,
\end{equation}
where, for some $\mu,\nu\in\{+,-\}$,
\begin{equation}\label{AssL2Prop1}
\begin{split}
&\Phi(\xi,\eta)=\Lambda(\xi)-\Lambda_\mu(\xi-\eta)-\Lambda_\nu(\eta),\qquad a(\xi,\eta)=A(\xi,\eta)\chi_{ba}(\xi-\eta)\widehat{g}(\xi-\eta),\\
&\Vert D^{\alpha}A\Vert_{L^\infty_{x,y}}\lesssim_{|\alpha|} 2^{\vert\alpha\vert m/3},\qquad \Vert g\Vert_{Z_1\cap H^{N_1/3,0}_{\Omega}}\lesssim 1.
\end{split}
\end{equation}
Then
\begin{equation*}
\Vert L_{p,k}f\Vert_{L^2}\lesssim 2^{30\delta m}(2^{(3/2)(p-k/2)}+2^{p-k/2-m/3})\Vert f\Vert_{L^2}.
\end{equation*}
\end{lemma}

\begin{remark}\label{L2Rem}
(i) Lemma \ref{L2Prop0}, which is proved in section \ref{L2proof} below, plays a central role in the proof of Proposition \ref{EEMainProp}. A key role in its
proof is played by the ``curvature'' component
\begin{align}
\label{defY}
\begin{split}
&\Upsilon(\xi,\eta) := (\nabla^2_{\xi,\eta}\Phi)(\xi,\eta) \big[(\nabla^\perp_\xi\Phi)(\xi,\eta),(\nabla^\perp_\eta\Phi)(\xi,\eta)\big],
\end{split}
\end{align}
and in particular by its non-degeneracy close to the bad frequencies $\gamma_0$ and $\gamma_1$, and to the resonant hypersurface $\Phi(\xi,\eta) = 0$.
The properties of $\Upsilon$ that we are going to use are described in section \ref{upsilon}, and in particular in Lemmas \ref{Geomgamma0}, \ref{lemmaD1},
 and \ref{lemmaD2}.

(ii) We can insert $S^\infty$ symbols and bounded factors that depend only on $\xi$ or only on $\eta$ in the integral in \eqref{LMainDef}, without changing the conclusion. We will often use this lemma in the form
\begin{equation}\label{L2Use}
\begin{split}
\Big|\iint_{\mathbb{R}^2\times\mathbb{R}^2}e^{is\Lambda(\xi-\eta)}&\chi(2^{-p}\Phi(\xi,\eta))\mu(\xi,\eta)a(\xi,\eta)\widehat{P_{k_1}F_1}(\eta)\widehat{P_{k}F_2}(-\xi)\,d\xi d\eta\Big|\\
&\lesssim 2^{30\delta m}(2^{(3/2)(p-k/2)}+2^{p-k/2-m/3})\Vert P_{k_1}F_1\Vert_{L^2}\|P_kF_2\|_{L^2},
\end{split}
\end{equation}
provided that $k,k_1\geq -80$, \eqref{AssL2Prop} and \eqref{AssL2Prop1} hold, and $\|\mu\|_{S^\infty}\lesssim 1$. This follows by writing
\begin{equation*}
\mu(\xi,\eta)=\iint_{\mathbb{R}^2\times\mathbb{R}^2}P(x,y)e^{-ix\cdot\xi}e^{-iy\cdot\eta}\,d\xi d\eta,
\end{equation*}
with $\|P\|_{L^1(\mathbb{R}^2\times\mathbb{R}^2)}\lesssim 1$, and then combining the oscillatory factors with the functions $F_1,F_2$.
\end{remark}

\section{Energy estimates, II: proof of Proposition \ref{EEMainProp}}\label{secEE2}

In this section we prove Proposition \ref{EEMainProp}, thus completing the proof of Proposition \ref{MainBootstrapEn}. Recall the 
definitions \eqref{njk20}-\eqref{njk22}.  For $G\in\mathcal{G}'$ and $W_1,W_2\in\mathcal{W}'$ let
\begin{equation}\label{bnm1}
\mathcal{A}^l_{Y,m}[G,W_1,W_2]:=\int_{\mathbb{R}}q_m(s)\iint_{\mathbb{R}^2\times\mathbb{R}^2}\mu_l(\xi,\eta)\chi_Y(\xi-\eta)\widehat{G}(\xi-\eta,s)\widehat{W_1}(\eta,s)\widehat{W_2}(-\xi,s)\,d\xi d\eta ds,
\end{equation}  
where $l\in\{0,1\}$, $m\in[\D^2,L]$, $Y\in\{go,ba\}$, and the symbols $\mu_l$ are as in \eqref{njk14}. The conclusion of Proposition \ref{EEMainProp} is equivalent to the uniform bound
\begin{equation}\label{bnm2}
|\mathcal{A}^l_{Y,m}[G,W_1,W_2]|\lesssim \varep_1^32^{2\delta^2m}.
\end{equation}

In proving this bound we further decompose the functions $W_1$ and $W_2$ dyadically and consider several cases. We remark that the most difficult case (which is treated in Lemma \ref{EELemmaMain}) is when the ``bad'' frequencies of $G$ interact with the high frequencies of the functions $W_1$ and $W_2$.

\subsection{The main interactions}\label{secMainbulk} We prove the following lemma.

\begin{lemma}\label{EELemmaMain}
For $l\in\{0,1\}$, $m\in[\D^2,L]$, $G\in\mathcal{G}'$, and $W_1,W_2\in\mathcal{W}'$ we have
\begin{equation}
\label{EELemmaMainbound}
\sum_{\min(k_1,k_2)\geq -40}|\mathcal{A}^l_{ba,m}[G,P_{k_1}W_1,P_{k_2}W_2]|\lesssim \varep_1^3.
\end{equation}
\end{lemma}

The rest of the subsection is concerned with the proof of this lemma. We need to further decompose our operators based on the size of the modulation. 
Assuming that $\overline{W_2}\in\mathcal{W}'_{\sigma}$, $W_1\in\mathcal{W}'_\nu$, $G\in\mathcal{G}'_\mu$, $\sigma,\nu,\mu\in\{+,-\}$, see \eqref{njk21}--\eqref{njk22}, we define the associated phase
\begin{equation}\label{bnm3}
\Phi(\xi,\eta)=\Phi_{\sigma\mu\nu}(\xi,\eta)=\Lambda_\sigma(\xi)-\Lambda_\mu(\xi-\eta)-\Lambda_{\nu}(\eta).
\end{equation}

Notice that in proving \eqref{EELemmaMainbound} we may assume that $\sigma=+$ (otherwise take complex conjugates) and that the sum is over $|k_1-k_2|\leq 50$ (due to localization in $\xi-\eta$). 

Some care is needed to properly sum the dyadic pieces in $k_1$ and $k_2$. For this we use frequency envelopes. More precisely, for $k\geq -30$ let
\begin{equation}\label{bnm3.1}
\begin{split}
&\rho_k(s):=\sum_{i\in\{1,2\}}\|P_{[k-40,k+40]}W_i(s)\|_{L^2}+2^{5m/6-\delta m}2^{-k/2}\sum_{i\in\{1,2\}}\|P_{[k-40,k+40]}\mathcal{E}_{W_i}(s)\|_{L^2},\\
&\rho_{k,m}^2:=\int_{\mathbb{R}}\rho_k(s)^2[2^{-m}q_m(s)+|q'_m(s)|]\,ds,
\end{split}
\end{equation}
where $\mathcal{E}_{W_{1,2}}$ are the ``semilinear'' nonlinearities defined in \eqref{njk34}. In view of \eqref{njk33} and \eqref{njk34.1}, 
\begin{equation}\label{bnm3.2}
\sum_{k\geq -30}\rho_{k,m}^2\lesssim \varep_1^22^{2\delta^2m}.
\end{equation}

Given $k\geq -30$, let $\underline{p}=\lfloor k/2-7m/9\rfloor$ (the largest integer $\leq k/2-7m/9$). We define 
\begin{equation}\label{bnm4}
\mathcal{A}^{l,p}_{ba}[F,H_1,H_2]:=\iint_{\mathbb{R}^2\times\mathbb{R}^2}\mu_l(\xi,\eta)\varphi_p^{[\underline{p},\infty)}(\Phi(\xi,\eta))\chi_{ba}(\xi-\eta)\widehat{F}(\xi-\eta)\widehat{H_1}(\eta)\widehat{H_2}(-\xi)\,d\xi d\eta,
\end{equation}
where $p\in[\underline{p},\infty)$ (here $\varphi_p^{[\underline{p},\infty)}=\varphi_p$ if $p\geq \underline{p}+1$ and $\varphi_p^{[\underline{p},\infty)}=\varphi_{\leq \underline{p}}$ if $p=\underline{p}$). Assuming that $|k_1-k|\leq 30$, $|k_2-k|\leq 30$, let
\begin{equation}\label{bnm5}
\mathcal{A}^{l,p}_{ba,m}[G,P_{k_1}W_1,P_{k_2}W_2]:=\int_{\mathbb{R}}q_m(s)\mathcal{A}^{l,p}_{ba}[G(s),P_{k_1}W_1(s),P_{k_2}W_2(s)]\,ds.
\end{equation} 
This gives a decomposition $\mathcal{A}^{l}_{ba,m}=\sum_{p\geq \underline{p}}\mathcal{A}^{l,p}_{ba,m}$ as a sum of operators localized in modulation. Notice that the sum is either over $p\in[\underline{p},k/2+\D]$ (if $\nu=+$ or if $\nu=-$ and $k\leq \D/2$) or over $|p-3k/2|\leq \D$ (if $\nu=-$ and $k\geq \D/2$). For \eqref{EELemmaMainbound} it remains to prove that
\begin{equation}\label{bnm5.1}
\big|\mathcal{A}^{l,p}_{ba,m}[G,P_{k_1}W_1,P_{k_2}W_2]\big|\lesssim \varep_12^{-\delta m}\rho_{k,m}^2,
\end{equation} 
for any $k\geq -30$, $p\geq \underline{p}$, and $k_1,k_2\in\mathbb{Z}$ satisfying $|k_1-k|\leq 30$, $|k_2-k|\leq 30$.

Using Lemma \ref{EstimAKP} (see \eqref{SimpleBdAKP}), we have
\begin{equation*}
|\mathcal{A}^{l,p}_{ba}[G(s),P_{k_1}W_1(s),P_{k_2}W_2(s)]|\lesssim \varep_12^{2p^+-k}2^{-5m/6+\delta m}\|P_{k_1}W_1(s)\|_{L^2}\|P_{k_2}W_2(s)\|_{L^2},
\end{equation*}
for any $p\geq \underline{p}$, due to the $L^\infty$ bound in \eqref{Gdecay}. The desired bound \eqref{bnm5.1} follows if $2p^+-k\leq -m/5+\D$. Also, using Lemma \ref{L2Prop0}, we have
\begin{equation*}
|\mathcal{A}^{l,\underline{p}}_{ba}[G,P_{k_1}W_1,P_{k_2}W_2](s)|\lesssim \varep_12^{-m-\delta m}\|P_{k_1}W_1(s)\|_{L^2}\|P_{k_2}W_2(s)\|_{L^2},
\end{equation*}
using \eqref{L2Use}, since $2^{\underline{p}-k/2}\lesssim 2^{-7m/9}$ and $\Vert e^{is\Lambda_{\mu}}G(s)\Vert_{Z_1\cap H^{N_1/3,0}_{\Omega}}\lesssim \varep_12^{\delta m}$ (see \eqref{njk30}). Therefore \eqref{bnm5.1} follows if $p=\underline{p}$. It remains to prove \eqref{bnm5.1} when
\begin{equation}\label{bnm8}
p\geq \underline{p}+1\quad\text{ and }\quad k\in[-30,2p^++m/5],\,|k_1-k|\leq 30,\,|k_2-k|\leq 30.
\end{equation}

In the remaining range in \eqref{bnm8} we integrate by parts in $s$. We define
\begin{equation}\label{DefTerms}
\widetilde{\mathcal{A}}^{l,p}_{ba}[F,H_1,H_2] :=  \iint_{\R^2\times \R^2} \mu_l(\xi,\eta) \widetilde{\varphi}_p(\Phi(\xi,\eta)) \chi_{ba}(\xi-\eta)
  \widehat{F}(\xi-\eta) \widehat{H_1}(\eta) \widehat{H_2}(-\xi)\,d\xi d\eta,
\end{equation}
where $\widetilde{\varphi}_p(x):=2^px^{-1}\varphi_p(x)$. This is similar to the definition in \eqref{bnm4}, but with $\varphi_p$ replaced by $\widetilde{\varphi}_p$. Then we let $W_{k_1}:=P_{k_1}W_1$, $W_{k_2}:=P_{k_2}W_2$ and write
\begin{equation}\label{IBPtEE}
\begin{split}
0 &= \int_{\R}\frac{d}{ds}\Big\{q_m(s)\widetilde{\mathcal{A}}^{l,p}_{ba} [G(s),W_{k_1}(s),W_{k_2}(s)]\Big\}\,ds\\
&= \int_{\R} q'_m(s) \widetilde{\mathcal{A}}^{l,p}_{ba} [G(s),W_{k_1}(s),W_{k_2}(s)]\,ds
+\mathcal{J}_0^{l,p}(k_1,k_2)+\mathcal{J}_1^{l,p}(k_1,k_2)+\mathcal{J}_2^{l,p}(k_1,k_2)\\
&+i2^p\int_{\R}q_m(s)\mathcal{A}^{l,p}_{ba} [G(s),W_{k_1}(s),W_{k_2}(s)]\,ds,
\end{split}
\end{equation}
where
\begin{equation}\label{bnm9}
\begin{split}
&\mathcal{J}_{ba,0}^{l,p}(k_1,k_2):=\int_{\R} q_m(s) \widetilde{\mathcal{A}}^{l,p}_{ba} [(\partial_s+i\Lambda_\mu)G(s),W_{k_1}(s),W_{k_2}(s)]\,ds, \\
&\mathcal{J}_{ba,1}^{l,p}(k_1,k_2):=\int_{\R} q_m(s) \widetilde{\mathcal{A}}^{l,p}_{ba} [G(s),(\partial_s+i\Lambda_\nu)W_{k_1}(s),W_{k_2}(s)]\,ds, \\
&\mathcal{J}_{ba,2}^{l,p}(k_1,k_2):=\int_{\R} q_m(s) \widetilde{\mathcal{A}}^{l,p}_{ba} [G(s),W_{k_1}(s),(\partial_s+i\Lambda_{-\sigma})W_{k_2}(s)]\,ds.
\end{split}
\end{equation}
The integral in the last line of \eqref{IBPtEE} is the one we have to estimate. Notice that
\begin{equation*}
2^{-p}\big|\widetilde{\mathcal{A}}^{l,p}_{ba} [G(s),W_{k_1}(s),W_{k_2}(s)]\big|\lesssim 2^{-p}2^{-5m/6+\delta m}\|W_{k_1}(s)\|_{L^2}\|W_{k_2}(s)\|_{L^2}
\end{equation*}
as a consequence of Lemma \ref{EstimAKP} and \eqref{Gdecay}. It remains to prove that if \eqref{bnm8} holds then
\begin{equation}\label{bnm10}
2^{-p}\big|\mathcal{J}_{ba,0}^{l,p}(k_1,k_2)+\mathcal{J}_{ba,1}^{l,p}(k_1,k_2)+\mathcal{J}_{ba,2}^{l,p}(k_1,k_2)\big|\lesssim \varep_12^{-\delta m}\rho_{k,m}^2.
\end{equation}

This bound will be proved in several steps, in Lemmas \ref{EEmediummod}, \ref{EElargemod}, and \ref{EELemmaSemi} below. 

\subsubsection{Quasilinear terms}\label{Qterms}
We consider first the quasilinear terms appearing in \eqref{bnm10}, which are those where $(\partial_t+i\Lambda)$ hits the high frequency inputs $W_{k_1}$ and $W_{k_2}$. We start with the case when the frequencies $k_1,k_2$ are not too large relative to $p^+$.

\begin{lemma}\label{EEmediummod}
Assume that \eqref{bnm8} holds and, in addition, $k\leq 2p^+/3+m/4$. Then
\begin{equation}\label{bnm11}
2^{-p}\big[\big|\mathcal{J}_{ba,1}^{l,p}(k_1,k_2)\big|+\big|\mathcal{J}_{ba,2}^{l,p}(k_1,k_2)\big|\big]\lesssim \varep_12^{-\delta m}\rho_{k,m}^2.
\end{equation}
\end{lemma}

\begin{proof}
It suffices to bound the contributions of $\big|\mathcal{J}_{ba,1}^{l,p}(k_1,k_2)\big|$ in \eqref{bnm11}, since the contributions of $\big|\mathcal{J}_{ba,2}^{l,p}(k_1,k_2)\big|$ are similar. We estimate, for $s\in [2^{m-1},2^{m+1}]$,
\begin{equation}\label{bnm12}
\|(\partial_s+i\Lambda_\nu)W_{k_1}(s)\|_{L^2}\lesssim \varep_12^{-5m/6+\delta m}(2^{k_1}+2^{3k_1/2}2^{-5m/6})\rho_k(s),
\end{equation}
using \eqref{njk34}--\eqref{BoundsQuasiTerms}. As before, we use Lemma \ref{EstimAKP} and the pointwise bound \eqref{Gdecay} to estimate
\begin{equation}\label{bnm13}
\begin{split}
\big|\widetilde{\mathcal{A}}^{l,p}_{ba}& [G(s),(\partial_s+i\Lambda_\nu)W_{k_1}(s),W_{k_2}(s)]\big|\\
&\lesssim \min(1,2^{2p^+-k})\varep_12^{-5m/6+\delta m}\|(\partial_s+i\Lambda_\nu)W_{k_1}(s)\|_{L^2}\|W_{k_2}(s)\|_{L^2}.
\end{split}
\end{equation}
The bounds \eqref{bnm12}--\eqref{bnm13} suffice to prove \eqref{bnm11} when $p\geq 0$ or when $-m/2+k/2\leq p\leq 0$. 

It remains to prove \eqref{bnm11} when
\begin{equation}\label{bnm13.5}
\underline{p}+1\leq p \leq -m/2 + k/2, \qquad k \leq m/5.
\end{equation}
For this we would like to apply Lemma \ref{L2Prop0}. We claim that, for $s\in [2^{m-1},2^{m+1}]$,
\begin{equation}\label{bnm14}
\begin{split}
\big|\widetilde{\mathcal{A}}^{l,p}_{ba}& [G(s),(\partial_s+i\Lambda_\nu)W_{k_1}(s),W_{k_2}(s)]\big|\\
&\lesssim 2^{-k}\varep_12^{31\delta m}(2^{(3/2)(p-k/2)}+2^{p-k/2-m/3})\|(\partial_s+i\Lambda_\nu)W_{k_1}(s)\|_{L^2}\|W_{k_2}(s)\|_{L^2}.
\end{split}
\end{equation}
Assuming this and using also \eqref{bnm12}, it follows that
\begin{equation*}
\begin{split}
2^{-p}\big|\mathcal{J}_{ba,1}^{l,p}(k_1,k_2)\big|&\lesssim 2^{-p}2^m\cdot \varep_1\rho_{k,m}^22^{-5m/6+40\delta m}(2^{(3/2)(p-k/2)}+2^{p-k/2-m/3})\\
&\lesssim \varep_1\rho_{k,m}^22^{m/6+40\delta m}(2^{p/2-3k/4}+2^{-k/2-m/3}),
\end{split}
\end{equation*}
and the desired conclusion follows using also \eqref{bnm13.5}. 

On the other hand, to prove the bound \eqref{bnm14}, we use \eqref{L2Use}. Clearly, with $g=e^{is\Lambda_\mu} G$, we have $\|g\|_{Z_1\cap H_{\Omega}^{N_1/3,0}}\lesssim \varep_12^{\delta^2m}$, see \eqref{Gbound0}. The factor $2^{-k}$ in the right-hand side of \eqref{bnm14} is due to the symbols $\mu_0$ and $\mu_1$. This is clear for the symbols $\mu_1$, which already contain a factor of $2^{-k}$ (see \eqref{njk14}). For the symbols $\mu_0$, we notice that we can take 
\begin{equation*}
A(\xi,\eta):=2^{k}\mathfrak{d}(\xi,\eta)\varphi_{\leq 4}(\Phi(\xi,\eta))\varphi_{[k_2-2,k_2+2]}(\xi)\varphi_{[-10,10]}(\xi-\eta).
\end{equation*}
This satisfies the bounds required in \eqref{AssL2Prop1}, since $k\leq m/5$. This completes the proof.
\end{proof}

We now look at the remaining cases for the quasilinear terms and prove the following:

\begin{lemma}\label{EElargemod}
Assume that \eqref{bnm8} holds and, in addition, 
\begin{equation}\label{bnm14.5}
p\geq 0,\qquad k\in[2p/3+m/4,2p+m/5].
\end{equation} 
Then
\begin{equation}\label{EElargemodest}
2^{-p}\big|\mathcal{J}_{ba,1}^{l,p}(k_1,k_2)+\mathcal{J}_{ba,2}^{l,p}(k_1,k_2)\big|\lesssim \varep_12^{-\delta m}\rho_{k,m}^2.
\end{equation}
\end{lemma}

\begin{proof}
The main issue here is to deal with the case of large frequencies, relative to the time variable, and avoid the loss of derivatives coming from the terms $(\partial_t\pm i\Lambda)W_{1,2}$. For this we use ideas related to the local existence theory, such as symmetrization.
Notice that in Lemma \ref{EElargemod} we estimate the absolute value of the sum $\mathcal{J}_{ba,1}^{l,p}+\mathcal{J}_{ba,2}^{l,p}$, and not each term separately.

Notice first that we may assume that $\nu=+=\sigma$, since otherwise $\mathcal{J}_{ba,n}^{l,p}(k_1,k_2)=0$, $n\in\{1,2\}$, when $k\geq 2p/3+m/4$. In particular $2^p\lesssim 2^{k/2}$. We deal first with the semilinear part of the nonlinearity, which is $\mathcal{E}_{W_1}$ in equation \eqref{njk34}. Using Lemma \ref{EstimAKP} and the definition \eqref{bnm3.1},
\begin{equation*}
\begin{split}
\big|\widetilde{\mathcal{A}}^{l,p}_{ba} [G(s),P_{k_1}\mathcal{E}_{W_1}(s),W_{k_2}(s)]\big|&\lesssim \varep_12^{-5m/6+\delta m}\|P_{k_1}\mathcal{E}_{W_1}(s)\|_{L^2}\|W_{k_2}(s)\|_{L^2}\\
&\lesssim \varep_12^{-5m/3+2\delta m}2^{k/2}\rho_k(s)^2.
\end{split}
\end{equation*}
Therefore
\begin{equation*}
2^{-p}\int_{\mathbb{R}}q_m(s)\big|\widetilde{\mathcal{A}}^{l,p}_{ba} [G(s),P_{k_1}\mathcal{E}_{W_1}(s),W_{k_2}(s)]\big|\,ds\lesssim \varep_12^{-m/4}\rho_{k,m}^2.
\end{equation*}

It remains to bound the contributions of $\mathcal{Q}_{W_1}$ and $\mathcal{Q}_{W_2}$. Using again Lemma \ref{EstimAKP}, we can easily prove the estimate 
when $k \leq 6m/5$ or when $l=1$. It remains to show that
\begin{equation}\label{bnm20}
\begin{split}
2^{-p}\int_{\mathbb{R}}q_m(s)\big|\widetilde{\mathcal{A}}^{0,p}_{ba} [G(s),P_{k_1}\mathcal{Q}_{W_1}(s),W_{k_2}(s)]+\widetilde{\mathcal{A}}^{0,p}_{ba} [G(s),W_{k_1}(s),P_{k_2}\mathcal{Q}_{W_2}(s)]\big|\,ds\\
\lesssim \varep_12^{-\delta m}\rho_{k,m}^2,
\end{split}
\end{equation}
provided that
\begin{equation}\label{bnm21}
\nu=\sigma=+,\qquad k\in[2p-\D,2p+m/5],\qquad k\geq 6m/5.
\end{equation}

In this case we consider the full expression and apply a symmetrization procedure to recover the loss of derivatives. 
Since $W_1\in\mathcal{W}'_+$ and $W_2\in\mathcal{W}'_-$, recall from \eqref{njk34.1} that
\begin{equation*}
\begin{split}
\mathcal{Q}_{W_1} = -iT_{\Sigma_{\geq 2}} W_1 - iT_{V\cdot\zeta} W_1,\qquad \mathcal{Q}_{W_2} = i\overline{T_{\Sigma_{\geq 2}} \overline{W_2}}
+i\overline{T_{V\cdot\zeta} \overline{W_2}}.
\end{split}
\end{equation*}
Therefore, using the definition \eqref{DefTerms},
\begin{equation*}
\begin{split}
 \widetilde{\mathcal{A}}^{0,p}_{ba} &[G,P_{k_1}\mathcal{Q}_{W_1},W_{k_2}]=\sum_{\sigma\in\{\Sigma_{\geq 2},V\cdot\zeta\}}\iint_{\R^2\times \R^2} \mu_0(\xi,\eta)\\ 
&\times\widetilde{\varphi}_p(\Phi(\xi,\eta)) \chi_{ba}(\xi-\eta)
  \widehat{G}(\xi-\eta)\cdot\varphi_{k_1}(\eta)(-i)\widehat{T_\sigma W_1}(\eta)\cdot\varphi_{k_2}(\xi)\widehat{W_2}(-\xi)\,d\xi d\eta,
\end{split}
\end{equation*}
and
\begin{equation*}
\begin{split}
 \widetilde{\mathcal{A}}^{0,p}_{ba} &[G,W_{k_1},P_{k_2}\mathcal{Q}_{W_2}]=\sum_{\sigma\in\{\Sigma_{\geq 2},V\cdot\zeta\}}\iint_{\R^2\times \R^2} \mu_0(\xi,\eta)\\ 
&\times\widetilde{\varphi}_p(\Phi(\xi,\eta)) \chi_{ba}(\xi-\eta)
  \widehat{G}(\xi-\eta)\cdot\varphi_{k_1}(\eta)\widehat{W_1}(\eta)\cdot\varphi_{k_2}(\xi)i\widehat{\overline{T_\sigma\overline{W_2}}}(-\xi)\,d\xi d\eta.
\end{split}
\end{equation*}
We use the definition \eqref{Tsigmaf} and make suitable changes of variables to write
\begin{equation*}
\begin{split}
&\widetilde{\mathcal{A}}^{0,p}_{ba} [G,P_{k_1}\mathcal{Q}_{W_1},W_{k_2}]+\widetilde{\mathcal{A}}^{0,p}_{ba} [G,W_{k_1},P_{k_2}\mathcal{Q}_{W_2}]=\\
&=\sum_{\sigma\in\{\Sigma_{\geq 2},V\cdot\zeta\}}\frac{-i}{4\pi^2}\iiint_{(\R^2)^3} \widehat{W_1}(\eta)\widehat{W_2}(-\xi)\widehat{G_{ba}}(\xi-\eta-\alpha)
(\delta M)(\xi,\eta,\alpha)\,d\xi d\eta d\alpha,
\end{split}
\end{equation*}
where $\widehat{G_{ba}}:=\chi_{ba}\cdot\widehat{G}$ and
\begin{equation*}
\begin{split}
(\delta M)(\xi,\eta,\alpha)&=\mu_0(\xi,\eta+\alpha)\widetilde{\varphi}_p(\Phi(\xi,\eta+\alpha)) 
  \widetilde{\sigma}(\alpha,\frac{2\eta+\alpha}{2})\chi(\frac{|\alpha|}{|2\eta+\alpha|})\varphi_{k_1}(\eta+\alpha)\varphi_{k_2}(\xi)\\
&-\mu_0(\xi-\alpha,\eta)\widetilde{\varphi}_p(\Phi(\xi-\alpha,\eta)) 
  \widetilde{\sigma}(\alpha,\frac{2\xi-\alpha}{2})\chi(\frac{|\alpha|}{|2\xi-\alpha|})\varphi_{k_1}(\eta)\varphi_{k_2}(\xi-\alpha).
\end{split}
\end{equation*}
For \eqref{bnm20} it suffices to prove that for any $s\in[2^{m-1},2^{m+1}]$ and $\sigma\in\{\Sigma_{\geq 2},V\cdot\zeta\}$
\begin{equation}\label{bnm25}
\begin{split}
2^{-p}\Big|\iiint_{(\R^2)^3} \widehat{W_1}(\eta,s)\widehat{W_2}(-\xi,s)\widehat{G_{ba}}(\xi-\eta-\alpha,s)
(\delta M)(\xi,\eta,\alpha,s)\,d\xi d\eta d\alpha\Big|\\
\lesssim \varep_1\rho_k(s)^22^{-m-\delta m}.
\end{split}
\end{equation}

Let
\begin{align}
 \label{formulaM}
\begin{split}
M(\xi,\eta,\alpha;\theta_1,\theta_2) &:=
  \mu_0(\xi-\theta_1,\eta+\alpha-\theta_1) \widetilde{\varphi}_p(\Phi(\xi-\theta_1,\eta+\alpha-\theta_1))\\
&\times\varphi_{k_2}(\xi-\theta_1)\varphi_{k_1}(\eta+\alpha-\theta_1)
\widetilde{\sigma}(\alpha,\eta+\frac{\alpha}{2}+\theta_2)\chi(\frac{|\alpha|}{|2\eta+\alpha+2\theta_2|}),
\end{split}
\end{align}
therefore
\begin{align*}
&(\delta M)(\xi,\eta,\alpha) = M(\xi,\eta,\alpha;0,0) - M(\xi,\eta,\alpha;\alpha,\xi-\eta-\alpha)\\
& = \varphi_{\leq k-\D}(\alpha)[\alpha\cdot\nabla_{\theta_1} M(\xi,\eta,\alpha;0,0) + (\xi-\eta-\alpha)\cdot\nabla_{\theta_2} M(\xi,\eta,\alpha;0,0)]+(eM)(\xi,\eta,\alpha).
\end{align*}
Using the formula for $\mu_0$ in \eqref{njk14} and recalling that $\sigma\in\varep_1\mathcal{M}^{3/2,1}_{N_3-2}$ (see Definition \ref{DefSym}), 
it follows that, in the support of the integral,
\begin{align*}
\big|(eM)(\xi,\eta,\alpha) \big|\lesssim (1+|\alpha|^2)P(\alpha)2^{-2k} 2^{3k/2},\qquad \|(1+|\alpha|)^{8}P\|_{L^2}\lesssim 2^{\delta m}.
\end{align*}
The contribution of $(eM)$ in \eqref{bnm25} can then be estimated by $2^{-p}2^{\delta m}2^{-k/2}\varep_1\rho_k(s)^2$ which suffices due to the 
assumptions \eqref{bnm21}.

We are thus left with estimating the integrals
\begin{align*}
I& :=\iiint_{(\R^2)^3} \widehat{G_{ba}}(\xi-\eta-\alpha) \varphi_{\leq k-\D}(\alpha)\big[ (\xi-\eta-\alpha)\cdot\nabla_{\theta_2}M(\xi,\eta,\alpha;0,0) \big]
  \widehat{W_1}(\eta)\widehat{W_2}(-\xi) \, d\alpha d\eta d\xi,\\
II & :=\iiint_{(\R^2)^3} \widehat{G_{ba}}(\xi-\eta-\alpha) \varphi_{\leq k-\D}(\alpha)\big[ \alpha\cdot\nabla_{\theta_1} M(\xi,\eta,\alpha;0,0) \big]
  \widehat{W_1}(\eta) \widehat{W_2}(-\xi) \, d\alpha d\eta d\xi.
\end{align*}

If $|\alpha|\ll 2^k$ we have
\begin{equation*}
\begin{split}
(\xi-\eta-\alpha)\cdot\nabla_{\theta_2}M(\xi,\eta,\alpha;0,0)&=\mu_0(\xi,\eta+\alpha) \widetilde{\varphi}_p(\Phi(\xi,\eta+\alpha))\varphi_{k_2}(\xi)\varphi_{k_1}(\eta+\alpha)\\
&\times(\xi-\eta-\alpha)\cdot(\nabla_{\zeta}\widetilde{\sigma})(\alpha,\eta+\frac{\alpha}{2}).
\end{split}
\end{equation*}
We make the change of variable $\alpha=\beta-\eta$ to rewrite
\begin{equation*}
\begin{split}
I=&c\iiint_{(\R^2)^3} \widehat{G_{ba}}(\xi-\beta) \mu_0(\xi,\beta) \widetilde{\varphi}_p(\Phi(\xi,\beta))(\xi-\beta)\cdot \mathcal{F}\{P_{k_1}T_{P_{\leq k-\D}\nabla_{\zeta}\sigma}W_1\}(\beta)\widehat{P_{k_2}W_2}(-\xi) \, d\beta d\xi.
	\end{split}
\end{equation*}
Then we use Lemma \ref{EstimAKP}, \eqref{Gdecay}, and \eqref{LqBdTa} (recall $\sigma\in \varep_1\mathcal{M}^{3/2,1}_{N_3-2}$) to estimate
\begin{equation*}
\begin{split}
2^{-p}|I(s)|&\lesssim 2^{-p}2^{2p-k}\varep_12^{-5m/6+\delta m}\|P_{k_1}T_{P_{\leq k-\D}\nabla_{\zeta}\sigma}W_1(s)\|_{L^2}\|P_{k_2}W_2(s)\|_{L^2}\\
&\lesssim \varep_12^{-3m/2}2^{p-k/2}\rho_k(s)^2.
\end{split}
\end{equation*}
This is better than the desired bound \eqref{bnm25}. One can estimate $2^{-p}|II(s)|$ in a similar way, using the flexibility in Lemma \ref{EstimAKP} due to the fact that the symbol $\underline{\mu}_0$ is allowed to contain additional $S^\infty$ symbols. This completes the proof of the bound \eqref{bnm25} and the lemma.
\end{proof}

\subsubsection{Semilinear terms}\label{secSemiterms}

The only term in \eqref{IBPtEE} that remains to be estimated is $\mathcal{J}^{l,p}_0(k_1,k_2)$. This is a semilinear term, since the $\partial_t$ derivative hits the low-frequency component, for which we will show the following:

\begin{lemma}\label{EELemmaSemi}
Assume that \eqref{bnm8} holds. Then
\begin{equation}
\label{EELemmaSemiconc}
2^{-p}|\mathcal{J}^{l,p}_{ba,0}(k_1,k_2)| \lesssim \e_12^{-\delta m}\rho_{k,m}^2.
\end{equation}
\end{lemma}

\begin{proof} Assume first that $p \geq -m/4$. Using integration by parts we can see that, for $\rho\in\mathbb{R}$,
\begin{align}\label{bnm8.2}
\big\| \mathcal{F}^{-1} \big\{ e^{i\rho\Lambda(\xi)} \varphi_{[-20,20]}(\xi) \big\} \big\|_{L^1_x} \lesssim 1+|\rho|.
\end{align}
Combining this and the bounds in the second line of \eqref{Gbound0} we get
\begin{equation*}
\sup_{|\rho|\leq 2^{-p+\delta m}}\|e^{i\rho\Lambda}[(\partial_s+i\Lambda_\mu)P_{[-10,10]}G(s)]\|_{L^\infty}\lesssim (2^{-p}+1)2^{-5m/3+2\delta m}.
\end{equation*}
Using this in combination with Lemma \ref{EstimAKP} we get
\begin{equation}\label{bnm31}
\big| \widetilde{\mathcal{A}}^{l,p}_{ba}[(\partial_s+i\Lambda_\mu)G(s),W_{k_1}(s),W_{k_2}(s)] \big|
  \lesssim (2^{-p}+1) 2^{-5m/3+2\delta m}\rho_k(s)^2,
\end{equation}
which leads to an acceptable contribution.

Assume now that
\begin{equation*}
\underline{p}+1 \leq p \leq -m/4
\end{equation*}
Even though there is no loss of derivatives here, the information that we have so far is not sufficient to obtain the bound in this range. The main reason is that some components of $(\partial_s+i\Lambda_\mu)G(s)$ undergo oscillations which are not linear.
To deal with this term we are going to use the following decomposition of $(\partial_s+i\Lambda_\mu)G$, which follows from Lemma \ref{dtfLem3},
\begin{equation}\label{DecpartialtfEE}
\chi'_{ba}(\xi)\cdot\mathcal{F}\{(\partial_s+i\Lambda_\mu)G(s)\}(\xi) = g_d(\xi) + g_\infty(\xi) + g_2(\xi)
\end{equation}
for any $s\in [2^{m-1},2^{m+1}]$, where $\chi'_{ba}(x)=\varphi_{\leq 4}(2^{\D}(|x|-\gamma_0))+\varphi_{\leq 4}(2^{\D}(|x|-\gamma_1))$ and
\begin{equation}\label{bnm30}
\begin{split}
&\| g_2 \|_{L^2} \lesssim \e_1^2 2^{-3m/2+20\delta m},
  \qquad \| g_\infty\|_{L^\infty} \lesssim \e_1^22^{-m-4\delta m },\\
&\sup_{|\rho|\leq 2^{7m/9+4\delta m}}\|\mathcal{F}^{-1}\{e^{i\rho\Lambda}g_d\}\|_{L^\infty}\lesssim \varep_1^22^{-16m/9-4\delta m}.
\end{split}
\end{equation}

Clearly, the contribution of $g_d$ can be estimated as in \eqref{bnm31}, using Lemma \ref{EstimAKP}. On the other hand, we estimate the 
contributions of $g_2$ and $g_\infty$ in the Fourier space, using Schur's lemma. For this we need to use the volume bound 
in Proposition \ref{volume} (i). We have
\begin{equation*}
\sup_{\xi}\|\widetilde{\varphi}_p(\Phi(\xi,\eta))\chi_{ba}(\xi-\eta)g_\infty(\xi-\eta) \|_{L^1_{\eta}}
\lesssim 2^{(1-\delta)p} \|g_\infty \|_{L^\infty}\lesssim 2^{(1-\delta)p}2^{-(1+4\delta)m} \e_1^2,
\end{equation*}
and also a similar bound for the $\xi$ integral (keeping $\eta$ fixed). Therefore, using Schur's lemma
\begin{equation*}
\big| \widetilde{\mathcal{A}}^{l,p}_{ba}[\mathcal{F}^{-1}g_{\infty}(s),W_{k_1}(s),W_{k_2}(s)] \big|
  \lesssim 2^{(1-\delta)p}2^{-(1+4\delta)m} \e_1^2\rho_k(s)^2, 
\end{equation*}
and the corresponding contribution is bounded as claimed in \eqref{EELemmaSemiconc}. The contribution of $g_2$ can be bounded in a similar way, 
using Schur's lemma and the Cauchy--Schwarz inequality. This completes the proof of the lemma.
\end{proof}

\subsection{The other interactions}\label{secother}
In this subsection we show how to bound all the remaining contributions to the energy increment in \eqref{bnm1}. We remark that we do not use the 
main $L^2$ lemma in the estimates in this subsection.

\subsubsection{Small frequencies}\label{secsmallfreq}

We consider now the small frequencies and prove the following:
\begin{lemma}\label{BNV1}
For $l\in\{0,1\}$, $m\in[\D^2,L]$, $G\in\mathcal{G}'$, and $W_1,W_2\in\mathcal{W}'$ we have
\begin{equation}\label{bnm40}
\sum_{\min(k_1,k_2)\leq -40}|\mathcal{A}^l_{ba,m}[G,P_{k_1}W_1,P_{k_2}W_2]|\lesssim \varep_1^3.
\end{equation}
\end{lemma}

\begin{proof} Let $\underline{k} := \min \{ k_1,k_2 \}$. Notice that we may assume that $\underline{k}\leq -40$, $\max(k_1,k_2)\in[-10,10]$, and $l=1$. We can easily estimate
\begin{align*}
|\mathcal{A}^1_{ba,m}[G,P_{k_1}W_1,P_{k_2}W_2]|\lesssim \sup_{s\in[2^{m-1},2^{m+1}]}2^m2^{\underline{k}} {\| G(s)\|}_{L^2} {\| P_{k_1} W_1 (s)\|}_{L^2} {\| P_{k_2}W_2 (s)\|}_{L^2}.
\end{align*}
In view of \eqref{njk30} and \eqref{njk33}, this suffices to estimate the sum corresponding to $ \underline{k}\leq -m-3\delta m$. Therefore, it suffices to show that if $ -(1+3\delta)m \leq \bar{k} \leq -40$ then
\begin{align}
\label{EEsmallfreqbound}
|\mathcal{A}^1_{ba,m}[G,P_{k_1}W_1,P_{k_2}W_2]|\lesssim \varep_1^32^{-\delta m}.
\end{align}

As in the proof of Lemma \ref{EELemmaMain}, assume that $W_2\in\mathcal{W}'_{-\sigma}$, $W_1\in\mathcal{W}'_\nu$, $G\in\mathcal{G}'_\mu$, 
$\sigma,\nu,\mu\in\{+,-\}$, and define the associated phase $\Phi=\Phi_{\sigma\mu\nu}$ as in \eqref{bnm3}. The important observation is that
\begin{align}
\label{phaselb}
 |\Phi(\xi,\eta)| \approx 2^{\underline{k}/2}
\end{align}
in the support of the integral. We define $\mathcal{A}^{1,p}_{ba}$ and $\mathcal{A}^{1,p}_{ba,m}$ as in \eqref{bnm4}--\eqref{bnm5}, by introducing the the cutoff
function $\varphi_p(\Phi(\xi,\eta)$. In view of \eqref{phaselb} we may assume that $|p-\underline{k}/2|\lesssim 1$. Then we integrate by parts as 
in \eqref{IBPtEE} and similarly obtain
\begin{equation}\label{bnm50}
\begin{split}
|\mathcal{A}^{1,p}_{ba,m}[G,P_{k_1}W_1,P_{k_2}W_2]|&\lesssim 2^{-p}\Big|\int_{\R} q'_m(s) \widetilde{\mathcal{A}}^{1,p}_{ba} [G(s),W_{k_1}(s),W_{k_2}(s)]\,ds\Big|\\
&+2^{-p}|\mathcal{J}_{ba,0}^{1,p}(k_1,k_2)|+2^{-p}|\mathcal{J}_{ba,1}^{1,p}(k_1,k_2)|+2^{-p}|\mathcal{J}_{ba,2}^{1,p}(k_1,k_2)|,
\end{split}
\end{equation}
see \eqref{DefTerms} and \eqref{bnm9} for definitions. 

We apply Lemma \ref{EstimAKP} (see \eqref{SimpleBdAKP}) to control the terms in the right-hand side of \eqref{bnm50}. Using \eqref{njk33} and \eqref{Gdecay} (recall that 
$2^{-p}\leq 2^{-\underline{k}/2+\delta m}\leq 2^{m/2+3\delta m}$), the first term is dominated by
\begin{equation*}
 C\varep_1^32^{-p}2^{\delta m}2^{-5m/6+\delta m}\lesssim \varep_1^32^{-m/4}.
\end{equation*}
Similarly,
\begin{equation*}
 2^{-p}|\mathcal{J}_{ba,1}^{1,p}(k_1,k_2)|+2^{-p}|\mathcal{J}_{ba,2}^{1,p}(k_1,k_2)|\lesssim \varep_1^32^m2^{-p}2^{-5m/6+\delta m}2^{-5m/6+2\delta m}
\lesssim \varep_1^32^{-m/10}.
\end{equation*}
For $|\mathcal{J}_{ba,0}^{1,p}(k_1,k_2)|$ we estimate first, using also \eqref{bnm8.2} and \eqref{Gbound0}
\begin{equation*}
 2^{-p}|\mathcal{J}_{ba,0}^{1,p}(k_1,k_2)|\lesssim \varep_1^32^m2^{-p}(2^{-p}2^{-5m/3+\delta m})2^{\delta m}\lesssim \varep_1^32^{-2p}2^{-2m/3+2\delta m}.
\end{equation*}
We can also estimate directly in the Fourier space 
(placing the factor at low frequency in $L^1$ and the other two factors in $L^2$),
\begin{equation*}
 2^{-p}|\mathcal{J}_{ba,0}^{1,p}(k_1,k_2)|\lesssim \varep_1^32^m2^{-p}2^{\underline{k}}2^{-5m/6+3\delta m}\lesssim 
\varep_1^32^{p}2^{m/6+3\delta m}.
\end{equation*}
These last two bounds show that $2^{-p}|\mathcal{J}_{ba,0}^{1,p}(k_1,k_2)|\lesssim \varep_1^32^{-m/10}$. The desired conclusion \eqref{EEsmallfreqbound} follows using
\eqref{bnm50}.
\end{proof}

\subsubsection{The ``good'' frequencies}\label{secgoodfreq}
We estimate now the contribution of the terms in \eqref{bnm1}, corresponding to the cutoff $\chi_{go}$.
One should keep in mind that these terms are similar, but easier than the ones we have already estimated. 
We often use the sharp decay in \eqref{Gdecaygood} to bound the contribution of small modulations. 

We may assume that $\overline{W_2}\in\mathcal{W}'_\sigma$, $W_1\in\mathcal{W}'_\nu$, and $G\in\mathcal{G}'_+$. For \eqref{bnm2} it suffices to prove that
\begin{equation}\label{bnm60}
 \sum_{k,k_1,k_2\in\mathbb{Z}}|\mathcal{A}^l_{go,m}[P_kG,P_{k_1}W_1,P_{k_2}W_2]|\lesssim\varep_1^32^{2\delta^2m}.
\end{equation}
Recalling the assumptions \eqref{njk14} on the symbols $\mu_l$, we have the simple bound 
\begin{equation*}
\begin{split}
|\mathcal{A}^l_{go,m}&[P_kG,P_{k_1}W_1,P_{k_2}W_2]|\lesssim 2^m2^{\min(k,k_1,k_2)}2^{2k^+}\sup_{s\in I_m}\|P_kG(s)\|_{L^2}\|P_{k_1}W_1(s)\|_{L^2}\|P_{k_2}W_2(s)\|_{L^2}. 
\end{split}
\end{equation*}
Using now \eqref{njk30} and \eqref{njk33}, it follows that the sum over $k\geq 2\delta m$ or $k\leq -m-\delta m$ in \eqref{bnm60} is dominated as claimed. 
Using also the $L^\infty$ bounds \eqref{Gbound1} and Lemma \ref{touse}, we have
\begin{equation*}
\begin{split}
|\mathcal{A}^l_{go,m}[P_kG,P_{k_1}W_1,P_{k_2}W_2]|&\lesssim 2^m2^{2k^+}\sup_{s\in I_m}\|P_kG(s)\|_{L^\infty}\|P_{k_1}W_1(s)\|_{L^2}\|P_{k_2}W_2(s)\|_{L^2}\\
&\lesssim 2^m2^{2k^+}\sup_{s\in I_m}\varep_12^{k-m+50\delta m}\|P_{k_1}W_1(s)\|_{L^2}\|P_{k_2}W_2(s)\|_{L^2}
\end{split} 
\end{equation*}
if $|k|\geq 10$. This suffices to control the part of the sum over $k\leq -52\delta m$. Moreover 
\begin{equation*}
\sum_{\min(k_1,k_2)\leq -\D-|k|}|\mathcal{A}^l_{go,m}[P_kG,P_{k_1}W_1,P_{k_2}W_2]|\lesssim \varep_1^32^{-\delta m}
\end{equation*}
if $k\in[-52\delta m,2\delta m]$. This follows as in the proof of Lemma \ref{BNV1}, once we notice that $\Phi(\xi,\eta)\approx 2^{\min(k_1,k_2)/2}$ in 
the support of the integral, so we can integrate by parts in $s$. After these reductions, for \eqref{bnm60} it suffices to 
prove that, for any $k\in[-52\delta m,2\delta m]$,
\begin{equation}\label{bnm61}
 \sum_{k_1,k_2\in[-\D-|k|,\infty)}|\mathcal{A}^l_{go,m}[P_kG,P_{k_1}W_1,P_{k_2}W_2]|\lesssim\varep_1^32^{2\delta^2m}2^{-\delta|k|}.
\end{equation}

To prove \eqref{bnm61} we further decompose in modulation. Let $\overline{k}:=\max(k,k_1,k_2)$ and 
$\underline{p}:=\lfloor \overline{k}^+/2-110\delta m\rfloor$. We define, as in \eqref{bnm4}--\eqref{bnm5},
\begin{equation}\label{bnm72}
\mathcal{A}^{l,p}_{go}[F,H_1,H_2]:=\iint_{\mathbb{R}^2\times\mathbb{R}^2}\mu_l(\xi,\eta)
\varphi_p^{[\underline{p},\infty)}(\Phi(\xi,\eta))\chi_{go}(\xi-\eta)\widehat{F}(\xi-\eta)\widehat{H_1}(\eta)\widehat{H_2}(-\xi)\,d\xi d\eta,
\end{equation}
and
\begin{equation}\label{bnm73}
\mathcal{A}^{l,p}_{go,m}[P_kG,P_{k_1}W_1,P_{k_2}W_2]:=\int_{\mathbb{R}}q_m(s)\mathcal{A}^{l,p}_{go}[P_kG(s),P_{k_1}W_1(s),P_{k_2}W_2(s)]\,ds.
\end{equation}

For $p\geq \underline{p}+1$ we integrate by parts in $s$. As in \eqref{DefTerms} and \eqref{bnm9} let
\begin{equation}\label{bnm80}
\widetilde{\mathcal{A}}^{l,p}_{go}[F,H_1,H_2] :=  \iint_{\R^2\times \R^2} \mu_l(\xi,\eta) \widetilde{\varphi}_p(\Phi(\xi,\eta)) \chi_{go}(\xi-\eta)
  \widehat{F}(\xi-\eta) \widehat{H_1}(\eta) \widehat{H_2}(-\xi)\,d\xi d\eta,
\end{equation}
where $\widetilde{\varphi}_p(x):=2^px^{-1}\varphi_p(x)$. Let $W_{k_1}=P_{k_1}W_1$, $W_{k_2}=P_{k_2}W_2$, and
\begin{equation*}
\begin{split}
&\mathcal{J}_{go,0}^{l,p}(k_1,k_2):=\int_{\R} q_m(s) \widetilde{\mathcal{A}}^{l,p}_{go} [P_k(\partial_s+i\Lambda_\mu)G(s),W_{k_1}(s),W_{k_2}(s)]\,ds, \\
&\mathcal{J}_{go,1}^{l,p}(k_1,k_2):=\int_{\R} q_m(s) \widetilde{\mathcal{A}}^{l,p}_{go} [P_kG(s),(\partial_s+i\Lambda_\nu)W_{k_1}(s),W_{k_2}(s)]\,ds, \\
&\mathcal{J}_{go,2}^{l,p}(k_1,k_2):=\int_{\R} q_m(s) \widetilde{\mathcal{A}}^{l,p}_{go} [P_kG(s),W_{k_1}(s),(\partial_s+i\Lambda_{-\sigma})W_{k_2}(s)]\,ds.
\end{split}
\end{equation*}
As in \eqref{IBPtEE}, we have 
\begin{equation}\label{bnm85}
\begin{split}
\big|\mathcal{A}^{l,p}_{go,m}&[P_kG,P_{k_1}W_1,P_{k_2}W_2]\big|\lesssim 2^{-p}\Big|\int_{\R} q'_m(s) \widetilde{\mathcal{A}}^{l,p}_{go} [P_kG(s),W_{k_1}(s),W_{k_2}(s)]\,ds\Big|\\
&+2^{-p}\big|\mathcal{J}_{go,0}^{l,p}(k_1,k_2)+\mathcal{J}_{go,1}^{l,p}(k_1,k_2)+\mathcal{J}_{go,2}^{l,p}(k_1,k_2)\big|.
\end{split}
\end{equation}

Using Lemma \ref{EstimAKP}, \eqref{njk33}, and \eqref{Gdecay}, it is easy to see that
\begin{equation}\label{bnm86}
 \sum_{k_1,k_2\in[-\D-|k|,\infty)}\sum_{p\geq\underline{p}+1}2^{-p}\Big|\int_{\R} q'_m(s) \widetilde{\mathcal{A}}^{l,p}_{go} [P_kG(s),W_{k_1}(s),W_{k_2}(s)]\,ds\Big|\lesssim\varep_1^32^{-\delta m}.
\end{equation}
Using also \eqref{bnm8.2} and \eqref{Gbound0}, as in the first part of the proof of Lemma \ref{EELemmaSemi}, we have
\begin{equation}\label{bnm87}
 \sum_{k_1,k_2\in[-\D-|k|,\infty)}\sum_{p\geq\underline{p}+1}2^{-p}\big|\mathcal{J}_{go,0}^{l,p}(k_1,k_2)\big|\lesssim\varep_1^32^{-\delta m}.
\end{equation}
Using Lemma \ref{EstimAKP}, \eqref{Gdecay}, and \eqref{bnm12}, it follows that
\begin{equation}\label{bnm87.1}
 \sum_{k_1,k_2\in[-\D-|k|,6m/5]}\sum_{p\geq\underline{p}+1}2^{-p}\big[\big|\mathcal{J}_{go,1}^{l,p}(k_1,k_2)\big|+\big|\mathcal{J}_{go,2}^{l,p}(k_1,k_2)\big|\big]\lesssim\varep_1^32^{-\delta m}.
\end{equation}
Finally, a symmetrization argument as in the proof of Lemma \ref{EElargemod} shows that
\begin{equation}\label{bnm87.2}
 \sum_{k_1,k_2\in[6m/5-10,\infty)}\sum_{p\geq\underline{p}+1}2^{-p}\big|\mathcal{J}_{go,1}^{l,p}(k_1,k_2)+\mathcal{J}_{go,2}^{l,p}(k_1,k_2)\big|\lesssim\varep_1^32^{-\delta m}.
\end{equation}

In view of \eqref{bnm85}--\eqref{bnm87.2}, to complete the proof of \eqref{bnm61} it remains to bound the contribution of small modulations. In the case of ``bad'' frequencies, this is done using the main $L^2$ lemma. Here we need a different argument.

\begin{lemma}\label{BNV2}
Assume that $k\in[-52\delta m,2\delta m]$ and 
$\underline{p}=\lfloor \overline{k}^+/2-110\delta m\rfloor$. Then
\begin{equation}\label{bnm71}
\sum_{\min(k_1,k_2)\geq -\D-|k|}|\mathcal{A}^{l,\underline{p}}_{go,m}[P_kG,P_{k_1}W_1,P_{k_2}W_2]|\lesssim\varep_1^32^{2\delta^2m}2^{-\delta|k|}.
\end{equation}
\end{lemma}

\begin{proof} We need to further decompose the function $G$. Recall that $G\in\mathcal{G}'_+$ and let, for $(k,j)\in\mathcal{J}$
\begin{equation}\label{bnm62}
 f(s)=e^{is\Lambda}G(s),\quad f_{j,k}=P_{[k-2,k+2]}Q_{jk}f,\quad g_{j,k}:=A_{\leq 2\D,\gamma_0}A_{\leq 2\D,\gamma_1}f_{j,k}.
\end{equation}
Compare with Lemma \ref{LinEstLem}. The functions $g_{j,k}$ are supported away from the bad frequencies $\gamma_0$ and $\gamma_1$ and 
$\sum_jg_{j,k}(s)=e^{is\Lambda}G(s)$ away from these frequencies. This induces a decomposition
\begin{equation*}
 \mathcal{A}^{l,\underline{p}}_{go,m}[P_kG,P_{k_1}W_1,P_{k_2}W_2]=\sum_{j\geq \max(-k,0)}\mathcal{A}^{l,\underline{p}}_{go,m}[e^{-is\Lambda}g_{j,k},P_{k_1}W_1,P_{k_2}W_2].
\end{equation*}

Notice that for $j\leq m-\delta m$ we have the stronger estimate \eqref{Gdecaygood} on $\|e^{-is\Lambda}g_{j,k}\|_{L^\infty}$. Therefore, using Lemma \ref{EstimAKP}, if $j\leq m-\delta m$ then
\begin{equation*}
\begin{split}
|\mathcal{A}^{l,\underline{p}}_{go,m}[e^{-is\Lambda}g_{j,k},P_{k_1}W_1,P_{k_2}W_2]|\lesssim \varep_12^k2^{-2k^+}2^{-j/4}\sup_{s\in I_m}\|P_{k_1}W_1(s)\|_{L^2}\|P_{k_2}W_2(s)\|_{L^2}.
\end{split} 
\end{equation*}
Therefore\footnote{This is the only place in the proof of the bound \eqref{bnm2} where one needs the $2^{2\delta^2m}$ factor in the right-hand side.}
\begin{equation*}
 \sum_{j\leq m-\delta m}\sum_{\min(k_1,k_2)\geq -\D-|k|}|\mathcal{A}^{l,\underline{p}}_{go,m}[e^{-is\Lambda}g_{j,k},P_{k_1}W_1,P_{k_2}W_2]|\lesssim\varep_1^32^{2\delta^2m}2^{-\delta|k|}.
\end{equation*}
Similarly, if $j\geq m+60\delta m$ then we also have a stronger bound on $\|e^{-is\Lambda}g_{j,k}\|_{L^\infty}$ in the first line of \eqref{Gbound1}, and the corresponding contributions are controlled in the same way.

It remains to show, for any $j\in[m-\delta m,m+60\delta m]$,
\begin{equation}\label{bnm88}
\sum_{\min(k_1,k_2)\geq -\D-|k|}|\mathcal{A}^{l,\underline{p}}_{go,m}[e^{-is\Lambda}g_{j,k},P_{k_1}W_1,P_{k_2}W_2]|\lesssim\varep_1^32^{-\delta m}.
\end{equation} 
For this we use Schur's test. Since $\min(k,k_1,k_2)\geq -53\delta m$ it follows from Proposition \ref{volume} (i)  and the 
bound $\|\widehat{g_{j,k}}\|_{L^2}\lesssim \varep_12^{-8k^+}2^{-j+50\delta j}$ that
\begin{equation*}
\int_{\mathbb{R}^2}|\mu_l(\xi,\eta)|\varphi_{\leq\underline{p}}(\Phi(\xi,\eta))|\widehat{g_{j,k}}(\xi-\eta)|
\varphi_{[k_1-2,k_1+2]}(\eta)\,d\eta\lesssim \varep_12^{(\underline{p}-\overline{k}^+/2)/2+\delta m}2^{-j+50\delta j}
\end{equation*}
for any $\xi\in\mathbb{R}^2$ fixed with $|\xi|\in [2^{k_2-4},2^{k_2+4}]$. 
The integral in $\xi$ (for $\eta$ fixed) can be estimated in the same way. Given the choice of $\underline{p}$, the desired bound \eqref{bnm88} follows using Schur's lemma.
\end{proof}

\section{Energy estimates, III: proof of the main $L^2$ lemma}\label{L2proof}

In this section we prove Lemma \ref{L2Prop0}. We divide the proof into several cases. Let
\begin{equation*}
\chi_{\gamma_l}(x):=\varphi(2^{\D}(|x|-\gamma_l)),\qquad l\in\{0,1\}.
\end{equation*}
We start the most difficult case when $|\xi-\eta|$ is close to $\gamma_0$ and $2^k \gg 1$.
In this case $\hat{\Upsilon}$ can vanish up to order $1$ (so we can have $2^q \ll 1$ in the notation of the Lemma \ref{L2Lem1} below). 

\begin{lemma}\label{L2Lem1}
The conclusion of Lemma \ref{L2Prop0} holds if $k\geq 3\D_1/2$ and $\widehat{g}$ is supported in the set $\{||\xi|-\gamma_0|\leq 2^{-100}\}$.
\end{lemma}

\begin{proof} We will often use the results in Lemma \ref{Geomgamma0} below. We may assume that $\sigma=\nu=+$ in the definition of $\Phi$, since otherwise the operator is trivial. We may also assume that $\mu=+$, in view of the formula \eqref{ph11.1}. 

In view of Lemma \ref{Geomgamma0} (ii) we may assume that either $(\xi-\eta)\cdot\xi^\perp\approx 2^k$ or $-(\xi-\eta)\cdot\xi^\perp\approx 2^k$ in the support of the integral, due to the factor $\chi(2^{-p}\Phi(\xi,\eta))$. Thus we may define
\begin{equation}\label{njk40}
a^\pm(\xi,\eta)=a(\xi,\eta)\mathbf{1}_{\pm}((\xi-\eta)\cdot\xi^{\perp}),
\end{equation}
and decompose the operator $L_{p,k}=L_{p,k}^++L_{p,k}^-$ accordingly. The two operators can be treated in similar ways, so we will concentrate on the operator $L_{p,k}^+$.

To apply the main $TT^\ast$ argument we need to first decompose the operators $L_{p,k}$. For $\kappa:=2^{-\D^{3/2}}$ (a small parameter) and $\psi\in C^\infty_0(-2,2)$ satisfying $\sum_{v\in\mathbb{Z}}\psi(.+v)\equiv 1$, we write
\begin{equation}\label{njk50}
\begin{split}
L_{p,k}^+&=\sum_{q,r\in\mathbb{Z}}\sum_{j\geq 0} L_{p,k,q}^{r,j},\\
L_{p,k,q}^{r,j}f(x)&:=\int_{\mathbb{R}^2}e^{is\Phi(x,y)}\chi(2^{-p}\Phi(x,y))\varphi_q(\hat{\Upsilon}(x,y))\psi(\kappa^{-1}2^{-q}\hat{\Upsilon}(x,y)-r)\varphi_k(y)a^+_j(x,y)f(y)dy,\\
 a^+_j(x,y)&:=A(x,y)\chi_{\gamma_0}(x-y)\mathbf{1}_+((x-y)\cdot x^{\perp})\widehat{g_j}(x-y),\qquad g_j:=A_{\geq 0,\gamma_0}[\varphi_j^{[0,\infty)}\cdot g].
\end{split}
\end{equation}
In other words, we insert the decompositions
\begin{equation*}
g=\sum_{j\geq 0}g_j,\qquad 1=\sum_{q,r\in\mathbb{Z}}\varphi_q(\hat{\Upsilon}(x,y))\psi(\kappa^{-1}2^{-q}\hat{\Upsilon}(x,y)-r)
\end{equation*} 
in the formula \eqref{LMainDef} defining the operators $L_{p,k}$. The parameters $j$ and $r$ play a somewhat minor role in the proof (one can focus on the main case $j=0$) 
but the parameter $q$ is important. Notice that $q\leq-\D/2$, in view of \eqref{ModUp}. 
The hypothesis $\Vert g\Vert_{Z_1\cap H^{N_1/3,0}_{\Omega}}\lesssim 1$ and Lemma \ref{LinEstLem} (i) show that
\begin{equation}\label{njk41}
\|\widehat{g_j}\|_{L^\infty}\lesssim 2^{-j(1/2-55\delta)},\qquad \|\sup_{\theta\in\mathbb{S}^1}|\widehat{g_j}(r\theta)|\|_{L^2(rdr)}\lesssim 2^{-j(1-55\delta)}.
\end{equation}

Note that, for fixed $x$ (respectively $y$) the support of integration is included in $\mathcal{S}^{1,-}_{p,q,r}(x)$ (respectively $\mathcal{S}^{2,-}_{p,q,r}(y)$), see \eqref{DefRect1}--\eqref{DefRect2}. We can use this to estimate the Schur norm of the kernel. It follows from \eqref{VolumeSpq2} and the first bound in \eqref{njk41} that
\begin{align}\label{SchurL2Lem}
\sup_{x}\int_{\mathbb{R}^2}\vert \chi(2^{-p}\Phi(x,y))\varphi_q(\hat{\Upsilon}(x,y))\varphi_k(y)a_j^+(x,y)\vert dy\lesssim \Vert a_j^+ \Vert_{L^\infty}\vert \mathcal{S}^{1,-}_{p,q,r}(x)\vert\lesssim 2^{q+p-k/2}2^{-j/3}.
\end{align}
A similar estimate holds for the $x$ integral (keeping $y$ fixed). Moreover, using \eqref{VolumeSpq} and the second bound in \eqref{njk41} to estimate the left-hand side of \eqref{SchurL2Lem} by $C2^{-j+55\delta j}2^{p-k/2}$. In view of Schur's lemma, we have
\begin{equation*}
\|L^{r,j}_{p,k,q}\|_{L^2\to L^2}\lesssim \min(2^{q+p-k/2}2^{-j/3},2^{-j+55\delta j}2^{p-k/2}).
\end{equation*}
These bounds suffice to control the contribution of the operators $L^{r,j}_{p,k,q}$, unless
\begin{equation}\label{L2Lem1Reduced1}
q \geq \D + \max \Big\{\frac{1}{2}(p-\frac{k}{2}),-\frac{m}{3} \Big\}
  \quad \text{ and } \quad 0 \leq j \leq \min \Big\{\frac{4m}{9},-\frac{2}{3}(p-\frac{k}{2})\Big\}.
\end{equation}

Therefore, in the rest of the proof we may assume that \eqref{L2Lem1Reduced1} holds, so $\kappa 2^q\gg 2^{p-\frac{k}{2}}$. We use the $TT^\ast$ argument and Schur's test. It suffices to show that
\begin{align}\label{SuffKernel}
\sup_x\int_{\mathbb{R}^2} & \vert K(x,\xi) \vert \,d\xi+\sup_{\xi}\int_{\mathbb{R}^2}\vert K(x,\xi) \vert \,dx
  \lesssim 2^{6\delta^2 m} \big(2^{3(p-\frac{k}{2})} + 2^{2(p-\frac{k}{2})} 2^{-\frac{2}{3}m} \big)
\end{align}
for $p,k,q,r,j$ fixed (satisfying \eqref{AssL2Prop} and \eqref{L2Lem1Reduced1}), where
\begin{equation}\label{SuffKernel2}
\begin{split}
K(x,\xi)&:=\int_{\mathbb{R}^2}e^{is\Theta(x,\xi,y)}\chi(2^{-p}\Phi(x,y))\chi(2^{-p}\Phi(\xi,y))\psi_{q,r}(x,\xi,y)a_j^+(x,y)\overline{a_j^+(\xi,y)}dy,\\
\Theta(x,\xi,y)&:=\Phi(x,y)-\Phi(\xi,y)=\Lambda(x)-\Lambda(\xi)-\Lambda(x-y)+\Lambda(\xi-y),\\
\psi_{q,r}(x,\xi,y) & := \varphi_q(\hat{\Upsilon}(x,y))\varphi_q(\hat{\Upsilon}(\xi,y))
  \psi(\kappa^{-1}2^{-q}\hat{\Upsilon}(x,y)-r)\psi(\kappa^{-1}2^{-q}\hat{\Upsilon}(\xi,y)-r)\varphi_k(y)^2.
	\end{split}
\end{equation}

Since $K(x,\xi)=\overline{K(\xi,x)}$, it suffices to prove the bound on the first term in the left-hand side of \eqref{SuffKernel}. The main idea of the proof is to show that $K$ is essentially supported in the set where $\omega:=x-\xi$ is small. Note first that, in view of \eqref{VolumeSpq},  we may assume that
\begin{equation}\label{RestricDiam}
\vert\omega\vert=\vert x-\xi\vert\lesssim\kappa 2^q\ll 1.
\end{equation}

{\bf{Step 1:}} We will show in {\bf{Step 2}} below that if
\begin{equation}\label{C3Ass}
\text{ if }\quad\vert\omega\vert\ge L:=2^{2\delta^2m}\big[2^{p-k/2}2^{-q}+2^{j-q-m}+2^{-2m/3-q}\big]\quad\text{ then }\quad\vert K(x,\xi)\vert\lesssim 2^{-4m}.
\end{equation}

Assuming this, we show now how prove the bound on the first term in \eqref{SuffKernel}. Notice that $L\ll 1$, in view of \eqref{AssL2Prop} and \eqref{L2Lem1Reduced1}. We decompose, for fixed $x$,
\begin{equation*}
\int_{\mathbb{R}^2}\vert K(x,\xi)\vert \,d\xi\lesssim \int_{\{\vert\omega\vert\le L\}}\vert K(x,x-\omega)\vert \,d\omega+\int_{\{\vert\omega\vert\ge L\}}\vert K(x,x-\omega)\vert \,d\omega.
\end{equation*}
Combining \eqref{RestricDiam} and \eqref{C3Ass}, we obtain a suitable bound for the second integral. We now turn to the first integral, which we bound using Fubini and the formula \eqref{SuffKernel2} by
\begin{equation}\label{njk42}
C\Vert a_j^+\Vert_{L^\infty}\int_{\mathbb{R}^2}\vert a_j^+(x,y)\vert\chi(2^{-p}\Phi(x,y))\varphi_{q}(\hat{\Upsilon}(x,y))\varphi_k(y)^2\Big(\int_{\{\vert\omega\vert\le L\}}|\chi(2^{-p}\Phi(x-\omega,y))|\,d\omega\Big)dy.
\end{equation}

We observe that, for fixed $x,y$ satisfying $|\vert x-y\vert-\gamma_0|\ll 1$, $\vert x\vert\approx 2^k\gg 1$, we have
\begin{equation}\label{njk43}
\int_{\{\vert\omega\vert\le L\}}|\chi(2^{-p}\Phi(x-\omega,y))|\,d\omega\lesssim 2^{p-k/2}L.
\end{equation}
Indeed, it follows from \eqref{Deftheta1} that if $z=(x-y-\omega)=(\rho\cos\theta,\rho\sin\theta)$, $|\omega|\leq L$, and $|\Phi(y+z,y)|\leq 2^{p}$, then $|\rho-|x-y||\lesssim L$ and $\theta$ belongs to a union of two intervals of length $\lesssim 2^{p-k/2}$. The desired bound \eqref{njk43} follows.

Using also \eqref{SchurL2Lem} and $\|a_j\|_{L^\infty}\lesssim 2^{-j/3}$, it follows that the expression in \eqref{njk42} is bounded by $C2^{2(p-k/2)} 2^{-2j/3}  2^q L$. The desired bound \eqref{SuffKernel} follows, using also the restrictions \eqref{L2Lem1Reduced1}.

{\bf{Step 2:}} We prove now \eqref{C3Ass}. We define orthonormal frames $(e_1,e_2)$ and $(V_1,V_2)$,
\begin{equation}\label{Frames}
\begin{split}
&e_1:=\frac{\nabla_x\Phi(x,y)}{\vert\nabla_x\Phi(x,y)\vert},\quad e_2=e_1^\perp,\qquad V_1:=\frac{\nabla_y\Phi(x,y)}{\vert\nabla_y\Phi(x,y)\vert},\quad V_2=V_1^\perp,\\
&\omega=x-\xi=\omega_1e_1+\omega_2e_2.
\end{split}
\end{equation}
Note that $\omega_1$, $\omega_2$ are functions of $(x,y,\xi)$. We first make a useful observation: if $\vert \Theta(x,\xi,y)\vert\lesssim 2^{p}$, and $|\omega|\ll 1$ then
\begin{equation}\label{Restricw1}
\vert\omega_1\vert\lesssim 2^{-k/2}\left(2^{p}+|\omega|^2\right).
\end{equation}
This follows from a simple Taylor expansion, since
\begin{equation*}
|\Phi(x,y)-\Phi(\xi,y)-\omega\cdot\nabla_x\Phi(x,y)|\lesssim |\omega|^2.
\end{equation*}

We turn now to the proof of \eqref{C3Ass}. Assuming that $x,\xi$ are fixed with $|x-\xi|\geq L$ and using \eqref{Restricw1}, we see that, on the support of integration, $|\omega_2|\approx \vert\omega\vert$ and
\begin{equation}\label{CompPhaseDer}
\begin{split}
V_2\cdot\nabla_y\Theta(x,\xi,y)&=V_2\cdot\nabla_y\left\{-\Lambda(x-y)+\Lambda(\xi-y)\right\}\\
&=V_2\cdot\nabla^2_{x,y}\Phi(x,y)\cdot (x-\xi)+O(\vert \omega\vert^2)\\
&=\omega_2\hat{\Upsilon}(x,y)+O(\vert\omega_1\vert+\vert\omega\vert^2).
\end{split}
\end{equation}
Using \eqref{L2Lem1Reduced1}, \eqref{C3Ass}, \eqref{Restricw1} and \eqref{RestricDiam} (this is where we need $\kappa\ll1$), we obtain that
\begin{equation*}
\vert V_2\cdot\nabla_y\Theta(x,\xi,y)\vert\approx 2^q\vert\omega_2\vert\approx 2^q\vert\omega\vert
\end{equation*}
in the support of the integral. Using that
\begin{equation*}
e^{is\Theta}=\frac{-i}{sV_2\cdot\nabla_y\Theta}V_2\cdot\nabla_y e^{is\Theta},\qquad\vert D^\alpha_y\Theta\vert\lesssim\vert\omega\vert,
\end{equation*}
and letting $\Theta_{(1)}:=V_2\cdot\nabla_y\Theta$, after integration by parts we have
\begin{equation*}
\begin{split}
K(x,\xi)&=i\int_{\mathbb{R}^2}e^{is\Theta}\partial_l\Big\{V_2^l\frac{1}{s\Theta_{(1)}}\chi(2^{-p}\Phi(x,y))\chi(2^{-p}\Phi(\xi,y))\psi_{q,r}(x,\xi,y)a^+_j(x,y)\overline{a^+_j(\xi,y)}\Big\}dy.
\end{split}
\end{equation*}
We observe that
\begin{equation*}
V_2^l\partial_l[\chi(2^{-p}\Phi(x,y))\chi(2^{-p}\Phi(\xi,y))]=-2^{-p}\Theta_{(1)}\cdot[\chi(2^{-p}\Phi(x,y))\chi'(2^{-p}\Phi(\xi,y))].
\end{equation*}
This identity is the main reason for choosing $V_2$ as in \eqref{Frames}, and this justifies the definition of the function $\Upsilon$ (intuitively, we can only integrate by parts in $y$ along the level sets of the function $\Phi$, due to the very large $2^{-p}$ factor). Moreover
\begin{equation*}
\|D^\alpha_y\psi_{q,r}(x,\xi,y)|\lesssim 2^{-q|\alpha|},\qquad |D_y^\alpha a_j^+(v,y)|\lesssim_\alpha 2^{\vert\alpha\vert j}+2^{|\alpha|m/3},\,v\in\{x,\xi\},
\end{equation*}
in the support of the integral defining $K(x,\xi)$. We integrate by parts many times in $y$ as above. At every step we gain a factor of $2^m2^q|\omega|$ and lose a factor of $2^{-p}2^q|\omega|+2^{-q}+2^j+2^{m/3}$. The desired bound in \eqref{C3Ass} follows. This completes the proof.
\end{proof}

We consider now the (easier) case when $\vert\xi-\eta\vert$ is close to $\gamma_1$ and $k$ is large.

\begin{lemma}\label{L2Lem2}
The conclusion of Lemma \ref{L2Prop0} holds if $k\geq 3\D_1/2$ and $\widehat{g}$ is supported in the set $\{||\xi|-\gamma_1|\leq 2^{-100}\}$.
\end{lemma}

\begin{proof} Using \eqref{ModUp}, we see that on the support of integration we have $\vert\hat{\Upsilon}(\xi,\eta)\vert\approx 1$. The proof is similar to the proof of Lemma \ref{L2Lem1} in the case $2^q\approx 1$. The new difficulties come from the less favorable decay in $j$ close to $\gamma_1$ and from the fact that the conclusions in Lemma \ref{Geomgamma0} (iii) do not apply. We define $a_j^{\pm}$ as in \eqref{njk50} (with $\gamma_1$ replacing $\gamma_0$ and $g_j:=A_{\geq 4,\gamma_1}[\varphi_j^{[0,\infty)}\cdot g]$), and
\begin{equation}\label{njk51.1}
L_{p,k}^{x_0,j}f(x):=\varphi_{\leq -\D}(x-x_0)\int_{\mathbb{R}^2}e^{is\Phi(x,y)}\chi(2^{-p}\Phi(x,y))\varphi_k(y)a^+_j(x,y)f(y)dy,
\end{equation}
for any $x_0\in\mathbb{R}^2$. We have 
\begin{equation}\label{njk51}
\|\widehat{g_j}\|_{L^\infty}\lesssim 2^{6\delta j},\qquad \|\sup_{\theta\in\mathbb{S}^1}|\widehat{A_{n,\gamma_1}g_j}(r\theta)|\|_{L^2(rdr)}\lesssim 2^{(1/2-49\delta)n-j(1-55\delta)},
\end{equation}
for $n\geq 1$, as a consequence of Lemma \ref{LinEstLem} (i). Notice that these bounds are slightly weaker than the bounds in \eqref{njk41}. However, we can still estimate (compare with \eqref{SchurL2Lem}) 
\begin{align}\label{SchurL2Lem2}
\sup_x\int_{\R^2}\vert \chi(2^{-p}\Phi(x,y))\varphi_k(y) a_j^+(x,y)\vert dy\lesssim 2^{p-k/2}\cdot 2^{-(1-55\delta)j}.
\end{align}
Indeed, we use only the second bound in \eqref{njk51}, decompose the integral as a sum of integrals over the dyadic sets $||x-y|-\gamma_1|\approx 2^{-n}$, $n\geq 1$, and use \eqref{Deftheta1} and the Cauchy--Schwarz in each dyadic set. As a consequence of \eqref{SchurL2Lem2}, it remains to consider the sum over $j\leq 4m/9$.

We can then proceed as in the proof of Lemma \ref{L2Lem1}. Using the $TT^\ast$ argument for the operators $L_{p,k}^{x_0,j}$ and Schur's lemma, it suffices to prove bounds similar to those in \eqref{SuffKernel}. Let $\omega=x-\xi$, and notice that $|\omega|\leq 2^{-\D+10}$. This replaces the diameter bound \eqref{RestricDiam} and is the main reason for adding the localization factors $\varphi_{\leq -\D}(x-x_0)$ in \eqref{njk51.1}. The main claim is that
\begin{equation}\label{njk52}
\text{ if }\quad|\omega| \geq L:=2^{2\delta^2m}(2^{p-k/2}+2^{j-m}+2^{-2m/3})\quad\text{ then }\quad |K(x,\xi)|\lesssim 2^{-4m}.
\end{equation}
The same argument as in {\bf{Step 1}} in the proof of Lemma \ref{L2Lem1} shows that this claim suffices. Moreover, this claim can be proved using integration by parts, as in {\bf{Step 2}} in the proof of Lemma \ref{L2Lem1}. The conclusion of the lemma follows. 
\end{proof}

Finally, we now consider the case of low frequencies.

\begin{lemma}\label{L2Lem3}
The conclusion of Lemma \ref{L2Prop0} holds if $k\in[-100,7\D_1/4]$. 
\end{lemma}

\begin{proof} For small frequencies, the harder case is when $|\xi-\eta|$ is close to $\gamma_1$, since the conclusions of Lemma \ref{lemmaD2} are weaker than the conclusions of Lemma \ref{lemmaD1}, and the decay in $j$ is less favorable. So we will concentrate on this case.

We need to first decompose our operator. For $j\geq 0$ and $l\in\mathbb{Z}$ we define
\begin{equation}\label{njk70}
 a^{\pm}_{j,l}(x,y):=A(x,y)\chi_{\gamma_1}(x-y)\varphi_l^{\pm}((x-y)\cdot x^{\perp})\widehat{g_j}(x-y),\qquad g_j:=A_{\geq 4,\gamma_1}[\varphi_j^{[0,\infty)}\cdot P_{[-8,8]}g],
\end{equation}
where $\varphi_l^{\pm}(v):=\mathbf{1}_{\pm}(v)\varphi_l(v)$. This is similar to \eqref{njk50}, but with the additional dyadic decomposition in terms of the angle $|(x-y)\cdot x^{\perp}|\approx 2^l$. Then we decompose, as in \eqref{njk50},
\begin{equation}\label{njk71}
L_{p,k}=\sum_{q,r\in\mathbb{Z}}\sum_{j\geq 0} \sum_{l\in\mathbb{Z}}\sum_{\iota\in\pm}L_{p,k,q}^{r,j,l,\iota},
\end{equation} 
where, with $\kappa=2^{-\D^{3/2}}$ and $\psi\in C^\infty_0(-2,2)$ satisfying $\sum_{v\in\mathbb{Z}}\psi(.-v)\equiv 1$ as before,
\begin{equation}\label{njk72}
\begin{split}
L_{p,k,q}^{r,j,l,\iota}f(x):=\varphi_{\geq -100}(x)\int_{\mathbb{R}^2}&e^{is\Phi(x,y)}\chi(2^{-p}\Phi(x,y))\\
&\times\varphi_q(\Upsilon(x,y))\psi(\kappa^{-1}2^{-q}{\Upsilon}(x,y)-r)\varphi_k(y)a^\iota_{j,l}(x,y)f(y)dy.
\end{split}
\end{equation} 
 
We consider two main cases, depending on the size of $q$.

{\bf{Case 1.}} $q\leq -\D_1$. As a consequence of \eqref{kn3}, the operators $L_{p,k,q}^{r,j,l,\iota}$ are nontrivial only if $2^k\approx 1$ and $2^l\approx 1$. Using also \eqref{kn2} it follows that
\begin{equation}\label{njk73}
\begin{split}
&|\nabla_x\Phi|\approx 1,\qquad|\nabla_x\Upsilon\cdot\nabla_x^\perp\Phi|\approx 1,\\
&|\nabla_y\Phi|\approx 1,\qquad|\nabla_y\Upsilon\cdot\nabla_y^\perp\Phi|\approx 1,
\end{split}
\end{equation}
in the support of the integrals defining the operators $L_{p,k,q}^{r,j,l,\iota}$. 

{\bf{Step 1.}} The proof proceeds as in Lemma \ref{L2Lem1}. For simplicity, we assume that $\iota=+$. Let
\begin{equation}\label{njk74}
\begin{split}
\mathcal{S}^1_{p,q,r,l}(x):=\{z:&||z|-\gamma_1|\leq 2^{-\D+1},\,|\Phi(x,x-z)|\leq 2^{p+1},\,|{\Upsilon}(x,x-z)|\leq 2^{q+2},\\
&|{\Upsilon}(x,x-z)-r\kappa 2^q|\leq 10\kappa 2^q,\,z\cdot x^\perp\in[2^{l-2},2^{l+2}]\}.
\end{split}
\end{equation}
Recall that, if $z=(\rho\cos\theta,\rho\sin\theta)$ and $x=(|x|\cos\alpha,|x|\sin\alpha)$ then
\begin{equation}\label{njk75}
\Phi(x,x-z)=\lambda(|x|)-\mu\lambda(\rho)-\nu\lambda\big(\sqrt{|x|^2+\rho^2-2\rho|x|\cos(\theta-\alpha)}\big).
\end{equation}
It follows from \eqref{njk73} and the change of variables argument in the proof of Lemma \ref{Geomgamma0} (iii) that
\begin{equation}\label{njk76}
|\mathcal{S}^1_{p,q,r,l}(x)|\lesssim 2^{p+q},\qquad\mathrm{diam}(\mathcal{S}^1_{p,q,r,l}(x))\lesssim 2^p+\kappa 2^q
\end{equation}
if $|x|\approx 1$ and $2^l\approx 1$. Moreover, using \eqref{njk75}, for any $x$ and $\rho$,
\begin{equation}\label{njk77}
|\{\theta:z=(\rho\cos\theta,\rho\sin\theta)\in \mathcal{S}^1_{p,q,r,l}(x)\}|\lesssim 2^p.
\end{equation}
Therefore, using \eqref{njk51} and these last two bounds, if $|x|\approx 1$ then
\begin{equation}\label{njk78}
\int_{\mathbb{R}^2}|\chi(2^{-p}\Phi(x,y))\varphi_q({\Upsilon}(x,y))\varphi_k(y)a^+_{j,l}(x,y)|dy\lesssim \min(2^{p+q}2^{6\delta j},2^p2^{-j+55\delta j}).
\end{equation}
One can prove a similar bound for the $x$ integral, keeping $y$ fixed. In view of Schur's lemma, it remains the bound the contribution of the terms for which
\begin{equation}\label{L2Lem3Reduced1}
q \geq \D + \max \Big\{\frac{p}{2},-\frac{m}{3}\Big\} \quad \hbox{ and } \quad 0\leq j\leq \min\Big\{\frac{4m}{9},-\frac{2p}{3}\Big\}.
\end{equation}

{\bf{Step 2.}} Assuming \eqref{L2Lem3Reduced1}, we use the $TT^\ast$ argument and Schur's test. It suffices to show that
\begin{align}\label{SuffKernel6}
\sup_x\int_{\mathbb{R}^2} & \vert K(x,\xi) \vert \,d\xi+\sup_{\xi}\int_{\mathbb{R}^2}\vert K(x,\xi) \vert \,dx
  \lesssim 2^{6\delta m} \big(2^{3p} + 2^{2p-2m/3} \big)
\end{align}
for $p,k,q,r,j,l$ fixed satisfying \eqref{L2Lem3Reduced1}, where
\begin{equation}\label{njk80}
\begin{split}
K(x,\xi):=\varphi_{\geq -100}&(x)\varphi_{\geq -100}(\xi)\int_{\mathbb{R}^2}e^{is\Theta(x,\xi,y)}\\
&\times\chi(2^{-p}\Phi(x,y))\chi(2^{-p}\Phi(\xi,y))\psi_{q,r}(x,\xi,y)a_{j,l}^+(x,y)\overline{a_{j,l}^+(\xi,y)}dy,
	\end{split}
\end{equation}
and, as in \eqref{SuffKernel2},
\begin{equation*}
\begin{split}
\Theta(x,\xi,y)&:=\Phi(x,y)-\Phi(\xi,y)=\Lambda(x)-\Lambda(\xi)-\Lambda_\mu(x-y)+\Lambda_\mu(\xi-y),\\
\psi_{q,r}(x,\xi,y) & := \varphi_q({\Upsilon}(x,y))\varphi_q({\Upsilon}(\xi,y))
  \psi(\kappa^{-1}2^{-q}{\Upsilon}(x,y)-r)\psi(\kappa^{-1}2^{-q}{\Upsilon}(\xi,y)-r)\varphi_k(y)^2.
	\end{split}
\end{equation*}

Let $\omega:=x-\xi$. As in the proof of Lemma \ref{L2Lem1} the main claim is that 
\begin{equation}\label{C3AssLF}
\text{ if }\quad \vert\omega\vert\geq L:=2^{2\delta^2 m}(2^{p-q}+2^{j-q-m}+2^{-q-2m/3})\quad\text{ then }\quad |K(x,\xi)|\lesssim 2^{-4m}.
\end{equation}
The same argument as in {\bf{Step 1}} in the proof of Lemma \ref{L2Lem1}, using \eqref{njk78}, shows that this claim suffices. Moreover, this claim can be proved using integration by parts, as in {\bf{Step 2}} in the proof of Lemma \ref{L2Lem1}. The desired bound \eqref{SuffKernel6} follows.

{\bf{Case 2.}} $q\geq -\D_1$. There is one new issue in this case, namely when the angular parameter $2^l$ is very small and bounds like \eqref{njk77} fail. As in the proof of Lemma \ref{L2Lem2}, we also need to modify the main decomposition \eqref{njk71}. Let
\begin{equation}\label{njk90}
L_{p,k,q}^{x_0,j,l}f(x):=\varphi_{\leq-\D}(x-x_0)\int_{\mathbb{R}^2}e^{is\Phi(x,y)}\chi(2^{-p}\Phi(x,y))\varphi_q(\Upsilon(x,y))\varphi_k(y)a^+_{j,l}(x,y)f(y)dy.
\end{equation} 
Here $x_0\in\mathbb{R}^2$, $|x_0|\geq 2^{-110}$, and the localization factor on $x-x_0$ leads to a good upper bound on $|x-\xi|$ in the $TT^\ast$ argument below. It remains to prove that if $q\geq -\D_1$ then
\begin{equation}\label{njk91}
\|L_{p,k,q}^{x_0,j,l}\|_{L^2\to L^2}\lesssim 2^{\delta^2l}2^{-\delta^2j}2^{30\delta m}(2^{(3/2)p}+2^{p-m/3}).
\end{equation}

{\bf{Step 1.}} We start with a Schur bound. For $x\in\mathbb{R}^2$ with $|x|\in[2^{-120},2^{\D_1+10}]$ let
\begin{equation}\label{njk91.1}
\begin{split}
\mathcal{S}^1_{p,q,l}(x):=\{z:&||z|-\gamma_1|\leq 2^{-\D+1},\,|\Phi(x,x-z)|\leq 2^{p+1},\\
&\,|{\Upsilon}(x,x-z)|\in [2^{q-2},2^{q+2}],\,z\cdot x^\perp\in[2^{l-2},2^{l+2}]\}.
\end{split}
\end{equation}
The condition $|\Upsilon(x,x-z)|\geq 2^{-\D_1-4}$ shows that $|\nabla_z[\Phi(x,x-z)]|\in[2^{-4\D_1},2^{\D_1}]$ for $z\in \mathcal{S}^1_{p,q,l}(x)$. The formula \eqref{njk75} shows that
\begin{equation}\label{njk93}
|\{\theta:z=(\rho\cos\theta,\rho\sin\theta)\in\mathcal{S}^1_{p,q,l}(x)\}|\lesssim 2^{p-l}.
\end{equation}
Moreover, we claim that for any $x$,
\begin{equation}\label{njk94}
|\mathcal{S}^1_{p,q,l}(x)|\lesssim 2^{p+l}.
\end{equation}
Indeed, this follows from \eqref{njk93} if $l\geq -\D$. On the other hand, if $l\leq -\D$ then $\partial_\theta[\Phi(x,x-z)]\leq 2^{-\D/2}$ (due to \eqref{njk75}), so $\partial_\rho[\Phi(x,x-z)]\geq 2^{-5\D_1}$ (due to the inequality $|\nabla_z[\Phi(x,x-z)]|\in[2^{-4\D_1},2^{\D_1}]$). Recalling also \eqref{njk51}, it follows from these last two bounds that
\begin{equation}\label{njk95}
\int_{\mathbb{R}^2}|\chi(2^{-p}\Phi(x,y))\varphi_q(\Upsilon(x,y))\varphi_k(y)a^+_{j,l}(x,y)|dy\lesssim \min(2^{6\delta j}2^{p+l},2^{-j+55\delta j}2^{p-l}),
\end{equation}
if $|x|\in[2^{-120},2^{\D_1+10}]$. In particular, the integral is also bounded by $C2^p2^{-j/2+31\delta j}$. The integral in $x$, keeping $y$ fixed, can be estimated in a similar way. The desired bound \eqref{njk91} follows unless
\begin{equation}\label{njk96}
j\leq \min(2m/3,-p)-\D,\qquad l\geq \max (p/2,-m/3)+\D.
\end{equation}

{\bf{Step 2.}} Assuming \eqref{njk96}, we use the $TT^\ast$ argument and Schur's test. 
It suffices to show that
\begin{align}\label{njk97}
\sup_x\int_{\mathbb{R}^2} \vert K(x,\xi) \vert \,d\xi
  \lesssim 2^{55\delta m} \big(2^{3p} + 2^{2p-2m/3} \big)
\end{align}
for $p,k,q,x_0,j,l$ fixed, where $\Theta(x,\xi,y)=\Phi(x,y)-\Phi(\xi,y)$ and
\begin{equation}\label{njk98}
\begin{split}
K(x,\xi):=\varphi_{\leq -\D}&(x-x_0)\varphi_{\leq -\D}(\xi-x_0)\int_{\mathbb{R}^2}e^{is\Theta(x,\xi,y)}\chi(2^{-p}\Phi(x,y))\chi(2^{-p}\Phi(\xi,y))\\
&\times\varphi_q({\Upsilon}(x,y))\varphi_q({\Upsilon}(\xi,y))
  \varphi_k(y)^2a_{j,l}^+(x,y)\overline{a_{j,l}^+(\xi,y)}dy.
	\end{split}
\end{equation}

Let $\omega=x-\xi$. As before, the main claim is that 
\begin{equation}\label{njk99}
\text{ if }\quad \vert\omega\vert\geq L:=2^{2\delta^2 m}(2^{p}+2^{j-m}+2^{-2m/3})\quad\text{ then }\quad |K(x,\xi)|\lesssim 2^{-4m}.
\end{equation}

To see that this claim suffices, we use an argument similar to the one in {\bf{Step 1}} in the proof of Lemma \ref{L2Lem1}. Indeed, up to acceptable errors, the left-hand side of \eqref{njk97} is bounded by
\begin{equation}\label{njk100}
\begin{split}
C\Vert a_{j,l}^+\Vert_{L^\infty}\sup_{|x-x_0|\leq 2^{-\D+2}}\int_{\mathbb{R}^2}\vert a_{j,l}^+(x,y)\vert&\chi(2^{-p}\Phi(x,y))\varphi_{q}({\Upsilon}(x,y))\\
&\times\Big(\int_{\{\vert\omega\vert\le L\}}|\chi(2^{-p}\Phi(x-\omega,y))|\,d\omega\Big)dy.
\end{split}
\end{equation}
Notice that if $|{\Upsilon}(x,y)|\geq 2^{-\D_1-2}$ then $|(\nabla_x\Phi)(x,y)|\geq 2^{-4\D_1}$, thus $|(\nabla_w\Phi)(x-w,y)|\geq 2^{-4\D_1-1}$ if $|\omega|\leq L\leq 2^{-\D}$. Therefore, the integral in $\omega$ in the expression above is bounded by $C2^pL$. Using also \eqref{njk95}, the expression in \eqref{njk100} is bounded by
\begin{equation*}
C2^{6\delta j}2^pL\cdot 2^p2^{-j/2+32\delta j}\lesssim 2^{\delta m}2^{3p}+2^{40\delta m}2^{2p+j/2-m}+2^{\delta m}2^{2p-2m/3}
\end{equation*}
The desired bound \eqref{njk97} follows using also that $j\leq 2m/3$, see \eqref{njk96}.

The claim \eqref{njk99} follows by the same integration by parts argument as in {\bf{Step 2}} in the proof of Lemma \ref{L2Lem1}, once we recall that $|(\nabla_x\Phi)(x,y)|\geq 2^{-4\D_1}$ and $|(\nabla_y\Phi)(x,y)|\geq 2^{-4\D_1}$ in the support of the integral, while $|\omega|\leq 2^{-\D+4}$. This completes the proof of the lemma.
\end{proof}

\section{Dispersive analysis, I: setup and the main proposition}\label{NotationsF}

\subsection{The Duhamel formula and the main proposition}\label{Duhamel}

In this section we start the proof of Proposition \ref{MainBootstrapDisp}. With $\mathcal{U}=\langle\nabla\rangle h+i|\nabla|^{1/2}\phi$, 
assume that $\mathcal{U}$ is a solution of the equation
\begin{equation}\label{system8}
(\partial_{t}+i\Lambda)\mathcal{U}=\mathcal{N}_2+\mathcal{N}_3+\mathcal{N}_{\geq 4},
\end{equation}
on some time interval $[0,T]$, $T\geq 1$, where $\mathcal{N}_2$ is a quadratic nonlinearity in $\mathcal{U},\overline{\mathcal{U}}$, $\mathcal{N}_3$ 
is a cubic nonlinearity, and $\mathcal{N}_{\geq 4}$ is a higher order nonlinearity. Such an equation will be verified below, see subsection \ref{ProofYu}, 
starting from the main system \eqref{WW0} and using the expansion of the Dirichlet--Neumann operator in section \ref{DNLinear}. The nonlinearity $\mathcal{N}_2$ is of the form
\begin{equation}\label{system9}
\mathcal{N}_2=\sum_{\mu,\nu\in\{+,-\}}\mathcal{N}_{\mu\nu}(\mathcal{U}_{\mu},\mathcal{U}_{\nu}),\qquad 
\big(\mathcal{F}\mathcal{N_{\mu\nu}}(f,g)\big)(\xi)=\int_{\mathbb{R}^{2}}\mathfrak{m}_{\mu\nu}(\xi,\eta)\widehat{f}(\xi-\eta)\widehat{g}(\eta)\,d\eta,
\end{equation}
where $\mathcal{U}_+=\mathcal{U}$ and $\mathcal{U}_-=\overline{\mathcal{U}}$. The cubic nonlinearity is of the form
\begin{equation}\label{system9.1}
\begin{split}
&\mathcal{N}_3=\sum_{\mu,\nu,\beta\in\{+,-\}}\mathcal{N}_{\mu\nu\beta}(\mathcal{U}_{\mu},\mathcal{U}_{\nu},\mathcal{U}_\beta),\\
&\left(\mathcal{F}\mathcal{N_{\mu\nu\beta}}(f,g,h)\right)(\xi)=\int_{\mathbb{R}^{2}\times\mathbb{R}^2}\mathfrak{n}_{\mu\nu\beta}(\xi,\eta,\sigma)\widehat{f}(\xi-\eta)\widehat{g}(\eta-\sigma)\widehat{h}(\sigma)\,d\eta d\sigma.
\end{split}
\end{equation}
The multipliers $\mathfrak{m}_{\mu\nu}$ and $\mathfrak{n}_{\mu\nu\beta}$ satisfy suitable symbol-type estimates. 
We define the profiles $\mathcal{V}_{\sigma}(t)=e^{it\Lambda_{\sigma}}\mathcal{U}_{\sigma}(t)$, $\sigma\in\{+,-\}$, as in \eqref{Ama30}. 
The Duhamel formula is
\begin{equation}\label{duhamelDER}
(\partial_t\widehat{\mathcal{V}})(\xi,s)=e^{is\Lambda(\xi)}\widehat{\mathcal{N}_{2}}(\xi,s)+e^{is\Lambda(\xi)}\widehat{\mathcal{N}_{3}}(\xi,s)+e^{is\Lambda(\xi)}\widehat{\mathcal{N}_{\geq 4}}(\xi,s),
\end{equation}
or, in integral form,
\begin{equation}\label{duhamel}
\widehat{\mathcal{V}}(\xi,t)=\widehat{\mathcal{V}}(\xi,0)+\widehat{W_2}(\xi,t)+\widehat{W_3}(\xi,t)+\int_{0}^{t}e^{is\Lambda(\xi)}\widehat{\mathcal{N}_{\geq 4}}(\xi,s)\,ds,
\end{equation}
where, with the definitions in \eqref{phasedef},
\begin{equation}\label{W2}
\widehat{W_2}(\xi,t):=\sum_{\mu,\nu\in\{+.-\}}\int_{0}^{t}\int_{\mathbb{R}^2}e^{is\Phi_{+\mu\nu}(\xi,\eta)}\mathfrak{m}_{\mu\nu}(\xi,\eta)
\widehat{\mathcal{V}_{\mu}}(\xi-\eta,s)\widehat{\mathcal{V}_{\nu}}(\eta,s)\,d\eta ds,
\end{equation}
\begin{equation}\label{W3}
\begin{split}
\widehat{W_3}(\xi,t):=\sum_{\mu,\nu,\beta\in\{+.-\}}\int_{0}^{t}\int_{\mathbb{R}^2\times\mathbb{R}^2}&e^{is\widetilde{\Phi}_{+\mu\nu\beta}(\xi,\eta,\sigma)}
\mathfrak{n}_{\mu\nu\beta}(\xi,\eta,\sigma)\\
&\times\widehat{\mathcal{V}_{\mu}}(\xi-\eta,s)\widehat{\mathcal{V}_{\nu}}(\eta-\sigma,s)\widehat{\mathcal{V}_{\beta}}(\sigma,s)\,d\eta d\sigma ds.
\end{split}
\end{equation}

The vector-field $\Omega$ acts on the quadratic part of the nonlinearity according to the identity
\begin{equation*}
\Omega_\xi\widehat{\mathcal{N}_2}(\xi,s)=\sum_{\mu,\nu\in\{+,-\}}\int_{\mathbb{R}^2}(\Omega_\xi+\Omega_\eta)\big[\mathfrak{m}_{\mu\nu}(\xi,\eta)
\widehat{\mathcal{U}_{\mu}}(\xi-\eta,s)\widehat{\mathcal{U}_{\nu}}(\eta,s)\big]\,d\eta.
\end{equation*}
A similar formula holds for $\Omega_\xi\widehat{\mathcal{N}_3}(\xi,s)$. Therefore, for $1\leq a\leq N_1$, letting $\mathfrak{m}_{\mu\nu}^{b}:=(\Omega_\xi+\Omega_\eta)^{b}\mathfrak{m}_{\mu\nu}$ and $\mathfrak{n}_{\mu\nu\beta}^{b}:=(\Omega_\xi+\Omega_\eta+\Omega_\sigma)^{b}\mathfrak{n}_{\mu\nu\beta}$ we have
\begin{equation}\label{DuhamelDER2}
\Omega^a_\xi(\partial_t\widehat{\mathcal{V}})(\xi,s)=e^{is\Lambda(\xi)}\Omega^a_\xi\widehat{\mathcal{N}_{2}}(\xi,s)+e^{is\Lambda(\xi)}\Omega^a_\xi\widehat{\mathcal{N}_{3}}(\xi,s)+e^{is\Lambda(\xi)}\Omega^a_\xi\widehat{\mathcal{N}_{\geq 4}}(\xi,s),
\end{equation}
where
\begin{equation}\label{OmegaAN2}
\begin{split}
e^{is\Lambda(\xi)}\Omega^a_\xi\widehat{\mathcal{N}_{2}}(\xi,s)=\sum_{\mu,\nu\in\{+,-\}}&\sum_{a_1+a_2+b=a}\int_{\mathbb{R}^2}e^{is\Phi_{+\mu\nu}(\xi,\eta)}\mathfrak{m}_{\mu\nu}^{b}(\xi,\eta)\\
&\times(\Omega^{a_1}\widehat{\mathcal{V}_{\mu}})(\xi-\eta,s)(\Omega^{a_2}\widehat{\mathcal{V}_{\nu}})(\eta,s)\,d\eta
\end{split}
\end{equation}
and
\begin{equation}\label{OmegaAN3}
\begin{split}
e^{is\Lambda(\xi)}\Omega^a_\xi\widehat{\mathcal{N}_{3}}(\xi,s)=\sum_{\mu,\nu,\beta\in\{+,-\}}&\sum_{a_1+a_2+a_3+b=a}\int_{\mathbb{R}^2\times\mathbb{R}^2}e^{is\widetilde{\Phi}_{+\mu\nu\beta}(\xi,\eta,\sigma)}\mathfrak{n}_{\mu\nu\beta}^{b}(\xi,\eta,\sigma)\\
&\times(\Omega^{a_1}\widehat{\mathcal{V}_{\mu}})(\xi-\eta,s)(\Omega^{a_2}\widehat{\mathcal{V}_{\nu}})(\eta-\sigma,s)(\Omega^{a_3}\widehat{\mathcal{V}_{\beta}})(\sigma,s)\,d\eta d\sigma.
\end{split}
\end{equation}

To state our main proposition we need to make suitable assumptions on the nonlinearities $\mathcal{N}_2$, $\mathcal{N}_3$, and $\mathcal{N}_{\geq 4}$. 
Recall the class of symbols $S^\infty$ defined in \eqref{defclassA}. 

$\bullet$ Concerning the multipliers defining $\mathcal{N}_2$, we assume that $(\Omega_\xi+\Omega_\eta)m(\xi,\eta)\equiv 0$ and
\begin{equation}\label{Assumptions2}
\begin{split}
\|m^{k,k_1,k_2}\|_{S^\infty}&\lesssim 2^k2^{\min(k_1,k_2)/2},\\
\|D_\eta^\alpha m^{k,k_1,k_2}\|_{L^\infty}&\lesssim_{|\alpha|} 2^{(|\alpha|+3/2)\max(|k_1|,|k_2|)},\\
\|D_\xi^\alpha m^{k,k_1,k_2}\|_{L^\infty}&\lesssim_{|\alpha|} 2^{(|\alpha|+3/2)\max(|k|,|k_1|,|k_2|)},
\end{split}
\end{equation}
for any $k,k_1,k_2\in\mathbb{Z}$ and $m\in\{\mathfrak{m}_{\mu\nu}:\,\mu,\nu\in\{+,-\}\}$, where
\begin{equation*} 
m^{k,k_1,k_2}(\xi,\eta):=m(\xi,\eta)\cdot \varphi_k(\xi)\varphi_{k_1}(\xi-\eta)\varphi_{k_2}(\eta).
\end{equation*}

$\bullet$ Concerning the multipliers defining $\mathcal{N}_3$, we assume that $(\Omega_\xi+\Omega_\eta+\Omega_\sigma)n(\xi,\eta,\sigma)\equiv 0$ and
\begin{equation}\label{Assumptions3}
\begin{split}
\|n^{k,k_1,k_2,k_3}\|_{S^\infty}&\lesssim 2^{\min(k,k_1,k_2,k_3)/2}2^{3\max(k,k_1,k_2,k_3,0)},\\
\|D_{\eta,\sigma}^\alpha n^{k,k_1,k_2,k_3;l}\|_{L^\infty}&\lesssim_{|\alpha|} 2^{|\alpha|\max(|k_1|,|k_2|,|k_3|,|l|)}2^{(7/2)\max(|k_1|,|k_2|,|k_3|)},\\
\|D_\xi^\alpha n^{k,k_1,k_2,k_3}\|_{L^\infty}&\lesssim_{|\alpha|} 2^{(|\alpha|+7/2)\max(|k|,|k_1|,|k_2|,|k_3|)},
\end{split}
\end{equation}
for any $k,k_1,k_2,k_3,l\in\mathbb{Z}$ and $n\in\{\mathfrak{n}_{\mu\nu\beta}:\,\mu,\nu\in\{+,-\}\}$, where
\begin{equation*} 
\begin{split}
&n^{k,k_1,k_2,k_3}(\xi,\eta,\sigma):=n(\xi,\eta,\sigma)\cdot \varphi_k(\xi)\varphi_{k_1}(\xi-\eta)\varphi_{k_2}(\eta-\sigma)\varphi_{k_3}(\sigma),\\
&n^{k,k_1,k_2,k_3;l}(\xi,\eta,\sigma):=n(\xi,\eta,\sigma)\cdot \varphi_k(\xi)\varphi_{k_1}(\xi-\eta)\varphi_{k_2}(\eta-\sigma)\varphi_{k_3}(\sigma)\varphi_{l}(\eta).
\end{split}
\end{equation*}

Our main result is the following:

\begin{proposition}\label{bootstrap} Assume that $\mathcal{U}$ is a solution of the equation
\begin{equation}\label{MainEq}
(\partial_{t}+i\Lambda)\mathcal{U}=\mathcal{N}_2+\mathcal{N}_3+\mathcal{N}_{\geq 4},
\end{equation}
on some time interval $[0,T]$, $T\geq 1$, with initial data $\mathcal{U}_0$. Define, as before, 
$\mathcal{V}(t)=e^{it\Lambda}\mathcal{U}(t)$ and $\mathcal{V}_0=\mathcal{U}_0$. With $\delta$ as in Definition \ref{MainZDef}, assume that
\begin{equation}\label{bootstrap1}
\|\mathcal{U}_0\|_{H^{N_{0}}\cap H^{N_1,N_3}_\Omega}+\|\mathcal{V}_0\|_{Z}\leq\varep_{0}\ll 1
\end{equation} 
and 
\begin{equation}\label{bootstrap2}
\begin{split}
&(1+t)^{-\delta^2}\|\mathcal{U}(t)\|_{H^{N_{0}}\cap H^{N_1,N_3}_\Omega}+\|\mathcal{V}(t)\|_{Z}\leq \varep_{1}\ll 1,\\
&(1+t)^2\|\mathcal{N}_{\geq 4}(t)\|_{H^{N_0-N_3}\cap H_{\Omega}^{N_1,0}}+(1+t)^{1+\delta^2}\|e^{it\Lambda}\mathcal{N}_{\geq 4}(t)\|_{Z}\leq \varep_{1}^2,
\end{split}
\end{equation} 
for all $t\in[0,T]$. Moreover, assume that the nonlinearities $\mathcal{N}_2$ and $\mathcal{N}_3$ satisfy \eqref{system9}--\eqref{system9.1} and \eqref{Assumptions2}--\eqref{Assumptions3}. Then, for any $t\in[0,T]$
\begin{equation}\label{bootstrap3}
\|\mathcal{V}(t)\|_{Z}\lesssim \varep_0+\varep_1^2.
\end{equation} 
\end{proposition}

We will show in appendix \ref{ProofYu} below how to use this proposition and a suitable expansion of the Dirichlet--Neumann operator to complete the proof of the main Proposition \ref{MainBootstrapDisp}.

\subsection{Some lemmas}\label{lemmas}

In this subsection we collect several important lemmas which are used often in the proofs in the next two sections. 
Let $\Phi=\Phi_{\sigma\mu\nu}$ as in \eqref{phasedef}.

\subsubsection{Integration by parts} In this subsection we state two lemmas that are used in the paper in integration by parts arguments. 
We start with an oscillatory integral estimate. See \cite[Lemma 5.4]{IP2} for the proof of (i), and the proof of (ii) is similar.

\begin{lemma}\label{tech5} (i) Assume that $0<\eps\leq 1/\eps\leq K$, $N\geq 1$ is an integer, and $f,g\in C^N(\mathbb{R}^2)$. Then
\begin{equation}\label{ln1}
\Big|\int_{\mathbb{R}^2}e^{iKf}g\,dx\Big|\lesssim_N (K\eps)^{-N}\big[\sum_{|\alpha|\leq N}\eps^{|\alpha|}\|D^\alpha_xg\|_{L^1}\big],
\end{equation}
provided that $f$ is real-valued, 
\begin{equation}\label{ln2}
|\nabla_x f|\geq \mathbf{1}_{{\mathrm{supp}}\,g},\quad\text{ and }\quad\|D_x^\alpha f \cdot\mathbf{1}_{{\mathrm{supp}}\,g}\|_{L^\infty}\lesssim_N\eps^{1-|\alpha|},\,2\leq |\alpha|\leq N+1.
\end{equation}

(ii) Similarly, if $0<\rho\leq 1/\rho\leq K$ then
\begin{equation}\label{ln3}
\Big|\int_{\mathbb{R}^2}e^{iKf}g\,dx\Big|\lesssim_N (K\rho)^{-N}\big[\sum_{m\leq N}\rho^{m}\|\Omega^m g\|_{L^1}\big],
\end{equation}
provided that $f$ is real-valued, 
\begin{equation}\label{ln4}
|\Omega f|\geq \mathbf{1}_{{\mathrm{supp}}\,g},\quad\text{ and }\quad\|\Omega^m f \cdot\mathbf{1}_{{\mathrm{supp}}\,g}\|_{L^\infty}\lesssim_N\rho^{1-m},\,2\leq m\leq N+1.
\end{equation}
\end{lemma}

We will need another result about integration by parts using the vector-field $\Omega$. This lemma is more subtle. It is needed many times in the 
next two sections to localize and then estimate bilinear expressions. The point is to be able to take advantage of the fact that our profiles are ``almost radial'' 
(due to the bootstrap assumption involving many copies of $\Omega$), and prove that for such functions one has better localization 
properties than for general functions.  

\begin{lemma}\label{RotIBP}
Assume that $N\geq 100$, $m\geq 0$, $p,k,k_1,k_2\in\mathbb{Z}$, and
\begin{equation}\label{hypo0}
2^{-k_1}\leq 2^{2m/5},\qquad 2^{\max(k,k_1,k_2)}\le U\leq U^2\le 2^{m/10},\qquad U^2+2^{3|k_1|/2}\leq 2^{p+m/2}.
\end{equation}
For some $A\geq \max(1,2^{-k_1})$ assume that
\begin{equation}\label{hypo}
\begin{split}
\sup_{0\le a\le 100}\big[\Vert\Omega^a g\Vert_{L^2}+\Vert \Omega^a f\,\Vert_{L^2}\big]+\sup_{|\alpha|\leq N}A^{-|\alpha|}\|D^\alpha f\|_{L^2}&\le 1,\\
\sup_{\xi,\eta}\sup_{\vert\alpha\vert\le N}(2^{-m/2}\vert\eta\vert)^{\vert\alpha\vert}\vert D^\alpha_\eta m(\xi,\eta)\vert&\le 1.
\end{split}
\end{equation}
Fix $\xi\in\mathbb{R}^2$ and let, for $t\in[2^{m}-1,2^{m+1}]$,
\begin{equation*}
I_p(f,g):=\int_{\mathbb{R}^2}e^{it\Phi(\xi,\eta)}m(\xi,\eta)\varphi_p(\Omega_\eta\Phi(\xi,\eta))\varphi_k(\xi)\varphi_{k_1}(\xi-\eta)\varphi_{k_2}(\eta)f(\xi-\eta)g(\eta)d\eta.
\end{equation*}
If $2^{p}\leq U2^{|k_1|/2+100}$ and $A\leq 2^mU^{-2}$ then
\begin{equation}\label{OmIBP}
\vert I_p(f,g)\vert\lesssim_N  (2^{p+m})^{-N} U^{2N}\big[2^{m/2}+A2^p\big]^N+2^{-10m}.
\end{equation}
In addition, assuming that $(1+\delta/4)\nu\geq -m$, the same bound holds when $I_p$ is replaced by
\begin{equation*}
\widetilde{I_p}(f,g):=\int_{\mathbb{R}^2}e^{it\Phi(\xi,\eta)}\varphi_\nu(\Phi(\xi,\eta))m(\xi,\eta)\varphi_p(\Omega_\eta\Phi(\xi,\eta))\varphi_k(\xi)\varphi_{k_1}(\xi-\eta)\varphi_{k_2}(\eta)f(\xi-\eta)g(\eta)d\eta.
\end{equation*}
\end{lemma}

A slightly simpler version of this integration by parts lemma was used recently in \cite{DIP}. The main interest of this lemma is that we have essentially no assumption on $g$ and very mild assumptions on $f$.

\begin{proof}[Proof of Lemma \ref{RotIBP}] We decompose first $f=\mathcal{R}_{\leq m/10}f+[I-\mathcal{R}_{\leq m/10}]f$, $g=\mathcal{R}_{\leq m/10}g+[I-\mathcal{R}_{\leq m/10}]g$, where the operators $\mathcal{R}_{\leq L}$ are defined in polar coordinates by
\begin{equation}\label{RadOp}
(\mathcal{R}_{\leq L}h)(r\cos\theta,r\sin\theta):=\sum_{n\in\mathbb{Z}}\varphi_{\leq L}(n)h_n(r)e^{in\theta}\quad\text { if }\quad h(r\cos\theta,r\sin\theta):=\sum_{n\in\mathbb{Z}}h_n(r)e^{in\theta}.
\end{equation}
Since $\Omega$ corresponds to $d/d\theta$ in polar coordinates, using \eqref{hypo} we have,
\begin{equation*}
\big\|[I-\mathcal{R}_{\leq m/10}]f\big\|_{L^2}+\big\|[I-\mathcal{R}_{\leq m/10}]g\big\|_{L^2}\lesssim 2^{-10m}.
\end{equation*}
Therefore, using the H\"{o}lder inequality,
\begin{equation*}
\vert I_p\big([I-\mathcal{R}_{\leq m/10}]f,g\big)\vert+\vert I_p\big(\mathcal{R}_{\leq m/10}f,[I-\mathcal{R}_{\leq m/10}]g\big)\vert\lesssim 2^{-10m}.
\end{equation*}

It remains to prove a similar inequality for $I_p:=I_p\big(f_1,g_1\big)$, where $f_1:=\varphi_{[k_1-2,k_1+2]}\cdot \mathcal{R}_{\leq m/10}f$, $g_1:=\varphi_{[k_2-2,k_2+2]}\cdot \mathcal{R}_{\leq m/10}g$. It follows from \eqref{hypo} and the definitions that
\begin{equation}\label{RadOp2}
\|\Omega^ag_1\|_{L^2}\lesssim_a 2^{am/10},\qquad \|\Omega^a D^\alpha f_1\|_{L^2}\lesssim_a 2^{am/10}A^{|\alpha|},
\end{equation}
for any $a\geq 0$ and $|\alpha|\leq N$. Integration by parts gives
\begin{equation*}
I_p=c\varphi_k(\xi)\int_{\mathbb{R}^2}e^{it\Phi(\xi,\eta)}\Omega_\eta\left\{\frac{m(\xi,\eta)\varphi_{k_1}(\xi-\eta)\varphi_{k_2}(\eta)}{t\Omega_\eta\Phi(\xi,\eta)}\varphi_p(\Omega_\eta\Phi(\xi,\eta))f_1(\xi-\eta)g_1(\eta)\right\}d\eta.
\end{equation*}
Iterating $N$ times, we obtain an integrand made of a linear combination of terms like
\begin{equation*}
\begin{split}
e^{it\Phi(\xi,\eta)}&\varphi_k(\xi)\left(\frac{1}{t\Omega_\eta\Phi(\xi,\eta)}\right)^N\times\Omega^{a_1}_\eta\left\{m(\xi,\eta)\varphi_{k_1}(\xi-\eta)\varphi_{k_2}(\eta)\right\}\\
&\times\Omega_\eta^{a_2}f_1(\xi-\eta)\cdot\Omega_\eta^{a_3}g_1(\eta)\cdot\Omega_\eta^{a_4}\varphi_p(\Omega_\eta\Phi(\xi,\eta))\cdot \frac{\Omega^{a_5+1}_\eta\Phi}{\Omega_\eta\Phi}\dots\frac{\Omega_\eta^{a_q+1}\Phi}{\Omega_\eta\Phi},
\end{split}
\end{equation*}
where $\sum a_i=N$. The desired bound follows from the pointwise bounds
\begin{equation}\label{BdsIBP}
\begin{split}
\left\vert\Omega^{a}_\eta\left\{m(\xi,\eta)\varphi_{k_1}(\xi-\eta)\varphi_{k_2}(\eta)\right\}\right\vert&\lesssim 2^{am/2},\\
\left\vert \Omega^{a}_\eta\varphi_p(\Omega_\eta\Phi(\xi,\eta))\right\vert+ \left\vert\frac{\Omega^{a+1}_\eta\Phi}{\Omega_\eta\Phi}\right\vert&\lesssim U^{2a}2^{am/2},\\
\end{split}
\end{equation}
which hold in the support of the integral, and the $L^2$ bounds
\begin{equation}\label{l2ibp}
\begin{split}
\Vert \Omega^{a}_\eta g_1(\eta)\Vert_{L^2}&\lesssim 2^{am/4},\\
\Vert \Omega^a_\eta f_1(\xi-\eta)\varphi_{k}(\xi)\varphi_{[k_2-2,k_2+2]}(\eta)\varphi_{\leq p+2}(\Omega_\eta\Phi(\xi,\eta))\Vert_{L^2_\eta}&\lesssim U^{2a}\big[2^{m/2}+A2^p\big]^a.
\end{split}
\end{equation}

The first bound in \eqref{BdsIBP} is direct (see \eqref{hypo0}). For the second bound we notice that
\begin{equation}\label{OmegaBounds}
\begin{split}
&\Omega_\eta(\xi\cdot\eta^\perp)=-\xi\cdot\eta,\qquad\Omega_\eta(\xi\cdot\eta)=\xi\cdot\eta^\perp,\qquad\Omega_\eta\Phi(\xi,\eta)=\frac{\lambda^\prime_\mu(\vert\xi-\eta\vert)}{\vert\xi-\eta\vert}(\xi\cdot\eta^\perp),\\
&\vert\Omega_\eta^a\Phi(\xi,\eta)\vert\lesssim \lambda(|\xi-\eta|)\big[|\xi-\eta|^{-2a}|\xi\cdot\eta^\perp|^a+|\xi-\eta|^{-a}U^a\big].
\end{split}
\end{equation}
Since $\lambda'(|\xi-\eta|)\approx 2^{|k_1|/2}$, in the support of the integral, we have $|\xi-\eta|^{-2}|\xi\cdot\eta^\perp|\approx 2^{p}2^{-k_1-|k_1|/2}$. 
The second bound in \eqref{BdsIBP} follows once we recall the assumptions in \eqref{hypo0}.

We turn now to the proof of \eqref{l2ibp}. The first bound follows from the construction of $g_1$. For the second bound, if $2^p\gtrsim 2^{|k_1|/2+\min(k,k_2)}$ then we have the simple bound
\begin{equation*}
\Vert \Omega^a_\eta f_1(\xi-\eta)\varphi_{k}(\xi)\varphi_{[k_2-2,k_2+2]}(\eta)\Vert_{L^2_\eta}\lesssim [A2^{\min(k,k_2)}+2^{m/10}]^a,
\end{equation*}
which suffices. On the other hand, if $2^p\ll 2^{|k_1|/2+\min(k,k_2)}$ then we may assume that $\xi=(s,0)$, $s\approx 2^k$. The identities \eqref{OmegaBounds} show that $\varphi_{\leq p+2}(\Omega_\eta\Phi(\xi,\eta))\neq 0$ only if $|\xi\cdot\eta^\perp|\leq 2^{p+20}2^{k_1-|k_1|/2}$, which gives $|\eta_2|\leq 2^{p+30}2^{k_1-|k_1|/2}2^{-k}$. Therefore $|\eta_2|\ll 2^{k_1}$, so we may assume that $|\eta_1-s|\approx 2^{k_1}$.

We write now
\begin{equation*}
-\Omega_\eta f_1(\xi-\eta)=(\eta_1\partial_2f_1-\eta_2\partial_1f_1)(\xi-\eta)=\frac{\eta_1}{s-\eta_1}(\Omega f_1)(\xi-\eta)-\frac{s\eta_2}{s-\eta_1}(\partial_1f_1)(\xi-\eta).
\end{equation*}
By iterating this identity we see that $\Omega^a_\eta f_1(\xi-\eta)$ can be written as a sum of terms of the form
\begin{equation*}
P(s,\eta)\cdot\Big(\frac{1}{s-\eta_1}\Big)^{c+d+e}\Big(\frac{s\eta_2}{s-\eta_1}\Big)^{|b|-d}(D^b\Omega^cf_1)(\xi-\eta),
\end{equation*}
where $|b|+c+d+e\leq a$, $|b|,c,d,e\in\mathbb{Z}_+$, $|b|\geq d$, and $P(s,\eta)$ is a polynomial of degree at most $a$ in $s$ and at most $a$ in $(\eta_1,\eta_2)$. The second bound in \eqref{l2ibp} follows using the bounds on $f_1$ in \eqref{RadOp2} and the bounds proved earlier, $|s\eta_2|\lesssim 2^p2^{k_1-|k_1|/2}$, $|\eta_1-s|\approx 2^{k_1}$.

The last claim follows using the formula \eqref{loca2}, as in Lemma \ref{PhiLocLem} below.
\end{proof}

\subsubsection{Localization in modulation} Our lemma in this subsection shows that localization with respect to the phase is often a bounded operation:

\begin{lemma}\label{PhiLocLem}
Let $s\in[2^{m}-1,2^{m+1}]$, $m\geq 0$, and $-p\le m-2\delta^2m$. Let $\Phi=\Phi_{\sigma\mu\nu}$ as in \eqref{phasedef} and assume that $1/2=1/q+1/r$ and $\chi$ is a Schwartz function. Then, if $\|m\|_{S^\infty}\leq 1$,
\begin{equation}\label{loca}
\begin{split}
\Big\Vert \varphi_{\leq 10m}(\xi)&\int_{\mathbb{R}^2}e^{is\Phi(\xi,\eta)}m(\xi,\eta)\chi(2^{-p}\Phi(\xi,\eta))\widehat{f}(\xi-\eta)\widehat{g}(\eta)d\eta\Big\Vert_{L^2_\xi}\\
&\lesssim\sup_{|\rho|\leq 2^{-p+\delta^2m}}\Vert e^{-i(s+\rho)\Lambda_\mu}f\Vert_{L^q}\Vert e^{-i(s+\rho)\Lambda_\nu}g\Vert_{L^r}+2^{-10m}\Vert f\Vert_{L^2}\Vert g\Vert_{L^2},
\end{split}
\end{equation}
where the constant in the inequality only depends on the function $\chi$.
\end{lemma}

\begin{proof}
We may assume that $m\geq 10$ and use the Fourier transform to write
\begin{equation}\label{loca2}
\chi(2^{-p}\Phi(\xi,\eta))=c\int_{\mathbb{R}}e^{i\rho 2^{-p}\Phi(\xi,\eta)}\widehat{\chi}(\rho)d\rho.
\end{equation}
The left-hand side of \eqref{loca} is dominated by
\begin{equation*}
C\int_{\mathbb{R}}|\widehat{\chi}(\rho)|\Big\Vert\varphi_{\leq 10m}(\xi)\int_{\mathbb{R}^2}e^{i(s+2^{-p}\rho)\Phi(\xi,\eta)}m(\xi,\eta)\widehat{f}(\xi-\eta)\widehat{g}(\eta)d\eta\Big\Vert_{L^2_\xi}\,d\rho.
\end{equation*}
Using \eqref{mk6}, the contribution of the integral over $|\rho|\leq 2^{\delta^2m}$ is dominated by the first term in the right-hand side of \eqref{loca}. The contribution of the integral over $|\rho|\geq 2^{\delta^2m}$ is arbitrarily small and is dominated by the second term in the right-hand side of \eqref{loca}.
\end{proof}

\subsubsection{Linear estimates} We note first the straightforward estimates,
\begin{equation}\label{Straightforward}
\Vert P_kf\Vert_{L^2}\lesssim\min\{2^{(1-50\delta)k},2^{-Nk}\}\Vert f\Vert_{Z_1\cap H^{N}}, 
\end{equation}
for $N\geq 0$. We prove now several linear estimates for functions in $Z_1\cap H^N_\Omega$. As in Lemma \ref{RotIBP}, it is important to take advantage of 
the fact that our functions are ``almost radial''. The bounds we prove here are much stronger than the bounds one would normally expect for 
general functions with the same localization properties, and this is important in the next two sections.   

\begin{lemma}\label{LinEstLem}
Assume that $N\geq 10$ and
\begin{equation}\label{Zs}
\|f\|_{Z_1}+\sup_{k\in\mathbb{Z},\,a\leq N}\|\Omega^a P_kf\|_{L^2}\leq 1.
\end{equation}
Let $\delta':=50\delta+1/(2N)$. For any $(k,j)\in\mathcal{J}$ and $n\in\{0,\ldots,j+1\}$ let (recall the notation \eqref{Alx80})
\begin{equation}\label{Alx100}
f_{j,k}:=P_{[k-2,k+2]}Q_{jk}f,\qquad \widehat{f_{j,k,n}}(\xi):=\varphi_{-n}^{[-j-1,0]}(2^{100}(|\xi|-\gamma_1))\widehat{f_{j,k}}(\xi).
\end{equation}
For any $\xi_0\in\mathbb{R}^2\setminus\{0\}$ and $\kappa,\rho\in[0,\infty)$ let $\mathcal{R}(\xi_0;\kappa,\rho)$ denote the rectangle
\begin{equation}\label{Alx100.1}
\mathcal{R}(\xi_0;\kappa,\rho):=\{\xi\in\mathbb{R}^2:\big|(\xi-\xi_0)\cdot \xi_0/|\xi_0|\big|\leq\rho,\,\big|(\xi-\xi_0)\cdot \xi_0^\perp/|\xi_0|\big|\leq\kappa\}.
\end{equation}
(i) Then, for any $(k,j)\in\mathcal{J}$, $n\in [0,j+1]$, and $\kappa,\rho\in(0,\infty)$ satisfying $\kappa+\rho\leq 2^{k-10}$
\begin{equation}\label{RadL2}
\begin{split}
\big\Vert \sup_{\theta\in\mathbb{S}^1}|\widehat{f_{j,k,n}}(r\theta)|\,\big\Vert_{L^2(rdr)}&\lesssim 2^{(1/2-49\delta)n-(1-\delta')j},
\end{split}
\end{equation}
\begin{equation}\label{FL1bd}
\int_{\mathbb{R}^2}|\widehat{f_{j,k,n}}(\xi)|\mathbf{1}_{\mathcal{R}(\xi_0;\kappa,\rho)}(\xi)\,d\xi\lesssim \kappa 2^{-j+\delta'j}2^{-49\delta n}\min(1,2^n\rho 2^{-k})^{1/2},
\end{equation}
\begin{equation}\label{FLinftybd}
\|\widehat{f_{j,k,n}}\|_{L^\infty}\lesssim 
\begin{cases}
2^{(\delta+(1/2N)) n}2^{-(1/2-\delta')(j-n)}\,\,&\text{ if }|k|\leq 10,\\
2^{-\delta'k}2^{-(1/2-\delta')(j+k)}\,\,&\text{ if }|k|\geq 10,
\end{cases}
\end{equation}
and
\begin{equation}\label{FLinftybdDER}
\|D^\beta\widehat{f_{j,k,n}}\|_{L^\infty}\lesssim_{|\beta|}
\begin{cases}
2^{|\beta|j}2^{(\delta+1/(2N)) n}2^{-(1/2-\delta')(j-n)}\,\,&\text{ if }|k|\leq 10,\\
2^{|\beta|j}2^{-\delta' k}2^{-(1/2-\delta')(j+k)}\,\,&\text{ if }|k|\geq 10.
\end{cases}
\end{equation}

(ii) (Dispersive bounds) If $m\geq 0$ and $|t|\in[2^{m}-1,2^{m+1}]$ then
\begin{equation}\label{LinftyBd}\big\| e^{-it\Lambda}f_{j,k,n}\big\|_{L^\infty}\lesssim\big\| \widehat{f_{j,k,n}}\big\|_{L^1}\lesssim 2^k2^{-j+50\delta j}2^{-49\delta n},
\end{equation}
\begin{equation}\big\| e^{-it\Lambda}f_{j,k,0}\big\|_{L^\infty}\lesssim 2^{3k/2}2^{-m+50\delta j},\quad\text{ if }\,\,|k|\geq 10.\label{LinftyBd1.1}
\end{equation}
Recall the operators $A_{n,\gamma_0}$ defined in \eqref{aop}. If $j\leq (1-\delta^2)m+|k|/2$ and $|k|+\D\leq m/2$ then we have the more precise bounds
\begin{equation}\label{LinftyBd2}
\big\| e^{-it\Lambda}A_{\leq 0,\gamma_0}f_{j,k,n}\big\|_{L^\infty}\lesssim 
\begin{cases}
2^{-m+2\delta^2m}2^{-(j-n)(1/2-\delta')}2^{n(\delta+1/(2N))}&\text{ if }\,\,n\geq 1,\\
2^{-m+2\delta^2m}2^k2^{-(1/2-\delta')j}&\text{ if }\,\,n=0.
\end{cases}
\end{equation}
Moreover, for $l\geq 1$,
\begin{equation}\label{LinftyBd2.5}
\big\| e^{-it\Lambda}A_{l,\gamma_0}f_{j,k,0}\big\|_{L^\infty}\lesssim 
\begin{cases}
2^{-m+2\delta^2m}2^{\delta'j}2^{m/2-j/2-l/2-\max(j,l)/2}\,\,&\text{ if }\,\,2l+\max(j,l)\geq m,\\
2^{-m+2\delta^2m}2^{\delta'j}2^{(l-j)/2}\,\,&\text{ if }\,\,2l+\max(j,l)\leq m.
\end{cases}
\end{equation}
In particular, if $j\leq (1-\delta^2)m+|k|/2$ and $|k|+\D\leq m/2$ then
\begin{equation}\label{LinftyBd3}
\begin{split}
&\big\| e^{-it\Lambda}A_{\leq 0,\gamma_0}f_{j,k}\big\|_{L^\infty}\lesssim 2^{-m+2\delta^2m}2^k2^{j(\delta+1/(2N))},\\
&\sum_{l\geq 1}\big\| e^{-it\Lambda}A_{l,\gamma_0}f_{j,k}\big\|_{L^\infty}\lesssim 2^{-m+2\delta^2m}2^{\delta'j}2^{(m-3j)/6}.
\end{split}
\end{equation}
For all $k\in\mathbb{Z}$ we have the bound
\begin{equation}\label{LinftyBd3.5}
\begin{split}
&\big\| e^{-it\Lambda}A_{\leq 0,\gamma_0}P_kf\big\|_{L^\infty}\lesssim (2^{k/2}+2^{2k})2^{-m}\big[2^{51\delta m}+2^{m(2\delta+1/(2N))}\big],\\
&\big\| e^{-it\Lambda}A_{\geq 1,\gamma_0}P_kf\big\|_{L^\infty}\lesssim 2^{-5m/6+2\delta^2m}.
\end{split}
\end{equation}
\end{lemma}

\begin{proof} (i) The hypothesis gives 
\begin{equation}\label{Alx101}
\Vert f_{j,k,n}\Vert_{L^2}\lesssim 2^{(1/2-49\delta)n-(1-50\delta)j},\qquad \big\Vert \Omega^{N}f_{j,k,n}\big\Vert_{L^2}\lesssim \Vert \Omega^{N}P_kf\Vert_{L^2}\lesssim 1.
\end{equation} 
The bounds \eqref{RadL2} follow using the general interpolation inequality
\begin{equation}\label{Alx101.1}
\begin{split}
\big\Vert \sup_{\theta\in\mathbb{S}^1}|h(r\theta)|\,\big\Vert_{L^2(rdr)}
&\lesssim L^{1/2}\Vert h\Vert_{L^2}+L^{1/2-N}\Vert \Omega^{N}h\Vert_{L^2},
\end{split}
\end{equation}
for any $h\in L^2(\mathbb{R}^2)$ and $L\geq 1$, which follows easily using the operators $\mathcal{R}_{\leq L}$ defined in \eqref{RadOp}.

Inequality \eqref{FL1bd} follows from \eqref{RadL2}. Indeed, the left-hand side is dominated by
\begin{equation*}
C(\kappa 2^{-k})\sup_{\theta\in\mathbb{S}^1}\int_{\mathbb{R}}|\widehat{f_{j,k,n}}(r\theta)|\mathbf{1}_{\mathcal{R}(\xi_0;\kappa,\rho)}(r\theta)\,rdr\lesssim \sup_{\theta\in\mathbb{S}^1}\big\Vert \widehat{f_{j,k,n}}(r\theta)\,\big\Vert_{L^2(rdr)}(\kappa 2^{-k})[2^k\min(\rho,2^{k-n})]^{1/2},
\end{equation*}
which gives the desired result.

We now consider \eqref{FLinftybd}. For any $\theta\in\mathbb{S}^1$ fixed we have
\begin{equation*}
\begin{split}
\|\widehat{f_{j,k,n}}(r\theta)\|_{L^\infty}&\lesssim 2^{j/2}\|\widehat{f_{j,k,n}}(r\theta)\|_{L^2(dr)}+2^{-j/2}\|(\partial_r\widehat{f_{j,k,n}})(r\theta)\|_{L^2(dr)}\\
&\lesssim 2^{j/2}2^{-k/2}\|\widehat{f_{j,k,n}}(r\theta)\|_{L^2(rdr)},
\end{split}
\end{equation*} 
using the support property of $Q_{jk}f$ in the physical space. The desired bound follows using \eqref{RadL2} and the observation that $\widehat{f_{j,k,n}}=0$ unless $n=0$ or $k\in[-10,10]$. The bound \eqref{FLinftybdDER} follows also since differentiation in the Fourier space corresponds essentially to multiplication by factors of $2^j$, due to space localization.

(ii) The bound \eqref{LinftyBd} follows directly from Hausdorff-Young and \eqref{Alx101}. To prove \eqref{LinftyBd1.1}, if $|k|\geq 10$ then the standard dispersion estimate 
\begin{equation}\label{DiEs}
\Big|\int_{\mathbb{R}^2}e^{-it\lambda(|\xi|)}\varphi_k(\xi)e^{ix\cdot\xi}\,d\xi\Big|\lesssim 2^{2k}(1+|t|2^{k+|k|/2})^{-1}
\end{equation}
gives
\begin{equation}\label{LinEstLemA1}
\Vert e^{-it\Lambda}f_{j,k,n}\Vert_{L^\infty}\lesssim \frac{2^{2k}}{1+|t|2^{k/2}}\Vert f_{j,k,n}\Vert_{L^1}\lesssim  \frac{2^{2k}}{1+|t|2^{k/2}}2^{50\delta j}.
\end{equation}
The bound \eqref{LinftyBd1.1} follows (in the case $m\leq 10$ and $k\geq 0$ one can use \eqref{LinftyBd}).

We prove now \eqref{LinftyBd2}. The operator $A_{\leq 0,\gamma_0}$ is important here, because the function $\lambda$ has an inflection point at $\gamma_0$, see \eqref{ph3}. Using Lemma \ref{tech5} (i)  and the observation that $|(\nabla\Lambda)(\xi)|\approx 2^{|k|/2}$ if $|\xi|\approx 2^k$, it is easy to see that
\begin{equation*}
\big|\big(e^{-it\Lambda}A_{\leq 0,\gamma_0}f_{j,k,n}\big)(x)\big|\lesssim 2^{-10m} \qquad \text{ unless }|x|\approx 2^{m+|k|/2}.
\end{equation*}
Also, letting $f'_{j,k,n}:=\mathcal{R}_{\leq m/5}f_{j,k,n}$, see \eqref{RadOp}, we have $\|f_{j,k,n}-f'_{j,k,n}\|_{L^2}\lesssim 2^{-m(N/5)}$ therefore
\begin{equation}\label{Alx90.3}
\big\|e^{-it\Lambda}A_{\leq 0,\gamma_0}(f_{j,k,n}-f'_{j,k,n})\big\|_{L^\infty}\lesssim \big\|\widehat{f_{j,k,n}}-\widehat{f'_{j,k,n}}\big\|_{L^1}\lesssim 2^{-2m}2^k.
\end{equation}
On the other hand, if $|x|\approx 2^{m+|k|/2}$ then, using again Lemma \ref{tech5} and \eqref{FLinftybdDER},
\begin{equation}\label{Alx90.4}
\begin{split}
\big(e^{-it\Lambda}A_{\leq 0,\gamma_0}f'_{j,k,n}\big)(x)=C\int_{\mathbb{R}^2}&e^{i\Psi(\xi)}\varphi(\kappa_r^{-1}\nabla_\xi\Psi)\varphi(\kappa_\theta^{-1}\Omega_\xi\Psi)\\
&\times\widehat{f'_{j,k,n}}(\xi)\varphi_{\geq -100}(|\xi|-\gamma_0)d\xi+O(2^{-10m}),
\end{split}
\end{equation}
where
\begin{equation}\label{Alx90.5}
\Psi:=-t\Lambda(\xi)+x\cdot\xi,\quad\kappa_r:=2^{\delta^2m}\big(2^{(m+|k|/2-k)/2}+2^{j}\big),\quad\kappa_\theta:=2^{\delta^2m}2^{(m+k+|k|/2)/2}.
\end{equation}

We notice that the support of the integral in \eqref{Alx90.4} is contained in a $\kappa\times\rho$ rectangle in the direction of the vector $x$, where $\rho\lesssim \frac{\kappa_r}{2^{m+|k|/2-k}}$ and $\kappa \lesssim \frac{\kappa_\theta}{2^{m+|k|/2}}$, $\kappa\lesssim\rho$. This is because the function $\lambda''$ does not vanish in the support of the integral, so $\lambda''(|\xi|)\approx 2^{|k|/2-k}$. Therefore we can estimate the contribution of the integral in \eqref{Alx90.4} using either \eqref{FL1bd} or \eqref{FLinftybd}. More precisely, if $j\leq (m+|k|/2-k)/2$ then we use \eqref{FLinftybd} while if $j\geq (m+|k|/2-k)/2$ then we use \eqref{FL1bd} (and estimate $\min(1,2^n\rho 2^{-k})\leq 2^n\rho 2^{-k}$); in both cases the desired estimate follows.

We prove now \eqref{LinftyBd2.5}. We may assume that $|k|\leq 10$ and $m\geq \D$. As before, we may assume that $|x|\approx 2^m$ and replace $f_{j,k,0}$ with $f'_{j,k,0}$. As in \eqref{Alx90.4}, we have
\begin{equation}\label{Alx90.6}
\begin{split}
\big(e^{-it\Lambda}A_{l,\gamma_0}f_{j,k,0}\big)(x)=C\int_{\mathbb{R}^2}&e^{i\Psi(\xi)}\varphi(2^{-m/2-\delta^2m}\Omega_\xi\Psi)\\
&\times\widehat{f'_{j,k,0}}(\xi)\varphi_{-l-100}(|\xi|-\gamma_0)d\xi+O(2^{-2m}),
\end{split}
\end{equation}
where $\Psi$ is as in \eqref{Alx90.5}. The support of the integral above is contained in a $\kappa\times\rho$ rectangle in the direction of the vector $x$, where $\rho\lesssim 2^{-l}$ and $\kappa \lesssim 2^{-m/2+\delta^2m}$. Since $|\widehat{f'_{j,k,0}}(\xi)|\lesssim 2^{-j/2+\delta'j}$ in this rectangle (see \eqref{FLinftybd}), the bound in the first line of \eqref{LinftyBd2.5} follows if $l\geq j$. On the other hand, if $l\leq j$ then we use \eqref{FL1bd} to show that the absolute value of the integral in \eqref{Alx90.6} is dominated by $C2^{-j+\delta'j}\kappa\rho^{1/2}$, which gives again the bound in the first line of \eqref{LinftyBd2.5}.

It remains to prove the stronger bound in the second line of \eqref{LinftyBd2.5} in the case $2l+\max(j,l)\leq m$. We notice that $\lambda''(|\xi|)\approx 2^{-l}$ in the support of the integral. Assume that $x=(x_1,0)$, $x_1\approx 2^m$, and notice that we can insert an additional cutoff function of the form
\begin{equation*} 
\varphi[\kappa_r^{-1}(x_1-t\lambda'(|\xi_1|)\,\mathrm{sgn}\,(\xi_1))]\qquad\text{where}\qquad\kappa_r:=2^{\delta^2m}(2^{(m-l)/2}+2^j+2^l),
\end{equation*}
in the integral in \eqref{Alx90.6}, at the expense of an acceptable error. This can be verified using Lemma \ref{tech5} (i). The support of the integral is then contained in a $\kappa\times\rho$ rectangle in the direction of the vector $x$, where $\rho\lesssim \kappa_r2^{-m}2^l$ and $\kappa \lesssim 2^{-m/2+\delta^2m}$. The desired estimate then follows as before, using the $L^\infty$ bound \eqref{FLinftybd} if $2j\leq m-l$ and the integral bound \eqref{FL1bd} if $2j\geq m-l$.

The bounds in \eqref{LinftyBd3} follow from \eqref{LinftyBd2} and \eqref{LinftyBd2.5} by summation over $n$ and $l$ respectively. Finally, the bounds in \eqref{LinftyBd3.5} follow by summation (use \eqref{LinftyBd} if $j\geq (1-\delta^2)m$ or $m\leq 4\D$, use \eqref{LinftyBd1.1} if $j\leq (1-\delta^2)m$ and $|k|\geq 10$, and use \eqref{LinftyBd3} if $j\leq (1-\delta^2)m$ and $|k|\leq 10$).
\end{proof}

\begin{remark} We notice that we also have the bound (with no loss of $2^{2\delta^2m}$)
\begin{equation}\label{NewLinfty}
\big\| e^{-it\Lambda}A_{\leq 0,\gamma_0}f_{j,k,0}\big\|_{L^\infty}\lesssim 2^{-m}2^k2^{-(1/2-\delta'-\delta)j},
\end{equation}
provided that $j\leq (1-\delta^2)m+|k|/2$ and $|k|+\D\leq m/2$. Indeed, this follows from \eqref{LinftyBd2} if $j\geq m/10$. On the other hand, if $j\leq m/10$ then we can decompose (compare with \eqref{Alx90.4}),
\begin{equation*}
\big(e^{-it\Lambda}A_{\leq 0,\gamma_0}f_{j,k,0}\big)(x)=\sum_{p\geq 0}C\int_{\mathbb{R}^2}e^{i\Psi(\xi)}\varphi_p^{[0,\infty)}(\kappa^{-1}\nabla_\xi\Psi)\widehat{f_{j,k,0}}(\xi)\varphi_{\geq -100}(|\xi|-\gamma_0)d\xi,
\end{equation*}
where $\kappa:=2^{(m+|k|/2-k)/2}$. The contribution of $p=0$ is estimated as before, using \eqref{FLinftybd}, while for $p\geq 1$ we can first integrate by parts at most three times and then estimate the integral in the same way.
\end{remark}

\section{Dispersive analysis, II: the function $\partial_t\mathcal{V}$}\label{partialt}

In this section we prove several lemmas describing the function $\partial_t \mathcal{V}$. These lemmas rely on the Duhamel formula \eqref{DuhamelDER2},
\begin{equation}\label{mj1}
\Omega^a_\xi(\partial_t\widehat{\mathcal{V}})(\xi,s)=e^{is\Lambda(\xi)}\Omega^a_\xi\widehat{\mathcal{N}_{2}}(\xi,s)+e^{is\Lambda(\xi)}\Omega^a_\xi\widehat{\mathcal{N}_{3}}(\xi,s)+e^{is\Lambda(\xi)}\Omega^a_\xi\widehat{\mathcal{N}_{\geq 4}}(\xi,s),
\end{equation}
where
\begin{equation}\label{mj1.1}
\begin{split}
e^{is\Lambda(\xi)}\Omega^a_\xi\widehat{\mathcal{N}_{2}}(\xi,s)=\sum_{\mu,\nu\in\{+,-\}}&\sum_{a_1+a_2=a}\int_{\mathbb{R}^2}e^{is\Phi_{+\mu\nu}(\xi,\eta)}\mathfrak{m}_{\mu\nu}(\xi,\eta)(\Omega^{a_1}\widehat{\mathcal{V}_{\mu}})(\xi-\eta,s)(\Omega^{a_2}\widehat{\mathcal{V}_{\nu}})(\eta,s)\,d\eta
\end{split}
\end{equation}
and
\begin{equation}\label{mj1.2}
\begin{split}
e^{is\Lambda(\xi)}\Omega^a_\xi\widehat{\mathcal{N}_{3}}(\xi,s)=\sum_{\mu,\nu,\beta\in\{+,-\}}&\sum_{a_1+a_2+a_3=a}\int_{\mathbb{R}^2\times\mathbb{R}^2}e^{is\widetilde{\Phi}_{+\mu\nu\beta}(\xi,\eta,\sigma)}\mathfrak{n}_{\mu\nu\beta}(\xi,\eta,\sigma)\\
&\times(\Omega^{a_1}\widehat{\mathcal{V}_{\mu}})(\xi-\eta,s)(\Omega^{a_2}\widehat{\mathcal{V}_{\nu}})(\eta-\sigma,s)(\Omega^{a_3}\widehat{\mathcal{V}_{\beta}})(\sigma,s)\,d\eta d\sigma.
\end{split}
\end{equation}
Recall also the assumptions on the nonlinearity $\mathcal{N}_{\geq 4}$ and the profile $\mathcal{V}$ (see \eqref{bootstrap2}),
\begin{equation}\label{mj2}
\begin{split}
&\|\mathcal{V}(t)\|_{H^{N_0}\cap H_\Omega^{N_1,N_3}}\leq \varep_1(1+t)^{\delta^2},\qquad \|\mathcal{V}(t)\|_{Z}\leq \varep_1,\\
&\|\mathcal{N}_{\geq 4}(t)\|_{H^{N_0-N_3}\cap H_{\Omega}^{N_1}}\lesssim \varep_{1}^2(1+t)^{-2},
\end{split}
\end{equation}
and the symbol-type bounds \eqref{Assumptions2} on the multipliers $\mathfrak{m}_{\mu\nu}$. Given $\Phi=\Phi_{\sigma\mu\nu}$ as in \eqref{phasedef} let
\begin{equation}\label{Alx200}
\begin{split}
&\Xi=\Xi_{\mu\nu}(\xi,\eta):=(\nabla_\eta\Phi_{\sigma\mu\nu})(\xi,\eta)=(\nabla\Lambda_\mu)(\xi-\eta)-(\nabla\Lambda_\nu)(\eta),\qquad  \Xi:\mathbb{R}^2\times\mathbb{R}^2\to\mathbb{R}^2,\\
&\Theta=\Theta_\mu(\xi,\eta):=(\Omega_\eta\Phi_{\sigma\mu\nu})(\xi,\eta)=\frac{\lambda'_\mu(|\xi-\eta|)}{|\xi-\eta|}(\xi\cdot\eta^\perp),\qquad \Theta:\mathbb{R}^2\times\mathbb{R}^2\to\mathbb{R}.
\end{split}
\end{equation}

In this section we prove three lemmas describing the function $\partial_t\mathcal{V}$.

\begin{lemma}\label{dtfLem1}

(i) Assume \eqref{mj1}--\eqref{mj2}, $m\geq 0$, $s\in [2^{m}-1,2^{m+1}]$, $k\in\mathbb{Z}$, $\sigma\in\{+,-\}$. Then
\begin{equation}\label{Brc1}
\big\|(\partial_t\mathcal{V}_{\sigma})(s)\big\|_{H^{N_0-N_3}\cap H^{N_1}_{\Omega}}\lesssim \varep_1^22^{-5m/6+6\delta^2 m},
\end{equation}
\begin{equation}\label{Brc2}
\sup_{a\leq N_1/2+20,\,2a+|\alpha|\leq N_1+N_4}\|e^{-is\Lambda_\sigma}P_kD^\alpha\Omega^a(\partial_t\mathcal{V}_{\sigma})(s)\|_{L^{\infty}}\lesssim\varep_1^22^{-5m/3+6\delta^2 m}.
\end{equation}

(ii) In addition, if $a\leq N_1/2+20$ and $2a+|\alpha|\leq N_1+N_4$, then we may decompose
\begin{equation}\label{Brc4}
P_{k}D^\alpha\Omega^{a}(\partial_{t}\mathcal{V}_{\sigma})=\varepsilon_{1}^{2}\sum_{a_1+a_2=a,\,\alpha_1+\alpha_2=\alpha,\,\mu,\nu\in \{+,-\}}\sum_{[(k_1,j_1),(k_2,j_2)]\in X_{m,k}}A^{a_1,\alpha_1;a_2,\alpha_2}_{k;k_1,j_1;k_2,j_2}+\varepsilon_{1}^{2}P_kE_\sigma^{a,\alpha},
\end{equation}
where
\begin{equation}\label{vd2}
\|P_kE_\sigma^{a,\alpha}(s)\|_{L^2}\lesssim 2^{-3m/2+5\delta m}.
\end{equation}
Moreover, with $\mathfrak{m}_{+\mu\nu}(\xi,\eta):=\mathfrak{m}_{\mu\nu}(\xi,\eta)$, $\mathfrak{m}_{-\mu\nu}(\xi,\eta):=\overline{\mathfrak{m}_{(-\mu)(-\nu)}(-\xi,-\eta)}$, we have
\begin{equation}\label{Brc4.1}
\mathcal{F}\{A^{a_1,\alpha_1;a_2,\alpha_2}_{k;k_1,j_1;k_2,j_2}\}(\xi,s):=\int_{\mathbb{R}^2}e^{is\Phi(\xi,\eta)}\mathfrak{m}_{\sigma\mu\nu}(\xi,\eta)\varphi_k(\xi)\widehat{f^\mu_{j_1,k_1}}(\xi-\eta,s)\widehat{f^\nu_{j_2,k_2}}(\eta,s)d\eta,
\end{equation}
where
\begin{equation*}
f^\mu_{j_1,k_1}=\varep_1^{-1}P_{[k_1-2,k_1+2]}Q_{j_1k_1}D^{\alpha_1}\Omega^{a_1}\mathcal{V}_\mu,\qquad f^\nu_{j_2,k_2}=\varep_1^{-1}P_{[k_2-2,k_2+2]}Q_{j_2k_2}D^{\alpha_2}\Omega^{a_2}\mathcal{V}_\nu.
\end{equation*}
Let $N'_0=N_1-N_4=1/\delta$. The sets $X_{m,k}$ and the functions $A^{a_1,\alpha_1;a_2,\alpha_2}_{k;k_1,j_1;k_2,j_2}$ have the following properties:

(1) $X_{m,k}=\emptyset$ unless $m\geq \D^2$, $k\in[-3m/4,m/N'_0]$ and
\begin{equation}\label{vd4}
\begin{split}
X_{m,k}\subseteq\big\{[(k_1,j_1),(k_2,j_2)]\in \mathcal{J}\times\mathcal{J}: k_1,k_2\in[-3m/4,m/N'_0],\, \max(j_1,j_2)\leq 2m\big\}.
\end{split}
\end{equation}

(2) If $[(k_1,j_1),(k_2,j_2)]\in X_{m,k}$ and $\min(k_{1},k_{2})\leq -2m/N'_{0}$, then 
\begin{equation}\label{vd5}
\max(j_{1},j_{2})\leq (1-\delta^{2})m-|k|,\quad \max(|k_{1}-k|,|k_{2}-k|)\leq 100,\quad \mu=\nu,
\end{equation}
and
\begin{equation}\label{Dbound3}
\big\|A^{a_1,\alpha_1;a_2,\alpha_2}_{k;k_1,j_1;k_2,j_2}(s)\big\|_{L^{2}}\lesssim 2^{2k}2^{-m+6\delta^{2}m}.
\end{equation}

(3) If $[(k_1,j_1),(k_2,j_2)]\in X_{m,k}$, $\min(k_{1},k_{2})\geq -5m/N'_{0}$, and $k\leq \min(k_1,k_2)-200$, then
\begin{equation}\label{vd6}
\max(j_{1},j_{2})\leq (1-\delta^{2})m-|k|,\quad\max(|k_{1}|,|k_{2}|)\leq 10,\quad \mu=-\nu,
\end{equation}
and
\begin{equation}\label{Dbound4.0}
\big\|A^{a_1,\alpha_1;a_2,\alpha_2}_{k;k_1,j_1;k_2,j_2}(s)\big\|_{L^{2}}\lesssim 2^{k}2^{-m+4\delta m}.
\end{equation}

(4) If $[(k_1,j_1),(k_2,j_2)]\in X_{m,k}$ and $\min(k,k_{1},k_{2})\geq -6m/N'_0$ then 
\begin{equation}\label{vd6.6}
\text{ either }\quad j_1\leq 5m/6\quad\text{ or }\quad |k_1|\leq 10,
\end{equation}
\begin{equation}\label{vd6.65}
\text{ either }\quad j_2\leq 5m/6\quad\text{ or }\quad |k_2|\leq 10,
\end{equation}
and
\begin{equation}\label{vd6.7}
\min(j_{1},j_{2})\leq (1-\delta^{2})m.
\end{equation}
Moreover,
\begin{equation}\label{vd6.9}
\big\|A^{a_1,\alpha_1;a_2,\alpha_2}_{k;k_1,j_1;k_2,j_2}(s)\big\|_{L^{2}}\lesssim 2^{k}2^{-m+4\delta m},
\end{equation}
and
\begin{equation}\label{vd6.8}
\text{ if }\,\,\max(j_{1},j_{2})\geq (1-\delta^{2})m-|k|\,\,\text{ then }\,\,\big\|A^{a_1,\alpha_1;a_2,\alpha_2}_{k;k_1,j_1;k_2,j_2}(s)\big\|_{L^{2}}\lesssim 2^{-4m/3+4\delta m}.
\end{equation}

(iii) As a consequence of \eqref{vd2}, \eqref{Dbound3}, \eqref{Dbound4.0}, \eqref{vd6.9}, if $a\leq N_1/2+20$, and $2a+|\alpha|\leq N_1+N_4$ then we have the $L^2$ bound
\begin{equation}\label{vd7}
\big\|P_{k}D^\alpha\Omega^{a}(\partial_{t}\mathcal{V}_{\sigma})\big\|_{L^2}\lesssim \varep_1^2\big[2^{k}2^{-m+5\delta m}+2^{-3m/2+5\delta m}\big].
\end{equation}
\end{lemma}

\begin{proof} (i) We consider first the quadratic part of the nonlinearity. Let $I^{\sigma\mu\nu}$ denote the bilinear operator defined by
\begin{equation}\label{Isigmamunu}
\begin{split}
&\mathcal{F}\left\{I^{\sigma\mu\nu}[f,g]\right\}(\xi):=\int_{\mathbb{R}^2}e^{is\Phi_{\sigma\mu\nu}(\xi,\eta)}m(\xi,\eta)\widehat{f}(\xi-\eta)\widehat{g}(\eta)d\eta,\\
&\|m^{k,k_1,k_2}\|_{S^\infty}\leq 2^k2^{\min(k_1,k_2)/2},\qquad \|D_\eta^\alpha m^{k,k_1,k_2}\|_{L^\infty}\lesssim_{|\alpha|} 2^{(|\alpha|+3/2)\max(|k_1|,|k_2|)},
\end{split}
\end{equation}
where, for simplicity of notation, $m=\mathfrak{m}_{\sigma\mu\nu}$. For simplicity, we often write $\Phi$, $\Xi$, and $\Theta$ instead of $\Phi_{\sigma\mu\nu}$, $\Xi_{\mu\nu}$, and $\Theta_\mu$ in the rest of this proof.

We define the operators $P_k^+$ for $k\in\mathbb{Z}_+$ by $P_{k}^+:=P_k$ for $k\geq 1$ and $P_0^+:=P_{\leq 0}$. In view of Lemma \ref{touse} (ii), \eqref{mj2}, and \eqref{LinftyBd3.5}, for any $k\geq 0$ we have
\begin{equation}\label{Brc20}
\begin{split}
\|P_{k}^+I^{\sigma\mu\nu}[\mathcal{V}_\mu,\mathcal{V}_\nu](s)\|_{H^{N_0-N_3}}&\lesssim 2^{(N_0-N_3)k}\sum_{0\leq k_1\leq k_2,\, k_2\geq k-10}2^k2^{k_1/2}\|P_{k_2}^+\mathcal{V}(s)\|_{L^2}\|e^{-is\Lambda}P_{k_1}^+\mathcal{V}(s)\|_{L^\infty}\\
&\lesssim \varep_1^22^{-k}2^{-5m/6+6\delta^2m},
\end{split}
\end{equation}
which is consistent with \eqref{Brc1}. Similarly,
\begin{equation}\label{Brc20.1}
\|P_{k}^+I^{\sigma\mu\nu}[\Omega^{a_2}\mathcal{V}_\mu,\Omega^{a_3}\mathcal{V}_\nu](s)\|_{L^2}\lesssim 2^{-k}\varep_1^22^{-5m/6+6\delta^2m},\qquad a_2+a_3\leq N_1
\end{equation}
by placing the factor with less than $N_1/2$ $\Omega$-derivatives in $L^\infty$, and the other factor in $L^2$. Finally, using $L^\infty$ estimates on both factors,
\begin{equation}\label{Brc20.2}
\|e^{-is\Lambda_\sigma}P_{k}^+I^{\sigma\mu\nu}[D^{\alpha_2}\Omega^{a_2}\mathcal{V}_\mu,D^{\alpha_3}\Omega^{a_3}\mathcal{V}_\nu](s)\|_{L^\infty}\lesssim 
\begin{cases}
\varep_1^22^{-5m/3+6\delta^2m}&\text{ if }k\leq 20,\\
\varep_1^22^{4k}2^{-11m/6+52\delta m}&\text{ if }k\geq 20,
\end{cases}
\end{equation}
provided that $a_2+a_3=a$ and $\alpha_2+\alpha_3=\alpha$ (see also \eqref{AssonFmunu} below). The conclusions in part (i) follow for the quadratic components.

The conclusions for the cubic components follow by the same argument, using the assumption \eqref{Assumptions3} instead of \eqref{Assumptions2}, and the formula \eqref{mj1.2}. The contributions of the higher order nonlinearity $\mathcal{N}_\geq 4$ are estimated using directly the bootstrap hypothesis \eqref{mj2}.

(ii) We assume that $s$ is fixed and, for simplicity, drop it from the notation. In view of \eqref{mj2} and using interpolation, the functions $f^{\mu}:=\varep_1^{-1}D^{\alpha_2}\Omega^{a_2}\mathcal{V}_\mu$ and $f^{\nu}:=\varep_1^{-1}D^{\alpha_3}\Omega^{a_3}\mathcal{V}_\nu$ satisfy
\begin{equation}\label{AssonFmunu}
\begin{split}
&\Vert f^\mu\Vert_{H^{N'_0}\cap Z_1\cap H^{N'_1}_\Omega}+ \Vert f^\nu\Vert_{H^{N'_0}\cap Z_1\cap H^{N'_1}_\Omega}\lesssim 2^{\delta^2m}.
\end{split}
\end{equation}
where, compare with the notation in Theorem \ref{MainTheo},
\begin{equation}\label{vd1}
N'_1:=(N_1-N_4)/2=1/(2\delta),\qquad N'_0:=(N_0-N_3)/2-N_4=1/\delta.
\end{equation}
In particular, the dispersive bounds \eqref{LinftyBd}--\eqref{LinftyBd3.5} hold with $N=N'_1=1/(2\delta)$.

The contributions of the higher order nonlinearities $\mathcal{N}_3$ and $\mathcal{N}_{\geq 4}$ can all be estimated as part of the error term $P_kE_\sigma^{a,\alpha}$, so we focus on the quadratic nonlinearity $\mathcal{N}_2$. Notice that $$A^{a_1,\alpha_1;a_2,\alpha_2}_{k;k_1,j_1;k_2,j_2}=P_{k}I^{\sigma\mu\nu}(f_{j_{1},k_{1}}^{\mu},f_{j_{2},k_{2}}^{\nu}).$$

{\bf{Proof of property (1).}} In view of Lemma \ref{touse} and \eqref{LinftyBd3}, we have the general bound
\begin{equation*}
\big\|A^{a_1,\alpha_1;a_2,\alpha_2}_{k;k_1,j_1;k_2,j_2}\big\|_{L^{2}}\lesssim 2^{k+\min(k_1,k_2)/2}\cdot 2^{-5m/6+5\delta^2m}\min\big[2^{-(1/2-\delta)\max(j_1,j_2)},2^{-N'_0\max(k_1,k_2)}\big].
\end{equation*}
This bound suffices to prove the claims in (1). Indeed, if $k\geq m/N'_0$ or if $k\leq -3m/4+\D^2$ then the sum of all the terms can be bounded as in \eqref{vd2}. Similarly, if $k\in [-3m/4+\D^2,m/N'_0]$ then the sums of the $L^2$ norms corresponding to $\max(k_1,k_2)\geq m/N'_0$, or $\max(j_1,j_2)\geq 2m$, or $\min(k_1,k_2)\leq -3m/4+\D^2$ are all bounded by $2^{-3m/2}$ as desired.

{\bf{Proof of property (2).}} Assume now that $\min(k_{1},k_{2})\leq -2m/N'_{0}$ and $j_2=\max(j_1,j_2)\geq(1-\delta^{2})m-|k|$. Then, using the $L^2\times L^\infty$ estimate as before
\begin{equation*}
\big\|P_{k}I^{\sigma\mu\nu}[f_{j_{1},k_{1}}^{\mu},A_{\leq 0,\gamma_1}f_{j_{2},k_{2}}^{\nu}]\big\|_{L^2}\lesssim 2^{k+\min(k_1,k_2)/2}2^{-5m/6+5\delta^2m}2^{-j_2(1-50\delta)}\lesssim 2^{-3m/2}.
\end{equation*}
Moreover, we notice that if $A_{\geq 1,\gamma_1}f_{j_{2},k_{2}}^{\nu}$ is nontrivial then $|k_2|\leq 10$ and $k_1\leq -2m/N'_0$, therefore 
\begin{equation*}
\big\|P_{k}I^{\sigma\mu\nu}[f_{j_{1},k_{1}}^{\mu},A_{\geq 1,\gamma_1}f_{j_{2},k_{2}}^{\nu}]\big\|_{L^2}\lesssim 2^{k+k_1/2}2^{-m+5\delta^2m}2^{-j_2(1/2-\delta)}\lesssim 2^{-3m/2+3\delta m},
\end{equation*}
if $j_1\leq (1-\delta^2)m$, using \eqref{LinftyBd2} if $k_1\geq -m/2$ and \eqref{LinftyBd1.1} if $k_1\leq -m/2$. On the other hand, if $j_1\geq (1-\delta^2)m$ then we use again the $L^2\times L^\infty$ estimate (placing $f_{j_{1},k_{1}}^{\mu}$ in $L^2$) to conclude that
\begin{equation*}
\big\|P_{k}I^{\sigma\mu\nu}[f_{j_{1},k_{1}}^{\mu},A_{\geq 1,\gamma_1}f_{j_{2},k_{2}}^{\nu}]\big\|_{L^2}\lesssim 2^{k+k_1/2}2^{-j_1+50\delta j_1}2^{-m+52\delta m}\lesssim 2^{-3m/2}.
\end{equation*}
The last three bounds show that
\begin{equation}\label{vd9}
\big\|A^{a_1,\alpha_1;a_2,\alpha_2}_{k;k_1,j_1;k_2,j_2}\big\|_{L^2}\lesssim 2^{-3m/2+3\delta m}\quad\text{ if }\quad \max(j_1,j_2)\geq(1-\delta^{2})m-|k|.
\end{equation}

Assume now that 
\begin{equation*}
k_1=\min(k_1,k_2)\leq -2m/N'_0\quad\text{ and }\quad\max(j_1,j_2)\leq(1-\delta^{2})m-|k|.
\end{equation*}
If $k_2\geq k_1+20$ then $|\nabla_\eta\Phi(\xi,\eta)|\gtrsim 2^{|k_1|/2}$, so $\big\|A^{a_1,\alpha_1;a_2,\alpha_2}_{k;k_1,j_1;k_2,j_2}\big\|_{L^2}\lesssim 2^{-3m}$ in view of Lemma \ref{tech5} (i). On the other hand, if $k,k_2\leq k_1+30$ then, using again the $L^2\times L^\infty$ argument as before,
\begin{equation}\label{vd12}
\big\|P_{k}I^{\sigma\mu\nu}[f_{j_{1},k_{1}}^{\mu},f_{j_{2},k_{2}}^{\nu}]\big\|_{L^2}\lesssim 2^{k+k_1}2^{-m+5\delta^2 m}.
\end{equation}
The $L^2$ bound in \eqref{vd2} follows if $k+k_1\leq -m/2$. On the other hand, if $k+k_1\geq -m/2$ and
\begin{equation*}
\max(|k_{1}-k|,|k_{2}-k|)\geq 100\quad\text{or}\quad\mu=-\nu
\end{equation*}
then $|\nabla_\eta\Phi(\xi,\eta)|\gtrsim 2^{k-\max(k_1,k_2)}$ in the support of the integral, in view of \eqref{zc2.1}. Therefore $\big\|A^{a_1,\alpha_1;a_2,\alpha_2}_{k;k_1,j_1;k_2,j_2}\big\|_{L^2}\lesssim 2^{-3m}$ in view of Lemma \ref{tech5} (i). The inequalities in \eqref{vd5} follow. The bound \eqref{Dbound3} then follows from \eqref{vd12}. 

{\bf{Proof of property (3).}} Assume first that
\begin{equation}\label{vd20}
\min(k_{1},k_{2})\geq -5m/N'_{0},\quad k\leq \min(k_1,k_2)-200,\quad \max(j_{1},j_{2})\geq (1-\delta^{2})m-|k|-|k_2|.
\end{equation}
We may assume that $j_2\geq j_1$. Using the $L^2\times L^\infty$ estimate and Lemma \ref{LinEstLem} (ii) as before
\begin{equation*}
\big\|P_{k}I^{\sigma\mu\nu}[f_{j_{1},k_{1}}^{\mu},A_{n_2,\gamma_1}^{(j_2)}f_{j_{2},k_{2}}^{\nu}]\big\|_{L^2}\lesssim 2^{k+k_1/2}2^{-5m/6+5\delta^2m}2^{-j_2(1-50\delta)}\lesssim 2^{-3m/2}
\end{equation*}
if $n_2\leq\D$. On the other hand, if $n_2\in[\D,j_2]$ then
\begin{equation*}
P_{k}I^{\sigma\mu\nu}[f_{j_{1},k_{1}}^{\mu},A_{n_2,\gamma_1}^{(j_2)}f_{j_{2},k_{2}}^{\nu}]=P_{k}I^{\sigma\mu\nu}[A_{\geq 1,\gamma_1}f_{j_{1},k_{1}}^{\mu},A_{n_2,\gamma_1}^{(j_2)}f_{j_{2},k_{2}}^{\nu}].
\end{equation*}
If $j_1\leq (1-\delta^2)m$ then we estimate
\begin{equation*}
\big\|P_{k}I^{\sigma\mu\nu}[A_{\geq 1,\gamma_1}f_{j_{1},k_{1}}^{\mu},A_{n_2,\gamma_1}^{(j_2)}f_{j_{2},k_{2}}^{\nu}]\big\|_{L^2}\lesssim 2^{k}2^{-m+5\delta^2m+2\delta m}2^{-j_2(1/2-\delta)}\lesssim 2^{-3m/2+3\delta m+8\delta^2m}.
\end{equation*}
Finally, if $j_2\geq j_1\geq (1-\delta^2)m$ then we use Schur's lemma in the Fourier space and estimate
\begin{equation}\label{tla60.6}
\begin{split}
\big\|P_{k}I^{\sigma\mu\nu}[A_{n_1,\gamma_1}^{(j_1)}&f_{j_{1},k_{1}}^{\mu},A_{n_2,\gamma_1}^{(j_2)}f_{j_{2},k_{2}}^{\nu}]\big\|_{L^2}\lesssim 2^{k}2^{-\max(n_1,n_2)/2}\big\|A_{n_1,\gamma_1}^{(j_1)}f_{j_{1},k_{1}}^{\mu}\big\|_{L^2}\big\|A_{n_2,\gamma_1}^{(j_2)}f_{j_2,k_2}^{\nu}\big\|_{L^2}\\
&\lesssim 2^{k}2^{2\delta^2m}2^{-\max(n_1,n_2)/2}2^{-j_1(1-50\delta)}2^{(1/2-49\delta)n_1}\cdot 2^{-j_2(1-50\delta)}2^{(1/2-49\delta)n_2}\\
&\lesssim 2^{2\delta^2m}2^{\min(n_1,n_2)/2}2^{-j_1(1-50\delta)}2^{-49\delta(n_1+n_2)}2^{-j_2(1-50\delta)}\\
&\lesssim 2^{2\delta^2m}2^{-(2-2\delta^2)(1-50\delta)m}2^{(1/2-98\delta)m}
\end{split}
\end{equation}
for any $n_1\in[1,j_1+1]$, $n_2\in[1,j_2+1]$. Therefore, if \eqref{vd20} holds then
\begin{equation}\label{vd20.5}
\big\|A^{a_1,\alpha_1;a_2,\alpha_2}_{k;k_1,j_1;k_2,j_2}\big\|_{L^2}\lesssim 2^{-3m/2+4\delta m}.
\end{equation}

Assume now that
\begin{equation}\label{vd21}
\min(k_{1},k_{2})\geq -5m/N'_{0},\quad k\leq \min(k_1,k_2)-200,\quad \max(j_{1},j_{2})\leq (1-\delta^{2})m-|k|-|k_2|.
\end{equation}
If, in addition, $\max(|k_1|,|k_2|)\geq 11$ or $\mu=\nu$ then $|\nabla_\eta\Phi(\xi,\eta)|\gtrsim 2^{k-k_2}$ in the support of the integral. Indeed, this is a consequence of \eqref{zc2.1} if $k\leq -100$ and it follows easily from the formula \eqref{try2} if $k\geq -100$. Therefore, $\big\|A^{a_1,\alpha_1;a_2,\alpha_2}_{k;k_1,j_1;k_2,j_2}\big\|_{L^2}\lesssim 2^{-3m}$, using Lemma \ref{tech5} (i). As a consequence, the functions $A^{a_1,\alpha_1;a_2,\alpha_2}_{k;k_1,j_1;k_2,j_2}$ can be absorbed into the error term $P_kE^{a,\alpha}_\sigma$ unless all the inequalities in \eqref{vd6} hold.

Assume now that \eqref{vd6} holds and we are looking to prove \eqref{Dbound4.0}. It suffices to prove that
\begin{equation}\label{vd23}
\big\|P_{k}I^{\sigma\mu\nu}[A_{\geq 1,\gamma_0}f_{j_{1},k_{1}}^{\mu},A_{\geq 1,\gamma_0}f_{j_{2},k_{2}}^{\nu}]\big\|_{L^2}\lesssim 2^k2^{-m+4\delta m},
\end{equation}
after using \eqref{LinftyBd2} and the $L^2\times L^\infty$ argument. We may assume that $\max(j_1,j_2)\leq m/3$; otherwise \eqref{vd23} follows from the $L^2\times L^\infty$ estimate. Using \eqref{FLinftybd} and the more precise bound \eqref{LinftyBd2.5},
\begin{equation*}
\|A_{p,\gamma_0}h\|_{L^2}\lesssim 2^{\delta^2m}2^{-p/2},\qquad \|e^{-it\Lambda}A_{p,\gamma_0}h\|_{L^\infty}\lesssim 2^{-m+3\delta^2m}\min\big(2^{p/2},2^{m/2-p}\big),
\end{equation*}
where $h\in \{f_{j_1,k_1},g_{j_2,k_2}\}$, $p\geq 1$. Therefore, using Lemma \ref{touse},
\begin{equation*}
\big\|P_{k}I^{\sigma\mu\nu}[A_{p_1,\gamma_0}f_{j_{1},k_{1}}^{\mu},A_{p_2,\gamma_0}f_{j_{2},k_{2}}^{\nu}]\big\|_{L^2}\lesssim 2^k 2^{-m+5\delta^2m}2^{-\max(p_1,p_2)/2}2^{\min(p_1,p_2)/2}.
\end{equation*}
The desired bound \eqref{vd23} follows, using also the simple estimate 
\begin{equation*}
\big\|P_{k}I^{\sigma\mu\nu}[A_{p_1,\gamma_0}f_{j_{1},k_{1}}^{\mu},A_{p_2,\gamma_0}f_{j_{2},k_{2}}^{\nu}]\big\|_{L^2}\lesssim 2^k2^{2\delta^2m}2^{-(p_1+p_2)/2}.
\end{equation*}
This completes the proof of \eqref{Dbound4.0}.

{\bf{Proof of property (4).}} The same argument as in the proof of \eqref{vd20.5}, using just $L^2\times L^\infty$ estimates shows that $\|A^{a_1,\alpha_1;a_2,\alpha_2}_{k;k_1,j_1;k_2,j_2}\|_{L^2}\lesssim 2^{-3m/2+4\delta m}$ if either \eqref{vd6.6} or \eqref{vd6.7} do not hold. The bounds \eqref{vd6.8} follow in the same way. The same argument as in the proof of \eqref{vd23}, together with $L^2\times L^\infty$ estimates using \eqref{LinftyBd3} and \eqref{LinftyBd}, gives \eqref{vd6.9}.
\end{proof}

In our second lemma we give a more precise description of the basic functions $A^{a_1,\alpha_1;a_2,\alpha_2}_{k;k_1,j_1;k_2,j_2}(s)$ in the case $\min(k,k_{1},k_{2})\geq -6m/N'_0$.

\begin{lemma}\label{dtfLem2}
Assume $[(k_1,j_1),(k_2,j_2)]\in X_{m,k}$ and $k,k_{1},k_{2}\in [-6m/N'_0,m/N'_0]$ (as in Lemma \ref{dtfLem1} (ii) (4)), and recall the functions $A^{a_1,\alpha_1;a_2,\alpha_2}_{k;k_1,j_1;k_2,j_2}(s)$ defined in \eqref{Brc4.1}. 

(i) We can decompose 
\begin{equation}\label{Ddecomposition}
A^{a_1,\alpha_1;a_2,\alpha_2}_{k;k_1,j_1;k_2,j_2}=\sum_{i=1}^{3}A^{a_1,\alpha_1;a_2,\alpha_2;[i]}_{k;k_1,j_1;k_2,j_2}=\sum_{i=1}^{3}G^{[i]},
\end{equation}
\begin{equation}\label{Brc4.2}
\mathcal{F}A^{a_1,\alpha_1;a_2,\alpha_2;[i]}_{k;k_1,j_1;k_2,j_2}(\xi,s):=\int_{\mathbb{R}^2}e^{is\Phi(\xi,\eta)}\mathfrak{m}_{\sigma\mu\nu}(\xi,\eta)\varphi_k(\xi)\chi^{[i]}(\xi,\eta)\widehat{f^\mu_{j_1,k_1}}(\xi-\eta,s)\widehat{f^\nu_{j_2,k_2}}(\eta,s)d\eta,
\end{equation} 
where $\chi^{[i]}$ are defined as 
\begin{equation*}
\begin{split}
\chi^{[1]}(\xi,\eta)&=\varphi(2^{10\delta m}\Phi(\xi,\eta))\varphi(2^{30\delta m}\nabla_{\eta}\Phi(\xi,\eta))\mathbf{1}_{[0,5m/6]}(\max(j_1,j_2)),\\
\chi^{[2]}(\xi,\eta)&=\varphi_{\geq 1}(2^{10\delta m}\Phi(\xi,\eta))\varphi(2^{20\delta m}\Omega_\eta\Phi(\xi,\eta)),\\
\chi^{[3]}&=1-\chi^{[1]}-\chi^{[2]}.
\end{split}
\end{equation*}
The functions $A^{a_1,\alpha_1;a_2,\alpha_2;[1]}_{k;k_1,j_1;k_2,j_2}(s)$ are nontrivial only when $\max(|k|,|k_1|,|k_2|)\leq 10$. Moreover
\begin{equation}\label{Dbound5.1}
\big\|G^{[1]}(s)\big\|_{L^{2}}\lesssim 2^{-m+4\delta m}2^{-(1-50\delta)\max(j_1,j_2)},
\end{equation}
\begin{equation}\label{Dbound5.2}
\big\|G^{[2]}(s)\big\|_{L^{2}}\lesssim 2^k2^{-m+4\delta m},\qquad\big\|G^{[3]}(s)\big\|_{L^{2}}\lesssim 2^{-3m/2+4\delta m}.
\end{equation}

(ii) We have
\begin{equation}\label{vd70}
\big\|\mathcal{F}\{A_{\leq \D,2\gamma_0}A^{a_1,\alpha_1;a_2,\alpha_2}_{k;k_1,j_1;k_2,j_2}\}(s)\big\|_{L^\infty}\lesssim (2^{-k}+2^{3k})2^{-m+14\delta m}.
\end{equation}
As a consequence, if $k\geq -6m/N'_0+\D$ then we can decompose
\begin{equation}\label{vd70.1}
\begin{split}
&A_{\leq \D-10,2\gamma_0}\partial_t f^{\sigma}_{j,k}=h_2+h_\infty,\\
&\|h_2(s)\|_{L^2}\lesssim 2^{-3m/2+5\delta m},\qquad \|\widehat{h_\infty}(s)\|_{L^\infty}\lesssim (2^{-k}+2^{3k})2^{-m+15\delta m}.
\end{split}
\end{equation}

(iii) If $j_1,j_2\leq m/2+\delta m$ then we can write
\begin{equation}\label{Decdtf}
\begin{split}
&\widehat{G^{[1]}}(\xi,s)=e^{is[\Lambda_\sigma(\xi)-2\Lambda_\sigma(\xi/2)]}g^{[1]}(\xi,s)\varphi(2^{3\delta m}(|\xi|-\gamma_1))+h^{[1]}(\xi,s),\\
&\|D^\alpha_\xi g^{[1]}(s)\|_{L^\infty}\lesssim _{\alpha}2^{-m+4\delta m}2^{|\alpha|(m/2+4\delta m)},\quad \|\partial_sg^{[1]}(s)\|_{L^\infty}\lesssim 2^{-2m+18\delta m},\\
&\|h^{[1]}(s)\|_{L^\infty}\lesssim 2^{-4m}.
\end{split}
\end{equation}
\end{lemma}

\begin{proof}
(i) To prove the bounds \eqref{Dbound5.1}--\eqref{Dbound5.2} we decompose
\begin{equation}\label{vd30}
A^{a_1,\alpha_1;a_2,\alpha_2}_{k;k_1,j_1;k_2,j_2}=\sum_{i=1}^{5}A_i,\qquad A_i:=P_kI_i[f^\mu_{j_1,k_1},f^\nu_{j_2,k_2}],
\end{equation}
\begin{equation}\label{vd31}
\mathcal{F}\{I_i[f,g]\}(\xi):=\int_{\mathbb{R}^2}e^{is\Phi(\xi,\eta)}m(\xi,\eta)\chi_i(\xi,\eta)\widehat{f}(\xi-\eta)\widehat{g}(\eta)d\eta,
\end{equation} 
where $m=\mathfrak{m}^{a_1}_{\sigma\mu\nu}$ and $\chi_i$ are defined as 
\begin{equation}\label{vd31.5}
\begin{split}
\chi_1(\xi,\eta)&:=\varphi_{\geq 1}(2^{20\delta m}\Theta(\xi,\eta)),\\
\chi_2(\xi,\eta)&:=\varphi_{\geq 1}(2^{10\delta m}\Phi(\xi,\eta))\varphi(2^{20\delta m}\Theta(\xi,\eta)),\\
\chi_3(\xi,\eta)&:=\varphi(2^{10\delta m}\Phi(\xi,\eta))\varphi(2^{20\delta m}\Theta(\xi,\eta))\mathbf{1}_{(5m/6,\infty)}(\max(j_1,j_2)),\\
\chi_4(\xi,\eta)&:=\varphi(2^{10\delta m}\Phi(\xi,\eta))\varphi(2^{20\delta m}\Theta(\xi,\eta))\varphi_{\geq 1}(2^{30\delta m}\Xi(\xi,\eta))\mathbf{1}_{[0,5m/6]}(\max(j_1,j_2)),\\
\chi_5(\xi,\eta)&:=\varphi(2^{10\delta m}\Phi(\xi,\eta))\varphi(2^{20\delta m}\Theta(\xi,\eta))\varphi(2^{30\delta m}\Xi(\xi,\eta))\mathbf{1}_{[0,5m/6]}(\max(j_1,j_2)).
\end{split}
\end{equation}

Notice that $A_2=G^{[2]}$, $A_5=G^{[1]}$, and $A_1+A_3+A_4=G^{[3]}$. We will show first that
\begin{equation}\label{vd32}
\|A_1\|_{L^2}+\|A_3\|_{L^2}+\|A_4\|_{L^2}\lesssim 2^{-3m/2+4\delta m}.
\end{equation}

It follows from Lemma \ref{RotIBP} and \eqref{vd6.6}--\eqref{vd6.7} that $\|A_1\|_{L^2}\lesssim 2^{-2m}$, as desired. Also, $\|A_4\|_{L^2}\lesssim 2^{-4m}$, as a consequence of Lemma \ref{tech5} (i).  It remains to prove that
\begin{equation}\label{vd40}
\|A_3\|_{L^2}\lesssim 2^{-3m/2+4\delta m}.
\end{equation}
Assume that $j_2> 5m/6$ (the proof of \eqref{vd40} when $j_1> 5m/6$ is similar). We may assume that $|k_2|\leq 10$ (see \eqref{vd6.65}), and then $|k|,|k_1|\in[0,100]$ (due to the restrictions $|\Phi(\xi,\eta)|\lesssim 2^{-10\delta m}$ and $|\Theta(\xi,\eta)|\lesssim 2^{-20\delta m}$, see also \eqref{try5.5}). We show first that
\begin{equation}\label{vd41}
\big\|P_kI_3[f^\mu_{j_1,k_1},A_{\leq 0,\gamma_1}f^\nu_{j_2,k_2}]\big\|_{L^2}\lesssim 2^{-3m/2+4\delta m}.
\end{equation}
Indeed, we notice that, as a consequence of the $L^2\times L^\infty$ argument,
\begin{equation*}
\big\|P_{k}I^{\sigma\mu\nu}[f_{j_{1},k_{1}}^{\mu},A_{\leq 0,\gamma_1}f_{j_{2},k_{2}}^{\nu}]\big\|_{L^2}\lesssim 2^{-3m/2},
\end{equation*}
where $I^{\sigma\mu\nu}$ is defined as in \eqref{Isigmamunu}. Let $I^{||}$ be defined by
\begin{equation}\label{vd36}
\mathcal{F}\{I^{||}[f,g]\}(\xi):=\int_{\mathbb{R}^2}e^{is\Phi(\xi,\eta)}m(\xi,\eta)\varphi(2^{20\delta m}\Theta(\xi,\eta))\widehat{f}(\xi-\eta)\widehat{g}(\eta)d\eta.
\end{equation}
Using Lemma \ref{RotIBP} and \eqref{vd6.7}, it follows that
\begin{equation*}
\big\|P_{k}I^{||}[f_{j_{1},k_{1}}^{\mu},A_{\leq 0,\gamma_1}f_{j_{2},k_{2}}^{\nu}]\big\|_{L^2}\lesssim 2^{-3m/2}.
\end{equation*}
The same averaging argument as in the proof of Lemma \ref{PhiLocLem} gives \eqref{vd41}.

We show now that
\begin{equation}\label{vd42}
\big\|P_kI_3[f^\mu_{j_1,k_1},A_{\geq 1,\gamma_1}f^\nu_{j_2,k_2}]\big\|_{L^2}\lesssim 2^{-3m/2+4\delta m}.
\end{equation}
Recall that $|k_2|\leq 10$ and $|k|,|k_1|\in[0,100]$. It follows that $|\nabla_\eta\Phi(\xi,\eta)|\geq 2^{-\D}$ in the support of the integral (otherwise $|\eta|$ would be close to $\gamma_1/2$, as a consequence of Proposition \ref{spaceres11} (iii), which is not the case). The bound \eqref{vd42} (in fact rapid decay) follows using Lemma \ref{tech5} (i) unless
\begin{equation}\label{vd43}
j_2\geq (1-\delta^2)m.
\end{equation}

Finally, assume that \eqref{vd43} holds. Notice that $P_kI_3[A_{\geq 1,\gamma_0}f^\mu_{j_1,k_1},A_{\geq 1,\gamma_1}f^\nu_{j_2,k_2}]\equiv 0$. This is due to the fact that $|\lambda(\gamma_1)\pm\lambda(\gamma_0)\pm\lambda(\gamma_1\pm\gamma_0)|\gtrsim 1$, see Lemma \ref{LambdaBasic} (iv). Moreover, 
\begin{equation*}
\big\|P_{k}I^{\sigma\mu\nu}[A_{\leq 0,\gamma_0}f_{j_{1},k_{1}}^{\mu},A_{\geq 1,\gamma_1}f_{j_{2},k_{2}}^{\nu}]\big\|_{L^2}\lesssim 2^{-3m/2+3\delta m+6\delta^2m}
\end{equation*}
as a consequence of the $L^2\times L^\infty$ argument and the bound \eqref{LinftyBd3}. Therefore, using Lemma \ref{RotIBP},
\begin{equation*}
\big\|P_{k}I^{||}[A_{\leq 0,\gamma_0}f_{j_{1},k_{1}}^{\mu},A_{\geq 1,\gamma_1}f_{j_{2},k_{2}}^{\nu}]\big\|_{L^2}\lesssim 2^{-3m/2+3\delta m+6\delta^2m}.
\end{equation*}
The same averaging argument as in the proof of Lemma \ref{PhiLocLem} shows that
\begin{equation*}
\big\|P_{k}I_3[A_{\leq 0,\gamma_0}f_{j_{1},k_{1}}^{\mu},A_{\geq 1,\gamma_1}f_{j_{2},k_{2}}^{\nu}]\big\|_{L^2}\lesssim 2^{-3m/2+3\delta m+6\delta^2m},
\end{equation*}
and the desired bound \eqref{vd42} follows in this case as well. This completes the proof of \eqref{vd40}.

We prove now the bounds \eqref{Dbound5.1}. We notice that $|\eta|$ and $|\xi-\eta|$ are close to $\gamma_1/2$ in the support of the integral, due to Proposition \ref{spaceres11} (iii), so
\begin{equation*}
\widehat{G^{[1]}}(\xi)=\int_{\mathbb{R}^2}e^{is\Phi(\xi,\eta)}m(\xi,\eta)\varphi_k(\xi)\chi^{[1]}(\xi,\eta)\mathcal{F}\{A_{\geq 1,\gamma_1/2}f^\mu_{j_1,k_1}\}(\xi-\eta)\mathcal{F}\{A_{\geq 1,\gamma_1/2}f^\nu_{j_2,k_2}\}(\eta)d\eta.
\end{equation*}
Then we notice that the factor $\varphi(2^{30\delta m}\nabla_{\eta}\Phi(\xi,\eta))$ can be removed at the expense of negligible errors (due to Lemma \ref{tech5} (i)). The bound follows using the $L^2\times L^\infty$ argument and Lemma \ref{PhiLocLem}.

The bound on $G^{[2]}(s)$ in \eqref{Dbound5.2} follows using \eqref{vd6.9}, \eqref{Dbound5.1}, and \eqref{vd32}.

(ii) The plan is to localize suitably, in the Fourier space both in the radial and the angular directions, and use \eqref{FL1bd} or \eqref{FLinftybd}. More precisely, let
\begin{equation}\label{vd71}
B_{\kappa_\theta,\kappa_r}(\xi):=\int_{\mathbb{R}^2}e^{is\Phi(\xi,\eta)}m(\xi,\eta)\varphi_k(\xi)\varphi(\kappa_r^{-1}\Xi(\xi,\eta))\varphi(\kappa_\theta^{-1}\Theta(\xi,\eta))\widehat{f^\mu_{j_1,k_1}}(\xi-\eta)\widehat{f^\nu_{j_2,k_2}}(\eta)d\eta,
\end{equation}
where $\kappa_\theta$ and $\kappa_r$ are to be fixed. 

Let $\overline{j}:=\max(j_1,j_2)$. If 
\begin{equation*}
\min(k_1,k_2)\geq -2m/N'_0,\qquad\overline{j}\leq m/2
\end{equation*} 
then we set $\kappa_r=2^{2\delta m-m/2}$ (we do not localize in the angular variable in this case). Notice that $|\mathcal{F}\{A^{a_1,\alpha_1;a_2,\alpha_2}_{k;k_1,j_1;k_2,j_2}\}(\xi)-B_{\kappa_\theta,\kappa_r}(\xi)|\lesssim 2^{-4m}$ in view of Lemma \ref{tech5} (i). If $||\xi|-2\gamma_0|\geq 2^{-2\D}$ then we use Proposition \ref{spaceres11} (ii) and conclude that the integration in $\eta$ is over a ball of radius $\lesssim 2^{|k|}\kappa_r$. Therefore 
\begin{equation}\label{vd72}
|B_{\kappa_\theta,\kappa_r}(\xi)|\lesssim 2^{k+\min(k_1,k_2)/2}(2^{|k|}\kappa_r)^2\|\widehat{f^\mu_{j_1,k_1}}\|_{L^\infty}\|\widehat{f^\nu_{j_2,k_2}}\|_{L^\infty}\lesssim (2^{-k}+2^{3k})2^{-m+10\delta m}.
\end{equation} 

If 
\begin{equation*}
\min(k_1,k_2)\geq -2m/N'_0,\qquad\overline{j}\in[m/2,m-10\delta m]
\end{equation*} 
then we set $\kappa_r=2^{2\delta m+\overline{j}-m}$, $\kappa_\theta=2^{3\delta m-m/2}$. Notice that $|\mathcal{F}\{A^{a_1,\alpha_1;a_2,\alpha_2}_{k;k_1,j_1;k_2,j_2}\}(\xi)-B_{\kappa_\theta,\kappa_r}(\xi)|\lesssim 2^{-2m}$ in view of Lemma \ref{tech5} (i) and Lemma \ref{RotIBP}. If $||\xi|-2\gamma_0|\geq 2^{-2\D}$ then we use Proposition \ref{spaceres11} (ii) (notice that the hypothesis \eqref{zc1} holds in our case) to conclude that the integration in $\eta$ in the integral defining $B_{\kappa_\theta,\kappa_r}(\xi)$ is over a $O(\kappa\times\rho)$ rectangle in the direction of the vector $\xi$, where $\kappa:=2^{|k|}2^{\delta m}\kappa_\theta$, $\rho:=2^{|k|}\kappa_r$. Then we use \eqref{FL1bd} for the function corresponding to the larger $j$ and \eqref{FLinftybd} to the other function to estimate
\begin{equation}\label{vd73}
|B_{\kappa_\theta,\kappa_r}(\xi)|\lesssim 2^{k}\kappa 2^{-\overline{j}+51\delta\overline{j}}\rho^{49\delta}2^{2\delta\overline{j}}2^{2\delta m}\lesssim (2^{-k}+2^{3k})2^{-m+10\delta m}.
\end{equation}

If 
\begin{equation*}
\min(k_1,k_2)\geq -2m/N'_0,\qquad\overline{j}\geq m-10\delta m
\end{equation*}
then we have two subcases: if $\min(j_1,j_2)\leq m-10\delta m$ then we still localize in the angular direction (with $\kappa_\theta=2^{3\delta m-m/2}$ as before) and do not localize in the radial direction. The same argument as above, with $\rho\lesssim 2^{2\delta m}$, gives the same pointwise bound \eqref{vd73}. On the other hand, if $\min(j_1,j_2)\geq m-10\delta m$ then the desired conclusion follows by H\"{o}lder's inequality. The bound \eqref{vd70} follows if $\min(k_1,k_2)\geq -2m/N'_0$.

On the other hand, if $\min(k_1,k_2)\leq -2m/N'_0$ then $2^k\approx 2^{k_1}\approx 2^{k_2}$ (due to \eqref{vd5}) and the bound \eqref{vd70} can be proved in a similar way. The decomposition \eqref{vd70.1} is a consequence of \eqref{vd70} and the $L^2$ bounds \eqref{vd2}.

(iii) We prove now the decomposition \eqref{Decdtf}. With $\kappa:=2^{-m/2+\delta m+\delta^2m}$ we define
\begin{equation}\label{vd60}
\begin{split}
&g^{[1]}(\xi,s):=\int_{\mathbb{R}^2}e^{is\Phi'(\xi,\eta)}m(\xi,\eta)\varphi_k(\xi)\chi^{[1]}(\xi,\eta)\widehat{f^\mu_{j_1,k_1}}(\xi-\eta,s)\widehat{f^\nu_{j_2,k_2}}(\eta,s)\varphi(\kappa^{-1}\Xi(\xi,\eta))d\eta,\\
&h^{[1]}(\xi,s):=\int_{\mathbb{R}^2}e^{is\Phi(\xi,\eta)}m(\xi,\eta)\varphi_k(\xi)\chi^{[1]}(\xi,\eta)\widehat{f^\mu_{j_1,k_1}}(\xi-\eta,s)\widehat{f^\nu_{j_2,k_2}}(\eta,s)\varphi_{\geq 1}(\kappa^{-1}\Xi(\xi,\eta))d\eta,
\end{split}
\end{equation}
where $\Phi'(\xi,\eta)=\Phi_{\sigma\mu\nu}(\xi,\eta)-\Lambda_{\sigma}(\xi)+2\Lambda_\sigma(\xi/2)$. In view of Proposition \ref{spaceres11} (iii) and the definition of $\chi^{[1]}$, the function $G^{[1]}$ is nontrivial only when $\mu=\nu=\sigma$, and it is supported in the set $\{||\xi|-\gamma_1|\lesssim 2^{-10\delta m}\}$. The conclusion $\|h^{[1]}(s)\|_{L^\infty}\lesssim 2^{-4m}$ in \eqref{Decdtf} follows from Lemma \ref{tech5} (i) and the assumption $j_1,j_2\leq m/2+\delta m$.

To prove the bounds on $g^{[1]}$ we notice that $\Phi'(\xi,\eta)=2\Lambda_\sigma(\xi/2)-\Lambda_\sigma(\xi-\eta)-\Lambda_\sigma(\eta)$ and $|\eta-\xi/2|\lesssim \kappa$ (due to \eqref{try1.2}). Therefore $|\Phi'(\xi,\eta)|\lesssim \kappa^2$, $|(\nabla_\xi\Phi')(\xi,\eta)| \lesssim \kappa$, and $|(D^\alpha_\xi\Phi')(\xi,\eta)|\lesssim_{|\alpha|}1$ in the support of the integral. The bounds on $\|\D^\alpha_{\xi}g^{[1]}(s)\|_{L^\infty}$ in \eqref {Decdtf} follow using $L^\infty$ bounds on $\widehat{f^\mu_{j_1,k_1}}(s)$ and $\widehat{f^\nu_{j_2,k_2}}(s)$. The bounds on $\|\partial_sg^{[1]}(s)\|_{L^\infty}$ follow in the same way, using also the decomposition \eqref{vd70.1} when the $s$-derivative hits either $\widehat{f^\mu_{j_1,k_1}}(s)$ or $\widehat{f^\nu_{j_2,k_2}}(s)$ (the contribution of the $L^2$ component is estimated using H\"{o}lder's inequality). This completes the proof.
\end{proof}

Our last lemma concerning $\partial_t\mathcal{V}$ is a refinement of Lemma \ref{dtfLem2} (ii). It is only used in the proof of the decomposition \eqref{DecpartialtfEE}--\eqref{bnm30} in Lemma \ref{EELemmaSemi}.

\begin{lemma}\label{dtfLem3}
For $s\in[2^m-1,2^{m+1}]$ and $k\in[-10,10]$ we can decompose 
\begin{equation}\label{poi1}
\mathcal{F}\{P_{k}A_{\leq \D,2\gamma_0}(D^\alpha \Omega^a\partial_t\mathcal{V}_\sigma)(s)\}(\xi) = g_d(\xi) + g_\infty(\xi) + g_2(\xi)
\end{equation}
provided that $a\leq N_1/2+20$ and $2a+|\alpha|\leq N_1+N_4$, where
\begin{equation}\label{poi2}
\begin{split}
&\| g_2 \|_{L^2} \lesssim \e_1^2 2^{-3m/2+20\delta m},
  \qquad \| g_\infty\|_{L^\infty} \lesssim \e_1^22^{-m-4\delta m },\\
&\sup_{|\rho|\leq 2^{7m/9+4\delta m}}\|\mathcal{F}^{-1}\{e^{-i(s+\rho)\Lambda_\sigma}g_d\}\|_{L^\infty}\lesssim \varep_1^22^{-16m/9-4\delta m}.
\end{split}
\end{equation} 
\end{lemma}

\begin{proof} Starting from Lemma \ref{dtfLem1} (ii), we notice that the error term $E^{a,\alpha}_{\sigma}$ can be placed in the $L^2$ 
component $g_2$ (due to \eqref{vd2}). It remains to decompose the functions $A^{a_1,\alpha_1;a_2,\alpha_2}_{k;k_1,j_1;k_2,j_2}$. We may assume that 
we are in case (4), $k_1,k_2\in[-2m/N'_0,m/N'_0]$. We define the functions $B_{\kappa_\theta,\kappa_r}$ as in \eqref{vd71}. We notice that the argument in Lemma \ref{dtfLem2} (ii) already gives the desired conclusion if $\overline{j}=\max(j_1,j_2)\geq m/2+20\delta m$ (without having to use the function $g_d$). 

It remains to decompose the functions $A_{\leq \D,2\gamma_0}A^{a_1,\alpha_1;a_2,\alpha_2}_{k;k_1,j_1;k_2,j_2}(s)$ when
\begin{equation}\label{poi3}
\overline{j}=\max(j_1,j_2)\leq m/2+20\delta m. 
\end{equation}
As in \eqref{vd71} let
\begin{equation}\label{poi4}
B_{\kappa_r}(\xi):=\int_{\mathbb{R}^2}e^{is\Phi(\xi,\eta)}m(\xi,\eta)\varphi_k(\xi)\varphi(\kappa_r^{-1}\Xi(\xi,\eta))\widehat{f^\mu_{j_1,k_1}}(\xi-\eta)\widehat{f^\nu_{j_2,k_2}}(\eta)d\eta,
\end{equation}
where $\kappa_r:=2^{30\delta m-m/2}$ (we do not need angular localization here). In view of Lemma \ref{tech5} (i), $|\mathcal{F}A^{a_1,\alpha_1;a_2,\alpha_2}_{k;k_1,j_1;k_2,j_2}(\xi)-B_{\kappa_r}(\xi)|\lesssim 2^{-4m}$. It remains to prove that
\begin{equation}\label{poi5}
\big\|\mathcal{F}^{-1}\big\{e^{-i(s+\rho) \Lambda_\sigma(\xi)}\varphi_{\geq -\D}(2^{100}||\xi|-2\gamma_0|)B_{\kappa_r}(\xi)\big\}\big\|_{L^\infty}\lesssim 2^{-16m/9-5\delta m}
\end{equation}
for any $k,j_1,k_1,j_2,k_2,\rho$ fixed, $|\rho|\leq 2^{7m/9+4\delta m}$.

In proving \eqref{poi5}, we may assume that $m\geq \D^2$. The condition $|\Xi(\xi,\eta)|\leq 2\kappa_r$ shows that the variable $\eta$ is localized to a small ball. More precisely, using Lemma \ref{spaceres11}, we have
\begin{equation}\label{poi6}
|\eta-p(\xi)|\lesssim \kappa_r,\qquad \text{ for some }\qquad p(\xi)\in P_{\mu\nu}(\xi), 
\end{equation}
provided that $||\xi|-2\gamma_0|\gtrsim 1$. The sets $P_{\mu\nu}(\xi)$ are defined in \eqref{try1} and contain two or three points. We parametrize these points by $p_{\ell}(\xi)=q_{\ell}(|\xi|)\xi/|\xi|$, where $q_1(r)=r/2, q_2(r)=p_{++2}(r), q_3(r)=r-p_{++2}(r)$ if $\mu=\nu$, or $q_1(r)=p_{+-1}(r), q_2(r)=r-p_{+-1}(r)$ if $\mu=-\nu$. Then we rewrite
\begin{equation}\label{poi7}
B_{\kappa_r}(\xi)=\sum_{\ell}e^{is\Lambda_\sigma(\xi)}e^{-is[\Lambda_\mu(\xi- p_\ell(\xi))+\Lambda_\nu(p_\ell(\xi))]}H_{\ell}(\xi)
\end{equation}
where
\begin{equation}\label{poi8}
\begin{split}
H_{\ell}(\xi):=\int_{\mathbb{R}^2}&e^{is[\Phi(\xi,\eta)-\Phi(\xi,p_\ell(\xi)]}m(\xi,\eta)\varphi_k(\xi)\varphi(\kappa_r^{-1}\Xi(\xi,\eta))\\
&\widehat{f^\mu_{j_1,k_1}}(\xi-\eta)\widehat{f^\nu_{j_2,k_2}}(\eta)\varphi(2^{m/2-31\delta m}(\eta-p_{\ell}(\xi))d\eta.
\end{split}
\end{equation} 
Clearly, $|\Phi(\xi,\eta)-\Phi(\xi,p_\ell(\xi)|\lesssim |\eta-p_\ell(\xi)|^2$, $|\nabla_\xi[\Phi(\xi,\eta)-\Phi(\xi,p_\ell(\xi)]|\lesssim |\eta-p_\ell(\xi)|$. Therefore
\begin{equation}\label{poi9}
|D^\beta H_\ell(\xi)|\lesssim_\beta 2^{-m+70\delta m}2^{|\beta|(m/2+35\delta m)}\qquad\text{ if }\quad ||\xi|-2\gamma_0|\gtrsim 1.
\end{equation} 

We can now prove \eqref{poi5}. Notice that the factor $e^{is\Lambda_\sigma(\xi)}$ simplifies and that the remaining phase $\xi\to \Lambda_\mu(\xi- p_\ell(\xi))+\Lambda_\nu(p_\ell(\xi))$ is radial. Let $\Gamma_l=\Gamma_{l;\mu\nu}$ be defined such that $\Gamma_l(|\xi|)=\Lambda_\mu(\xi- p_\ell(\xi))+\Lambda_\nu(p_\ell(\xi))$. Standard stationary phase estimates, using also \eqref{poi9}, show that \eqref{poi5} holds provided that
\begin{equation}\label{poi10}
|\Gamma'_{\ell}(r)|\approx 1\,\,\text{ and }\,\,|\Gamma''_{\ell}(r)|\approx 1\qquad\text{ if }\qquad r\in[2^{-20},2^{20}],\,|r-2\gamma_0|\geq 2^{-3\D/2}.
\end{equation}

To prove \eqref{poi10}, assume first that $\mu=\nu$. If $\ell=1$ then $p_\ell(\xi)=\xi/2$ and the desired conclusion is clear. If $\ell\in\{2,3\}$ then $\pm\Gamma_\ell(r)=\lambda(r-p_{++2}(r))+\lambda (p_{++2}(r))$. In view of Proposition \ref{spaceres11} (i), $r-2\gamma_0\geq 2^{-2\D}$, $p_{++2}(r)\in(0,\gamma_0-2^{-2\D}]$, and $\lambda'(r-p_{++2}(r))=\lambda' (p_{++2}(r))$. Therefore
\begin{equation*}
|\Gamma'_{\ell}(r)|=\lambda'(r-p_{++2}(r)),\qquad |\Gamma''_{\ell}(r)|=|\lambda''(r-p_{++2}(r))(1-p'_{++2}(r))|.
\end{equation*}
The desired conclusions in \eqref{poi10} follow since $|1-p'_{++2}(r)|\approx 1$ in the domain of $r$ (due to the identity $\lambda''(r-p_{++2}(r))(1-p'_{++2}(r))=\lambda''(p_{++2}(r))p'_{++2}(r)$). 

The proof of \eqref{poi10} in the case $\mu=-\nu$ is similar. This completes the proof of the lemma.
\end{proof}

\section{Dispersive analysis, III: proof of Proposition \ref{bootstrap}}\label{Sec:Z1Norm} 

\subsection{Quadratic interactions}\label{QuadInt} In this section we prove Proposition \ref{bootstrap}. We start with the quadratic component in the Duhamel formula \eqref{duhamel} and show how to control its $Z$ norm.

\begin{proposition}\label{bootstrapIMP3} With the hypothesis in Proposition \ref{bootstrap}, for any $t\in[0,T]$ we have 
\begin{equation}\label{bootstrapIMP3.1}
\sup_{0\le a\leq N_1/2+20,\,2a+|\alpha|\leq N_1+N_4}\|D^\alpha\Omega^a W_2(t)\|_{Z_1}\lesssim \varepsilon_1^{2}.
\end{equation} 
\end{proposition}

The rest of this section is concerned with the proof of this proposition. Notice first that
\begin{equation}\label{ni2}
\begin{split}
\Omega^a_{\xi}\widehat{W_2}(\xi,t)=\sum_{\mu,\nu\in\{+,-\}}\sum_{a_1+a_2=a}\int_{0}^{t}\int_{\mathbb{R}^2}&e^{is\Phi_{+\mu\nu}(\xi,\eta)}\mathfrak{m}_{\mu\nu}(\xi,\eta)(\Omega^{a_1}\widehat{\mathcal{V}_{\mu}})(\xi-\eta,s)(\Omega^{a_2}\widehat{\mathcal{V}_{\nu}})(\eta,s)\,d\eta ds.
\end{split}
\end{equation}
 Given $t\in[0,T]$, we fix a suitable decomposition 
of the function $\mathbf{1}_{[0,t]}$, i.e. we fix functions $q_0,\ldots,q_{L+1}:\mathbb{R}\to[0,1]$, $|L-\log_2(2+t)|\leq 2$ as in \eqref{nh2}.
For $\mu,\nu\in\{+,-\}$, and $m\in[0,L+1]$ we define the operator $T^{\mu\nu}_{m}$ by
\begin{equation}\label{nh3}
\mathcal{F}\big\{T^{\mu\nu}_{m}[f,g]\big\}(\xi):=\int_{\mathbb{R}}q_m(s)\int_{\mathbb{R}^2}e^{is\Phi_{+\mu\nu}(\xi,\eta)}\mathfrak{m}_{\mu\nu}(\xi,\eta)\widehat{f}(\xi-\eta,s)\widehat{g}(\eta,s)d\eta ds.
\end{equation}

In view of Definition \ref{MainZDef}, Proposition \ref{bootstrapIMP3} follows from Proposition \ref{ZNormProp} below:

\begin{proposition}\label{ZNormProp}
Assume that $t\in[0,T]$ is fixed and define the operators $T^{\mu\nu}_{m}$ as above. If $a_1+a_2=a$, $\alpha_1+\alpha_2=\alpha$, $\mu,\nu\in\{+,-\}$, $m\in[0,L+1]$, and $(k,j)\in\mathcal{J}$, then 
\begin{equation}\label{nh4}
\sum_{k_1,k_2\in\mathbb{Z}}\big\| Q_{jk}T^{\mu\nu}_{m}[P_{k_1}D^{\alpha_1}\Omega^{a_1}\mathcal{V}_\mu,P_{k_2}D^{\alpha_2}\Omega^{a_2}\mathcal{V}_\nu]\big\|_{B_j}\lesssim 2^{-\delta^2m}\varepsilon_1^2.
\end{equation}
\end{proposition}

Assume that $a_1$, $a_2$, $b$, $\alpha_1$, $\alpha_2$, $\mu$, $\nu$ are fixed and let, for simplicity of notation,
\begin{equation}\label{ni3}
f^\mu:=\varepsilon_1^{-1}D^{\alpha_1}\Omega^{a_1}\mathcal{V}_{\mu},\quad f^\nu:=\varepsilon_1^{-1}D^{\alpha_2}\Omega^{a_2}\mathcal{V}_{\nu},\quad \Phi:=\Phi_{+\mu\nu},\quad  m_0:=\mathfrak{m}_{\mu\nu},\quad T_m:=T^{\mu\nu}_{m}.
\end{equation}
The bootstrap assumption \eqref{bootstrap2} gives, for any $s\in[0,t]$,
\begin{equation}\label{ni3.11}
\Vert f^\mu(s)\Vert_{H^{N'_0}\cap Z_1\cap H^{N'_1}_\Omega}+\Vert f^\nu(s)\Vert_{H^{N'_0}\cap Z_1\cap H^{N'_1}_\Omega}\lesssim (1+s)^{\delta^2}.
\end{equation}
We recall also the symbol-type bounds, which hold for any $k,k_1,k_2\in\mathbb{Z}$, $|\alpha|\geq 0$, 
\begin{equation}\label{ni3.111}
\begin{split}
\|m_0^{k,k_1,k_2}\|_{S^\infty}&\lesssim 2^k2^{\min(k_1,k_2)/2},\\
\|D_\eta^\alpha m_0^{k,k_1,k_2}\|_{L^\infty}&\lesssim_{|\alpha|} 2^{(|\alpha|+3/2)\max(|k_1|,|k_2|)},\\
\|D_\xi^\alpha m_0^{k,k_1,k_2}\|_{L^\infty}&\lesssim_{|\alpha|} 2^{(|\alpha|+3/2)\max(|k_1|,|k_2|,|k|)},
\end{split}
\end{equation}
where $m_0^{k,k_1,k_2}(\xi,\eta)=m_0(\xi,\eta)\cdot\varphi_k(\xi)\varphi_{k_1}(\xi-\eta)\varphi_{k_2}(\eta)$.

We consider first a few simple cases before moving to the main analysis in the next subsections. Recall (see \eqref{LinftyBd3.5}) that, for any $k\in\mathbb{Z}$, $m\in\{0,\ldots,L+1\}$, and $s\in I_m:=\mathrm{supp}\,q_m$,
\begin{equation}\label{DirectBds1}
\begin{split}
\Vert P_kf^\mu(s)\Vert_{L^2}+\Vert P_kf^\nu(s)\Vert_{L^2}&\lesssim 2^{\delta^2m}\min\{2^{(1-50\delta)k},2^{-N'_0k}\},\\
\Vert P_ke^{-is\Lambda_\mu}f^\mu(s)\Vert_{L^\infty}+\Vert P_ke^{-is\Lambda_\nu}f^\nu(s)\Vert_{L^\infty}&\lesssim 2^{3\delta^2m}\min\{2^{(2-50\delta)k},2^{-5m/6}\}.
\end{split}
\end{equation}

\begin{lemma}\label{ZNormEstSimpleLem1}
Assume that $f^\mu,f^\nu$ are as in \eqref{ni3} and let $(k,j)\in\mathcal{J}$. Then
\begin{equation}\label{Alx21}
\sum_{\max\{k_1,k_2\}\ge 1.01(j+m)/N'_0-\D^2}\Vert Q_{jk}T_m[P_{k_1}f^\mu,P_{k_2}f^\nu]\Vert_{B_j}\lesssim 2^{-\delta^2m},
\end{equation}
\begin{equation}\label{Alx22}
\sum_{\min\{k_1,k_2\}\le -(j+m)/2+\D^2}\Vert Q_{jk}T_m[P_{k_1}f^\mu,P_{k_2}f^\nu]\Vert_{B_j}\lesssim 2^{-\delta^2m},
\end{equation}
\begin{equation}\label{Alx23}
\text{ if }2k\leq -j-m+49\delta j-\delta m\text{ then }\,\,\sum_{k_1,k_2\in\mathbb{Z}}\Vert Q_{jk}T_m[P_{k_1}f^\mu,P_{k_2}f^\nu]\Vert_{B_j}\lesssim 2^{-\delta^2m},
\end{equation}
\begin{equation}\label{Alx24}
\text{ if }j\geq 2.1m\text{ then }\,\,\sum_{-j\le k_1,k_2\le 2j/N'_0}\Vert Q_{jk}T_m[P_{k_1}f^\mu,P_{k_2}f^\nu]\Vert_{B_j}\lesssim 2^{-\delta^2m}.
\end{equation}
\end{lemma}

\begin{proof} Using \eqref{DirectBds1}, the left-hand side of \eqref{Alx21} is dominated by
\begin{equation*}
\begin{split}
C\sum_{\max\{k_1,k_2\}\ge 1.01(m+j)/N'_0-\D^2}2^{j+m}2^{2k^+}2^{\min(k_1,k_2)/2}\sup_{s\in I_m}\Vert P_{k_1}f^\mu(s)\Vert_{L^2}\Vert P_{k_2}f^\nu(s)\Vert_{L^2}\lesssim 2^{-\delta m},
\end{split}
\end{equation*}
which is acceptable. Similarly, if $k_1\leq k_2$ and $k_1\leq \D^2$ then
\begin{equation*}
\begin{split}
2^j\Vert P_kT_m[P_{k_1}f^\mu,P_{k_2}f^\nu]\Vert_{L^2}&\lesssim 2^{j+m}2^{k+k_1/2}\sup_{s\in I_m}\Vert \widehat{P_{k_1}f^\mu}(s)\Vert_{L^1}\Vert P_{k_2}f^\nu(s)\Vert_{L^2}\\
&\lesssim 2^{j+m}2^{(5/2-50\delta)k_1}2^{-(N'_0-1)\max(k_2,0)},
\end{split}
\end{equation*}
and the bound \eqref{Alx22} follows by summation over $\min\{k_1,k_2\}\le -(j+m)/2+2\D^2$.

To prove \eqref{Alx23} we may assume that
\begin{equation}\label{Alx25}
2k\leq -j-m+49\delta j-\delta m,\qquad-(j+m)/2\le k_1,k_2\le 1.01(j+m)/N'_0.
\end{equation}
Then
\begin{equation*}
\begin{split}
\Vert Q_{jk}T_m&[P_{k_1}f^\mu,P_{k_2}f^\nu]\Vert_{B_j}\lesssim 2^{j(1-50\delta)}\Vert P_kT_m[P_{k_1}f^\mu,P_{k_2}f^\nu]\Vert_{L^2}\\
&\lesssim 2^{j(1-50\delta)}2^m 2^{k+\min(k_1,k_2)/2}2^{k}\sup_{s\in I_m}\Vert P_{k_1}f^\mu(s)\Vert_{L^2}\Vert P_{k_2}f^\nu(s)\Vert_{L^2}\\
&\lesssim 2^{-\delta(j+m)/3}.
\end{split}
\end{equation*}
Summing in $k_1,k_2$ as in \eqref{Alx25}, we obtain an acceptable contribution.

Finally, to prove \eqref{Alx24} we may assume that
\begin{equation*}
j\ge 2.1m,\qquad j+k\ge j/10+\D,\qquad -j\le k_1,k_2\le 2j/N'_0,
\end{equation*}
and define 
\begin{equation}\label{Alx25.5}
f^\mu_{j_1,k_1}:=P_{[k_1-2,k_1+2]}Q_{j_1k_1}f^\mu,\qquad f^\nu_{j_2,k_2}:=P_{[k_2-2,k_2+2]}Q_{j_2k_2}f^\nu.
\end{equation}
If $\min\{j_1,j_2\}\ge 99j/100-\D$ then, using also \eqref{FL1bd},
\begin{equation*}
\begin{split}
\Vert P_kT_m[f^\mu_{j_1,k_1},f^\nu_{j_2,k_2}]\Vert_{L^2}&\lesssim 2^m2^{k+\min(k_1,k_2)/2}\sup_{s\in I_m}\Vert \widehat{f^\mu_{j_1,k_1}}(s)\Vert_{L^1}\Vert f^\nu_{j_2,k_2}(s)\Vert_{L^2}\\
&\lesssim 2^m2^{k+3k_1/2}2^{-(1-\delta')j_1-(1/2-\delta)j_2}2^{4\delta^2m},
\end{split}
\end{equation*}
and therefore
\begin{equation*}
\begin{split}
\sum_{-j\le k_1,k_2\le 2j/N'_0}\sum_{\min\{j_1,j_2\}\ge 99j/100-\D}\Vert Q_{jk}T_m[f^\mu_{j_1,k_1},f^\nu_{j_2,k_2}]\Vert_{B_j}\lesssim 2^{-\delta m}.
\end{split}
\end{equation*}
On the other hand, if $j_1\le 99j/100-\mathcal{D}$ then we rewrite
\begin{equation}\label{Alx26.4}
\begin{split}
&Q_{jk}T_{m}[f^\mu_{j_1,k_1},f^\nu_{j_2,k_2}](x)=C\phii_j^{(k)}(x)\\
&\times\int_{\mathbb{R}}q_m(s)\int_{\mathbb{R}^2}\left[\int_{\mathbb{R}^2}e^{i\left[s\Phi(\xi,\eta)+x\cdot \xi\right]}\varphi_k(\xi)m_0(\xi,\eta)\widehat{f^\mu_{j_1,k_1}}(\xi-\eta,s)d\xi\right] \widehat{f^\nu_{j_2,k_2}}(\eta,s)d\eta ds.
\end{split}
\end{equation}
In the support of integration, we have the lower bound $\left\vert\nabla_\xi\left[s\Phi(\xi,\eta)+x\cdot\xi\right]\right\vert\approx\vert x\vert\approx 2^j$. Integration by parts in $\xi$ using Lemma \ref{tech5} gives
\begin{equation}\label{Alx26.5}
\left\vert Q_{jk}T_{m}[f^\mu_{j_1,k_1},f^\nu_{j_2,k_2}](x)\right\vert\lesssim 2^{-10j}
\end{equation}
which gives an acceptable contribution. This finishes the proof.
\end{proof}

\subsection{The main decomposition}\label{MainDecomp}

We may assume that
\begin{equation}\label{Ass2}
k_1,k_2\in\Big[-\frac{j+m}{2},\frac{1.01(j+m)}{N'_0}\Big],\quad k\geq \frac{-j-m+49\delta j-\delta m}{2},\quad j\le 2.1m,\quad m\geq \D^2/8.
\end{equation}
Recall the definition \eqref{Alx80}. We fix $l_-:=\lfloor-(1-\delta/2) m\rfloor$, and decompose
\begin{equation}\label{TML}
\begin{split}
T_m[f,g]&=\sum_{l_-\leq l}T_{m,l}[f,g],\\
\widehat{T_{m,l}[f,g]}(\xi)&:=\int_{\mathbb{R}}q_m(s)\int_{\mathbb{R}^2}e^{is\Phi(\xi,\eta)}\varphi_l^{[l_-,m]}(\Phi(\xi,\eta))m_0(\xi,\eta)\widehat{f}(\xi-\eta,s)\widehat{g}(\eta,s)d\eta ds.
\end{split}
\end{equation}
Assuming \eqref{Ass2}, we notice that $T_{m,l}[P_{k_1}f^\mu,P_{k_2}f^\nu]\equiv 0$ if $l\geq 10m/N'_0$. When $l>l_-$, we may integrate by parts in time to rewrite $T_{m,l}[P_{k_1}f^\mu,P_{k_2}f^\nu]$,
\begin{equation}\label{Alx41}
\begin{split}
&T_{m,l}[P_{k_1}f^\mu,P_{k_2}f^\nu]=i\mathcal{A}_{m,l}[P_{k_1}f^\mu,P_{k_2}f^\nu]+i\mathcal{B}_{m,l}[P_{k_1}\partial_sf^\mu,P_{k_2}f^\nu]+i\mathcal{B}_{m,l}[P_{k_1}f^\mu,P_{k_2}\partial_sf^\nu],\\
&\widehat{\mathcal{A}_{m,l}[f,g]}(\xi):=\int_{\mathbb{R}}q_m^\prime(s)\int_{\mathbb{R}^2}e^{is\Phi(\xi,\eta)}2^{-l}\widetilde{\varphi}_{l}(\Phi(\xi,\eta))m_0(\xi,\eta)\widehat{f}(\xi-\eta,s)\widehat{g}(\eta,s)\,d\eta ds,\\
&\widehat{\mathcal{B}_{m,l}[f,g]}(\xi):=\int_{\mathbb{R}}q_m(s)\int_{\mathbb{R}^2}e^{is\Phi(\xi,\eta)}2^{-l}\widetilde{\varphi}_{l}(\Phi(\xi,\eta))m_0(\xi,\eta)\widehat{f}(\xi-\eta,s)\widehat{g}(\eta,s)\,d\eta ds,
\end{split}
\end{equation}
where $\widetilde{\varphi}_l(x):=2^lx^{-1}\varphi_{l}(x)$. For $s$ fixed let $\mathcal{I}_l$ denote the bilinear operator defined by
\begin{equation}\label{il}
\widehat{\mathcal{I}_l[f,g]}(\xi):=\int_{\mathbb{R}^2}e^{is\Phi(\xi,\eta)}2^{-l}\widetilde{\varphi}_{l}(\Phi(\xi,\eta))m_0(\xi,\eta)\widehat{f}(\xi-\eta)\widehat{g}(\eta)\,d\eta.
\end{equation}

It is easy to see that Proposition \ref{ZNormProp} follows from Lemma \ref{ZNormEstSimpleLem1} and Lemmas \ref{FSPLem}--\ref{lsmallBound} below.

\begin{lemma}\label{FSPLem}
Assume that \eqref{Ass2} holds and, in addition, 
\begin{equation}\label{case1}
j\geq m+2\D+\max(|k|,|k_1|,|k_2|)/2.
\end{equation}
Then, for $l_-\le l\le  10m/N'_0$,
\begin{equation*}
2^{(1-50\delta)j}\Vert Q_{jk}T_{m,l}[P_{k_1}f^\mu,P_{k_2}f^\nu]\Vert_{L^2}\lesssim 2^{-2\delta^2m}.
\end{equation*}
\end{lemma}

Notice that the assumptions \eqref{Ass2} and $j\leq m+2\D+\max(|k|,|k_1|,|k_2|)/2$ show that 
\begin{equation}\label{ki7}
k,k_1,k_2\in[-4m/3-2\D,3.2 m/N'_0],\qquad m\geq\D^2/8.
\end{equation}

\begin{lemma}\label{KsmallLem}
Assume that \eqref{ki7} holds and, in addition, 
\begin{equation}\label{case1.5}
j\leq m+2\D+\max(|k|,|k_1|,|k_2|)/2,\qquad\min(k,k_1,k_2)\leq -3.5m/N'_0.
\end{equation}
Then, for $l_-\le l\leq 10m/N'_0$,
\begin{equation*}
2^{(1-50\delta)j}\Vert Q_{jk}T_{m,l}[P_{k_1}f^\mu,P_{k_2}f^\nu]\Vert_{L^2}\lesssim 2^{-2\delta^2m}.
\end{equation*}
\end{lemma}

\begin{lemma}\label{StronglyResLem}
Assume that \eqref{ki7} holds and, in addition, 
\begin{equation}\label{case2}
j\leq m+2\D+\max(|k|,|k_1|,|k_2|)/2,\qquad\min(k,k_1,k_2)\geq -3.5m/N'_0.
\end{equation}
Then, for $l_-<l\leq 10m/N'_0$
\begin{equation*}
\Vert Q_{jk}T_{m,l_-}[P_{k_1}f^\mu,P_{k_2}f^\nu]\Vert_{B_j}+\Vert Q_{jk}\mathcal{A}_{m,l}[P_{k_1}f^\mu,P_{k_2}f^\nu]\Vert_{B_j}\lesssim 2^{-2\delta^2m}.
\end{equation*}
\end{lemma}

\begin{lemma}\label{lLargeBBound}
Assume that \eqref{ki7} holds and, in addition, 
\begin{equation}\label{case3}
j\leq m+2\D+\max(|k|,|k_1|,|k_2|)/2,\qquad \min(k,k_1,k_2)\geq -3.5m/N'_0, \qquad l\geq -m/14.
\end{equation}
Then
\begin{equation*}
2^{(1-50\delta)j}\Vert Q_{jk}\mathcal{B}_{m,l}[P_{k_1}f^\mu,P_{k_2}\partial_sf^\nu]\Vert_{L^2}\lesssim 2^{-2\delta^2m}.
\end{equation*}
\end{lemma}

\begin{lemma}\label{lsmallBound}
Assume that \eqref{ki7} holds and, in addition, 
\begin{equation}\label{case3.5}
j\leq m+2\D+\max(|k|,|k_1|,|k_2|)/2,\quad \min(k,k_1,k_2)\geq -3.5m/N'_0, \quad l_-<l\leq -m/14.
\end{equation}
Then
\begin{equation*}
\Vert Q_{jk}T_{m,l}[P_{k_{1}}f^\mu,P_{k_2}f^\nu]\Vert_{B_{j}}\lesssim 2^{-2\delta^2m}.
\end{equation*}
\end{lemma}

We prove these lemmas in the following five subsections. Lemma \ref{FSPLem} takes advantage of the approximate finite of propagation. Lemma \ref{KsmallLem} uses the null structure at low frequencies. Lemma \ref{StronglyResLem} controls interactions that lead to the creation of a space-time resonance. Lemmas \ref{lLargeBBound} and \ref{lsmallBound} correspond to interactions that are particularly difficult to control in dimension $2$ and contain the main novelty of our analysis (see also \cite{DIP}). They rely on all the estimates in Lemmas \ref{dtfLem1} and \ref{dtfLem2}, and on the ``slow propagation of iterated resonances'' properties in Lemma \ref{cubicphase}. 

We will use repeatedly the symbol bounds \eqref{ni3.111} and the main assumption \eqref{ni3.11}.

\subsection{Approximate finite speed of propagation} In this subsection we prove Lemma \ref{FSPLem}. 
We define the functions $f^\mu_{j_1,k_1}$ and $f^\nu_{j_2,k_2}$ as before, see \eqref{Alx25.5}, and further decompose 
\begin{equation}\label{Alx70}
f^\mu_{j_1,k_1}=\sum_{n_1=0}^{j_1+1} f^\mu_{j_1,k_1,n_1},\qquad f^\nu_{j_2,k_2}=\sum_{n_2=0}^{j_2+1} f^\nu_{j_2,k_2,n_2}
\end{equation}
as in \eqref{Alx100}. If $\min\{j_1,j_2\}\le j-\delta m$ then the same argument as in the proof of \eqref{Alx24} leads to rapid decay, as in \eqref{Alx26.5}. To bound the sum over $\min\{j_1,j_2\}\ge j-\delta m$ we consider several cases.

{\bf{Case 1.}} Assume first that
\begin{equation}\label{ki1}
\min(k,k_1,k_2)\leq -m/2.
\end{equation}
Then we notice that
\begin{equation*}
\begin{split}
\big\Vert \mathcal{F}\big\{P_kT_{m,l}[f^\mu_{j_1,k_1},f^\nu_{j_2,k_2}]\big\}\big\Vert_{L^\infty}&\lesssim 2^m2^{k+\min(k_1,k_2)/2}\sup_{s\in I_m}\big[\big\Vert\widehat{f^\mu_{j_1,k_1}}(s)\big\Vert_{L^2}\big\Vert\widehat{f^\nu_{j_2,k_2}}(s)\big\Vert_{L^2}\big]\\
&\lesssim 2^m2^{2\delta^2m}2^{k}2^{-(1/2-\delta)(j_1+j_2)}.
\end{split}
\end{equation*}
Therefore, the sum over $j_1,j_2$ with $\min(j_1,j_2)\geq j-\delta m$ is controlled as claimed provided that $k\leq -m/2$. On the other hand, if $k_1=\min(k_1,k_2)\leq -m/2$ then we estimate
\begin{equation}\label{ki1.5}
\begin{split}
\Vert P_kT_{m,l}[f^\mu_{j_1,k_1},f^\nu_{j_2,k_2}]\Vert_{L^2}&\lesssim 2^m2^{k+k_1/2}\sup_{s\in I_m}\big[\big\Vert\widehat{f^\mu_{j_1,k_1}}(s)\big\Vert_{L^1}\big\Vert\widehat{f^\nu_{j_2,k_2}}(s)\big\Vert_{L^2}\big]\\
&\lesssim 2^m2^{2\delta^2m}2^{k+k_1/2}2^{k_1}2^{-(1-50\delta)j_1}2^{-(1/2-\delta)j_2}2^{-4\max(k_2,0)}.
\end{split}
\end{equation}
The sum over $j_1,j_2$ with $\min(j_1,j_2)\geq j-\delta m$ is controlled as claimed in this case as well.

{\bf{Case 2.}} Assume now that
\begin{equation}\label{ki2}
\min(k,k_1,k_2)\geq -m/2,\qquad l\leq \min(k,k_1,k_2,0)/2-m/5.
\end{equation}
We use Lemma \ref{Shur2Lem}: we may assume that $\min(k,k_1,k_2)+\max(k,k_1,k_2)\geq -100$ and estimate
\begin{equation*}
\begin{split}
\big\Vert P_kT_{m,l}&[f^\mu_{j_1,k_1,n_1},f^\nu_{j_2,k_2,n_2}]\big\Vert_{L^2}\lesssim 2^m2^{k+\min(k_1,k_2)/2}2^{5\max(k_1,k_2,0)}2^{l/2-n_1/2-n_2/2}\\
&\sup_{s\in I_m}\big[\big\Vert \sup_\theta|\widehat{f^\mu_{j_1,k_1,n_1}}(r\theta,s)|\big\Vert_{L^2(rdr)}\big\Vert \sup_\theta|\widehat{f^\nu_{j_2,k_2,n_2}}(r\theta,s)|\big\Vert_{L^2(rdr)}\big].
\end{split}
\end{equation*}
Using \eqref{RadL2}, \eqref{ni3.11}, and summing over $n_1,n_2$, we have
\begin{equation*}
2^{(1-50\delta)j}\big\Vert P_kT_{m,l}[f^\mu_{j_1,k_1},f^\nu_{j_2,k_2}]\big\Vert_{L^2}\lesssim 2^{7\max(k_1,k_2,0)}2^m2^{2\delta^2m}2^{(1-50\delta)j}2^{l/2}2^{-(1-\delta')(j_1+j_2)}.
\end{equation*}
The sum over $j_1,j_2$ with $\min(j_1,j_2)\geq j-\delta m$ is controlled as claimed.

{\bf{Case 3.}} Finally, assume that
\begin{equation}\label{Alx73}
\min(k,k_1,k_2)\geq -m/2,\qquad l\geq \min(k,k_1,k_2,0)/2-m/5.
\end{equation}
We use the formula \eqref{Alx41}. The contribution of $\mathcal{A}_{m,l}$ can be estimated as in \eqref{ki1.5}, with $2^m$ replaced by $2^{-l}$, and we focus on the contribution of $\mathcal{B}_{m,l}[P_{k_1}f^\mu,P_{k_2}\partial_{s}f^{\nu}]$. We decompose $\partial_{s}f^{\nu}(s)$, according to \eqref{Brc4}. The contribution of $P_{k_2}E^{a_2,\alpha_2}_\nu$ can be estimated easily,
\begin{equation}\label{vd50}
\begin{split}
\Vert P_k\mathcal{B}_{m,l}&[f^\mu_{j_1,k_1},P_{k_2}E^{a_2,\alpha_2}_\nu]\Vert_{L^2}\lesssim 2^m2^{-l}2^{k+\min(k_1,k_2)/2}\sup_{s\in I_m}\big[\big\Vert\widehat{f^\mu_{j_1,k_1}}(s)\big\Vert_{L^1}\big\Vert P_{k_2}E^{a_2,\alpha_2}_\nu(s)\big\Vert_{L^2}\big]\\
&\lesssim 2^m2^{2\delta^2m}2^{m/5-\min(k,k_1,k_2,0)/2}2^{k+k_2/2}2^{k_1}2^{-(1-51\delta)j_1}2^{-3m/2+5\delta m}\\
&\lesssim 2^{-(1-51\delta)j_1}2^{-m/4},
\end{split}
\end{equation}
and the sum over $j_1\geq j-\delta m$ of $2^{(1-50\delta)j}\Vert P_k\mathcal{B}_{m,l}[f^\mu_{j_1,k_1},P_{k_2}E^{a_2,\alpha_2}_\nu]\Vert_{L^2}$ is suitably bounded. 

We consider now the terms $A_{k_{2};k_{3},j_{3},k_{4},j_{4}}^{a_3,\alpha_3;a_4,\alpha_4}(s)$ in \eqref{Brc4}, $[(k_3,j_3),(k_4,j_4)]\in X_{m,k_2}$, $\alpha_3+\alpha_4=\alpha_2$, $a_3+a_4\leq a_2$. In view of \eqref{vd5}, \eqref{vd6}, and \eqref{vd6.8}, $\Vert A_{k_{2};k_{3},j_{3},k_{4},j_{4}}^{a_3,\alpha_3;a_4,\alpha_4}(s)\Vert_{L^2}\lesssim 2^{-4m/3+4\delta m}$
\begin{equation*}
\text{ if }\,\,\max(j_{3},j_{4})\geq (1-\delta^{2})m-|k_{2}|\,\,\text{ or if }\,\,|k_2|+\D/2\leq\min(|k_3|,|k_4|).
\end{equation*}
The contributions of these terms can be estimated as in \eqref{vd50}. On the other hand, to control the contribution 
of $Q_{jk}\mathcal{B}_{m,l}[f^\mu_{j_1,k_1},A_{k_{2};k_{3},j_{3},k_{4},j_{4}}^{a_3,\alpha_3;a_4,\alpha_4}]$ when 
$\max(j_{3},j_{4})\leq (1-\delta^{2})m-|k_{2}|$ and $|k_2|+\D/2\geq |k_3|$, we simply rewrite this in the form 
\begin{equation}\label{New01}
\begin{split}
c&\phii_j^{(k)}(x)\int_{\mathbb{R}}q_m(s)\int_{\mathbb{R}^2}\widehat{f^\mu_{j_1,k_1}}(\eta,s)\Big[\int_{\mathbb{R}^2\times\mathbb{R}^2}e^{i[x\cdot\xi+s\widetilde{\Phi}'(\xi,\eta,\sigma)]}2^{-l}\widetilde{\varphi}_{l}(\Phi_{\sigma\mu\nu}(\xi,\xi-\eta))\\
&\times\varphi_k(\xi)\varphi_{k_2}(\xi-\eta)\mathfrak{m}_{\mu\nu}(\xi,\xi-\eta)\mathfrak{m}_{\nu\beta\gamma}(\xi-\eta,\sigma)\widehat{f^\beta_{j_3,k_3}}(\xi-\eta-\sigma,s)\widehat{f^\gamma_{j_4,k_4}}(\sigma,s)\,d\xi d\sigma\Big]d\eta ds,
\end{split}
\end{equation}
where $\widetilde{\Phi}'(\xi,\eta,\sigma):=\Lambda(\xi)-\Lambda_\mu(\eta)-\Lambda_{\beta}(\xi-\eta-\sigma)-\Lambda_{\gamma}(\sigma)$. Notice that
\begin{equation}\label{New02}
\big|\nabla_\xi[x\cdot\xi+s\Lambda(\xi)-s\Lambda_\mu(\eta)-s\Lambda_{\beta}(\xi-\eta-\sigma)-s\Lambda_{\gamma}(\sigma)]\big|\approx |x|\approx 2^j.
\end{equation}
We can integrate by parts in $\xi$ using Lemma \ref{tech5} (i) to conclude that these are 
negligible contributions, pointwise bounded by $C2^{-5m}$. This completes the proof of the lemma.

\subsection{The case of small frequencies}\label{Gty1} In this subsection we prove Lemma \ref{KsmallLem}. The main point is that if $\underline{k}:=\min(k,k_1,k_2)\leq -3.5m/N'_0$ then $|\Phi(\xi,\eta)|\gtrsim 2^{\underline{k}/2}$ for any $(\xi,\eta)\in \mathcal{D}_{k,k_1,k_2}$, as a consequence of \eqref{try5.5} and \eqref{ki7}. Therefore the operators $T_{m,l}$ are nontrivial only if 
\begin{equation}\label{Gty2}
l\geq \underline{k}/2-\D.
\end{equation}

{\bf{Step 1.}} We consider first the operators $\mathcal{A}_{m,l}$. Since $l\geq -2m/3-2\D$, it suffices to prove that
\begin{equation}\label{ki7.1}
2^{(1-50\delta)(m-\underline{k}/2)}\big\|P_k\mathcal{I}_l[f^\mu_{j_1,k_1}(s),f^\nu_{j_2,k_2}(s)]\big\|_{L^2}\lesssim 2^{-3\delta^2m},
\end{equation}
for any $s\in I_m$ and $j_1,j_2$, where $\mathcal{I}_l$ are the operators defined in \eqref{il}, and $f^\mu_{j_1,k_1}$ and $f^\nu_{j_2,k_2}$ are as in \eqref{Alx25.5}. We may assume $k_1\leq k_2$ and consider two cases.

{\bf{Case 1.}} If $\underline{k}=k_1$ then we estimate first the left-hand side of \eqref{ki7.1} by
\begin{equation*}
\begin{split}
C2^{(1-50\delta)(m-\underline{k}/2)}\cdot 2^{k+\underline{k}/2}2^{-l}\big[\sup_{s,t\approx 2^m}\big\|e^{-it\Lambda_\mu}f^\mu_{j_1,k_1}(s)\big\|_{L^\infty}\big\|f^\nu_{j_2,k_2}(s)\big\|_{L^2}+2^{-8m}\big]\\
\lesssim 2^{(1-50\delta)(m-\underline{k}/2)}\cdot 2^k2^{6\delta^2m}\big[2^{\underline{k}}2^{-m+50\delta j_1}2^{-4k^+}+2^{-8m}\big],
\end{split}
\end{equation*}
using Lemma \ref{PhiLocLem} and \eqref{LinftyBd1.1}. This suffices to prove \eqref{ki7.1} if $j_1\leq 9m/10$. On the other hand, if $j_1\geq 9m/10$ then we estimate the left-hand side of \eqref{ki7.1} by
\begin{equation*}
\begin{split}
C2^{(1-50\delta)(m-\underline{k}/2)}\cdot 2^{k+\underline{k}/2}2^{-l}\big[\sup_{s\approx 2^m}\big\|f^\mu_{j_1,k_1}(s)\big\|_{L^2}\big\|e^{-it\Lambda_\nu}f^\nu_{j_2,k_2}(s)\big\|_{L^\infty}+2^{-8m}\big]\\
\lesssim 2^{(1-50\delta)(m-\underline{k}/2)}\cdot 2^k2^{6\delta^2m}\big[2^{-(1-50\delta)j_1}2^{-5m/6}2^{-2k^+}+2^{-8m}\big],
\end{split}
\end{equation*}
using Lemma \ref{PhiLocLem} and \eqref{LinftyBd3.5}. This suffices to prove the desired bound \eqref{ki7.1}.

{\bf{Case 2.}} If $\underline{k}=k$ then \eqref{ki7.1} follows using the $L^2\times L^\infty$ estimate, as in Case 1, unless
\begin{equation}\label{ki7.2}
\max(|k_1|,|k_2|)\leq 20,\qquad \max(j_1,j_2)\leq m/3.
\end{equation}
On the other hand, if \eqref{ki7.2} holds then it suffices to prove that, for $|\rho|\leq 2^{m-\D}$,
\begin{equation}\label{ki7.3}
\begin{split}
&2^{(1-50\delta)(m-k/2)}2^{-k/2}\big\|P_kI_0[f^\mu_{j_1,k_1}(s),f^\nu_{j_2,k_2}(s)]\big\|_{L^2}\lesssim 2^{-3\delta^2m},\\
&\widehat{I_0[f,g]}(\xi):=\int_{\mathbb{R}^2}e^{i(s+\rho)\Phi(\xi,\eta)} m_0(\xi,\eta)\widehat{f}(\xi-\eta)\widehat{g}(\eta)\,d\eta.
\end{split}
\end{equation}
Indeed, \eqref{ki7.1} would follow from \eqref{ki7.3} and the inequality $l\geq \underline{k}/2-\D\geq -2m/3-2\D$ (see \eqref{ki7}--\eqref{Gty2}), using the superposition argument in Lemma \ref{PhiLocLem}. On the other hand, the proof of \eqref{ki7.3} is similar to the proof of \eqref{Dbound4.0} in Lemma \ref{dtfLem1}.

{\bf{Step 2.}} We consider now the operators $\mathcal{B}_{m,l}$. In some cases we prove the stronger bound
\begin{equation}\label{kj7.1}
2^{(1-50\delta)(m-\underline{k}/2)}2^m\big\|P_k\mathcal{I}_l[f^\mu_{j_1,k_1}(s),P_{k_2}\partial_sf^\nu(s)]\big\|_{L^2}\lesssim 2^{-3\delta^2m},
\end{equation}
for any $s\in I_m$ and $j_1$. We consider three cases.

{\bf{Case 1.}} If $\underline{k}=k_1$ then we use the bounds
\begin{equation}\label{kj7.2}
\begin{split}
&\|P_{k_2}\partial_sf^\nu(s)\|_{L^2}\lesssim 2^{-m+5\delta m}(2^{k_2}+2^{-m/2}),\\
&\|e^{-is\Lambda_\nu}P_{k_2}\partial_sf^\nu(s)\|_{L^\infty}\lesssim 2^{-5m/3+6\delta^2m},
\end{split} 
\end{equation}
see \eqref{vd7} and \eqref{Brc2}. We also record the bound, which can be verified easily using integration by parts and Plancherel for any $\rho\in\mathbb{R}$ and $k'\in\mathbb{Z}$,
\begin{equation}\label{kj7.45}
\big\|e^{-i\rho\Lambda}P_{k'}\big\|_{L^\infty\to L^\infty}\lesssim \big\|\mathcal{F}^{-1}\big\{e^{-i\rho\Lambda(\xi)}\varphi_{k'}(\xi)\big\}\big\|_{L^1}\lesssim 1+2^{k'/2}2^{k'_+}|\rho|.
\end{equation}

If
\begin{equation}\label{kj7.3}
k_1\geq -m/4,\qquad j_1\leq (1-\delta^2)m
\end{equation}
then we use \eqref{LinftyBd3}, \eqref{kj7.2}, and Lemma \ref{PhiLocLem} to estimate the left-hand side of \eqref{kj7.1} by
\begin{equation*}
\begin{split}
C&2^{k+k_1/2}2^{(1-50\delta)(m-\underline{k}/2)}2^m\big[2^{-l}\sup_{|\rho|\leq 2^{m/2}}\|e^{-i(s+\rho)\Lambda_\mu}f^\mu_{j_1,k_1}(s)\|_{L^\infty}\|P_{k_2}\partial_sf^\nu(s)\|_{L^2}+2^{-8m}\big]\\
&\lesssim 2^{6k^+}2^{k_1/2}2^{-40\delta m}.
\end{split}
\end{equation*}
This suffices to prove \eqref{kj7.1} when \eqref{kj7.3} holds (recall the choice of $\delta,N_0,N_1$ in Definition \ref{MainZDef}). On the other hand, if
\begin{equation}\label{kj7.4}
k_1\geq -m/4,\qquad j_1\geq (1-\delta^2)m
\end{equation}
then we use \eqref{kj7.45}, \eqref{LinftyBd}, \eqref{kj7.2}, and Lemma \ref{PhiLocLem} to estimate the left-hand side of \eqref{kj7.1} by
\begin{equation*}
\begin{split}
C&2^{k+k_1/2}2^{(1-50\delta)(m-\underline{k}/2)}2^m\big[2^{-l}\|f^\mu_{j_1,k_1}(s)\|_{L^2}\sup_{|\rho|\leq 2^{-l+4\delta^2m}}\|e^{-i(s+\rho)\Lambda_\nu}P_{k_2}\partial_sf^\nu(s)\|_{L^\infty}+2^{-8m}\big]\\
&\lesssim 2^{10k^+}2^{-2m/3+10\delta m}2^{-2l}.
\end{split}
\end{equation*}
This suffices to prove \eqref{kj7.1}, provided that \eqref{kj7.4} holds.

Finally, if $k_1\leq -m/4$ then we use the bound 
\begin{equation*}
\sup_{|\rho|\leq 2^{m-\D}}\|e^{-i(s+\rho)\Lambda_\mu}f^\mu_{j_1,k_1}(s)\|_{L^\infty}\lesssim 2^{(3/2-25\delta)k_1}2^{-m+50\delta m}2^{\delta^2m},
\end{equation*}
which follows from \eqref{LinftyBd}--\eqref{LinftyBd1.1}. Then we estimate the left-hand side of \eqref{kj7.1} by
\begin{equation*}
C2^{2k^++k_1/2}2^{(1-50\delta)(m-\underline{k}/2)}2^m\cdot 2^{-l}2^{(3/2-25\delta)k_1}2^{-m+51\delta m}2^{-m+5\delta m}\lesssim 2^{6k^+}2^{10\delta m}2^{k_1}.
\end{equation*}
The desired bound \eqref{kj7.1} follows, provided that $k_1\leq -m/4$.

{\bf{Case 2.}} If $\underline{k}=k$ then \eqref{kj7.1} follows using $L^2\times L^\infty$ estimates, as in Case 1, unless
\begin{equation}\label{kj7.7}
\max(|k_1|,|k_2|)\leq 20.
\end{equation}
Assuming \eqref{kj7.7}, we notice that 
\begin{equation}\label{kj7.8}
\begin{split}
&\sup_{|\rho|\leq 2^{m-\D}}\|e^{-i(s+\rho)\Lambda_\mu}A_{\leq 0,\gamma_0}f^\mu_{j_1,k_1}(s)\|_{L^\infty}\lesssim 2^{-m+3\delta m}\qquad\text{ if }\quad j_1\leq (1-\delta^2)m,\\
&\sup_{|\rho|\leq 2^{m-\D}}\|e^{-i(s+\rho)\Lambda_\mu}A_{\geq 1,\gamma_0}f^\mu_{j_1,k_1}(s)\|_{L^\infty}\lesssim 2^{-m}\qquad\text{ if }\quad m/2\leq j_1\leq (1-\delta^2)m,
\end{split}
\end{equation}
as a consequence of \eqref{LinftyBd3}. Therefore, using the $L^2\times L^\infty$ estimate and \eqref{kj7.2}, as before,
\begin{equation}\label{kj7.9}
2^{(1-50\delta)(m-k/2)}2^m\big\|P_k\mathcal{I}_l[A_{\leq 0,\gamma_0}f^\mu_{j_1,k_1}(s),P_{k_2}\partial_sf^\nu(s)]\big\|_{L^2}\lesssim 2^{-3\delta^2m},
\end{equation}
if $j_1\leq (1-\delta^2)m$, and 
\begin{equation}\label{kj7.10}
2^{(1-50\delta)(m-k/2)}2^m\big\|P_k\mathcal{I}_l[A_{\geq 1,\gamma_0}f^\mu_{j_1,k_1}(s),P_{k_2}\partial_sf^\nu(s)]\big\|_{L^2}\lesssim 2^{-3\delta^2m},
\end{equation}
if $m/2\leq j_1\leq (1-\delta^2)m$. 

On the other hand, if $j_1\geq (1-\delta^2)m$ then we can use the $L^\infty$ bound $\|e^{-is\Lambda_\nu}P_{k_2}\partial_sf^\nu(s)\|_{L^\infty}\lesssim 2^{-5m/3+6\delta^2m}$ in \eqref{kj7.2}, together with the general bound \eqref{kj7.45}. As in \eqref{Alx70} we decompose $f^\mu_{j_1,k_1}=\sum_{n_1=0}^{j_1} f^\mu_{j_1,k_1,n_1}$, and record the bound $\|f^\mu_{j_1,k_1,n_1}(s)\|_{L^2}\lesssim 2^{-j_1+50\delta j_1}2^{n_1/2-49\delta n_1}2^{\delta^2m}$. Let $X:=2^{(1-50\delta)(m-\underline{k}/2)}2^m\big\|P_k\mathcal{I}_l[f^\mu_{j_1,k_1,n_1}(s),P_{k_2}\partial_sf^\nu(s)]\big\|_{L^2}$. Using Lemma \ref{PhiLocLem} it follows that
\begin{equation*}
\begin{split}
X&\lesssim 2^{(1-50\delta)(m-k/2)}2^m \big[2^k2^{-l}\|f^\mu_{j_1,k_1,n_1}(s)\|_{L^2}\sup_{|\rho|\leq 2^{-l+2\delta^2m}}\|e^{-i(s+\rho)\Lambda_\nu}P_{k_2}\partial_sf^\nu(s)\|_{L^\infty}+2^{-8m}\big]\\
&\lesssim 2^{-k/2}2^{-2m/3}2^{n_1/2-49\delta n_1}2^{4\delta m}.
\end{split}
\end{equation*}
Using only $L^2$ bounds, see \eqref{kj7.2}, and Cauchy--Schwarz we also have
\begin{equation*}
\begin{split}
X&\lesssim 2^{(1-50\delta)(m-k/2)}2^m\cdot  2^{2k}2^{-l}\|f^\mu_{j_1,k_1,n_1}(s)\|_{L^2}\|P_{k_2}\partial_sf^\nu(s)\|_{L^2}\lesssim 2^{k}2^{n_1/2-49\delta n_1}2^{6\delta m}.
\end{split}
\end{equation*}
Finally, using \eqref{FL1bd}, we have
\begin{equation*}
\begin{split}
X\lesssim 2^{(1-50\delta)(m-k/2)}2^m\cdot  2^{k}2^{-l}\|\widehat{f^\mu_{j_1,k_1,n_1}}(s)\|_{L^1}\|P_{k_2}\partial_sf^\nu(s)\|_{L^2}\lesssim 2^{-49\delta n_1}2^{7\delta m}.
\end{split}
\end{equation*}
We can combine the last three estimates (using the last one for $n_1\geq m/4$ and the first two for $n_1\leq m/4$) to conclude that if $j_1\geq (1-\delta^2)m$ then 
\begin{equation}\label{kj7.11}
2^{(1-50\delta)(m-k/2)}2^m\big\|P_k\mathcal{I}_l[f^\mu_{j_1,k_1}(s),P_{k_2}\partial_sf^\nu(s)]\big\|_{L^2}\lesssim 2^{-3\delta^2m}.
\end{equation}

In view of \eqref{kj7.9}--\eqref{kj7.11}, it remains to prove that, for $j_1\leq m/2$,
\begin{equation}\label{kj7.15}
2^{(1-50\delta)(m-k/2)}2^m\big\|P_k\mathcal{I}_l[A_{\geq 1,\gamma_0}f^\mu_{j_1,k_1}(s),P_{k_2}\partial_sf^\nu(s)]\big\|_{L^2}\lesssim 2^{-3\delta^2m}.
\end{equation}
To prove \eqref{kj7.15} we decompose $P_{k_2}\partial_sf^\nu(s)$ as in \eqref{Brc4}. The terms that are bounded in $L^2$ by $2^{-4m/3+4\delta m}$ lead to acceptable contributions, using the $L^2\times L^\infty$ argument with Lemma \ref{PhiLocLem} and \eqref{LinftyBd3.5}. It remains to consider the terms  $A_{k_{2};k_{3},j_{3},k_{4},j_{4}}^{a_3,\alpha_3;a_4,\alpha_4}(s)$ when $\max(j_3,j_4)\leq (1-\delta^2)m$ and $k_3,k_4\in [-2m/N'_0,300]$. For these terms, it suffices to prove that
\begin{equation}\label{kj7.16}
\big\|P_k\mathcal{I}_l[A_{\geq 1,\gamma_0}f^\mu_{j_1,k_1}(s), A_{k_{2};k_{3},j_{3},k_{4},j_{4}}^{a_3,\alpha_3;a_4,\alpha_4}(s)]\big\|_{L^2}\lesssim 2^{-4m}.
\end{equation}
Notice that  $A_{k_{2};k_{3},j_{3},k_{4},j_{4}}^{a_3,\alpha_3;a_4,\alpha_4}(s)$ is given by an expression similar to (\ref{Brc4.1}). Therefore
\begin{equation}\label{kj7.165}
\begin{split}
\mathcal{F}\{P_k\mathcal{I}_l[A_{\geq 1,\gamma_0}&f^\mu_{j_1,k_1}(s), A_{k_{2};k_{3},j_{3},k_{4},j_{4}}^{a_3,\alpha_3;a_4,\alpha_4}(s)]\}(\xi)=\int_{\mathbb{R}^2\times\mathbb{R}^2}e^{is\widetilde{\Phi}(\xi,\eta,\sigma)}\widehat{f^\mu_{j_1,k_1}}(\xi-\eta,s)\\
&\times \varphi_{\leq -101}(|\xi-\eta|-\gamma_0)2^{-l}\widetilde{\varphi}_{l}(\Phi_{+\mu\nu}(\xi,\eta))\varphi_k(\xi)\varphi_{k_2}(\eta)\\
&\times\mathfrak{m}_{\mu\nu}(\xi,\eta)\mathfrak{m}_{\nu\beta\gamma}(\eta,\sigma)\widehat{f^\beta_{j_3,k_3}}(\eta-\sigma,s)\widehat{f^\gamma_{j_4,k_4}}(\sigma,s)\,d\sigma d\eta,
\end{split}
\end{equation} 
where $\widetilde{\Phi}(\xi,\eta,\sigma)=\Lambda(\xi)-\Lambda_\mu(\xi-\eta)-\Lambda_{\beta}(\eta-\sigma)-\Lambda_{\gamma}(\sigma)$. The main observation is that either
\begin{equation}\label{kj7.17}
\big|\nabla_\eta\widetilde{\Phi}(\xi,\eta,\sigma)\big|=\big|\nabla\Lambda_\mu(\xi-\eta)-\nabla\Lambda_{\beta}(\eta-\sigma)\big|\gtrsim 1,
\end{equation}
or
\begin{equation}\label{kj7.18}
\big|\nabla_\sigma\widetilde{\Phi}(\xi,\eta,\sigma)\big|=\big|\nabla\Lambda_{\beta}(\eta-\sigma)-\nabla\Lambda_{\gamma}(\sigma)\big|\gtrsim 1,
\end{equation}
in the support of the integral. Indeed, $||\eta|-\gamma_0|\leq 2^{-95}$ in view of the cutoffs on the variables $\xi$ and $\xi-\eta$. If $\big|\nabla_\sigma\widetilde{\Phi}(\xi,\eta,\sigma)\big|\leq 2^{-\D}$ then $\max(|k_3|,|k_4|)\leq 300$ and, using Proposition \ref{spaceres11} (ii) (in particular \eqref{zc2}), it follows that $|\eta-\sigma|$ is close to either $\gamma_0/2$, or $p_{+-1}(\gamma_0)\geq 1.1\gamma_0$, or $p_{+-1}(\gamma_0)-\gamma_0\leq 0.9\gamma_0$. In these cases the lower bound \eqref{kj7.17} follows. The desired bound \eqref{kj7.16} then follows using Lemma \ref{tech5} (i).

{\bf{Case 3.}} If $\underline{k}=k_2$ then we do not prove the stronger estimate \eqref{kj7.1}. In this case the desired bound follows from Lemma \ref {lLargeBBound2} below.

\begin{lemma}\label{lLargeBBound2}
Assume that \eqref{ki7} holds and, in addition, 
\begin{equation}\label{case5}
j\leq m+2\D+\max(|k|,|k_1|,|k_2|)/2,\qquad k_2\leq -2\D,\qquad 2^{-l}\leq 2^{10\delta m}+2^{-k_2/2+\D}.
\end{equation}
Then, for any $j_1$,
\begin{equation}\label{kl7.1}
2^{(1-50\delta)j}\Vert Q_{jk}\mathcal{B}_{m,l}[f^\mu_{j_1,k_1},P_{k_2}\partial_sf^\nu]\Vert_{L^2}\lesssim 2^{-3\delta^2m}.
\end{equation}
\end{lemma}

\begin{proof} We record the bounds
\begin{equation}\label{kl7.2}
\begin{split}
&\|P_{k_2}\partial_sf^\nu(s)\|_{L^2}\lesssim 2^{-m+5\delta m}(2^{k_2}+2^{-m/2}),\\
&\sup_{|\rho|\leq 2^{-l+2\delta^2m}}\|e^{-i(s+\rho)\Lambda_\nu}P_{k_2}\partial_sf^\nu(s)\|_{L^\infty}\lesssim 2^{-5m/3+10\delta^2m}(2^{k_2/2+10\delta m}+1),
\end{split} 
\end{equation}
see \eqref{Brc2}, \eqref{vd7}, and \eqref{kj7.45}. We will prove that for any $s\in\mathcal{I}_m$
\begin{equation}\label{kl7.3}
2^{(1-50\delta)j}2^m\Vert Q_{jk}\mathcal{I}_{l}[f^\mu_{j_1,k_1}(s),P_{k_2}\partial_sf^\nu(s)]\Vert_{L^2}\lesssim 2^{-3\delta^2m}.
\end{equation}

{\bf{Step 1.}} We notice the identity
\begin{equation*}
\begin{split}
Q_{jk}\mathcal{I}_{l}[f^\mu_{j_1,k_1}(s),&P_{k_2}\partial_sf^\nu(s)](x)=C\phii_j^{(k)}(x)\int_{\mathbb{R}^2\times\mathbb{R}^2}e^{i\left[s\Phi(\xi,\eta)+x\cdot \xi\right]}2^{-l}\widetilde{\varphi}_l(\Phi(\xi,\eta))\\
&\times\varphi_k(\xi)m_0(\xi,\eta)\widehat{f^\mu_{j_1,k_1}}(\xi-\eta,s) \widehat{P_{k_2}\partial_sf^\nu}(\eta,s)\,d\xi d\eta.
\end{split}
\end{equation*} 
Therefore $\big\|Q_{jk}\mathcal{I}_{l}[f^\mu_{j_1,k_1}(s),P_{k_2}\partial_sf^\nu(s)]\big\|_{L^2}\lesssim 2^{-4m}$, using integration by parts in $\xi$ and Lemma \ref{tech5} (i), unless
\begin{equation}\label{kl7.4}
2^j\leq\max\big\{2^{j_1+\delta m},2^{m+\max(|k|,|k_1|)/2+\D}\big\}.
\end{equation}
On the other hand, assuming \eqref{kl7.4}, $L^2\times L^\infty$ bounds using Lemma \ref{PhiLocLem}, the bounds \eqref{kl7.2}, and Lemma \ref{LinEstLem} show that \eqref{kl7.3} holds in the following cases:
\begin{equation}\label{kl7.5}
\begin{split}
\text{ either }&k_1\leq -10\quad\text{ and }\quad j_1\leq m-\delta m,\\
\text{ or }&k_1\leq -10\quad\text{ and }\quad j_1\geq m-\delta m,\\
\text{ or }&k_1\geq 10\,\,\,\,\,\quad\text{ and }\quad j_1\leq 2m/3,\\
\text{ or }&k_1\geq 10\,\,\,\,\,\quad\text{ and }\quad j_1\geq 2m/3.
\end{split}
\end{equation}
See the similar estimates in the proof of Lemma \ref{KsmallLem} above, in particular those in (Step 2, Case 1) and (Step 2, Case 2). In each case we estimate $e^{-i(s+\rho)\Lambda_\mu}f^\mu_{j_1,k_1}(s)$ in $L^\infty$ and $e^{-i(s+\rho)\Lambda_\nu}P_{k_2}\partial_sf^\nu(s)$ in $L^2$ when $j_1$ is small, and we estimate $e^{-i(s+\rho)\Lambda_\mu}f^\mu_{j_1,k_1}(s)$ in $L^2$ and $e^{-i(s+\rho)\Lambda_\nu}P_{k_2}\partial_sf^\nu(s)$ in $L^\infty$ when $j_1$ is large. We estimate the contribution of the symbol $m_0$ by $2^{(k+k_1+k_2)/2}$ in all cases.

It remains to prove the desired bound \eqref{kl7.3} when $k,k_1\in[-20,20]$. We can still prove this when $f^\mu_{j_1,k_1}(s)$ is replaced by $A_{\leq 0,\gamma_0}f^\mu_{j_1,k_1}(s)$, or when $j_1\geq m/3-\delta m$, or when $k_2\leq-m/3+\delta m$, using $L^2\times L^\infty$ estimates as before. 

{\bf{Step 2.}} To deal with the remaining cases we use the decomposition \eqref{Brc4}. The contribution of the error component $P_{k_2}E_\nu^{a_2,\alpha_2}$ can also be estimated in the same way when $j_1\leq m/3-\delta m$. After these reductions, we may assume that 
\begin{equation}\label{kl7.9}
\begin{split}
&k,k_1\in[-20,20],\quad j_1\leq m/3-\delta m,\quad j\leq m+2\D,\quad k_2\in[-m/3+\delta m,-2\D],\\
&2^{-l}\lesssim 2^{10\delta m}+2^{-k_2/2}.
\end{split}
\end{equation}
It remains to prove that for any $[(k_3,j_3),(k_4,j_4)]\in X_{m,k_2}$
\begin{equation}\label{kl7.10}
2^{(1-50\delta)j}2^m\big\Vert Q_{jk}\mathcal{I}_{l}[A_{\geq 1,\gamma_0}f^\mu_{j_1,k_1},A_{k_2;k_3,j_3;k_4,j_4}^{a_3,\alpha_3;a_4,\alpha_4}]\big\Vert_{L^2}\lesssim 2^{-4\delta^2m}.
\end{equation}

The $L^2\times L^\infty$ argument still works to prove \eqref{kl7.10} if
\begin{equation}\label{kl7.11}
\big\|A_{k_2;k_3,j_3;k_4,j_4}^{a_3,\alpha_3;a_4,\alpha_4}(s)\big\|_{L^2}\lesssim 2^{-7m/6+25\delta m}.
\end{equation}
We notice that this bound holds if $\max(j_3,j_4)\geq m/3-\delta m$. Indeed, since $k_2\leq -2\D$, we have $$P_{k_2}I^{\nu\beta\gamma}[A_{\geq 1,\gamma_1}f_{j_3,k_3}^\beta(s),A_{\geq 1,\gamma_0}f_{j_4,k_4}^\gamma(s)]\equiv 0,$$ and the bound \eqref{kl7.11} follows by $L^2\times L^\infty$ arguments as in the proof of Lemma \ref{dtfLem1}.

Therefore we may assume that $j_3,j_4\leq m/3-\delta m$. We examine the explicit formula \eqref{kj7.165}. We claim that $|\mathcal{F}\{P_k\mathcal{I}_l[A_{\geq 1,\gamma_0}f^\mu_{j_1,k_1}(s),A^{a_3,\alpha_3;a_4,\alpha_4}_{k_2;k_3,j_3;k_4,j_4}(s)]\}(\xi)|\lesssim 2^{-10m}$ if $|k_3|\geq \D/10$. Indeed, in this case the $\eta$ derivative of the phase $\widetilde{\Phi}$ is $\gtrsim 2^{|k_3|/2}$ in the support of the integral (recall that $|k_1|\leq 20$). Integration by parts in $\eta$, using Lemma \ref{tech5} (i), shows that the resulting integral is negligible, as desired. 

In view of Lemma \ref{dtfLem1} (ii) (3), it remains to prove \eqref{kl7.10} when, in addition to \eqref{kl7.9},
\begin{equation}\label{kl7.13}
k_3,k_4\in[-10,10],\qquad j_3,j_4\leq m/3-\delta m,\qquad\beta=-\gamma.
\end{equation}
We examine again the formula \eqref{kj7.165} and notice that the $(\eta,\sigma)$ derivative of the phase $\widetilde{\Phi}$ is $\gtrsim 1$ unless $||\eta-\sigma|-\gamma_0|\leq 2^{-98}$ and $||\sigma|-\gamma_0|\leq 2^{-98}$. Therefore we may replace $f^\beta_{j_3,k_3}$ with $A_{\geq -5,\gamma_0}f^\beta_{j_3,k_3}$ and $f^\gamma_{j_4,k_4}$ with $A_{\geq -5,\gamma_0}f^\gamma_{j_4,k_4}$, at the expense of negligible errors. Finally, we may assume that $l\geq -\D$ if $\mu=-$, and we may assume that $j\leq m+k_2+\D$ if $\mu=+$ (otherwise the approximate finite speed of propagation argument used in the proof of \eqref{Alx24} and Lemma \ref{FSPLem}, which relies on integration by parts in $\xi$, gives rapid decay). Therefore, in proving \eqref{kl7.10} we may assume that
\begin{equation}\label{kl7.14}
2^{-l}2^{(1-50\delta)j}\lesssim 2^{m-50\delta m}(1+2^{k_2/2+10\delta m}).
\end{equation}

Let $\kappa_r:=2^{\delta^2m}2^{k_2/2-m/2}$. We observe now that if $||\eta-\sigma|-\gamma_0|+||\sigma|-\gamma_0|\leq 2^{-90}$ and $|\Xi_{\beta\gamma}(\eta,\sigma)|=|(\nabla_\sigma\Phi_{\nu\beta\gamma})(\eta,\sigma)|\leq 2\kappa_r$ then
\begin{equation}\label{kl7.15}
||\sigma|-\gamma_0|\geq 2^{k_2-10},\qquad ||\eta-\sigma|-\gamma_0|\geq 2^{k_2-10}.
\end{equation}
Indeed, we may assume that $\sigma=(\sigma_1,0)$, $\eta=(\eta_1,\eta_2)$, $|\sigma_1-\gamma_0|\leq 2^{-90}$, $|\eta|\in [2^{k_2-2},2^{k_2+2}]$. Recalling that $\beta=-\gamma$ and using the formula \eqref{try2}, the condition $|\Xi_{\beta\gamma}(\eta,\sigma)|\leq 2\kappa_r$ gives
\begin{equation*}
\Big|\lambda'(\sigma_1)-\frac{\sigma_1-\eta_1}{|\sigma-\eta|}\lambda'(|\sigma-\eta|)\Big|\leq 2\kappa_r, \qquad \frac{|\eta_2|}{|\sigma-\eta|}\lambda'(|\sigma-\eta|)\leq 2\kappa_r.
\end{equation*}
Since $k_2\in[-m/3+\delta m,-2\D]$ and $\kappa_r=2^{\delta^2m+k_2/2-m/2}$ it follows that $|\eta_2|\leq \kappa_r2^{\D}\leq 2^{k_2-\D}$, $|\eta_1|\in [2^{k_2-3},2^{k_2+3}]$, and $|\lambda'(\sigma_1)-\lambda'(\sigma_1-\eta_1)|\leq 4\kappa_r$. On the other hand, if $|\sigma_1-\gamma_0|\leq 2^{k_2-10}$ and $|\eta_1|\in [2^{k_2-3},2^{k_2+3}]$ then $|\lambda'(\sigma_1)-\lambda'(\sigma_1-\eta_1)|\gtrsim 2^{2k_2}$ (since $\lambda''(\gamma_0)=0$ and $\lambda'''(\gamma_0)\approx 1$), which gives a contradiction. The claims in \eqref{kl7.15} follow.

We examine now the formula \eqref{kj7.165} and recall \eqref{kl7.13} and \eqref{kl7.15}. Using Lemma \ref{tech5} (i) and integration by parts in $\sigma$, we notice that we may insert the factor $\varphi(\kappa_r^{-1}\Xi_{\beta\gamma}(\eta,\sigma))$, at the expense of a negligible error. It remains to prove that
\begin{equation}\label{kl7.20}
2^{(1-50\delta)j}2^m\Vert H\Vert_{L^2}\lesssim 2^{-4\delta^2m},
\end{equation}
where, with $g_1:=A_{\geq 1,\gamma_0}f^\mu_{j_1,k_1}(s)$, $g_3:=A_{[-20,20-k_2],\gamma_0}f^\beta_{j_3,k_3}(s)$, $g_4:=A_{[-20,20-k_2],\gamma_0}f^\gamma_{j_4,k_4}(s)$,
\begin{equation*}
\begin{split}
\widehat{H}(\xi):=\varphi_k(\xi)\int_{\mathbb{R}^2}e^{is[\Lambda(\xi)-\Lambda_\mu(\xi-\eta)-\Lambda_\nu(\eta)]}\widehat{g_1}(\xi-\eta)2^{-l}\widetilde{\varphi}_{l}(\Phi_{+\mu\nu}(\xi,\eta))\mathfrak{m}_{\mu\nu}(\xi,\eta) \widehat{G_2}(\eta)\,d\eta,\\
\widehat{G_2}(\eta):=\varphi_{k_2}(\eta)\int_{\mathbb{R}^2}e^{is[\Lambda_\nu(\eta)-\Lambda_\beta(\eta-\sigma)-\Lambda_\gamma(\sigma)]}\mathfrak{m}_{\nu\beta\gamma}(\eta,\sigma)\varphi(\kappa_r^{-1}\Xi_{\beta\gamma}(\eta,\sigma))\widehat{g_3}(\eta-\sigma)\widehat{g_4}(\sigma)\,d\sigma.
\end{split}
\end{equation*}

We use now the more precise bound \eqref{LinftyBd2.5} to see that
\begin{equation*}
\big\|e^{-is\Lambda_\beta}g_3\big\|_{L^\infty}+\big\|e^{-is\Lambda_\gamma}g_4\big\|_{L^\infty}\lesssim 2^{-m+4\delta^2m}2^{-k_2/2}.
\end{equation*}
This bound is the main reason for proving \eqref{kl7.15}. After removing the factor $\varphi(\kappa_r^{-1}\Xi_{\beta\gamma}(\eta,\sigma))$ at the expense of a small error, and using also \eqref{mk6} and \eqref{kj7.45}, it follows that
\begin{equation*}
\big\|e^{-i(s+\rho)\Lambda_\nu}G_2\big\|_{L^\infty}\lesssim (1+|\rho|2^{k_2/2})2^{k_2}\cdot 2^{-2m+8\delta^2m}2^{-k_2}\lesssim (1+|\rho|2^{k_2/2})2^{-2m+8\delta^2m},
\end{equation*}
for any $\rho\in\mathbb{R}$. We use now the $L^2\times L^\infty$ argument, together with Lemma \ref{PhiLocLem}, to estimate
\begin{equation*}
\|H\|_{L^2}\lesssim 2^{k_2/2}2^{-l}\cdot (1+2^{-l}2^{k_2/2})2^{-2m+12\delta^2m}\lesssim 2^{-2m+12\delta^2m}2^{k_2/2}2^{-l}(1+2^{10\delta m+k_2/2}).
\end{equation*}
The desired bound \eqref{kl7.20} follows using also \eqref{kl7.14}.
\end{proof}  

\subsection{The case of strongly resonant interactions, I}\label{ProofZ2} In this subsection we prove Lemma \ref{StronglyResLem}. This is where we need the localization operators $A^{(j)}_{n,\gamma_1}$ to control the output. It is an instantaneous estimate, in the sense that the time evolution will play no role. Hence, it suffices to show the following: let $\chi\in C^\infty_c(\mathbb{R}^2)$ be supported in $[-1,1]$ and assume that $j, l, s, m$ satisfy
\begin{equation}\label{Ass3}
-m+\delta m/2\le l\le 10 m/N'_0,\qquad 2^{m-4}\le s\le 2^{m+4}.
\end{equation}
Assume that
\begin{equation}\label{Ass3.1}
\Vert f\Vert_{H^{N'_0}\cap H^{N'_1}_\Omega\cap Z_1}+ \Vert g\Vert_{H^{N'_0}\cap H^{N'_1}_\Omega\cap Z_1}\le 1,
\end{equation}
and define, with $\chi_l(x)=\chi(2^{-l}x)$,
\begin{equation*}
\begin{split}
\widehat{I[f,g]}(\xi)&:=\int_{\mathbb{R}^2}e^{is\Phi(\xi,\eta)}\chi_l(\Phi(\xi,\eta))m_0(\xi,\eta)\widehat{f}(\xi-\eta)\widehat{g}(\eta)d\eta.
\end{split}
\end{equation*}
Assume also that $k,k_1,k_2,j,m$ satisfy \eqref{ki7} and \eqref{case2}. Then
\begin{equation}\label{NRCCLNewBd}
2^{\delta m/2}2^{-l}\Vert Q_{jk}I[P_{k_1}f,P_{k_2}g]\Vert_{B_j} \lesssim 2^{-5\delta^2m}.
\end{equation}

To prove \eqref{NRCCLNewBd} we define $f_{j_1,k_1}, g_{j_2,k_2}, f_{j_1,k_1,n_1}, g_{j_2,k_2,n_2}$ as in \eqref{Alx100}, $(k_1,j_1),(k_2,j_2)\in\mathcal{J}$, $n_1\in[0,j_1+1]$, $n_2\in[0,j_2+1]$. We will analyze several cases depending on the relative sizes of the main parameters $m,l,k,j,k_1,j_1,k_2,j_2$. In many cases we will prove the stronger bound
\begin{equation}\label{SuffNRCCLNewBd}
2^{\delta m/2}2^{-l}2^{(1-50\delta)j}\Vert Q_{jk}I[f_{j_1,k_1},g_{j_2,k_2}]\Vert_{L^2}\lesssim 2^{-6\delta^2m}.
\end{equation}
However, in the main case \eqref{BulkCase}, we can only prove the weaker bound
\begin{equation}\label{SuffNRCCLNewBdPrev}
2^{\delta m/2}2^{-l}\Vert Q_{jk}I[f_{j_1,k_1},g_{j_2,k_2}]\Vert_{B_{j}}\lesssim 2^{-6\delta^2m}.
\end{equation}
These bounds clearly suffice to prove \eqref{NRCCLNewBd}.

{\bf Case 1:} We prove first the bound \eqref{SuffNRCCLNewBdPrev} under the assumption
\begin{equation}\label{BulkCase}
\begin{split}
&\max(j_1,j_2)\le 9m/10,\qquad 2l\leq \min (k,k_1,k_2,0)-\D.
\end{split}
\end{equation}
We may assume $j_1\leq j_2$. With 
\begin{equation*}
\kappa_\theta:=2^{-m/2+\delta^2m},\qquad \kappa_r:=2^{\delta^2m}\big(2^{-m/2+3\max(|k|,|k_1|,|k_2|)/4}+2^{j_2-m}\big)
\end{equation*}
we decompose
\begin{equation*}
\begin{split}
&\mathcal{F}I[f_{j_1,k_1},g_{j_2,k_2}]=\mathcal{R}_{1}+\mathcal{R}_2+\mathcal{NR},\\
&\mathcal{R}_{1}(\xi):=\int_{\mathbb{R}^2}e^{is\Phi(\xi,\eta)}\chi_{l}(\Phi(\xi,\eta))m_0(\xi,\eta)\varphi(\kappa_r^{-1}\Xi(\xi,\eta))\varphi(\kappa_\theta^{-1}\Theta(\xi,\eta))\widehat{f_{j_1,k_1}}(\xi-\eta)\widehat{g_{j_2,k_2}}(\eta)d\eta,\\
&\mathcal{R}_{2}(\xi):=\int_{\mathbb{R}^2}e^{is\Phi(\xi,\eta)}\chi_{l}(\Phi(\xi,\eta))m_0(\xi,\eta)\varphi(\kappa_r^{-1}\Xi(\xi,\eta))\varphi_{\geq 1}(\kappa_\theta^{-1}\Theta(\xi,\eta))\widehat{f_{j_1,k_1}}(\xi-\eta)\widehat{g_{j_2,k_2}}(\eta)d\eta,\\
&\mathcal{NR}(\xi):=\int_{\mathbb{R}^2}e^{is\Phi(\xi,\eta)}\chi_{l}(\Phi(\xi,\eta))m_0(\xi,\eta)\varphi_{\geq 1}(\kappa_r^{-1}\Xi(\xi,\eta))\widehat{f_{j_1,k_1}}(\xi-\eta)\widehat{g_{j_2,k_2}}(\eta)d\eta.
\end{split}
\end{equation*}

With $\psi_1:=\varphi_{\leq (1-\delta/4)m}$ and $\psi_2:=\varphi_{>(1-\delta/4)m}$, we rewrite 
\begin{equation*}
\begin{split}
&\mathcal{NR}(\xi)=C2^l[\mathcal{NR}_1(\xi)+\mathcal{NR}_2(\xi)],\\
&\mathcal{NR}_i(\xi):=\int_{\mathbb{R}}\int_{\mathbb{R}^2}e^{i(s+\lambda)\Phi(\xi,\eta)}\widehat{\chi}(2^l\lambda)\psi_i(\lambda)m_0(\xi,\eta)\varphi_{\geq 1}(\kappa_r^{-1}\Xi(\xi,\eta))\widehat{f_{j_1,k_1}}(\xi-\eta)\widehat{g_{j_2,k_2}}(\eta)\,d\eta d\lambda.
\end{split}
\end{equation*}
Since $\widehat{\chi}$ is rapidly decreasing we have $\|\varphi_k\cdot\mathcal{NR}_2\|_{L^\infty}\lesssim 2^{-4m}$, which gives an acceptable contribution. On the other hand, in the support of the integral defining $\mathcal{NR}_1$, we have that $\vert s+\lambda\vert\approx 2^m$ and integration by parts in $\eta$ (using Lemma \ref{tech5} (i)) gives $\|\varphi_k\cdot\mathcal{NR}_1\|_{L^\infty}\lesssim 2^{-4m}$.

The contribution of $\mathcal{R}=\mathcal{R}_1+\mathcal{R}_2$ is only present if we have a space-time resonance. In particular, in view of Proposition \ref{spaceres11} (iii) (notice that the assumption \eqref{try1.1} is satisfied due to \eqref{BulkCase}) we may assume that
\begin{equation}\label{Alx74.6}
-10\le k,k_1,k_2\le 10,\quad\pm(\sigma,\mu,\nu)=(+,+,+),\quad \big||\xi|-\gamma_1\big|+|\eta-\xi/2|\leq 2^{-\D}.
\end{equation}
Notice that, if $\mathcal{R}(\xi)\neq 0$ then
\begin{equation}\label{EstimPsi}
\begin{split}
\big||\xi|-\gamma_1\big|\lesssim \vert \Phi(\xi,\xi/2)\vert\lesssim\vert \Phi(\xi,\eta)\vert+\vert \Phi(\xi,\eta)-\Phi(\xi, \xi/2)\vert\lesssim 2^l+\kappa_r^2.
\end{split}
\end{equation}
Integration by parts using Lemma \ref{RotIBP} shows that $\|\varphi_k\cdot\mathcal{R}_2\|_{L^\infty}\lesssim 2^{-5m/2}$, which gives an acceptable contribution. To bound the contribution of $\mathcal{R}_1$ we will show that
\begin{equation}\label{Alx81}
2^{\delta m/2}2^{-l}\sup_{|\xi|\approx 1} \big\vert \big(1+2^{m}\big||\xi|-\gamma_1\big|\big)\mathcal{R}_1(\xi)\big\vert\lesssim 2^{9\delta m/10},
\end{equation}
which is stronger than the bound we need in \eqref{SuffNRCCLNewBdPrev}. Indeed for $j$ fixed we estimate
\begin{equation}\label{Alx67.5}
\begin{split}
\sup_{0\leq n\leq j}&2^{(1-50\delta)j}2^{-n/2+49\delta n}\big\|A_{n,\gamma_1}^{(j)}Q_{jk}\mathcal{F}^{-1}\mathcal{R}_1\big\|_{L^2}\\
&\lesssim \sup_{0\leq n\leq j}2^{(1-50\delta)j}2^{-n/2+49\delta n}\big\|\varphi_{-n}^{[-j,0]}(2^{100}||\xi|-\gamma_1|)\mathcal{R}_1(\xi)\big\|_{L^2_\xi}\\
&\lesssim \sum_{n\geq 0}2^{(1-50\delta)j}2^{-n/2-(1/2-49\delta)\min(n,j)}\big\|\varphi_{-n}^{(-\infty,0]}(2^{100}||\xi|-\gamma_1|)\mathcal{R}_1(\xi)\big\|_{L^\infty_\xi},
\end{split}
\end{equation} 
and notice that \eqref{SuffNRCCLNewBdPrev} would follow from \eqref{Alx81} and the assumption $j\leq m+3\D$.

Recall from Lemma \ref{LinEstLem} and \eqref{Alx74.6} (we may assume $f_{j_1,k_1}=f_{j_1,k_1,0}$, $g_{j_2,k_2}=g_{j_2,k_2,0}$) that
\begin{equation}\label{NewBdinput}
\begin{split}
2^{(1/2-\delta')j_1}\Vert \widehat{f_{j_1,k_1}}\Vert_{L^\infty}+2^{(1-\delta')j_1}\sup_{\theta\in\mathbb{S}^1}\Vert \widehat{f_{j_1,k_1}}(r\theta)\Vert_{L^2(rdr)}&\lesssim 1,\\
2^{(1/2-\delta')j_2}\Vert \widehat{g_{j_2,k_2}}\Vert_{L^\infty}+2^{(1-\delta')j_2}\sup_{\theta\in\mathbb{S}^1}\Vert \widehat{g_{j_2,k_2}}(r\theta)\Vert_{L^2(rdr)}&\lesssim 1.
\end{split}
\end{equation}
We ignore first the factor $\chi_{l}(\Phi(\xi,\eta))$. In view of Proposition \ref{spaceres11} (ii) the $\eta$ integration in the definition of $\mathcal{R}_1(\xi)$ takes place essentially over a $\kappa_\theta\times\kappa_r$ box in the neighborhood of $\xi/2$. Using \eqref{EstimPsi} and \eqref{NewBdinput}, and estimating $\Vert \widehat{f_{j_1,k_1}}\Vert_{L^\infty}\lesssim 1$, we have, if $j_2\geq m/2$,
\begin{equation*}
\vert (1+2^{m}||\xi|-\gamma_1|)\mathcal{R}_1(\xi)\vert\lesssim 2^{m}(2^l+\kappa_r^2)2^{-j_2+\delta' j_2}\kappa_\theta\kappa_r^{1/2}\lesssim (2^l+\kappa_r^2)2^{-j_2(1/2-\delta')}2^{2\delta^2m}.
\end{equation*}
On the other hand, if $j_2\leq m/2$ we estimate $\Vert \widehat{f_{j_1,k_1}}\Vert_{L^\infty}+\Vert \widehat{f_{j_2,k_2}}\Vert_{L^\infty}\lesssim 1$ and conclude that
\begin{equation*}
\vert (1+2^{m}||\xi|-\gamma_1|\mathcal{R}_1(\xi)\vert\lesssim 2^{m+l}\kappa_{\theta}\kappa_r\lesssim 2^l2^{2\delta^2m}.
\end{equation*}
The desired bound \eqref{Alx81} follows if $\kappa_r^22^{-l}\leq 2^{j_2/4}$.

Assume now that $\kappa_r^2\ge 2^l2^{j_2/4}$ (in particular $j_2\geq 11m/20$). In this case the restriction $|\Phi(\xi,\eta)|\leq 2^{l}$ is stronger and we have to use it. We decompose, with $p_-:=\lfloor\log_2(2^{l/2}\kappa_r^{-1})+\D\rfloor$,
\begin{equation*}
\begin{split}
\mathcal{R}_1(\xi)&=\sum_{p\in[p_-,0]}\mathcal{R}^p_1(\xi),\\
\mathcal{R}_1^p(\xi)&:=\int_{\mathbb{R}^2}e^{is\Phi(\xi,\eta)}\chi_{l}(\Phi(\xi,\eta))m_0(\xi,\eta)\varphi_p^{[p_-,1]}(\kappa_r^{-1}\Xi(\xi,\eta))\varphi(\kappa_\theta^{-1}\Theta(\xi,\eta))\widehat{f_{j_1,k_1}}(\xi-\eta)\widehat{g_{j_2,k_2}}(\eta)d\eta.
\end{split}
\end{equation*}
As in \eqref{EstimPsi}, notice that if $\mathcal{R}_1^p(\xi)\neq 0$ then $||\xi|-\gamma_1|\lesssim 2^{2p}\kappa_r^2$. The term $\mathcal{R}_1^{p_-}(\xi)$ can be bounded as before. Moreover, using the formula \eqref{cas7.1}, it is easy to see that if $\xi=(s,0)$ is fixed then the set of points $\eta$ that satisfy the three restrictions $|\Phi(\xi,\eta)|\lesssim 2^l$, $|\nabla_\eta\Phi(\xi,\eta)|\approx 2^p\kappa_r$, $|\xi\cdot\eta^\perp|\lesssim \kappa_\theta$ is essentially contained in a union of two $\kappa_\theta\times 2^l2^{-p}\kappa_r^{-1}$ boxes. Using \eqref{NewBdinput}, and estimating $\Vert \widehat{f_{j_1,k_1}}\Vert_{L^\infty}\lesssim 1$, we have
\begin{equation*}
\vert (1+2^{m}||\xi|-\gamma_1|)\mathcal{R}^p_1(\xi)\vert\lesssim 2^{m+2p}\kappa_r^22^{-j_2+\delta'j_2}\kappa_\theta(2^l2^{-p}\kappa_r^{-1})^{1/2}\lesssim 2^{3p/2}2^{-m+4\delta^2m}2^{l/2}2^{j_2/2+\delta' j_2}.
\end{equation*}
This suffices to prove \eqref{Alx81} since $2^p\leq 1$, $2^{-l/2}\leq 2^{m/2}$, and $2^{j_2}\leq 2^{9m/10}$, see \eqref{BulkCase}.

{\bf{Case 2.}} We assume now that
\begin{equation}\label{dr1}
2l\geq \min(k,k_1,k_2,0)-\D.
\end{equation}
In this case we prove the stronger bound \eqref{SuffNRCCLNewBd}. We can still use the standard $L^2\times L^\infty$ argument, with Lemma \ref{PhiLocLem} and Lemma \ref{LinEstLem}, to bound the contributions away from $\gamma_0$. For \eqref{SuffNRCCLNewBd} it remains to prove that
\begin{equation}\label{dr1.1}
2^{-l}2^{(1-50\delta)(m+|k|/2)}\Vert P_kI[A_{\geq 1,\gamma_0}f_{j_1,k_1},A_{\geq 1,\gamma_0}g_{j_2,k_2}]\Vert_{L^2}\lesssim 2^{-\delta m}.
\end{equation}

The bound \eqref{dr1.1} follows if $\max(j_1,j_2)\geq m/3$, using the same $L^2\times L^\infty$ argument. On the other hand, if $j_1,j_2\leq m/3$ then we use \eqref{FLinftybd} and the more precise bound \eqref{LinftyBd2.5} to see that
\begin{equation*}
\|A_{p,\gamma_0}h\|_{L^2}\lesssim 2^{-p/2},\qquad \|e^{-it\Lambda}A_{p,\gamma_0}h\|_{L^\infty}\lesssim 2^{-m+2\delta^2m}\min\big(2^{p/2},2^{m/2-p}\big),
\end{equation*}
where $h\in \{f_{j_1,k_1},g_{j_2,k_2}\}$, $p\geq 1$, and $t\approx 2^m$. Therefore, using Lemma \ref{PhiLocLem},
\begin{equation*}
\Vert P_kI[A_{p_1,\gamma_0}f_{j_1,k_1},A_{p_2,\gamma_0}g_{j_2,k_2}]\Vert_{L^2}\lesssim 2^k2^{-\max(p_1,p_2)/2}\cdot 2^{-m+2\delta^2m}2^{\min(p_1,p_2)/2}.
\end{equation*}
The desired bound \eqref{dr1.1} follows, using also the simple estimate 
\begin{equation*}
\Vert P_kI[A_{p_1,\gamma_0}f_{j_1,k_1},A_{p_2,\gamma_0}g_{j_2,k_2}]\Vert_{L^2}\lesssim 2^k2^{-(p_1+p_2)/2}.
\end{equation*}

{\bf Case 3.} Assume now that
\begin{equation*}
\max(j_1,j_2)\geq 9m/10,\qquad j\leq \min(j_1,j_2)+m/4,\qquad 2l\leq \min(k,k_1,k_2,0)-\D.
\end{equation*}
Using Lemma \ref{Shur2Lem} and \eqref{RadL2} we estimate
\begin{equation}\label{Alx74.2}
\begin{split}
&\Vert P_kI[f_{j_1,k_1,n_1},g_{j_2,k_2,n_2}]\Vert_{L^2}\\
&\lesssim 2^{k/2}2^{30\delta m}2^{l/2-n_1/2-n_2/2}\big\Vert\sup_{\theta\in\mathbb{S}^1}|\widehat{f_{j_1,k_1,n_1}}(r\theta)|\big\Vert_{L^2(rdr)}\big\Vert\sup_{\theta\in\mathbb{S}^1}|\widehat{g_{j_2,k_2,n_2}}(r\theta)|\big\Vert_{L^2(rdr)}\\
&\lesssim 2^{k/2}2^{l/2}2^{-j_1+\delta'j_1}2^{-j_2+\delta'j_2}2^{30\delta m},
\end{split}
\end{equation}
and the desired bound \eqref{SuffNRCCLNewBd} follows. 

{\bf Case 4.} Finally, assume that
\begin{equation}\label{Alx74.3}
j_2\ge 9m/10,\qquad j\geq j_1+m/4,\,\,\qquad 2l\leq \min(k,k_1,k_2,0)-\D.
\end{equation}
In particular, $j_1\leq 7m/8$. We decompose, with $\kappa_\theta=2^{-2m/5}$,
\begin{equation}\label{Alx74.35}
\begin{split}
&I[f_{j_1,k_1},g_{j_2,k_2}]=I_{||}[f_{j_1,k_1},g_{j_2,k_2}]+I_{\perp}[f_{j_1,k_1},g_{j_2,k_2}],\qquad\\
&\widehat{I_{||}[f,g]}(\xi)=\int_{\mathbb{R}^2}e^{is\Phi(\xi,\eta)}\chi_{l}(\Phi(\xi,\eta))\varphi(\kappa_\theta^{-1}\Omega_\eta\Phi(\xi,\eta))\widehat{f}(\xi-\eta)\widehat{g}(\eta)d\eta,\\
&\widehat{I_{\perp}[f,g]}(\xi)=\int_{\mathbb{R}^2}e^{is\Phi(\xi,\eta)}\chi_{l}(\Phi(\xi,\eta))(1-\varphi(\kappa_\theta^{-1}\Omega_\eta\Phi(\xi,\eta)))\widehat{f}(\xi-\eta)\widehat{g}(\eta)d\eta.
\end{split}
\end{equation}
Integration by parts using Lemma \ref{RotIBP} shows that $\big\|\mathcal{F}P_kI_\perp[f_{j_1,k_1},g_{j_2,k_2}]\big\|_{L^\infty}\lesssim 2^{-5m/2}$. In addition, using Schur's test and Proposition \ref{volume} (i), (iii),
\begin{equation*}
\begin{split}
\Vert P_kI_{||}[f_{j_1,k_1},g_{j_2,k_2,n_2}]\Vert_{L^2}&\lesssim 2^{80\delta m}2^{l}\kappa_\theta^{1/2}\Vert\widehat{f_{j_1,k_1}}\Vert_{L^\infty}\Vert \widehat{g_{j_2,k_2,n_2}}\Vert_{L^2}\lesssim 2^{95\delta m}2^{l-m/5}2^{-(1-50\delta)j_2}2^{n_2/2},
\end{split}
\end{equation*}
which gives an acceptable contribution if $n_2\leq \D$. 

It remains to estimate the contribution of $I_{||}[f_{j_1,k_1},g_{j_2,k_2,n_2}]$ for $n_2\geq\D$. Since $|\eta|$ is close to $\gamma_1$ and $|\Phi(\xi,\eta)|$ is sufficiently small (see \eqref{Alx74.3}), it follows from \eqref{try5.5} that $\min(k,k_1,k_2)\geq -40$; moreover, the vectors $\xi$ and $\eta$ are almost aligned and $|\Phi(\xi,\eta)|$ is small, so we may also assume that $\max(k,k_1,k_2)\leq 100$. Moreover, $\vert\nabla_\eta\Phi(\xi,\eta)\vert\gtrsim 1$ in the support of integration of $I_{||}[f_{j_1,k_1},g_{j_2,k_2,n_2}]$, in view of Proposition \ref{spaceres11} (iii). Integration by parts in $\eta$ using Lemma \ref{tech5} (i) then gives an acceptable contribution unless $j_2\ge(1-\delta^2)m$. We may also reset $\kappa_\theta=2^{\delta^2m-m/2}$, up to small errors, using Lemma \ref{RotIBP}. 

To summarize, we may assume that
\begin{equation}\label{Alx74.4}
j_2\ge(1-\delta^2)m,\quad j\geq j_1+m/4,\quad k,k_1,k_2\in[-100,100],\quad n_2\geq \D,\quad \kappa_\theta=2^{\delta^2m-m/2}.
\end{equation}
We decompose, with $p_-:=\lfloor l/2\rfloor$,
\begin{equation*}
\begin{split}
I_{||}[f_{j_1,k_1},&g_{j_2,k_2,n_2}]=\sum_{p_-\leq p\le \D}I^{p}_{||}[f_{j_1,k_1},g_{j_2,k_2,n_2}],\\
\widehat{I^{p}_{||}[f,g]}(\xi)&:=\int_{\mathbb{R}^2}e^{is\Phi(\xi,\eta)}\chi_{l}(\Phi(\xi,\eta))\varphi(\kappa_\theta^{-1}\Theta(\xi,\eta))\varphi_{p}^{[p_-,\D]}(\nabla_\xi\Phi(\xi,\eta))\widehat{f}(\xi-\eta)\widehat{g}(\eta)d\eta.
\end{split}
\end{equation*}
It suffices to prove that, for any $p$,
\begin{equation}\label{Alx74.45}
2^{-l}2^{(1-50\delta)j}\big\Vert Q_{jk}I^p_{||}[f_{j_1,k_1},g_{j_2,k_2,n_2}]\big\Vert_{L^2}\lesssim 2^{-\delta m}.
\end{equation}

As a consequence of Proposition \ref{volume} (iii), under our assumptions in \eqref{Alx74.4} and recalling that $\vert\nabla_\eta\Phi(\xi,\eta)\vert\gtrsim 1$ in the support of the integral,
\begin{equation*}
\sup_\xi\int_{\mathbb{R}^2}|\chi_{l}(\Phi(\xi,\eta))|\varphi(\kappa_\theta^{-1}\Theta(\xi,\eta))\varphi_{\leq -\D/2}(|\eta|-\gamma_1)\mathbf{1}_{\D_{k,k_1,k_2}}(\xi,\eta)d\eta\lesssim 2^{\delta^2m}2^{l}\kappa_\theta,
\end{equation*}
and, for any $p\geq p_-$,
\begin{equation*}
\sup_\eta\int_{\mathbb{R}^2}|\chi_{l}(\Phi(\xi,\eta))|\varphi(\kappa_\theta^{-1}\Theta(\xi,\eta))\varphi_{p}(\nabla_\xi\Phi(\xi,\eta))\varphi_{\leq -\D/2}(|\eta|-\gamma_1)\mathbf{1}_{\D_{k,k_1,k_2}}(\xi,\eta)d\xi\lesssim 2^{\delta^2m}2^{l-p}\kappa_\theta.
\end{equation*}
Using Schur's test we can then estimate, for $p\geq p_-$
\begin{equation*}
\begin{split}
\Vert P_kI^{p}_{||}[f_{j_1,k_1},g_{j_2,k_2,n_2}]\Vert_{L^2}&\lesssim 2^{-p/2}2^l2^{-m/2+4\delta^2m}\Vert \widehat{f_{j_1,k_1}}\Vert_{L^\infty}\Vert g_{j_2,k_2,n_2}\Vert_{L^2}\lesssim 2^{-p/2}2^l2^{-m+5\delta m}.
\end{split}
\end{equation*}

The desired bound \eqref{Alx74.4} follows if $j\leq m+p+4\delta m$. On the other hand, if $j\geq m+p+4\delta m$ then we use the approximate finite speed of propagation argument to show that
\begin{equation}\label{Alx74.9}
\Vert Q_{jk}I^{p}_{||}[f_{j_1,k_1},g_{j_2,k_2,n_2}]\Vert_{L^2}\lesssim 2^{-3m}.
\end{equation}
Indeed, we write, as in Lemma \ref{PhiLocLem}, $\chi_{l}(\Phi(\xi,\eta))=c2^l\int_{\mathbb{R}}\widehat{\chi}(2^l\rho) e^{i\rho\Phi(\xi,\eta)}\,d\rho$ and notice that $\big|\nabla_\xi[x\cdot\xi+(s+\rho)\Phi(\xi,\eta)]\big|\approx 2^j$ in the support of the integral, provided that $|x|\approx 2^j$ and $|\rho|\leq 2^m$. Then we recall that $j\geq j_1+m/4$, see \eqref{Alx74.4}, and use Lemma \ref{tech5} (i) to prove \eqref{Alx74.9}. This completes the proof of Lemma \ref{StronglyResLem}.

\subsection{The case of weakly resonant interactions} In this subsection we prove Lemma \ref{lLargeBBound}. We decompose $P_{k_{2}}\partial_{s}f^{\nu}$ as in \eqref{Brc4} and notice that the contribution of the error term can be estimated using the $L^2\times L^\infty$ argument as before.

To estimate the contributions of the terms $A_{k_{2};k_{3},j_{3};k_{4},j_{4}}^{a_3,\alpha_3;a_4,\alpha_4}$ we need more careful analysis of trilinear operators. With $\widetilde{\Phi}(\xi,\eta,\sigma)=\Lambda(\xi)-\Lambda_\mu(\xi-\eta)-\Lambda_{\beta}(\eta-\sigma)-\Lambda_{\gamma}(\sigma)$ and $p\in\mathbb{Z}$ we define the trilinear operators $\mathcal{J}_{l,p}$ by
\begin{equation}\label{hu1}
\begin{split}
\widehat{\mathcal{J}_{l,p}[f,g,h]}(\xi,s):=\int_{\mathbb{R}^2\times\mathbb{R}^2}e^{is\widetilde{\Phi}(\xi,\eta,\sigma)}&\widehat{f}(\xi-\eta)2^{-l}\widetilde{\varphi}_{l}(\Phi_{+\mu\nu}(\xi,\eta))\varphi_p(\widetilde{\Phi}(\xi,\eta,\sigma))\\
&\times\varphi_{k_2}(\eta)\mathfrak{m}_{\mu\nu}(\xi,\eta)\mathfrak{m}_{\nu\beta\gamma}(\eta,\sigma)\widehat{g}(\eta-\sigma)\widehat{h}(\sigma)\,d\sigma d\eta.
\end{split} 
\end{equation}
Let $\mathcal{J}_{l,\leq p}=\sum_{q\leq p}\mathcal{J}_{l,q}$ and $\mathcal{J}_{l}=\sum_{q\in\mathbb{Z}}\mathcal{J}_{l,q}$. Let
\begin{equation}\label{hu1.5}
\mathcal{C}_{l,p}[f,g,h]:=\int_{\mathbb{R}}q_m(s)\mathcal{J}_{l,p}[f,g,h](s)\,ds,\qquad \mathcal{C}_{l,\leq p}:=\sum_{q\leq p}\mathcal{C}_{l,q},\qquad \mathcal{C}_{l}=\sum_{q\in\mathbb{Z}}\mathcal{C}_{l,q}.
\end{equation}

Notice that
\begin{equation}\label{hu2}
\mathcal{B}_{m,l}[f^\mu_{j_1,k_1},A^{a_3,\alpha_3;a_4,\alpha_4}_{k_2;k_3,j_3;k_4,j_4}]=\mathcal{C}_l[f^\mu_{j_1,k_1},f^\beta_{j_3,k_3},f^\gamma_{j_4,k_4}].
\end{equation}
To prove the lemma it suffices to show that
\begin{equation}\label{hu3}
2^{(1-50\delta)j}\big\| Q_{jk}\mathcal{C}_l[f^\mu_{j_1,k_1},f^\beta_{j_3,k_3},f^\gamma_{j_4,k_4}]\big\|_{L^2}\lesssim 2^{-3\delta^2m}
\end{equation}
provided that 
\begin{equation}\label{hu4}
\begin{split}
&k,k_1,k_2\in[-3.5m/N'_0,3.2m/N'_0],\quad j\leq m+2\D+\max(|k|,|k_1|,|k_2|)/2,\\
&l\geq -m/14,\quad m\geq \D^2/8,\quad k_2,k_3,k_4\leq m/N'_0,\quad [(k_3,j_3),(k_4,j_4)]\in X_{m,k_2}.
\end{split}
\end{equation}

The bound \eqref{kj7.45} and the same argument as in the proof of Lemma \ref{PhiLocLem} show that
\begin{equation}\label{hu5}
\begin{split}
\big\|P_k\mathcal{J}_{l,\leq p}[f,g,h](s)\big\|_{L^2}\lesssim &2^{(k+k_1+k_2)/2}2^{(k_2+k_3+k_4)/2}2^{-l}\min\big\{
|f|_\infty |g|_2|h|_\infty, |f|_\infty |g|_\infty |h|_2,\\
&(1+2^{-l+2\delta^2m+3\max(k_2,0)/2})|f|_2|g|_\infty |h|_\infty\big\}+2^{-10m}|f|_2|g|_2|h|_2,
\end{split}
\end{equation}
provided that $s\in I_m$, $2^{-p}+2^{-l}\leq 2^{m-2\delta^2m}$, $f=P_{[k_1-8,k_1+8]}f$, $g=P_{[k_3-8,k_3+8]}g$, $h=P_{[k_4-8,k_4+8]}h$, and, for $F\in\{f,g,h\}$,
\begin{equation}\label{hu6}
|F|_q:=\sup_{|t|\in[2^{m-4},2^{m+4}]}\|e^{it\Lambda}F\|_{L^q}.
\end{equation} 

In particular, the bounds \eqref{hu5} and \eqref{LinftyBd3} show that
\begin{equation*}
2^{(1-50\delta)j}\big\| Q_{jk}\mathcal{C}_l[f^\mu_{j_1,k_1},f^\beta_{j_3,k_3},f^\gamma_{j_4,k_4}]\big\|_{L^2}\lesssim 2^{-\delta m}
\end{equation*}
provided that $\max(j_1,j_3,j_4)\geq 20m/21$. Therefore, it remains to prove \eqref{hu3} when
\begin{equation}\label{hu7}
\max(j_1,j_3,j_4)\leq 20m/21.
\end{equation}

{\bf{Step 1.}} We consider first the contributions of $\mathcal{C}_{l,p}[f^\mu_{j_1,k_1},f^\beta_{j_3,k_3},f^\gamma_{j_4,k_4}]$ for $p\geq -11m/21$. In this case we integrate by parts in $s$ and rewrite
\begin{equation*}
\begin{split}
\mathcal{C}_{l,p}&[f^\mu_{j_1,k_1},f^\beta_{j_3,k_3},f^\gamma_{j_4,k_4}]=i2^{-p}\Big\{\int_{\mathbb{R}}q'_m(s)\widetilde{\mathcal{J}}_{l,p}[f^\mu_{j_1,k_1},f^\beta_{j_3,k_3},f^\gamma_{j_4,k_4}](s)\,ds\\
&+\widetilde{\mathcal{C}}_{l,p}[\partial_sf^\mu_{j_1,k_1},f^\beta_{j_3,k_3},f^\gamma_{j_4,k_4}]+\widetilde{\mathcal{C}}_{l,p}[f^\mu_{j_1,k_1},\partial_sf^\beta_{j_3,k_3},f^\gamma_{j_4,k_4}]+\widetilde{\mathcal{C}}_{l,p}[f^\mu_{j_1,k_1},f^\beta_{j_3,k_3},\partial_sf^\gamma_{j_4,k_4}]\Big\},
\end{split}
\end{equation*}
where the operators $\widetilde{\mathcal{J}}_{l,p}$ and $\widetilde{\mathcal{C}}_{l,p}$ are defined in the same way as the operators $\mathcal{J}_{l,p}$ and $\mathcal{C}_{l,p}$, but with $\varphi_p(\widetilde{\Phi}(\xi,\eta,\sigma))$ replaced by $\widetilde{\varphi}_p(\widetilde{\Phi}(\xi,\eta,\sigma))$, $\widetilde{\varphi}_p(x)=2^px^{-1}\varphi_p(x)$, (see the formula \eqref{hu1}). The operator $\widetilde{\mathcal{J}}_{l,p}$ also satisfies the $L^2$ bound \eqref{hu5}. Recall the $L^2$ bounds \eqref{vd7} on $\partial_s P_{k'}f_\sigma$. Using \eqref{hu5} (with $\partial_s P_{k'}f_\sigma$ always placed in $L^2$, notice that $2^{-2l}\leq 2^{m/7}$), it follows that
\begin{equation*}
\sum_{p\geq -11m/21}2^{(1-50\delta)j}\big\|P_k\mathcal{C}_{l,p}[f^\mu_{j_1,k_1},f^\beta_{j_3,k_3},f^\gamma_{j_4,k_4}]\big\|_{L^2}\lesssim 2^{-3\delta^2m}.
\end{equation*}

{\bf{Step 2.}} For \eqref{hu3} it remains to prove that
\begin{equation}\label{hu20}
2^{(1-50\delta)j}\big\| Q_{jk}\mathcal{C}_{l,\leq -11m/21}[f^\mu_{j_1,k_1},f^\beta_{j_3,k_3},f^\gamma_{j_4,k_4}]\big\|_{L^2}\lesssim 2^{-3\delta^2m}.
\end{equation}
Since $\max(j_1,j_3,j_4)\leq 20 m/21$, see \eqref{hu7}, we have the pointwise approximate identity
\begin{equation}\label{hu20.5}
\begin{split}
P_k&\mathcal{C}_{l,\leq -11m/21}[f^\mu_{j_1,k_1},f^\beta_{j_3,k_3},f^\gamma_{j_4,k_4}]\\
&=P_k\mathcal{C}_{l,\leq -11m/21}[A_{\geq\D_1,\gamma_0}f^\mu_{j_1,k_1},A_{\geq\D_1-10,\gamma_0}f^\beta_{j_3,k_3},A_{\geq\D_1-20,\gamma_0}f^\gamma_{j_4,k_4}]\\
&+P_k\mathcal{C}_{l,\leq -11m/21}[A_{<\D_1,\gamma_0}f^\mu_{j_1,k_1},A_{\leq\D_1+10,\gamma_0}f^\beta_{j_3,k_3},A_{\leq\D_1+20,\gamma_0}f^\gamma_{j_4,k_4}]+O(2^{-4m}),
\end{split}
\end{equation}
where $\D_1$ is the large constant used in section \ref{phacolle}. This is a consequence of Lemma \ref{tech5} (i) and the observation that $|\nabla_{\eta,\sigma}\widetilde{\Phi}(\xi,\eta,\sigma)|\gtrsim 1$ in the other cases. Letting $g_1=A_{\geq\D_1,\gamma_0}f^\mu_{j_1,k_1}$, $g_3=A_{\geq\D_1-10,\gamma_0}f^\beta_{j_3,k_3}$, $g_4= A_{\geq\D_1-20,\gamma_0}f^\gamma_{j_4,k_4}$, $h_1=A_{<\D_1,\gamma_0}f^\mu_{j_1,k_1}$, $h_3=A_{\leq\D_1+10,\gamma_0}f^\beta_{j_3,k_3}$, $h_4=A_{\leq\D_1+20,\gamma_0}f^\gamma_{j_4,k_4}$, it remains to prove that
\begin{equation}\label{hu21}
2^{(1-50\delta)j}\big\| Q_{jk}\mathcal{C}_{l,\leq -11m/21}[g_1,g_3,g_4]\big\|_{L^2}\lesssim 2^{-3\delta^2m}.
\end{equation}
and
\begin{equation}\label{hu22}
2^{(1-50\delta)j}\big\| Q_{jk}\mathcal{C}_{l,\leq -11m/21}[h_1,h_3,h_4]\big\|_{L^2}\lesssim 2^{-3\delta^2m}.
\end{equation}

{\bf{Proof of \eqref{hu21}}.} We use Lemma \ref{cubicphase} (i). If $l\leq -4m/N'_0$ then $|\nabla_{\eta,\sigma}\widetilde{\Phi}(\xi,\eta,\sigma)|\gtrsim 1$ in the support of the integral (due to \eqref{ub2}) and the contribution is negligible (due to Lemma \ref{tech5} (i) and \eqref{hu7}). On the other hand, if
\begin{equation}\label{hu23}
l\geq -4m/N'_0\,\,\,\text{ and }\,\,\,j\leq 2m/3+\max(j_1,j_3,j_4)
\end{equation}
then we apply \eqref{hu5}. The left hand side of \eqref{hu21} is dominated by
\begin{equation*}
C2^{(1-50\delta)j}2^m(1+2^{-2l})2^{-5m/3+8\delta^2m}2^{-\max(j_1,j_3,j_4)(1-50\delta)}\lesssim 2^{-10\delta},
\end{equation*}
as we notice that $\max(k,k_1,k_2,k_3,k_4)\leq 20$. This suffices to prove \eqref{hu21} in this case.

Finally, if 
\begin{equation}\label{hu24}
l\geq -4m/N'_0\,\,\,\text{ and }\,\,\,j\geq 2m/3+\max(j_1,j_3,j_4)
\end{equation}
then $\max(j_1,j_3,j_4)\leq m/3+10\delta m$ and $j\geq 2m/3$. We define the localized trilinear operators 
\begin{equation}\label{hu30}
\begin{split}
\mathcal{F}\{\mathcal{J}_{l,\leq p,\kappa}[f,g,h]\}(\xi,s):=\int_{\mathbb{R}^2\times\mathbb{R}^2}e^{is\widetilde{\Phi}(\xi,\eta,\sigma)}\widehat{f}(\xi-\eta)2^{-l}\widetilde{\varphi}_{l}(\Phi_{+\mu\nu}(\xi,\eta))\varphi_{\leq p}(\widetilde{\Phi}(\xi,\eta,\sigma))\\
\times\varphi(\kappa^{-1}\nabla_{\eta,\sigma}\widetilde{\Phi}(\xi,\eta,\sigma))\varphi_{k_2}(\eta)\mathfrak{m}_{\mu\nu}(\xi,\eta)\mathfrak{m}_{\nu\beta\gamma}(\eta,\sigma)\widehat{g}(\eta-\sigma)\widehat{h}(\sigma)\,d\sigma d\eta,
\end{split} 
\end{equation}
which are similar to the trilinear operators defined in \eqref{hu1} with the additional cutoff factor in  $\nabla_{\eta,\sigma}\widetilde{\Phi}(\xi,\eta,\sigma)$ and $p=-11m/21$. Set $\kappa:=2^{-m/2+\delta^2m}$ and notice that 
\begin{equation*}
\|\mathcal{F}\{\mathcal{J}_{l,\leq -11m/21}[g_1,g_3,g_4]-\mathcal{J}_{l,\leq -11m/21,\kappa}[g_1,g_3,g_4]\}\|_{L^\infty}\lesssim 2^{-6m},
\end{equation*}
as a consequence of Lemma \ref{tech5} (i). Moreover, $|\nabla_\xi\widetilde{\Phi}(\xi,\eta,\sigma)|\lesssim 2^{2p/3}\approx 2^{-22m/63}$ in the support of the integral defining $\mathcal{J}_{l,\leq -11m/21,\kappa}[g_1,g_3,g_4]$, due to Lemma \ref{cubicphase} (i). Therefore, using the approximate finite speed of propagation of argument (integration by parts in $\xi$),
\begin{equation*}
\|Q_{jk}\mathcal{J}_{l,\leq -11m/21,\kappa}[g_1,g_3,g_4]\|_{L^\infty}\lesssim 2^{-6m}.
\end{equation*}
The desired bound \eqref{hu21} follows in this case as well (in fact, one has rapid decay if \eqref{hu24} holds).

{\bf{Proof of \eqref{hu22}}.} The desired estimate follows from \eqref{hu5} and the dispersive bounds \eqref{LinftyBd2}--\eqref{LinftyBd2.5} if $\max(j_1,j_3,j_4)\geq m/3$ or if $j\leq 2m/3$ or if $l\geq -10\delta m$. Assume that
\begin{equation}\label{hu33}
\max(j_1,j_3,j_4)\leq m/3,\qquad j\geq 2m/3,\qquad l\leq -10\delta m.
\end{equation}
As before, we may replace $\mathcal{J}_{l,\leq -11m/21}[h_1,h_3,h_4]$ with $\mathcal{J}_{l,\leq -11m/21,\kappa}[h_1,h_3,h_4]$, at the expense of a small error, where $\kappa=2^{-m/2+20\delta m}$. Moreover, $|\nabla_\xi\widetilde{\Phi}(\xi,\eta,\sigma)|\lesssim \kappa$ in the support of the integral defining $\mathcal{J}_{l,\leq -11m/21,\kappa}[h_1,h_3,h_4]$, due to Lemma \ref{cubicphase} (ii). The approximate finite speed of propagation of argument (integration by parts in $\xi$) then gives rapid decay in the case when \eqref{hu33} holds. This completes the proof.

\subsection{The case of strongly resonant interactions, II} In this subsection we prove Lemma \ref{lsmallBound}. Let $\overline{k}:=\max(k,k_1,k_2,0)$. It suffices to prove the lemma in the case
\begin{equation}\label{tla1}
k,k_1,k_2\in[-\overline{k}-20,\overline{k}],\quad j\leq m+3\D+\overline{k}/2,\quad \overline{k}\leq 7m/(6N'_0),\quad l_-<l\leq -m/14.
\end{equation}
Indeed, we can assume that $k,k_1,k_2\geq -\overline{k}-20$, since otherwise the operator is trivial (due to \eqref{try5.5}). Moreover, if $\max(k_1,k_2)\geq 7m/(6N'_0)-10$ then the $L^2\times L^\infty$ argument (with Lemma \ref{PhiLocLem}) easily gives the desired conclusion due to the assumption \eqref{ni3.11}.

We define (compare with the definition of the operators $T_{m,l}$ in \eqref{TML})
\begin{equation*}
\widehat{T_{m,l}^{\parallel}[f,g]}(\xi)=\int_{\mathbb{R}}q_m(s)\int_{\mathbb{R}^2}e^{is\Phi(\xi,\eta)}\varphi(\kappa_{\theta}^{-1}\Theta(\xi,\eta))\varphi_l(\Phi(\xi,\eta))m_0(\xi,\eta)\widehat{f}(\xi-\eta,s)\widehat{g}(\eta,s)d\eta ds,
\end{equation*} 
where $\kappa_{\theta}:=2^{-m/2+6\overline{k}+\delta^2m}$. Let $T_{m,l}^{\perp}=T_{m,l}-T_{m,l}^{\parallel}$, and define $\mathcal{A}_{m,l}^{\parallel}$ and $\mathcal{B}_{m,l}^{\parallel}$ similarly, by inserting the factor $\varphi(\kappa_{\theta}^{-1}\Theta(\xi,\eta))$ in the integrals in \eqref{Alx41}. We notice that \[T_{m,l}^{\parallel}[P_{k_1}f^\mu,P_{k_2}f^\nu]=i\mathcal{A}_{m,l}^{\parallel}[P_{k_1}f^\mu,P_{k_2}f^\nu]+i\mathcal{B}_{m,l}^{\parallel}[P_{k_1}\partial_sf^\mu,P_{k_2}f^\nu]+i\mathcal{B}_{m,l}^{\parallel}[P_{k_1}f^\mu,P_{k_2}\partial_sf^\nu].\] It remains to prove that for any $j_1,j_2$
\begin{equation}\label{choice01}
2^{(1-50\delta)j}\Vert Q_{jk}T_{m,l}^{\perp}[f_{j_{1},k_{1}}^\mu,f_{j_{2},k_{2}}^\nu]\Vert_{L^2}\lesssim 2^{-3\delta^2m},
\end{equation} 
\begin{equation}\label{choice01.5}
\Vert Q_{jk}\mathcal{A}_{m,l}^{\parallel}[f_{j_{1},k_{1}}^\mu,f_{j_{2},k_{2}}^\nu]\Vert_{B_j}\lesssim 2^{-3\delta^2m},
\end{equation} 
and
\begin{equation}\label{choice02}
\Vert Q_{jk}\mathcal{B}_{m,l}^{\parallel}[f_{j_{1},k_{1}}^\mu,\partial_{s}P_{k_{2}}f^{\nu}]\Vert_{B_j}\lesssim 2^{-3\delta^2m}.
\end{equation} 

{\bf{Proof of \eqref{choice01}.}} We may assume that $\min(j_{1},j_{2})\geq m-2\overline{k}-\delta^2m$, otherwise the conclusion follows from Lemma \ref{RotIBP}. We decompose $f_{j_{1},k_{1}}^\mu=\sum_{n_1=0}^{j_1+1}f_{j_1,k_1,n_1}$, $f_{j_{2},k_2}^\nu=\sum_{n_2=0}^{j_j+1}f_{j_2,k_2,n_2}$ and estimate, using Lemma \ref{Shur2Lem}, and \eqref{RadL2}
\begin{equation*}
\begin{split}
\big\| P_kT_{m,l}^{\perp}&[f_{j_{1},k_{1},n_1},f_{j_{2},k_{2},n_2}]\big\|_{L^2}\\
&\lesssim 2^{2\overline{k}}2^m2^{l/2-n_1/2-n_2/2}\big\Vert\sup_{\theta\in\mathbb{S}^1}|\widehat{f_{j_1,k_1,n_1}}(r\theta)|\big\Vert_{L^2(rdr)}\big\Vert\sup_{\theta\in\mathbb{S}^1}|\widehat{f_{j_2,k_2,n_2}}(r\theta)|\big\Vert_{L^2(rdr)}\\
&\lesssim 2^{2\overline{k}}2^m2^{l/2}2^{6\delta^2m}2^{-j_1+51\delta j_1}2^{-j_2+51\delta j_2}.
\end{split}
\end{equation*}
Therefore, using also \eqref{tla1}, the left-hand side of \eqref{choice01} is dominated by
\begin{equation*}
2^{(1-50\delta)j}\cdot 2^{6\delta^2m}2^{2\overline{k}}2^m2^{l/2}2^{-j_1+51\delta j_1}2^{-j_2+51\delta j_2}\lesssim 2^{8\overline{k}}2^{l/2}2^{54\delta m}.
\end{equation*}
This suffices to prove the desired bound, since $2^{l/2}\lesssim 2^{-m/28}$ and $2^{8\overline{k}}2^{54\delta m}\lesssim 2^{64\delta m}\lesssim 2^{m/30}$.

{\bf{Proof of \eqref{choice01.5}.}} In view of Lemma \ref{StronglyResLem}, it suffices to prove that
\begin{equation*}
2^{(1-50\delta)j}\Vert Q_{jk}\mathcal{A}_{m,l}^{\perp}[f_{j_{1},k_{1}}^\mu,f_{j_{2},k_{2}}^\nu]\Vert_{L^2}\lesssim 2^{-3\delta^2m}.
\end{equation*}
This is similar to the proof of \eqref{choice01} above, using Lemma \ref{Shur2Lem} and \eqref{RadL2}.

{\bf{Proof of \eqref{choice02}.}} This is the more difficult estimate, where we need to use the more precise information in Lemma \ref{dtfLem2}. We may assume $j_1\leq 3m$, since in the case $j_1\geq 3m$ we can simply estimate $\|\widehat{f_{j_{1},k_{1}}^\mu}\|_{L^1}\lesssim 2^{-j_1+51\delta j_1}$ (see \eqref{FL1bd}) and the desired estimate follows easily. We decompose $\partial_{s}P_{k_{2}}f^{\nu}$ as in \eqref{Brc4}, and then we decompose $A_{k_2;k_3,j_3;k_4,j_4}^{a_3,\alpha_3;a_4,\alpha_4}=\sum_{i=1}^3 A_{k_2;k_3,j_3;k_4,j_4}^{a_3,\alpha_3;a_4,\alpha_4;[i]}$ as in \eqref{Ddecomposition}. Notice that since $k_2\geq -3m/(2N'_0)$ (see \eqref{tla1}), it follows from Lemma \ref{dtfLem1} (ii) (2) that $\min(k_2,k_3,k_4)\geq -2m/N'_0$, so Lemma \ref{dtfLem2} applies. It remains to prove that
\begin{equation}\label{tla2}
\Vert Q_{jk}\mathcal{B}_{m,l}^{\parallel}[f_{j_{1},k_{1}}^\mu,P_{k_{2}}E_{\nu}^{a_2,\alpha_2}]\Vert_{B_j}\lesssim 2^{-4\delta^2m},
\end{equation}
and, for any $[(k_3,j_3),(k_4,j_4)]\in X_{m,k_2}$, $i\in\{1,2,3\}$,
\begin{equation}\label{tla3}
\Vert Q_{jk}\mathcal{B}_{m,l}^{\parallel}[f_{j_{1},k_{1}}^\mu,A_{k_2;k_3,j_3;k_4,j_4}^{a_3,\alpha_3;a_4,\alpha_4;[i]}]\Vert_{B_j}\lesssim 2^{-4\delta^2m}.
\end{equation}
These bounds follow from Lemmas \ref{Econtr}, \ref{G1contr}, and \ref{G2contr} below. Recall the definition
\begin{equation}\label{tla30.4}
\widehat{\mathcal{B}_{m,l}^{\parallel}[f,g]}(\xi)=\int_{\mathbb{R}}q_m(s)\int_{\mathbb{R}^2}e^{is\Phi(\xi,\eta)}\varphi(\kappa_{\theta}^{-1}\Theta(\xi,\eta))2^{-l}\widetilde{\varphi}_l(\Phi(\xi,\eta))m_0(\xi,\eta)\widehat{f}(\xi-\eta,s)\widehat{g}(\eta,s)d\eta ds.
\end{equation}

\begin{lemma}\label{Econtr}
Assume that \eqref{tla1} holds and $\kappa_\theta=2^{-m/2+6\overline{k}+\delta^2m}$. Then
\begin{equation}\label{tla20}
\Vert Q_{jk}\mathcal{B}_{m,l}^{\parallel}[f_{j_{1},k_{1}}^\mu,h]\Vert_{B_j}\lesssim 2^{-4\delta^2m},
\end{equation}
provided that, for any $s\in I_m$
\begin{equation}\label{tla21}
h(s)=P_{[k_2-2,k_2+2]}h(s),\qquad \|h(s)\|_{L^2}\lesssim 2^{-3m/2+35\delta m-22\overline{k}}.
\end{equation}
\end{lemma}

\begin{proof} The lemma is slightly stronger (with a weaker assumption on $h$) than we need to prove \eqref{tla2}, since we intend to apply it in some cases in the proof of \eqref{tla3} as well. We would like to use Schur's lemma and Proposition \ref{volume} (iii). For this we need to further decompose the operator $\mathcal{B}_{m,l}^{\parallel}$. For $p,q\in\mathbb{Z}$ we define the operators $\mathcal{B}'_{p,q}$ by
\begin{equation}\label{tla4}
\begin{split}
\widehat{\mathcal{B}'_{p,q}[f,g]}&(\xi):=\int_{\mathbb{R}}q_m(s)\int_{\mathbb{R}^2}e^{is\Phi(\xi,\eta)}\varphi(\kappa_{\theta}^{-1}\Theta(\xi,\eta))2^{-l}\widetilde{\varphi}_l(\Phi(\xi,\eta))\\
&\times\varphi_p(\nabla_\xi\Phi(\xi,\eta))\varphi_q(\nabla_\eta\Phi(\xi,\eta))m_0(\xi,\eta)\widehat{f}(\xi-\eta,s)\widehat{g}(\eta,s)d\eta ds.
\end{split}
\end{equation}
Let $H_{p,q}:=P_k\mathcal{B}'_{p,q}[f_{j_{1},k_{1}}^\mu,h]$. Using the bounds $\|\widehat{f_{j_{1},k_{1}}^\mu}\|_{L^\infty}\lesssim 2^{2\delta j_1}2^{5\delta^2m}2^{51\delta\overline{k}}\lesssim 2^{7\delta m}$ (see \eqref{FLinftybd}), Proposition \ref{volume} (iii), and \eqref{tla21}, we estimate
\begin{equation}\label{tla6}
\begin{split}
\Vert H_{p,q}\Vert_{L^2}&\lesssim 2^{2\overline{k}}2^m (2^{10\overline{k}}2^l\kappa_\theta 2^{-p_-/2}2^{-q_-/2}2^{\delta^2m})2^{-l}\sup_{s\in I_m}\|\widehat{f_{j_{1},k_{1}}^\mu}(s)\|_{L^\infty} \|h(s)\|_{L^2}\\
&\lesssim 2^{-4\overline{k}}2^{-p_-/2}2^{-q_-/2}2^{-m+43\delta m},
\end{split}
\end{equation}
where $x_-=\min(x,0)$. In particular
\begin{equation}\label{tla7}
\sum_{p\geq -4\delta m,\,q\geq -4\delta m}2^{j-50\delta j}\Vert P_k\mathcal{B}'_{p,q}[f_{j_{1},k_{1}}^\mu,h]\Vert_{L^2}\lesssim 2^{-\delta m}.
\end{equation}

We show now that
\begin{equation}\label{tla8}
\sum_{p\leq -4\delta m,\,q\in\mathbb{Z}}2^{j-50\delta j}\Vert P_k\mathcal{B}'_{p,q}[f_{j_{1},k_{1}}^\mu,h]\Vert_{L^2}\lesssim 2^{-\delta m}.
\end{equation}
For this we notice now that if $p\leq -4\delta m$ then $P_k\mathcal{B}'_{p,q}[f_{j_{1},k_{1}}^\mu,h]$ is nontrivial only when $|\eta|$ is close to $\gamma_1$, and $|\xi|,|\xi-\eta|$ are close to $\gamma_1/2$ (as a consequence of Proposition \ref{spaceres11} (iii)). In particular $2^{\overline{k}}\lesssim 1$, $2^q\approx 1$, and $|\widehat{f_{j_{1},k_{1}}^\mu}(\xi-\eta,s)|\lesssim 2^{2\delta^2m}2^{-j_1/2+51\delta j_1}$ in the support of the integral. Therefore we have the stronger estimate, using also \eqref{Alx64.3} (compare with \eqref{tla6})
\begin{equation}\label{tla10}
\begin{split}
\Vert H_{p,q}\Vert_{L^2}&\lesssim 2^{m-l} 2^l\kappa_\theta \min(2^{-p/2},2^{p/2-l/2})2^{\delta^2m}\sup_{s\in I_m}\|\widehat{f_{j_{1},k_{1}}^\mu}(s)\|_{L^\infty} \|h(s)\|_{L^2}\\
&\lesssim 2^{-j_1/2+51\delta j_1}\min(2^{-p/2},2^{p/2-l/2})2^{-m+36\delta m}.
\end{split}
\end{equation}
The desired bound \eqref{tla8} follows if $j_1\geq j-\delta m$ or if $j\leq 3m/4-5\delta m$, since $\min(2^{-p/2},2^{p/2-l/2})\lesssim 2^{-l/4}\lesssim 2^{m/4}$. On the other hand, if 
\begin{equation*}
j_1\leq j-\delta m\,\,\text{ and }\,\,j\geq 3m/4-5\delta m
\end{equation*}
then the sum over $p\geq (j-m)-10\delta m$ in \eqref{tla8} can also be estimated using \eqref{tla10}. The remaining sum over $p\leq (j-m)-10\delta m$ is negligible using the approximate finite speed of propagation argument (integration by parts in $\xi$). This completes the proof of \eqref{tla8}.

Finally we show that
\begin{equation}\label{tla13}
\sum_{p\in\mathbb{Z},\,q\leq -4\delta m}\Vert Q_{jk}\mathcal{B}'_{p,q}[f_{j_{1},k_{1}}^\mu,h]\Vert_{B_j}\lesssim 2^{-\delta m}.
\end{equation}
As before, we notice now that if $q\leq -4\delta m$ then $P_k\mathcal{B}'_{p,q}[f_{j_{1},k_{1}}^\mu,h]$ is nontrivial only when $|\xi|$ is close to $\gamma_1$, and $|\eta|,|\xi-\eta|$ are close to $\gamma_1/2$ (as a consequence of Proposition \ref{spaceres11} (iii)). In particular $2^{\overline{k}}\lesssim 1$, $2^p\approx 1$ and we have the stronger estimate (compare with \eqref{tla10})
\begin{equation}\label{tla14}
\begin{split}
\Vert H_{p,q}\Vert_{L^2}\lesssim 2^{-j_1/2+51\delta j_1}\min(2^{-q/2},2^{q/2-l/2})2^{-m+36\delta m}\lesssim \frac{2^{q/2}}{2^q+2^{l/2}}2^{-m+36\delta m}.
\end{split}
\end{equation}
Moreover, since $|\Phi(\xi,\eta)|\lesssim 2^l$ and $|\nabla_\eta\Phi(\xi,\eta)|\lesssim 2^q$, the function $\widehat{H_{p,q}}$ is supported in the set $\{||\xi|-\gamma_1|\lesssim 2^l+2^{2q}\}$ (see \eqref{try1.2}). The main observation is that the $B_j$ norm for functions supported in such a set carries an additional small factor. More precisely, after localization to a $2^j$ ball in the physical space, the function $\mathcal{F}\{Q_{jk}\mathcal{B}'_{p,q}[f_{j_{1},k_{1}}^\mu,h]\}(\xi)$ is supported in the set $\{||\xi|-\gamma_1|\lesssim 2^l+2^{2q}+2^{-j+2\delta m}\}$, up to a negligible error. Therefore, using \eqref{tla14},
\begin{equation*}
\begin{split}
\Vert Q_{jk}\mathcal{B}'_{p,q}[f_{j_{1},k_{1}}^\mu,P_{k_{2}}E_{\nu}^{a_3}]\Vert_{B_j}&\lesssim 2^{j-50\delta j}(2^l+2^{2q}+2^{-j+2\delta m})^{1/2-49\delta}\Vert H_{p,q}\Vert_{L^2}\\
&\lesssim 2^{j-50\delta j}2^{-m+36\delta m}(2^{l/2}+2^{q}+2^{-j/2+\delta m})\frac{2^{q/2-100\delta q}}{2^q+2^{l/2}}\\
&\lesssim 2^{q/8}2^{-4\delta m}.
\end{split}
\end{equation*}
The bound \eqref{tla13} follows. The bound \eqref{tla20} follows from \eqref{tla7}, \eqref{tla8}, and \eqref{tla13}.
\end{proof}

\begin{lemma}\label{G1contr}
Assume that \eqref{tla1} holds and $\kappa_\theta=2^{-m/2+6\overline{k}+\delta^2m}$. Then
\begin{equation}\label{tla30}
\Vert Q_{jk}\mathcal{B}_{m,l}^{\parallel}[f_{j_{1},k_{1}}^\mu,A_{k_2;k_3,j_3;k_4,j_4}^{a_3,\alpha_3;a_4,\alpha_4;[1]}]\Vert_{B_j}\lesssim 2^{-4\delta^2m}.
\end{equation}
\end{lemma}

\begin{proof} Notice that $A_{k_2;k_3,j_3;k_4,j_4}^{a_3,\alpha_3;a_4,\alpha_4;[1]}$ is supported in the set $||\eta|-\gamma_1|\leq 2^{-\D}$. Using also the conditions $\Phi(\xi,\eta)\lesssim 2^l$ and $\Theta(\xi,\eta)\lesssim \kappa_\theta$, we have
\begin{equation}\label{tla30.5}
||\eta|-\gamma_1|\leq 2^{-\D},\quad |\xi|,|\xi-\eta|\in[2^{-50},2^{50}],\quad \min(||\xi|-\gamma_1|,||\xi-\eta|-\gamma_1|)\geq 2^{-50}
\end{equation}
in the support of the integral defining $\mathcal{F}\{P_k\mathcal{B}_{m,l}^{\parallel}[f_{j_{1},k_{1}}^\mu,G^{[1]}](\xi)\}$, where $G^{[1]}=A_{k_2;k_3,j_3;k_4,j_4}^{a_3,\alpha_3;a_4,\alpha_4;[1]}$.

{\bf{Case 1.}} Assume first that 
\begin{equation}\label{tla31}
\max(j_3,j_4)\geq m/2.
\end{equation}
In this case $\|G^{[1]}\|_{L^2}\lesssim 2^{-3m/2+30\delta m}$ (see \eqref{Dbound5.1}), and the conclusion follows from Lemma \ref{Econtr}.

{\bf{Case 2.}} Assume now that 
\begin{equation}\label{tla32}
\max(j_3,j_4)\leq m/2,\qquad j_1\geq m/2.
\end{equation}
The bound \eqref{tla30} follows again by the same argument as in the proof of \eqref{tla20} above. In this case $\|\widehat{G^{[1]}}(s)\|_{L^\infty}\lesssim 2^{-m+4\delta m}$ (due to \eqref{Decdtf}) and $\|\mathcal{F}\{A_{\leq 0,\gamma_1}f_{j_{1},k_{1}}^\mu\}(s)\|_{L^2}\lesssim 2^{2\delta^2m}2^{-j_1+50\delta j_1}$ (see \eqref{FLinftybd}). We make the change of variables $\eta\to\xi-\eta$, define $\Phi'(\xi,\eta)=\Phi(\xi,\xi-\eta)$ and define the operators $\mathcal{B}''_{p,q}$ as in \eqref{tla4}, by inserting cutoff factors $\varphi_p((\nabla_\xi\Phi')(\xi,\eta))$ and $\varphi_q((\nabla_\eta\Phi')(\xi,\eta))$. In this case we notice that we may assume both $p\geq -\D$ and $q\geq -\D$. Indeed we have $|\Phi'(\xi,\eta)|\leq 2^{-\D}$ and $||\xi-\eta|-\gamma_1|\leq 2^{-\D}$, so $|(\nabla_\xi\Phi')(\xi,\eta)|\gtrsim 1$ and $|(\nabla_\eta\Phi')(\xi,\eta)|\gtrsim 1$ in the support of the integral (in view of Proposition \ref{spaceres11} (iii)). Then we estimate, using \eqref{Alx64.1},
\begin{equation*}
\Vert P_k\mathcal{B}''_{p,q}[A_{\leq 0,\gamma_1}f_{j_{1},k_{1}}^\mu,G^{[1]}]\Vert_{L^2}\lesssim 2^{-j_1+50\delta j_1}2^{-m/2+5\delta m}.
\end{equation*}
The bound \eqref{tla30} follows by summation over $p$ and $q$.

{\bf{Case 3.}} Assume now that 
\begin{equation}\label{tla33}
\max(j_1,j_3,j_4)\leq m/2,\qquad j\leq m/2+10\delta m.
\end{equation}
We use the bounds $\|\widehat{G^{[1]}}(s)\|_{L^\infty}\lesssim 2^{-m+4\delta m}$ (see \eqref{Decdtf}) and $\|\widehat{f_{j_{1},k_{1}}^\mu}(s)\|_{L^\infty}\lesssim 2^{3\delta m}$. Moreover, $|\nabla_\eta\Phi(\xi,\eta)|\gtrsim 1$ in the support of the integral. Therefore, using the first bound in \eqref{Alx64.1},
\begin{equation*}
\begin{split}
\big\|\mathcal{F}\{P_k\mathcal{B}^{\parallel}_{m,l}[f_{j_{1},k_{1}}^\mu,G^{[1]}]\}\big\|_{L^\infty}&\lesssim 2^{m-l}\kappa_\theta 2^l2^{\delta^2m}\sup_{s\in I_m}\|\widehat{G^{[1]}}(s)\|_{L^\infty}\|\widehat{f_{j_{1},k_{1}}^\mu}(s)\|_{L^\infty}\lesssim 2^{-m/2+8\delta m}.
\end{split}
\end{equation*}
The desired bound \eqref{tla30} follows when $j\leq m/2+10\delta m$.

{\bf{Case 4.}} Finally, assume that 
\begin{equation}\label{tla35}
\max(j_1,j_3,j_4)\leq m/2,\qquad j\geq m/2+10\delta m.
\end{equation}
We examine the formula \eqref{tla30.4}, decompose $G^{[1]}$ as in \eqref{Decdtf} and notice that the contribution of the error term is easy to estimate. To estimate the main term, we define the modified phase
\begin{equation}\label{tla36}
\mathfrak{p}(\xi,\eta):=\Phi_{+\mu\nu}(\xi,\eta)+\Lambda_\nu(\eta)-2\Lambda_\nu(\eta/2)=\Lambda(\xi)-\Lambda_\mu(\xi-\eta)-2\Lambda_\nu(\eta/2).
\end{equation}
For $r\in\mathbb{Z}$ we define the functions $\mathcal{G}_r=\mathcal{G}_{r,m,l,j,j_1}$ by
\begin{equation}\label{tla37}
\begin{split}
\widehat{\mathcal{G}_r}(\xi):=\int_{\mathbb{R}}q_m(s)\int_{\mathbb{R}^2}&e^{is\mathfrak{p}(\xi,\eta)}\varphi(\kappa_{\theta}^{-1}\Theta(\xi,\eta))2^{-l}\widetilde{\varphi}_l(\Phi(\xi,\eta))m_0(\xi,\eta)\\
&\times\varphi_r(\nabla_\eta\mathfrak{p}(\xi,\eta))\widehat{f^\mu_{j_1,k_1}}(\xi-\eta,s)g^{[1]}(\eta,s)\varphi(2^{3\delta m}(|\eta|-\gamma_1))d\eta ds.
\end{split}
\end{equation}
Notice that the functions $\mathcal{G}_r$ are negligible for, say, $r\leq -10m$. It suffices to prove that
\begin{equation}\label{tla38}
2^{j-50\delta j}\Vert Q_{jk}\mathcal{G}_r\Vert_{L^2}\lesssim 2^{-5\delta^2m}\qquad\text{ for any }r\in\mathbb{Z}.
\end{equation}

We notice first that $\|P_k\mathcal{G}_r\|_{L^2}\lesssim 2^{-4m}$ if $r\geq \max(\delta^2m-l-m,6\delta m-m/2)$, in view of Lemma \ref{tech5} (i). In particular, we may assume that $r\leq -\D$. In this case, the functions $\mathcal{G}_r$ are nontrivial only when $-\mu=\nu=+$ and $\xi$ is close to $\eta/2$. Therefore $\mathfrak{p}(\xi,\eta)=\Lambda(\xi)+\Lambda(\eta-\xi)-2\Lambda(\eta/2)$, and we have, in the support of the integral defining $\widehat{\mathcal{G}_r}(\xi)$
\begin{equation}\label{tla39}
\begin{split}
&|\nabla_\eta\mathfrak{p}(\xi,\eta)|\approx |\xi-\eta/2|\approx |\nabla_\xi\mathfrak{p}(\xi,\eta)|\approx |\nabla_\xi\Phi(\xi,\eta)|\approx 2^r,\\
&|\mathfrak{p}(\xi,\eta)|\approx |\xi-\eta/2|^2\approx 2^{2r},\\
&||\eta|-\gamma_1|\approx|\Lambda(\eta)-2\Lambda(\eta/2)|\lesssim |\Phi(\xi,\eta)|+|\mathfrak{p}(\xi,\eta)|\lesssim 2^l+2^{2r},\\
&||\xi|-\gamma_1/2|\lesssim 2^l+2^r.
\end{split}
\end{equation}
The finite speed of propagation argument (integration by parts in $\xi$) shows that $\|Q_{jk}\mathcal{G}_r\|_{L^2}\lesssim 2^{-4m}$ if $j\geq 3\delta^2m+\max(m+r,-r)$. To summarize, it remains to prove that
\begin{equation}\label{tla40}
\big(2^{m+r}+2^{-r}\big)^{1-50\delta}\|P_k\mathcal{G}_r\|_{L^2}\lesssim 2^{-\delta m}\qquad\text{ if }\qquad r\leq \max(\delta^2m-l-m,6\delta m-m/2).
\end{equation}

For $\xi$ fixed, the variable $\eta$ satisfies three restrictions: $|\eta\cdot\xi^{\perp}|\lesssim \kappa_\theta$, $|\Phi(\xi,\eta)|\lesssim 2^l$, and $|\eta-2\xi|\lesssim 2^r$. Therefore, using also \eqref{Decdtf}, we have the pointwise bound
\begin{equation}\label{tla41}
\begin{split}
|\widehat{\mathcal{G}_r}(\xi)|&\lesssim 2^{5\delta^2m}2^{m-l}\min(2^r,2^{-m/2})\min(2^r,2^l)\sup_{s\in I_m}\|\widehat{f^\mu_{j_1,k_1}}(s)\|_{L^\infty}\|g^{[1]}(s)\|_{L^\infty}\\
&\lesssim 2^{8\delta m}\min(2^r,2^{-m/2})\min(2^{r-l},1).
\end{split}
\end{equation}
The desired bound \eqref{tla40} follows, using also the support assumption $||\xi|-\gamma_1/2|\lesssim 2^l+2^r$ in \eqref{tla39}, if $r\leq -m/2$ or if $r\in[-m/2,-m/3]$.

It remains to prove \eqref{tla40} when $-m/3\leq r\leq -l-m+\delta^2m$. The main observation in this case is that $|\mathfrak{p}(\xi,\eta)|\approx 2^{2r}$ is large enough to be able to integrate by parts is $s$. It follows that
\begin{equation*}
\begin{split}
|\widehat{\mathcal{G}_r}(\xi)|\lesssim \int_{\mathbb{R}}\int_{\mathbb{R}^2}2^{-2r}\big|&\varphi(\kappa_{\theta}^{-1}\Theta(\xi,\eta))2^{-l}\widetilde{\varphi}_l(\Phi(\xi,\eta))\varphi_r(\nabla_\eta\mathfrak{p}(\xi,\eta))\varphi(2^{3\delta m}(|\eta|-\gamma_1))\big|\\
&\times\big|\partial_s[\widehat{f^\mu_{j_1,k_1}}(\xi-\eta,s)g^{[1]}(\eta,s)q_m(s)]\big|d\eta ds.
\end{split}
\end{equation*}
For $\xi$ fixed, the integral in supported in a $O(\kappa_\theta\times 2^l)$ rectangle centered at $\eta=2\xi$. In this support, we have the bounds, see Lemma \ref{dtfLem2} (ii) and (iii),
\begin{equation*}
\begin{split}
&\|\widehat{f^\mu_{j_1,k_1}}(s)\|_{L^\infty}\lesssim 2^{\delta^2m},\qquad \|g^{[1]}(s)\|_{L^\infty}\lesssim 2^{-m+4\delta m}\qquad \|\partial_sg^{[1]}(s)\|_{L^\infty}\lesssim 2^{-2m+18\delta m},\\
&\partial_s f^{\mu}_{j_1,k_1}=h_2+h_\infty,\quad\|h_2(s)\|_{L^2}\lesssim 2^{-3m/2+5\delta m},\quad \|\widehat{h_\infty}(s)\|_{L^\infty}\lesssim 2^{-m+15\delta m}.
\end{split}
\end{equation*}
The integrals that do not contain the function $h_2$ can all be estimated pointwise, as in \eqref{tla41} by $C2^{-2r}2^{-l}2^{-m+20\delta m}(2^l\kappa_\theta)\lesssim 2^{-2r}2^{-3m/2+21\delta m}$. The integral that contains the function $h_2$ can be estimated pointwise, using H\"{o}lder's inequality, by 
\begin{equation*}
C2^{-2r}2^{-l}2^{-3m/2+10\delta m}(2^l\kappa_\theta)^{1/2}\lesssim 2^{-2r}2^{-l/2}2^{-7m/4+11\delta m}\lesssim 2^{-2r}2^{-5m/4+11\delta m}.
\end{equation*}
Therefore, using also the support assumption $||\xi|-\gamma_1/2|\lesssim 2^r$ in \eqref{tla39}, and recalling that $r\geq -m/3$, $l\leq -m/2$, we have
\begin{equation*}
2^{m+r}\|P_k\mathcal{G}_r\|_{L^2}\lesssim 2^{-r/2}2^{-m/4+11\delta m}.
\end{equation*}
This suffices to prove \eqref{tla40}, which completes the proof of the lemma.
\end{proof}

\begin{lemma}\label{G2contr}
With the same notation as in Lemma \ref{G1contr}, and assuming \eqref{tla1}, we have
\begin{equation}\label{tla60}
\Vert Q_{jk}\mathcal{B}_{m,l}^{\parallel}[f_{j_{1},k_{1}}^\mu,A_{k_2;k_3,j_3;k_4,j_4}^{a_3,\alpha_3;a_4,\alpha_4;[2]}]\Vert_{B_j}\lesssim 2^{-4\delta^2m}.
\end{equation}
\end{lemma}

\begin{proof} The main observation here is that, since $|\Phi_{+\mu\nu}(\xi,\eta)|\lesssim 2^l$ and $|\Phi_{\nu\beta\gamma}(\eta,\sigma)|\gtrsim 2^{-10\delta m}$, we have $|\widetilde{\Phi}(\xi,\eta,\sigma)|\gtrsim 2^{-10\delta m}$, thus we can integrate by parts in $s$ once more. Before this, however, we notice that we may assume that
\begin{equation}\label{tla60.5}
k_3,k_4\in[-2m/N'_0,m/N'_0],\qquad \min(j_3,j_4)\leq m-4\delta m.
\end{equation}
Indeed, we first use Lemma \ref{dtfLem1} (ii) (2), (3). Moreover, if $\min(j_3,j_4)\geq m-4\delta m$ or $\max(k_3,k_4)\geq m/N'_0$ then we would have $\|A_{k_2;k_3,j_3;k_4,j_4}^{a_3,\alpha_3;a_4,\alpha_4;[2]}\|_{L^2}\lesssim 2^{-3m/2+8\delta m}$ (by the same argument as in the proof of \eqref{tla60.6} or an $L^2\times L^\infty$ estimate), and the desired bound would follow from Lemma \ref{Econtr}.

{\bf{Step 1.}} For $r\in\mathbb{Z}$ we define (compare with \eqref{hu1}) the trilinear operators $\mathcal{J}^{[2]}_{l,r}$ by
\begin{equation}\label{tla61}
\begin{split}
\mathcal{F}\{\mathcal{J}^{[2]}_{l,r}[f,g,h]\}&(\xi,s):=\int_{\mathbb{R}^2\times\mathbb{R}^2}e^{is\widetilde{\Phi}(\xi,\eta,\sigma)}\widehat{f}(\xi-\eta)\varphi(\kappa_{\theta}^{-1}\Theta(\xi,\eta))2^{-l}\widetilde{\varphi}_{l}(\Phi_{+\mu\nu}(\xi,\eta))\\
&\times\varphi_r(\widetilde{\Phi}(\xi,\eta,\sigma))\chi^{[2]}(\eta,\sigma)\varphi_{k_2}(\eta)\mathfrak{m}_{\mu\nu}(\xi,\eta)\mathfrak{m}_{\nu\beta\gamma}(\eta,\sigma)\widehat{g}(\eta-\sigma)\widehat{h}(\sigma)\,d\sigma d\eta.
\end{split} 
\end{equation}
Let
\begin{equation}\label{tla62}
\mathcal{C}^{[2]}_{l,r}[f,g,h]:=\int_{\mathbb{R}}q_m(s)\mathcal{J}^{[2]}_{l,r}[f,g,h](s)\,ds,
\end{equation} 
and notice that
\begin{equation*}
\mathcal{B}_{m,l}^{\parallel}[f_{j_{1},k_{1}}^\mu,A_{k_2;k_3,j_3;k_4,j_4}^{b_1,b_2,b_3,[2]}]=\sum_{r\geq -11\delta m}\mathcal{C}^{[2]}_{l,r}[f_{j_{1},k_{1}}^\mu,f_{j_3,k_3}^\beta,f_{j_4,k_4}^\gamma].
\end{equation*} 

We integrate by parts in $s$ to rewrite
\begin{equation*}
\begin{split}
\mathcal{C}^{[2]}_{l,r}&[f^\mu_{j_1,k_1},f^\beta_{j_3,k_3},f^\gamma_{j_4,k_4}]=i2^{-r}\Big\{\int_{\mathbb{R}}q'_m(s)\widetilde{\mathcal{J}^{[2]}_{l,r}}[f^\mu_{j_1,k_1},f^\beta_{j_3,k_3},f^\gamma_{j_4,k_4}](s)\,ds\\
&+\widetilde{\mathcal{C}^{[2]}_{l,r}}[\partial_sf^\mu_{j_1,k_1},f^\beta_{j_3,k_3},f^\gamma_{j_4,k_4}]+\widetilde{\mathcal{C}^{[2]}_{l,r}}[f^\mu_{j_1,k_1},\partial_sf^\beta_{j_3,k_3},f^\gamma_{j_4,k_4}]+\widetilde{\mathcal{C}^{[2]}_{l,r}}[f^\mu_{j_1,k_1},f^\beta_{j_3,k_3},\partial_sf^\gamma_{j_4,k_4}]\Big\},
\end{split}
\end{equation*}
where the operators $\widetilde{\mathcal{J}^{[2]}_{l,r}}$ and $\widetilde{\mathcal{C}^{[2]}_{l,r}}$ are defined in the same way as the operators $\mathcal{J}^{[2]}_{l,r}$ and $\mathcal{C}^{[2]}_{l,r}$, but with $\varphi_p(\widetilde{\Phi}(\xi,\eta,\sigma))$ replaced by $\widetilde{\varphi}_p(\widetilde{\Phi}(\xi,\eta,\sigma))$, $\widetilde{\varphi}_p(x)=2^px^{-1}\varphi_p(x)$, (see the formula \eqref{tla61}). It suffices to prove that for any $s\in I_m$ and $r\geq -11\delta m$,
\begin{equation}\label{tla64}
2^{j-50\delta j}\Vert Q_{jk}\widetilde{\mathcal{J}^{[2]}_{l,r}}[f,g,h]\Vert_{L^2}\lesssim 2^{-12\delta m},
\end{equation}
where $[f,g,h]=[f^\mu_{j_1,k_1},f^\beta_{j_3,k_3},f^\gamma_{j_4,k_4}](s)$ or $[f,g,h]=[2^m\partial_sf^\mu_{j_1,k_1},f^\beta_{j_3,k_3},f^\gamma_{j_4,k_4}](s)$ or $[f,g,h]=[f^\mu_{j_1,k_1},2^m\partial_sf^\beta_{j_3,k_3},f^\gamma_{j_4,k_4}](s)$ or $[f,g,h]=[f^\mu_{j_1,k_1},f^\beta_{j_3,k_3},2^m\partial_sf^\gamma_{j_4,k_4}](s)$.

{\bf{Step 2.}} As in the proof of Lemma \ref{PhiLocLem}, the function $\widetilde{\varphi}_r(\widetilde{\Phi}(\xi,\eta,\sigma))$ can be incorporated with the phase $e^{is\widetilde{\Phi}(\xi,\eta,\sigma)}$, using the formula \eqref{loca2} and the fact that $2^{-r}\leq 2^{11\delta m}$. Then we integrate the variable $\sigma$ and denote by $H_1$, $H_2$, and $H_3$ the resulting functions,
\begin{equation*}
\begin{split}
&H_1:=I^{[2]}[f^\beta_{j_3,k_3}(s),f^\gamma_{j_4,k_4}(s)],\,\,\,H_2:=I^{[2]}[\partial_sf^\beta_{j_3,k_3}(s),f^\gamma_{j_4,k_4}(s)],\,\,\,H_3:=I^{[2]}[f^\beta_{j_3,k_3}(s),\partial_sf^\gamma_{j_4,k_4}(s)],\\
&\mathcal{F}\{I^{[2]}[g,h]\}(\eta):=\int_{\mathbb{R}^2}e^{i(s+\lambda)\Phi_{\nu\beta\gamma}(\eta,\sigma)}\chi^{[2]}(\eta,\sigma)\varphi_{k_2}(\eta)\mathfrak{m}_{\nu\beta\gamma}(\eta,\sigma)\widehat{g}(\eta-\sigma)\widehat{h}(\sigma)\,d\sigma.
\end{split}
\end{equation*}
We claim that
\begin{equation}\label{tla65}
\|H_1\|_{L^2}+2^m\|H_2\|_{L^2}+2^m\|H_3\|_{L^2}\lesssim 2^{-5m/6+10\delta m}.
\end{equation}
Notice that the bound on $H_1$ is already proved (in a stronger form) in the proof of \eqref{Dbound5.2}. The bounds on $H_2$ and $H_3$ follow in the same way from the $L^2\times L^\infty$ argument: indeed, we have $\|\partial_sf^\beta_{j_3,k_3}(s)\|_{L^2}+\|\partial_sf^\gamma_{j_4,k_4}(s)\|_{L^2}\lesssim 2^{-m+7\delta m}$ (due to \eqref{vd7}). Then we notice that we can remove the factor $\varphi(2^{20\delta m}\Theta_\beta(\eta,\sigma))$ from the multiplier $\chi^{[2]}(\eta,\sigma)$, at the expenses of a small error (due to Lemma \ref{RotIBP} and \eqref{tla60.5}). The desired bounds in \eqref{tla65} follow using the $L^2\times L^\infty$ argument with Lemma \ref{PhiLocLem}.

{\bf{Step 3.}} We prove now \eqref{tla64} for $[f,g,h]=[f^\mu_{j_1,k_1},f^\beta_{j_3,k_3},f^\gamma_{j_4,k_4}](s)$. It suffices to show that
\begin{equation}\label{tla66}
2^{4\overline{k}}2^{m-30\delta m}\big\|S[f^\mu_{j_1,k_1}(s),H_1]\big\|_{L^2}\lesssim 1
\end{equation}
for any $s\in I_m$, where
\begin{equation}\label{tla66.5}
\mathcal{F}\{S[f,g]\}(\xi):=|\varphi_k(\xi)|\int_{\mathbb{R}^2}\big|\widehat{f}(\xi-\eta)\varphi(\kappa_{\theta}^{-1}\Theta(\xi,\eta))2^{-l}\widetilde{\varphi}_{l}(\Phi(\xi,\eta))\varphi_{[k_2-2,k+2]}(\eta)\widehat{g}(\eta)\big|\,d\eta.
\end{equation}
This follows using Schur's lemma, the bound \eqref{tla65}, and Proposition \ref{volume} (iii). Indeed, we have $|\nabla_\eta\Phi(\xi,\eta)|+|\nabla_\xi\Phi(\xi,\eta)|\gtrsim 2^{-4\delta m}$ in the support of the integral (due to the location of space-time resonances), therefore the left-hand side of \eqref{tla66} is dominated by
\begin{equation*}
C2^{4\overline{k}}2^{m-30\delta m}2^{-l}(2^{10\overline{k}}\kappa_\theta 2^{3l/4}2^{4\delta m})\|\widehat{f^\mu_{j_1,k_1}}(s)\|_{L^\infty}\|\widehat{H_1}\|_{L^2}\lesssim 2^{30\overline{k}}2^{-l/4}2^{-m/3}.
\end{equation*}
This suffices to prove \eqref{tla66} since $2^{-l}\leq 2^m$. Moreover, \eqref{tla64} follows in the same way for $[f,g,h]=[f^\mu_{j_1,k_1},2^m\partial_sf^\beta_{j_3,k_3},f^\gamma_{j_4,k_4}](s)$ or $[f,g,h]=[f^\mu_{j_1,k_1},f^\beta_{j_3,k_3},2^m\partial_sf^\gamma_{j_4,k_4}](s)$, since the $L^2$ bounds on $2^mH_2$ and $2^mH_3$ are the same as for $H_1$.

It remains to prove \eqref{tla64} for $[f,g,h]=[2^m\partial_sf^\mu_{j_1,k_1},f^\beta_{j_3,k_3},f^\gamma_{j_4,k_4}](s)$. It suffices to prove that
\begin{equation}\label{tla67}
2^{4\overline{k}}2^{m-30\delta m}\big\|S[2^m\partial_sf^\mu_{j_1,k_1}(s),H_1]\big\|_{L^2}\lesssim 1,
\end{equation}
for any $s\in I_m$. Let $f=2^m\partial_sf^\mu_{j_1,k_1}(s)$ and $f_{2\gamma_0}:=A_{\geq \D-11,2\gamma_0} f$. We decompose, using \eqref{vd70.1},
\begin{equation*}
f=f_{2\gamma_0}+f_2+f_\infty,\quad\|f_{2\gamma_0}\|_{L^2}\lesssim 2^{7\delta m},\quad \|f_2\|_{L^2}\lesssim 2^{-m/2+5\delta m},\quad \|\widehat{f_\infty}\|_{L^\infty}\lesssim 2^{3\overline{k}+15\delta m}.
\end{equation*}
The contribution of $f_\infty$ can be estimated as before, using Schur's lemma, \eqref{tla65}, and Proposition \ref{volume} (iii). To estimate the other contributions, we also use the bound (see \eqref{vd70})
\begin{equation*}
\|\widehat{H_{1,\infty}}\|_{L^\infty}\lesssim 2^{3\overline{k}}2^{-m+14\delta m}\quad\text{ where }\quad H_1=H_{1,2\gamma_0}+H_{1,\infty}=A_{\geq \D+1,2\gamma_0}H_1+A_{\leq \D,2\gamma_0}H_1.
\end{equation*}
As before, we use Schur's test and Proposition \ref{volume} (iii), together with the fact that space-time resonances are possible only when $|\xi|,|\eta|,|\xi-\eta|$ are all close to either $\gamma_1$ or $\gamma_1/2$. We estimate
\begin{equation*}
\begin{split}
\big\|S[f_2,H_{1,\infty}]\big\|_{L^2}&\lesssim 2^{-l}(2^{12\overline{k}}\kappa_\theta 2^{3l/4}2^{4\delta m})\|\widehat{f_2}\|_{L^2}\|\widehat{H_{1,\infty}}\|_{L^\infty}\lesssim 2^{20\overline{k}}2^{-l/4}2^{-2m+40\delta m},\\
\big\|S[f_{2\gamma_0},H_{1,\infty}]\big\|_{L^2}&\lesssim 2^{-l}(2^{12\overline{k}}\kappa_\theta 2^{3l/4}2^{4\delta m})\|\widehat{f_{2\gamma_0}}\|_{L^2}\|\widehat{H_{1,\infty}}\|_{L^\infty}\lesssim 2^{20\overline{k}}2^{-l/4}2^{-3m/2+40\delta m},\\
\big\|S[f_2,H_{1,2\gamma_0}]\big\|_{L^2}&\lesssim 2^{-l}(2^{10\overline{k}}\kappa_\theta 2^l2^{4\delta m})^{1/2}\|\widehat{f_2}\|_{L^2}\|\widehat{H_{1,2\gamma_0}}\|_{L^2}\lesssim 2^{15\overline{k}}2^{-l/2}2^{-19m/12+20\delta m},\\
S[f_{2\gamma_0},H_{1,2\gamma_0}]&\equiv 0.
\end{split}
\end{equation*}
These bounds suffice to prove \eqref{tla67}, which completes the proof of the lemma.
\end{proof}

\subsection{Higher order terms}\label{Sec:Z1Norm2}

In this subsection we consider the higher order components in the Duhamel formula \eqref{duhamel} and show how to control their $Z$ norms.

\begin{proposition}\label{rok1} With the hypothesis in Proposition \ref{bootstrap}, for any $t\in[0,T]$ we have 
\begin{equation}\label{rok2}
\|W_3(t)\|_{Z}+\Big\|\int_{0}^{t}e^{is\Lambda}\mathcal{N}_{\geq 4}(s)\,ds\Big\|_{Z}\lesssim \varepsilon_1^{2}.
\end{equation} 
\end{proposition}

The rest of this section is concerned with the proof of Proposition \ref{rok1}. The bound on $\mathcal{N}_{\geq 4}$ follows directly from the hypothesis $\|e^{is\Lambda}\mathcal{N}_{\geq 4}(s)\|_{Z}\leq \varep_{1}^2(1+s)^{-1-\delta^2}$, see \eqref{bootstrap2}. To prove the bound on $W_3$ we start from the formula 
\begin{equation}\label{rok3}
\begin{split}
\Omega^a_{\xi}\widehat{W_3}(\xi,t)=\sum_{\mu,\nu,\beta\in\{+.-\}}&\sum_{a_1+a_2+a_3=a}\int_{0}^{t}\int_{\mathbb{R}^2\times\mathbb{R}^2}e^{is\widetilde{\Phi}_{+\mu\nu\beta}(\xi,\eta,\sigma)}\mathfrak{n}_{\mu\nu\beta}(\xi,\eta,\sigma)\\
&\times(\Omega^{a_1}\widehat{\mathcal{V}_{\mu}})(\xi-\eta,s)(\Omega^{a_2}\widehat{\mathcal{V}_{\nu}})(\eta-\sigma,s)(\Omega^{a_3}\widehat{\mathcal{V}_{\beta}})(\sigma,s)\,d\eta d\sigma ds.
\end{split}
\end{equation}

We define the functions $q_m$ as in \eqref{nh2} and the trilinear operators $C_m=C^{\mu\nu\beta}_{m,b}$
\begin{equation}\label{rok4}
\begin{split}
\mathcal{F}\big\{C_m[f,g,h]\big\}(\xi):=&\int_{\mathbb{R}}q_m(s)\int_{\mathbb{R}^2\times\mathbb{R}^2}e^{is\widetilde{\Phi}(\xi,\eta,\sigma)} n_0(\xi,\eta,\sigma)\widehat{f}(\xi-\eta,s)\widehat{g}(\eta-\sigma,s)\widehat{h}(\sigma,s)\,d\eta d\sigma ds,
\end{split}
\end{equation}
where $\widetilde{\Phi}:=\widetilde{\Phi}_{+\mu\nu\beta}$ and $n_0:=\mathfrak{n}_{\mu\nu\beta}$. It remains to prove that, for any $(k,j)\in\mathcal{J}$ and $m\in[0,L+1]$,
\begin{equation}\label{rok5}
\sum_{k_1,k_2,k_3\in\mathbb{Z}}2^{j-50\delta j}\big\|Q_{jk}C_m[P_{k_1}D^{\alpha_1}\Omega^{a_1}\mathcal{V}_{\mu},P_{k_2}D^{\alpha_2}\Omega^{a_2}\mathcal{V}_{\nu},P_{k_3}D^{\alpha_3}\Omega^{a_3}\mathcal{V}_{\beta}]\big\|_{L^2}\lesssim 2^{-\delta^2m}\varepsilon_1^3
\end{equation} 
for any $\mu,\nu,\beta\in\{+,-\}$, provided that $a_1+a_2+a_3=a$ and $\alpha_1+\alpha_2+\alpha_3=\alpha$. Let
\begin{equation}\label{rok6}
f^\mu:=\varepsilon^{-1}D^{\alpha_1}\Omega^{a_1}\mathcal{V}_{\mu},\quad f^\nu:=\varepsilon^{-1}D^{\alpha_2}\Omega^{a_2}\mathcal{V}_{\nu},\quad 
f^\beta:=\varepsilon^{-1}D^{\alpha_3}\Omega^{a_3}\mathcal{V}_{\beta}.
\end{equation}
The bootstrap assumption \eqref{bootstrap2} gives, for any $s\in[0,t]$ and $\gamma\in\{\mu,\nu,\beta\}$,
\begin{equation}\label{rok7}
\Vert f^{\gamma}(s)\Vert_{H^{N'_0}\cap Z_1\cap H^{N'_1}_\Omega}\lesssim (1+s)^{\delta^2}.
\end{equation}

Simple estimates, as in the proof of Lemma \ref{ZNormEstSimpleLem1}, show that the parts of the sum in 
\eqref{rok5} over $\max(k_1,k_2,k_3)\geq 2(j+m)/N'_0-\D^2$ or over $\min(k_1,k_2,k_3)\leq -(j+m)/2$ are bounded as claimed. For \eqref{rok5} it remains to prove that
\begin{equation}\label{rok5.5}
2^{j-50\delta j}\big\|Q_{jk}C_m[P_{k_1}f^\mu,P_{k_2}f^{\nu},P_{k_3}f^{\beta}]\big\|_{L^2}\lesssim 2^{-2\delta^2m-\delta^2j}
\end{equation} 
for any fixed $m\in[0,L+1]$, $(k,j)\in\mathcal{J}$, and $k_1,k_2,k_3\in\mathbb{Z}$ satisfying
\begin{equation}\label{rok5.6}
k_1,k_2,k_3\in[-(j+m)/2,2(j+m)/N'_0-\D^2].
\end{equation}

Let $\overline{k}:=\max(k,k_{1},k_{2},k_{3},0)$, $\underline{k}:=\min(k,k_{1},k_{2},k_{3})$ and $[k]:=\max(|k|,|k_1|,|k_2|,|k_3|)$. The $S^\infty$ bound in \eqref{Assumptions3} and Lemma \ref{touse} (ii) show that
\begin{equation}\label{rok5.7}
\begin{split}
\big\|C_m&[P_{k_1}f^\mu,P_{k_2}f^{\nu},P_{k_3}f^{\beta}]\big\|_{L^2}\\
&\lesssim 2^{\underline{k}/2}2^{3\overline{k}}2^m\sup_{s\in I_m}\|e^{-is\Lambda_\mu}P_{k_1}f^\mu\|_{L^{p_1}}\|e^{-is\Lambda_\nu}P_{k_2}f^\nu\|_{L^{p_2}}\|e^{-is\Lambda_\beta}P_{k_3}f^\beta\|_{L^{p_3}},
\end{split}
\end{equation}
if $p_1,p_2,p_3\in\{2,\infty\}$ and $1/p_1+1/p_2+1/p_3=1/2$. The desired bound \eqref{rok5.5} follows unless
\begin{equation}\label{rok5.8}
j\geq 2m/3+[k]/2+\D^2,
\end{equation}
using the pointwise bounds in \eqref{LinftyBd3.5}. Also, by estimating $\|P_kH\|_{L^2}\lesssim 2^{k}\|P_kH\|_{L^1}$, and using a bound similar to \eqref{rok5.7}, the desired bound \eqref{rok5.5} follows unless
\begin{equation}\label{rok5.9}
k\geq -(2/3)(j+m/6+\delta m).
\end{equation} 

Next, we notice that if $j\geq m+\D+[k]/2$, and \eqref{rok5.9} holds then the desired bound \eqref{rok5.5} follows. Indeed, we use the approximate finite speed of propagation argument as in the proof of \eqref{Alx24}. First we define $f^\mu_{j_1,k_1}, f^\nu_{j_2,k_2},f^\beta_{j_3,k_3}$ as in \eqref{Alx25.5}. Then we notice that the contribution in the case $\min(j_1,j_2,j_3)\geq 9j/10$ is suitably controlled, due to \eqref{rok5.7}. On the other and, if
\begin{equation*}
\min(j_1,j_2,j_3)\leq 9j/10,
\end{equation*} 
then we may assume that $j_1\leq 9j/10$ (using changes of variables) and it follows that the contribution is negligible, using integration by parts in $\xi$ as before. To summarize, in proving \eqref{rok5.5} we may assume that
\begin{equation}\label{rok5.10}
2m/3+[k]/2+\D^2\leq j\leq m+\D+[k]/2,\quad \max(j,[k])\leq 2m+2\D,\quad \overline{k}\leq 6m/N'_0.
\end{equation}

We define now the functions $f^\mu_{j_1,k_1}, f^\nu_{j_2,k_2},f^\beta_{j_3,k_3}$ as in \eqref{Alx25.5}. The contribution in the case $\max(j_1,j_2,j_3)\geq 2m/3$ can be bounded using \eqref{rok5.7}. On the other hand, if $\max(j_{1},j_{2},j_{3})\leq 2m/3$ then we can argue as in the proof of Lemma \ref{lLargeBBound} when $2^l\approx 1$. More precisely, we define
\begin{equation}\label{rok5.12}
g_1:=A_{\geq \D_1,\gamma_0}f^\mu_{j_1,k_1},\qquad g_2:=A_{\geq \D_1-10,\gamma_0}f^\nu_{j_2,k_2},\qquad A_{\geq \D_1-20,\gamma_0}f^\beta_{j_3,k_3}.
\end{equation} 
As in the proof of Lemma \ref{lLargeBBound}, see \eqref{hu20.5}--\eqref{hu22}, (and after inserting cutoff functions of the form $\varphi_{\leq l}(\eta)$ and $\varphi_{>l}(\eta)$, $l=m-\delta m$, to bound the other terms) for \eqref{rok5.5} it suffices to prove that
\begin{equation}\label{rok5.15}
2^{j-50\delta j}\big\|Q_{jk}C_m[g_1,g_2,g_3]\big\|_{L^2}\lesssim 2^{-\delta m}.
\end{equation}

In proving \eqref{rok5.15}, we may assume that $\max(j_1,j_2,j_3)\leq m/3$ and $m\leq L$ (otherwise we could use directly \eqref{rok5.7}) and that $k\geq -100$ (otherwise the contribution is negligible, by integrating by parts in $\eta$ and $\sigma$). Therefore, using \eqref{rok5.10}, we may assume that
\begin{equation}\label{rok5.16}
[k]\leq 100,\qquad m\leq L,\qquad 2m/3+\D^2\leq j\leq m+2\D,\qquad j_1,j_2,j_3\in[0,m/3].
\end{equation}

As in the proof of Lemma \ref{lLargeBBound}, we decompose the operator $C_m$ in dyadic pieces depending on the size of the modulation. More precisely, let
\begin{equation*}
\widehat{\mathcal{J}_{p}[f,g,h]}(\xi,s):=\int_{\mathbb{R}^2\times\mathbb{R}^2}e^{is\widetilde{\Phi}(\xi,\eta,\sigma)}\varphi_p(\widetilde{\Phi}(\xi,\eta,\sigma))n_0(\xi,\eta,\sigma)\widehat{f}(\xi-\eta,s)\widehat{g}(\eta-\sigma,s)\widehat{h}(\sigma,s)\,d\sigma d\eta.
\end{equation*}
Let $\mathcal{J}_{\leq p}=\sum_{q\leq p}\mathcal{J}_{q}$ and
\begin{equation*}
C_{m,p}[f,g,h]:=\int_{\mathbb{R}}q_m(s)\mathcal{J}_{l,p}[f,g,h](s)\,ds.
\end{equation*}
For $p\geq -2m/3$ we integrate by parts in $s$. As in {\bf{Step 1}} in the proof of Lemma \ref{lLargeBBound}, using also the $L^2$ bound \eqref{vd7}, it follows easily that
\begin{equation*}
2^{j-50\delta j}\sum_{p\geq -2m/3}\big\|P_kC_{m,p}[g_1,g_2,g_3]\big\|_{L^2}\lesssim 2^{-\delta m}.
\end{equation*}

To complete the proof of \eqref{rok5.15}, it suffices to show that
\begin{equation}\label{rok5.20}
2^{j-50\delta j}2^m\sup_{s\in I_m}\big\|Q_{jk}\mathcal{J}_{\leq -m/2}[g_1,g_2,g_3](s)\big\|_{L^2}\lesssim 2^{-\delta m}.
\end{equation}
Let $\kappa=2^{-m/3}$ and define the operators $\mathcal{J}_{\leq -m/2,\leq 0}$ and $\mathcal{J}_{\leq -m/2,l}$ by inserting the factors $\varphi(\kappa^{-1}\nabla_{\eta,\sigma}\widetilde{\Phi}(\xi,\eta,\sigma))$ and $\varphi_{l}(\kappa^{-1}\nabla_{\eta,\sigma}\widetilde{\Phi}(\xi,\eta,\sigma))$, $l\geq 1$, in the definition of the operators $\mathcal{J}_p$ above. The point is to observe that $|\nabla_\xi\widetilde{\Phi}(\xi,\eta,\sigma)|\leq 2^{-m/3+\D}$ in the support of the integral defining the operator $\mathcal{J}_{\leq -m/2,\geq 0}$, due to Lemma \ref{cubicphase} (i). Since $j\geq 2m/3+\D^2$, see \eqref{rok5.16}, the contribution of this operator is negligible, using integration by parts in $\xi$.

To estimate the operators $\mathcal{J}_{\leq -m/2,l}$ notice that we may insert a factor of $\varphi(2^{2m/3+l-\delta m}\eta)$, at the expense of a negligible error (due to Lemma \ref{tech5} (i)). To summarize, we define
\begin{equation*}
\begin{split}
\widehat{\mathcal{J}'_{\leq -m/2,l}[f,g,h]}&(\xi,s):=\int_{\mathbb{R}^2\times\mathbb{R}^2}e^{is\widetilde{\Phi}(\xi,\eta,\sigma)}\varphi_{l}(\kappa^{-1}\nabla_{\eta,\sigma}\widetilde{\Phi}(\xi,\eta,\sigma))\varphi_{\leq -m/2}(\widetilde{\Phi}(\xi,\eta,\sigma))\\
&\times \varphi(2^{2m/3+l-\delta m}\eta)n_0(\xi,\eta,\sigma)\widehat{f}(\xi-\eta,s)\widehat{g}(\eta-\sigma,s)\widehat{h}(\sigma,s)\,d\sigma d\eta,
\end{split}
\end{equation*}
and it remains to show that, for $l\geq 1$ and $s\in I_m$,
\begin{equation}\label{rok5.25}
2^{j-50\delta j}2^m\big\|Q_{jk}\mathcal{J}'_{\leq -m/2,l}[g_1,g_2,g_3](s)\big\|_{L^2}\lesssim 2^{-2\delta m}.
\end{equation}

Using $L^\infty$ estimates in the Fourier space, \eqref{rok5.25} follows when $l\geq m/3-\delta m$, since $2^j\lesssim 2^m$ (see \eqref{rok5.16}). On the other hand, if $l\leq m/3-\delta m$ then the operator is nontrivial only if 
\begin{equation*}
\widetilde{\Phi}(\xi,\eta,\sigma)=\Lambda(\xi)-\Lambda(\xi-\eta)-\Lambda_\nu(\eta-\sigma)+\Lambda_\nu(\sigma),\qquad\nu\in\{+,-\},
\end{equation*}
due to the smallness of $|\eta|$, $|\nabla_{\sigma}\widetilde{\Phi}(\xi,\eta,\sigma)|$, and $|\widetilde{\Phi}(\xi,\eta,\sigma)|$ (recall the support restrictions in \eqref{rok5.12}). In this case $|\nabla_\xi\widetilde{\Phi}(\xi,\eta,\sigma)|\leq 2^{-m/2}$ in the support of the integral, and the contribution is again negligible using integration by parts in $\xi$. This completes the proof of Proposition \ref{rok1}.

\section{Analysis of phase functions}\label{phacolle}

In this section we collect and prove some important facts about the phase functions $\Phi$.

\subsection{Basic properties}  Recall that 
\begin{equation}\label{ph1}
\begin{split}
&\Phi(\xi,\eta)=\Phi_{\sigma\mu\nu}(\xi,\eta)=\Lambda_\sigma(\xi)-\Lambda_{\mu}(\xi-\eta)-\Lambda_{\nu}(\eta),\qquad\sigma,\mu,\nu\in \{+,-\},\\
&\Lambda_\kappa(\xi)=\lambda_\kappa(|\xi|)=\kappa\lambda(|\xi|)=\kappa\sqrt{|\xi|+|\xi|^3}.
\end{split}
\end{equation}
We have
\begin{equation}\label{ph2}
\begin{split}
&\lambda'(x)=\frac{1+3x^2}{2\sqrt{x+x^3}},\qquad \lambda''(x)=\frac{3x^4+6x^2-1}{4(x+x^3)^{3/2}},\qquad\lambda'''(x)=\frac{3(1+5x^2-5x^4-x^6)}{8(x+x^3)^{5/2}}.
\end{split}
\end{equation}
Therefore
\begin{equation}\label{ph3}
\lambda''(x)\geq 0\,\text{ if }\,x\geq \gamma_0,\qquad \lambda''(x)\leq 0\,\text{ if }\,x\in[0,\gamma_0],\qquad \gamma_0:=\sqrt{\frac{2\sqrt{3}-3}{3}}\approx 0.393.
\end{equation}
It follows that
\begin{equation}\label{ph3.1}
\lambda(\gamma_0)\approx 0.674,\qquad \lambda'(\gamma_0)\approx 1.086,\qquad \lambda'''(\gamma_0)\approx 4.452,\qquad \lambda''''(\gamma_0)\approx -28.701. 
\end{equation}

Let $\gamma_1:=\sqrt{2}\approx 1.414$ denote the radius of the space-time resonant sphere, and notice that
\begin{equation}\label{kn1}
\lambda(\gamma_1)=\sqrt{3\sqrt 2}\approx 2.060,\quad \lambda'(\gamma_1)=\frac{7}{2\sqrt{3\sqrt 2}}\approx 1.699,\quad \lambda''(\gamma_1)=\frac{23}{4\sqrt{54\sqrt 2}}\approx 0.658.
\end{equation}

The following simple observation will be used many times: if $U_2\geq 1$, $\xi,\eta\in\mathbb{R}^2$, $\max(|\xi|,|\eta|,|\xi-\eta|)\leq U_2$,  $\min(|\xi|,|\eta|,|\xi-\eta|)=a\leq 2^{-10}U_2^{-1}$, then 
\begin{equation}\label{try5.5}
|\Phi(\xi,\eta)|\geq\lambda(a)-\sup_{b\in[a,U_2]}(\lambda(a+b)-\lambda(b))\geq\lambda(a)-a\max\{\lambda'(a),\lambda'(U_2+1)\}\geq\lambda(a)/4.
\end{equation}

\begin{lemma}\label{LambdaBasic}
(i) The function $\lambda'$ is strictly decreasing on the interval $(0,\gamma_0]$ and strictly increasing on the interval $[\gamma_0,\infty)$, and
\begin{equation}\label{zc8}
\lim_{x\to\infty}\Big[\lambda'(x)-\frac{3\sqrt{x}}{2}\Big]=0,\qquad \lim_{x\to 0}\Big[\lambda'(x)-\frac{1}{2\sqrt{x}}\Big]=0.
\end{equation}
The function $\lambda'$ is concave up on the interval $(0,1]$ and concave down on the interval $[1,\infty)$. For any $y>\lambda'(\gamma_0)$ the equation $\lambda'(r)=y$ has two solutions $r_1(y)\in(0,\gamma_0)$ and $r_2(y)\in(\gamma_0,\infty)$.

(ii) If $a\neq b\in(0,\infty)$ then
\begin{equation}\label{Bas1}
\lambda'(a)=\lambda'(b)\quad\text{ if and only if }\quad (a-b)^2=\frac{(3ab+1)(3a^2b^2+6ab-1)}{1-9ab}.
\end{equation}
In particular, if $a\neq b\in(0,\infty)$ and $\lambda'(a)=\lambda'(b)$ then $ab\in(1/9,\gamma_0^2]$.

(iii) Let $b:[\gamma_0,\infty)\to(0,\gamma_0]$ be the implicit function defined by $\lambda'(a)=\lambda'(b(a))$. Then $b$ is a smooth decreasing function and\footnote{In a neighborhood of $\gamma_0$, $\lambda'(x)$ behaves like $A+B(x-\gamma_0)^2-C(x-\gamma_0)^3$, where $A,B,C>0$. The asymptotics described in \eqref{Bas1.5}--\eqref{Bas1.6} are consistent with this behaviour.}
\begin{equation}\label{Bas1.5}
\begin{split}
&b'(a)\in[-1,-b(a)/a],\qquad a+b(a)\text{ is increasing on }[\gamma_0,\infty),\qquad b(a)\approx 1/a,\\
&-b'(a)\approx 1/a^2,\qquad b'(a)+1\approx (a-\gamma_0)/a.
\end{split}
\end{equation}
In particular, 
\begin{equation}\label{Bas1.7}
a+b(a)-2\gamma_0\approx \frac{(a-\gamma_0)^2}{a}.
\end{equation}
Moreover,
\begin{equation}\label{Bas1.6}
-[\lambda''(b(a))+\lambda''(a)]\approx a^{-1/2}(a-\gamma_0)^2.
\end{equation}

(iv) If $a,b\in(0,\infty)$ then
\begin{equation}\label{Bas2}
\lambda(a+b)=\lambda(a)+\lambda(b)\quad\text{ if and only if }\quad (a-b)^2=\frac{4+8ab-32a^2b^2}{9ab-4}.
\end{equation}
In particular, if $a, b\in(0,\infty)$ and $\lambda(a+b)=\lambda(a)+\lambda(b)$ then $ab\in[4/9,1/2]$. Moreover,
\begin{equation}\label{Bas3}
\begin{split}
&\text{ if }\,\,ab>1/2\,\,\text{ then }\,\,\lambda(a+b)-\lambda(a)-\lambda(b)>0,\\
&\text{ if }\,\,ab<4/9\,\,\text{ then }\,\,\lambda(a+b)-\lambda(a)-\lambda(b)<0.
\end{split}
\end{equation}
\end{lemma} 

\begin{proof} The conclusions (i) and (ii) follow from \eqref{ph2}--\eqref{ph3.1} by elementary arguments. For part (iii) we notice that, with $Y=ab$.
\begin{equation*}
(a+b(a))^2=F(Y):=\frac{-9Y^3-21Y^2-3Y+1}{9Y-1}+4Y=\frac{32/81}{9Y-1}-Y^2+14Y/9-49/81,
\end{equation*}
as a consequence of \eqref{Bas1}. Taking the derivative with respect to $a$ it follows that
\begin{equation}\label{Bas3.5}
2(a+b(a))(1+b'(a))=[ab'(a)+b(a)]F'(Y).
\end{equation}
Since $F'(Y)\leq-1/10$ for all $Y\in(1/9,\gamma_0^2]$, it follows that $b'(a)\in[-1,-b(a)/a]$ for all $a\in[\gamma_0,\infty)$. The claims in the first line of \eqref{Bas1.5} follow.

The claim $-b'(a)\approx 1/a^2$ follows from the identity $\lambda''(a)-\lambda''(b(a))b'(a)=0$. The last claim in \eqref{Bas1.5} is clear if $a-\gamma_0\gtrsim 1$; on the other hand, if $a-\gamma_0=\rho\ll 1$ then \eqref{Bas3.5} gives 
\begin{equation*}
-\frac{1+b'(a)}{b'(a)+b(a)/a}\approx 1,\qquad \gamma_0-b(a)\approx \rho.
\end{equation*} 
In particular $1-b(a)/a\approx \rho$ and the last conclusion in \eqref{Bas1.5} follows.

The claim in \eqref{Bas1.7} follows by integrating the approximate identity $b'(x)+1\approx (x-\gamma_0)/x$ between $\gamma_0$ and $a$. To prove \eqref{Bas1.6} we recall that $\lambda''(a)-\lambda''(b(a))b'(a)=0$. Therefore
\begin{equation*}
-[\lambda''(b(a))+\lambda''(a)]=-\lambda''(b(a))(1+b'(a))=\lambda''(a)\frac{1+b'(a)}{-b'(a)},
\end{equation*}
and the desired conclusion follows using also \eqref{Bas1.5}.

To prove (iv), we notice that \eqref{Bas2} and the claim that $ab\in[4/9,1/2]$ follow from \eqref{ph2}--\eqref{ph3.1} by elementary arguments. To prove \eqref{Bas3}, let $G(x):=\lambda(a+x)-\lambda(a)-\lambda(x)$. For $a\in(0,\infty)$ fixed we notice that $G(x)>0$ if $x$ is sufficiently large and $G(x)<0$ if $x>0$ is sufficiently small. The desired conclusion follows from the continuity of $G$. 
\end{proof}

\subsection{Resonant sets} We prove now an important proposition describing the geometry of resonant sets.

\begin{proposition}(Structure of resonance sets)\label{spaceres11} The following claims hold:

(i) There are functions $p_{++1}=p_{--1}:(0,\infty)\to (0,\infty)$, $p_{++2}=p_{--2}:[2\gamma_0,\infty)\to (0,\gamma_0]$, $p_{+-1}=p_{-+1}:(0,\infty)\to (\gamma_0,\infty)$ such that, if $\sigma,\mu,\nu\in\{+,-\}$ and $\xi\neq 0$ then
\begin{equation}\label{try1}
(\nabla_\eta\Phi_{\sigma\mu\nu})(\xi,\eta)=0\,\,\text{ if and only if }\,\,\eta\in P_{\mu\nu}(\xi):=\Big\{p_{\mu\nu k}(|\xi|)\frac{\xi}{|\xi|},\xi-p_{\mu\nu k}(|\xi|)\frac{\xi}{|\xi|}:\,k\in\{1,2\}\Big\}.
\end{equation}

(ii) (Space resonances) With $\mathcal{D}_{k,k_1,k_2}$ as in \eqref{Rset}, assume that
\begin{equation}\label{zc1}
(\xi,\eta)\in\mathcal{D}_{k,k_1,k_2}\qquad\text{ and }\qquad |(\nabla_\eta\Phi_{\sigma\mu\nu})(\xi,\eta)|\leq \eps_2\leq 2^{-\D_1}2^{k-\max(k_1,k_2)},
\end{equation}
for some constant $\mathcal{D}_1$ sufficiently large. Then $\big||k_1|-|k_2|\big|\leq 20$ and, for some $p\in P_{\mu\nu}(\xi)$\footnote{The set $P_{\mu\nu}(\xi)$ contains $2$ points if $(\mu,\nu)\in\{(+.-),(-,+)\}$ and at most $3$ points if $(\mu,\nu)\in\{(+.+),(-,-)\}$.},

$\bullet\,\,$ if $|k|\leq 100$ then $\max(|k_1|,|k_2|)\leq 200$ and
\begin{equation}\label{zc2}
\begin{cases}
\text{ either }&\mu=-\nu,\,\,\,\big|\eta-p\big|\lesssim \eps_2,\\
\text{ or }&\mu=\nu,\,\,\,\big|(\eta-p)\cdot\xi^\perp/|\xi|\big|\lesssim \eps_2,\text{ and }\big|(\eta-p)\cdot\xi/|\xi|\big|\lesssim\frac{\eps_2}{\eps_2^{2/3}+\big||\xi|-2\gamma_0\big|};
\end{cases}
\end{equation}

$\bullet\,\,$ if $k\leq -100$ then
\begin{equation}\label{zc2.1}
\begin{cases}
\text{ either }&\mu=-\nu,\,\,k_1,k_2\in[-10,10],\text{ and }\big|\eta-p\big|\lesssim \eps_22^{|k|},\\
\text{ or }&\mu=\nu,\,\,k_1,k_2\in[k-10,k+10],\text{ and }|\eta-\xi/2|\lesssim 2^{-3|k|/2}\eps_2;
\end{cases}
\end{equation}

$\bullet\,\,$ if $k\geq 100$ then
\begin{equation}\label{zc2.2}
\big|\eta-p\big|\lesssim \eps_22^{k/2}.
\end{equation}

(iii) (Space-time resonances) Assume that $(\xi,\eta)\in\mathcal{D}_{k,k_1,k_2}$,
\begin{equation}\label{try1.1}
|\Phi_{\sigma\mu\nu}(\xi,\eta)|\leq\eps_1\leq 2^{-\D_1}2^{\min(k,k_1,k_2,0)/2},\quad |(\nabla_\eta\Phi_{\sigma\mu\nu})(\xi,\eta)|\leq \eps_2\leq 2^{-\D_1}2^{k-\max(k_1,k_2)}2^{-2k^+}.
\end{equation}
 Then, with $\gamma_1:=\sqrt{2}$,
\begin{equation}\label{try1.2}
\pm(\sigma,\mu,\nu)=(+,+,+),\quad |\eta-p_{++1}(\xi)|=|\eta-\xi/2|\lesssim\eps_2,\quad \big||\xi|-\gamma_1\big|\lesssim \eps_1+\eps_2^2.
\end{equation}
\end{proposition}

\begin{proof} (i) We have
\begin{equation}\label{try2}
(\nabla_\eta\Phi_{\sigma\mu\nu})(\xi,\eta)=\mu\lambda'(|\xi-\eta|)\frac{\xi-\eta}{|\xi-\eta|}-\nu\lambda'(|\eta|)\frac{\eta}{|\eta|}.
\end{equation}
Assume that $\xi=\alpha e$ for some $\alpha\in(0,\infty)$ and $e\in\mathbb{S}^1$. In view of \eqref{try2}, $(\nabla_\eta\Phi_{\sigma\mu\nu})(\xi,\eta)=0$ if and only if
\begin{equation}\label{try3}
\eta=\beta e,\,\,\,\,\beta\in\mathbb{R}\setminus \{0,\alpha\},\,\,\,\,\mu\lambda'(|\alpha-\beta|)\,\mathrm{sgn}(\alpha-\beta)=\nu\lambda'(|\beta|)\,\mathrm{sgn}(\beta).
\end{equation}
We observe that it suffices to define the functions $p_{++1},p_{++2}$, and $p_{+-1}$ satisfying \eqref{try1}, since clearly $p_{--1}=p_{++1}$, $p_{--2}=p_{++2}$, and $p_{-+1}=p_{+-1}$. 

If $(\mu,\nu)=(+,+)$ then, as a consequence of \eqref{try3}, $\beta\in(0,\alpha)$ and $\lambda'(\alpha-\beta)=\lambda'(\beta)$. Therefore, according to Lemma \ref{LambdaBasic} (i)--(iii), there are two possible solutions,
\begin{equation}\label{try4}
\begin{split}
&\beta=p_{++1}(\alpha):=\alpha/2,\\
&\beta=p_{++2}(\alpha)\quad\text{ uniquely determined by }\lambda'(\beta)=\lambda'(\alpha-\beta)\text{ and }\beta\in(0,\gamma_0].
\end{split}
\end{equation}
The uniqueness of the point $p_{++2}(\alpha)$ is due to the fact that the function $x\to x+b(x)$ is increasing on $[\gamma_0,\infty)$, see \eqref{Bas1.5}. On the other hand, if $(\mu,\nu)=(+,-)$ then, as a consequence of \eqref{try3}, $\beta<0$ or $\beta>\alpha$ and $\lambda'(|\alpha-\beta|)=\lambda'(|\beta|)$. Therefore, according to Lemma \ref{LambdaBasic}, there is only one solution $\beta\geq\gamma_0$, 
\begin{equation}\label{try5}
\beta=p_{+-1}(\alpha)\quad\text{ uniquely determined by }\lambda'(\beta)=\lambda'(\beta-\alpha)\text{ and }\beta\in[\max(\alpha,\gamma_0),\alpha+\gamma_0].
\end{equation}
The conclusions in part (i) follow.

(ii) Assume that \eqref{zc1} holds and that $(\mu,\nu)\in\{(+,+),(+,-)\}$. Let $\xi=\alpha e$, $|e|=1$, $\alpha\in [2^{k-4},2^{k+4}]$, $\eta=\beta e+v$, $v\cdot e=0$, $(\beta^2+|v|^2)^{1/2}\in [2^{k_2-4},2^{k_2+4}]$. The condition $|(\nabla_\eta\Phi_{\sigma\mu\nu})(\xi,\eta)|\leq \eps_2$ gives, using \eqref{try2}, $\big||k_1|-|k_2|\big|\leq 20$ and 
\begin{equation}\label{try6}
\Big|\mu\lambda'(|\xi-\eta|)\frac{(\alpha-\beta)}{|\xi-\eta|}-\nu\lambda'(|\eta|)\frac{\beta }{|\eta|}\Big|\leq\eps_2,\quad \Big|-\mu\frac{\lambda'(|\xi-\eta|)}{|\xi-\eta|}-\nu\frac{\lambda'(|\eta|)}{|\eta|}\Big||v|\leq\eps_2.
\end{equation}
Since $\alpha\gtrsim 2^k$ and $|\xi-\eta|^{-1}\lambda'(|\xi-\eta|)\gtrsim 2^{|k_1|/2-k_1}$, the first inequality in \eqref{try6} shows that
\begin{equation*}
\Big|\mu\lambda'(|\xi-\eta|)\frac{-\beta}{|\xi-\eta|}-\nu\lambda'(|\eta|)\frac{\beta }{|\eta|}\Big|\gtrsim 2^{k+|k_1|/2-k_1}.
\end{equation*}
Since $1/|\beta|\geq 2^{-k_2-4}$, using also the second inequality in \eqref{try6} it follows that
\begin{equation}\label{try7}
|v|\lesssim \eps_22^{-k-|k_1|/2+k_1+k_2}
\end{equation}
and
\begin{equation*}
\Big|-\mu\frac{\lambda'(|\xi-\eta|)}{|\xi-\eta|}-\nu\frac{\lambda'(|\eta|)}{|\eta|}\Big|\gtrsim 2^{k+|k_1|/2-k_1-k_2}.
\end{equation*}
In particular $|v|\leq 2^{-20}2^{\min(k_1,k_2)}$, 
\begin{equation}\label{try8}
\big||\eta|-|\beta|\big|\lesssim \eps_2^22^{-2k-|k_1|+2k_1+k_2},\qquad\big||\xi-\eta|-|\alpha-\beta|\big|\lesssim \eps_2^22^{-2k-|k_1|+k_1+2k_2}.
\end{equation}
Using the first inequality in \eqref{try6} it follows that
\begin{equation}\label{try8.1}
\big|\mu\lambda'(|\alpha-\beta|)\mathrm{sgn}(\alpha-\beta)-\nu\lambda'(|\beta|)\mathrm{sgn}(\beta)\big|\leq \eps_2+C\eps_2^22^{-2k-|k_1|/2+2\max(k_1,k_2)}.
\end{equation}

{\bf{Proof of \eqref{zc2}.}} Assume first that $|k|\leq 100$. Then $\max(|k_1|,|k_2|)\leq 200$, since otherwise \eqref{try8.1} cannot hold (so there are no points $(\xi,\eta)$ satisfying \eqref{zc1}). The conclusion $\big|(\eta-p)\cdot\xi^\perp/|\xi|\big|\lesssim \eps_2$ in \eqref{zc2} follows from \eqref{try7}.

{\bf{Case ${\mathbf{1}}$.}} If $(\mu,\nu)=(+,-)$ then \eqref{try8.1} gives
\begin{equation*}
\big|\lambda'(|\alpha-\beta|)-\lambda'(|\beta|)\big|\leq 2\eps_2,\qquad \mathrm{sgn}(\alpha-\beta)+\mathrm{sgn}(\beta)=0.
\end{equation*}
Therefore either $\beta>\alpha$ and $|\lambda'(\beta-\alpha)-\lambda'(\beta)\big|\leq 2\eps_2$, in which case $\beta-\alpha<\gamma_0$, $\beta>\gamma_0$, and $|\beta-p_{+-1}(\alpha)|\lesssim \eps_2$, or $\beta<0$ and $|\lambda'(\alpha-\beta)-\lambda'(-\beta)\big|\leq 2\eps_2$, in which case $\alpha-\beta>\gamma_0$, $-\beta<\gamma_0$, and $|\alpha-\beta-p_{+-1}(\alpha)|\lesssim \eps_2$. The desired conclusion follows in the stronger form $|\eta-p|\lesssim \eps_2$.

{\bf{Case $\mathbf{2}$.}} If $(\mu,\nu)=(+,+)$ then \eqref{try8.1} gives
\begin{equation*}
\big|\lambda'(|\alpha-\beta|)-\lambda'(|\beta|)\big|\leq 2\eps_2,\qquad \mathrm{sgn}(\alpha-\beta)=\mathrm{sgn}(\beta).
\end{equation*}
Therefore
\begin{equation}\label{try8.2}
\beta\in(0,\alpha)\qquad\text{ and }\qquad\big|\lambda'(\alpha-\beta)-\lambda'(\beta)\big|\leq 2\eps_2.
\end{equation}
Assume $\alpha$ fixed and let $G(\beta):=\lambda'(\beta)-\lambda'(\alpha-\beta)$. The function $G$ vanishes when $\beta=\alpha/2$ or $\beta\in\{p_{++2}(\alpha),\alpha-p_{++2}(\alpha)\}$ (if $\alpha\geq 2\gamma_0$).

Assume that $\alpha=2\gamma_0+\rho\geq 2\gamma_0$, $\rho\in[0,2^{110}]$. Then, using Lemma \ref{LambdaBasic} (iii),
\begin{equation}\label{zc6}
p_{++2}(\alpha)\leq\gamma_0\leq\alpha/2\leq\alpha-p_{++2}(\alpha),\qquad \alpha/2-\gamma_0=\rho/2,\,\,\gamma_0-p_{++2}(\alpha)\approx\sqrt\rho,
\end{equation}
where the last conclusion follows from \eqref{Bas1.7} with $a=\alpha-p_{++2}(\alpha)$, $b(a)=p_{++2}(\alpha)$. Moreover, $|G'(\beta)|=|\lambda''(\beta)+\lambda''(\alpha-\beta)|\approx \rho$ if $\beta\in\{\alpha/2,p_{++2}(\alpha),\alpha-p_{++2}(\alpha)\}$, using \eqref{Bas1.6} and \eqref{zc6}. Also, $|G''(\beta)|=|\lambda'''(\beta)-\lambda'''(\alpha-\beta)|\lesssim \sqrt{\rho}$ if $|\beta-\alpha/2|\lesssim \sqrt{\rho}$, therefore
\begin{equation}\label{zc6.1}
|G'(\beta)|\approx \rho\,\,\text{ if }\,\,\beta\in I_\alpha:=\{x:\,\min\big(|x-\alpha/2|,|x-p_{++2}(\alpha)|,|x-\alpha+p_{++2}(\alpha)|\big)\leq\sqrt{\rho}/C_0\},
\end{equation}
for some large constant $C_0$.

If $\rho\leq C_0^4\eps_2^{2/3}$ then the points $\alpha/2,p_{++2}(\alpha),\alpha-p_{++2}(\alpha)$ are within distance $\leq C_0^4\eps_2^{1/3}$. In this case it suffices to prove that $|G(\beta)|\geq 3\eps_2$ if $|\beta-\alpha/2|\geq 2C_0^4\eps_2^{1/3}$. Assume, for contradiction, that this is not true, so there is $\beta\leq \gamma_0-C_0^4\eps_2^{1/3}$ such that $|\lambda'(\beta)-\lambda'(\alpha-\beta)|\leq 3\eps_2$. So there is $x$ close to $\beta$, $|x-\beta|\lesssim \eps_2^{2/3}$, such that $\lambda'(x)=\lambda'(\alpha-\beta)$. In particular, using \eqref{Bas1.7} with $a=\alpha-\beta$, $b(a)=x$, we have $\alpha-\beta+x-2\gamma_0\geq C_0^7\eps_2^{2/3}$. Therefore $\alpha-2\gamma_0\geq C_0^6\eps_2^{2/3}$, in contradiction with the assumption $\alpha-2\gamma_0=\rho\leq C_0^4\eps_2^{2/3}$.

Assume now that $\rho\geq C_0^4\eps_2^{2/3}$. In view of \eqref{zc6.1}, it suffices to prove that if $\beta\notin I_\alpha$ then $|G(\beta)|\geq 3\eps_2$. Assume, for contradiction, that this is not true, so there is $\beta\in(0,\alpha/2]\setminus I_\alpha$ such that $|\lambda'(\beta)-\lambda'(\alpha-\beta)|\leq 3\eps_2$. Since $\beta\leq\alpha/2-\sqrt{\rho}/C_0$, we may in fact assume that $\beta\leq\gamma_0-\sqrt{\rho}/(2C_0)$, provided that the constant $\D_1$ in \eqref{zc1} is sufficiently large. So there is $x$ close to $\beta$, $|x-\beta|\lesssim \eps_2C_0/\sqrt{\rho}$, such that $\lambda'(x)=\lambda'(\alpha-\beta)$. Using \eqref{Bas1.5}, it follows there is a point $y$ close to $x$, $|y-x|\lesssim \eps_2C_0^2/\rho$, such that $\lambda'(y)=\lambda'(\alpha-y)$. Therefore $y=p_{++2}(\alpha)$. In particular $|\beta-p_{++2}(\alpha)|\lesssim \eps_2C_0^2/\rho$, in contradiction with the assumption $\beta\notin I_\alpha$, so $|\beta-p_{++2}(\alpha)|\geq \sqrt{\rho}/C_0$ (recall that $\rho\geq C_0^4\eps_2^{2/3}$).

The case $\alpha=2\gamma_0-\rho\leq 2\gamma_0$ is easier, since there is only one point to consider, namely $\alpha/2$. As in \eqref{zc6.1}, $|G'(\beta)|\approx \rho$ if $|\beta-\alpha/2|\leq\sqrt{\rho}/C_0$. The proof then proceeds as before, by considering the two cases $\rho\leq C_0^4\eps_2^{2/3}$ and $\rho\geq C_0^4\eps_2^{2/3}$.

{\bf{Proof of \eqref{zc2.1}.}} Assume now that $k\leq -100$, so $|k_1-k_2|\leq 20$, and consider two cases:

{\bf{Case ${\mathbf{1}}$.}} Assume first that $(\mu,\nu)=(+,-)$. In view of \eqref{try2} we have
\begin{equation}\label{zc6.2}
\Big|\lambda'(|\eta|)\frac{\eta}{|\eta|}-\lambda'(|w|)\frac{w}{|w|}\Big|\leq\eps_2,\qquad\text{ where }w=\eta-\xi.
\end{equation}
If $\max(|\eta|,|w|)\leq \gamma_0-2^{-10}$ or $\min(|\eta|,|w|)\geq \gamma_0+2^{-10}$ it follows from \eqref{zc6.2} that $\big|\lambda'(|\eta|)-\lambda'(|w|)\big|\leq \eps_2$, therefore $\big||\eta|-|w|\big|\lesssim \eps_22^{-|k_1|/2+k_1}$. Therefore 
\begin{equation*}
\Big|\frac{\eta}{|\eta|}-\frac{w}{|w|}\Big|\lesssim \eps_22^{-|k_1|/2}\quad\text{ and }\quad \Big|\frac{1}{|\eta|}-\frac{1}{|w|}\Big|\lesssim \eps_22^{-|k_1|/2-k_1}.
\end{equation*}
As a consequence $|\eta-w|\lesssim \eps_22^{-|k_1|/2+k_1}$. On the other hand $|\eta-w|=|\xi|\gtrsim 2^{k}$, in contradiction with the assumption $\eps_2\leq 2^{-\D_1}2^{k-k_1}$. 

Since $|\eta-w|\leq 2^{-90}$ it remains to consider the case
\begin{equation}\label{zc6.3}
|\eta|,|\eta-\xi|\in[\gamma_0-2^{-9},\gamma_0+2^{-9}].
\end{equation}

In particular $k_1,k_2\in[-10,10]$, as claimed. Moreover $|v|\lesssim \eps_22^{|k|}$ as desired, in view of \eqref{try7}. The condition \eqref{try8.1} gives
\begin{equation*}
\big|\lambda'(|\alpha-\beta|)-\lambda'(|\beta|)\big|\leq \eps_2+C\eps_2^22^{-2k},\qquad \mathrm{sgn}(\alpha-\beta)+\mathrm{sgn}(\beta)=0.
\end{equation*}  
Without loss of generality, we may assume that 
\begin{equation}\label{zc9.2}
\beta>\alpha,\qquad \big|\lambda'(\beta-\alpha)-\lambda'(\beta)\big|\leq \eps_2+C\eps_2^22^{-2k}.
\end{equation}

Notice that $p_{+-1}(\alpha)\in(\gamma_0,\alpha+\gamma_0)$. We have two cases: if $\eps_2\geq 2^{-\D_1}2^{2k}$ then we need to prove that $|\beta-\gamma_0|\leq 2^{4\D_1}\eps_22^{|k|}$. This follows from \eqref{zc6.2}: otherwise, if $|\beta-\gamma_0|=d\geq 2^{4\D_1}\eps_22^{|k|}\geq 2^{3\D_1}2^k$ then $\big||\eta|-\gamma_0\big|\approx d$ and $\big||w|-\gamma_0\big|\approx d$, using also \eqref{try7}. As a consequence of \eqref{zc6.2}, we have $\big||\eta|-|w|\big|\lesssim \eps_2d^{-1}$, so
\begin{equation*}
\Big|\frac{\eta}{|\eta|}-\frac{w}{|w|}\Big|\lesssim \eps_2\qquad\text{ and }\qquad \Big|\frac{1}{|\eta|}-\frac{1}{|w|}\Big|\lesssim \eps_2d^{-1}.
\end{equation*}
Thus $|\eta-w|\lesssim \eps_2+\eps_2d^{-1}\lesssim \eps_2+2^{k-4\D_1}$, in contradiction with the fact that $|\eta-w|=|\xi|\gtrsim 2^k$.

On the other hand, if $\eps_2\leq 2^{-\D_1}2^{2k}$ then \eqref{zc9.2} gives $\big|\lambda'(\beta-\alpha)-\lambda'(\beta)\big|\leq 2\eps_2$ and $\beta\in (\gamma_0,\gamma_0+\alpha)$. Let $H(\beta):=\lambda'(\beta)-\lambda'(\beta-\alpha)$ and notice that 
\begin{equation*}
|H'(\beta)|\gtrsim |\beta-\gamma_0|+|\beta-\alpha-\gamma_0|\gtrsim 2^k
\end{equation*}
if $\beta$ is in this set. The desired conclusion follows since $H(p_{+-1}(\alpha))=0$. 

{\bf{Case $\mathbf{2}$.}} If $(\mu,\nu)=(+,+)$ then \eqref{try8.1} gives
\begin{equation*}
\big|\lambda'(\alpha-\beta)-\lambda'(\beta)\big|\leq \eps_2+C\eps_2^22^{-2k-|k_1|/2+2\max(k_1,k_2)},\qquad \beta\in (0,\alpha).
\end{equation*}
This shows easily that $k_1,k_2\in[k-10,k+10]$ and $|\alpha-2\beta|\lesssim 2^{-3|k|/2}\eps_2$. The desired conclusion follows using also \eqref{try7}.

{\bf{Proof of \eqref{zc2.2}.}} Assume now that $k\geq 100$ and consider two cases:

{\bf{Case $\mathbf{1}$.}} If $(\mu,\nu)=(+,-)$ then \eqref{try8.1} gives
\begin{equation*}
\big|\lambda'(|\alpha-\beta|)-\lambda'(|\beta|)\big|\leq \eps_2+C\eps_2^22^{-2k-|k_1|/2+2\max(k_1,k_2)},\qquad \mathrm{sgn}(\alpha-\beta)+\mathrm{sgn}(\beta)=0.
\end{equation*}
We may assume $\beta>\alpha$, $|\max(k_1,k_2)-k|\leq 20$, and $\big|\lambda'(\beta-\alpha)-\lambda'(\beta)\big|\leq 2\eps_2$. In particular $\beta\in(\alpha,\alpha+\gamma_0)$. Let $H(\beta):=\lambda'(\beta)-\lambda'(\beta-\alpha)$ as before and notice that $|H'(\beta)|\gtrsim 2^{3k/2}$ in this set. The desired conclusion follows since $H(p_{+-1}(\alpha))=0$, using also \eqref{try7}.

{\bf{Case $\mathbf{2}$.}} If $(\mu,\nu)=(+,+)$ then \eqref{try8.1} gives
\begin{equation}\label{zc9}
\big|\lambda'(\alpha-\beta)-\lambda'(\beta)\big|\leq \eps_2+C\eps_2^22^{-2k-|k_1|/2+2\max(k_1,k_2)},\qquad \beta\in (0,\alpha).
\end{equation}
If both $\beta$ and $\alpha-\beta$ are in $[\gamma_0,\infty)$ then \eqref{zc9} gives $|\beta-\alpha/2|\lesssim \eps_22^{k/2}$, which suffices (using also \eqref{try7}). Otherwise, assuming for example that $\beta\in(0,\gamma_0)$, it follows from \eqref{zc9} that $\beta\leq 2^{-k+20}$. Let, as before, $G(\beta):= \lambda'(\beta)-\lambda'(\alpha-\beta)$ and notice that $|G'(\beta)|\gtrsim 2^{3k/2}$ if $\beta\in(0,2^{-k+20}]$. The desired conclusion follows since $G(p_{++2}(\alpha))=0$, using also \eqref{try7}.

(iii) If $k\leq -100$ then $\Phi_{\sigma\mu\nu}(\xi,\eta)\gtrsim 2^{k/2}$, in view of \eqref{try5.5} and \eqref{zc2.1}, which is not not allowed by the condition on $\eps_1$.

If $k\geq 100$ and $(\mu,\nu)=(+,-)$ then $p_{+-1}(\alpha)-\alpha\leq 2^{-k+10}\leq 2^{k-10}\leq \alpha$ and
\begin{equation*}
|\Phi(\xi,\eta)|\geq \big|\pm\lambda(\alpha)-\lambda(p_{+-1}(\alpha))+\lambda(p_{+-1}(\alpha)-\alpha)\big|-C\eps_22^{k},
\end{equation*}
for some constant $C$ sufficiently large. Moreover, in view of Lemma \ref{LambdaBasic} (i), $\alpha(p_{+-1}(\alpha)-\alpha)\leq\gamma_0^2\leq 0.2$. In particular, using also Lemma \ref{LambdaBasic} (iv), $|\Phi(\xi,\eta)|\gtrsim 2^{-k/2}$, which is impossible in view of the assumption on $\eps_1$. A similar argument works also in the case $k\geq 100$ and $(\mu,\nu)=(+,+)$ to show that there are no points $(\xi,\eta)$ satisfying \eqref{try1.1}.

Finally, assume that $|k|\leq 100$, so $|k_1|,|k_2|\in[0,200]$. If $(\mu,\nu)=(+,-)$ then there are still no solutions $(\xi,\eta)$ of \eqref{try1.1}, using the same argument as before: in view of Lemma \ref{LambdaBasic} (i), $\alpha(p_{+-1}(\alpha)-\alpha)\leq\gamma_0^2\leq 0.2$, so $|\Phi(\xi,\eta)|\gtrsim 1$ as a consequence of Lemma \ref{LambdaBasic} (iv). 

On the other hand, if $(\mu,\nu)=(+,+)$ then we may also assume that $\sigma=+$. If $\beta$ is close to $p_{++2}(\alpha)$ or to $\alpha-p_{++2}(\alpha)$ then $\Phi(\xi,\eta)\gtrsim 1$, for the same reason as before. We are left with the case $|\beta-\alpha/2|\lesssim \eps_2$ and $\alpha\geq 1$. Therefore $|\eta-\xi/2|\lesssim\eps_2$. We notice now that the equation $\lambda(x)-2\lambda(x/2)=0$ has the unique solution $x=\sqrt{2}=:\gamma_1$, and the desired bound on $\big||\xi|-\gamma_1\big|$ follows since
\begin{equation*}
\big||\xi|-\gamma_1\big|\lesssim |\Phi_{\sigma\mu\nu}(\xi,\xi/2)|\lesssim |\Phi_{\sigma\mu\nu}(\xi,\eta)|+|\Phi_{\sigma\mu\nu}(\xi,\xi/2)-\Phi_{\sigma\mu\nu}(\xi,\eta)|\lesssim \eps_1+\eps_2^2.
\end{equation*}
This completes the proof of the proposition.
\end{proof}

\subsection{Bounds on sublevel sets} In this subsection we analyze the sublevel sets of the phase functions $\Phi$, and the interaction of these
sublevel sets with several other structures. We start with a general bound on the size of sublevel sets of functions, see \cite[Lemma 8.5]{DIP} for the proof.

\begin{lemma}\label{lemma00}
Suppose $L,R,M\in\mathbb{R}$, $M\geq \max(1,L,L/R)$, and $Y:B_R:=\{x\in\mathbb{R}^n:|x|<R\}\to\mathbb{R}$ is a function satisfying $\|\nabla Y\|_{C^{l}(B_R)}\leq M$, for some $l\geq 1$. Then, for any $\epsilon>0$,
\begin{equation}\label{scale1}
\big|\big\{x\in B_R:|Y(x)|\leq\epsilon\text{ and }\sum_{|\alpha|\leq l}|\partial_{x}^{\alpha}Y(x)|\geq L\big\}\big|\lesssim R^{n}ML^{-1-1/l}\epsilon^{1/l}.
\end{equation} 
Moreover, if $n=l=1$, $K$ is a union of at most $A$ intervals, and $|Y'(x)|\geq L$ on K, then
\begin{equation}\label{scale2}\left|\{x\in K:|Y(x)|\leq\epsilon\}\right|\lesssim AL^{-1}\epsilon.\end{equation}
\end{lemma}

We prove now several important bounds on the sets of time resonances. Assume $\Phi=\Phi_{\sigma\mu\nu}$, for some choice of $\sigma,\mu,\nu\in\{+.-\}$, and $\mathcal{D}_1$ is the large constant fixed in Proposition \ref{spaceres11}.

\begin{proposition}[Volume bounds of sublevel sets]\label{volume} 
Assume that $k,k_1,k_2\in\mathbb{Z}$, define $\mathcal{D}_{k,k_1,k_2}$ as in \eqref{Rset}, let $\overline{k}:=\max(k,k_1,k_2)$, and assume that 
\begin{equation}\label{Alx64.5}
\min(k,k_1,k_2)+\max(k,k_1,k_2)\geq -100.
\end{equation} 

(i) Let
\begin{equation*} 
E_{k,k_1,k_2;\eps}:=\{(\xi,\eta)\in\mathcal{D}_{k,k_1,k_2}:\,|\Phi(\xi,\eta)|\leq \epsilon\}.
\end{equation*}
Then
\begin{equation}\label{cas4} 
\begin{split}
&\sup_{\xi}\int_{\mathbb{R}^{2}}\mathbf{1}_{E_{k,k_1,k_2;\eps}}(\xi,\eta)\,d\eta\lesssim 2^{-\overline{k}/2}\epsilon\log(2+1/\epsilon)2^{4\min(k_1^+,k_2^+)},\\
&\sup_{\eta}\int_{\mathbb{R}^{2}}\mathbf{1}_{E_{k,k_1,k_2;\eps}}(\xi,\eta)\,d\xi\lesssim 2^{-\overline{k}/2}\epsilon\log(2+1/\epsilon)2^{4\min(k_1^+,k^+)}.
\end{split}
\end{equation}

(ii) Assume that $r_0\in [2^{-\D_1},2^{\D_1}]$, $\epsilon\leq 2^{\min(k,k_1,k_2,0)/2-\D_1}$, $\eps'\leq 1$ and let
\begin{equation*}
E'_{k,k_1,k_2;\eps,\eps'}=\{(\xi,\eta)\in \mathcal{D}_{k,k_1,k_2},|\Phi(\xi,\eta)|\leq \epsilon,\big||\xi-\eta|-r_0\big|\leq \epsilon'\}.
\end{equation*}
Then we can write $E'_{k,k_1,k_2;\eps,\eps'}=E'_{1}\cup E'_{2}$ such that
\begin{equation}\label{cas5.5} \sup_{\xi}\int_{\mathbb{R}^{2}}\mathbf{1}_{E_{1}'}(\xi,\eta)\,d\eta+\sup_{\eta}\int_{\mathbb{R}^{2}}\mathbf{1}_{E_{2}'}(\xi,\eta)\,d\xi\lesssim \epsilon\log(1/\eps)\cdot 2^{2\overline{k}}(\epsilon')^{1/2}.
\end{equation}

(iii) Assume that $\epsilon\leq 2^{\min(k,k_1,k_2,0)/2-\D_1}$, $\kappa\leq 1$, $p,q\leq 0$, and let
\begin{equation*}
E''_{k,k_1,k_2;\eps,\kappa}=\{(\xi,\eta)\in \mathcal{D}_{k,k_1,k_2},|\Phi(\xi,\eta)|\leq \epsilon,|(\Omega_\eta\Phi)(\xi,\eta)|\leq \kappa\}.
\end{equation*}
Then
\begin{equation}\label{Alx64.1}
\begin{split}
&\sup_{\xi}\int_{\mathbb{R}^{2}}\mathbf{1}_{E''_{k,k_1,k_2;\eps,\kappa}}(\xi,\eta)\varphi_{\geq q}(\nabla_\eta\Phi(\xi,\eta))\,d\eta\lesssim 2^{8\min(|k_1|,|k_2|)}\epsilon\log(1/\epsilon)\cdot\kappa 2^{-q}2^{2\overline{k}},\\
&\sup_{\eta}\int_{\mathbb{R}^{2}}\mathbf{1}_{E''_{k,k_1,k_2;\eps,\kappa}}(\xi,\eta)\varphi_{\geq p}(\nabla_\xi\Phi(\xi,\eta))\,d\xi\lesssim 2^{8\min(|k_1|,|k|)}\epsilon\log(1/\epsilon)\cdot\kappa 2^{-p}2^{2\overline{k}}.
\end{split}
\end{equation}
As a consequence, we can write $E''_{k,k_1,k_2;\eps,\kappa}=E''_{1}\cup E''_{2}$ such that
\begin{equation}\label{Alx64.2}
\sup_{\xi}\int_{\mathbb{R}^{2}}\mathbf{1}_{E''_{1}}(\xi,\eta)\,d\eta+\sup_{\eta}\int_{\mathbb{R}^{2}}\mathbf{1}_{E''_{2}}(\xi,\eta)\,d\xi\lesssim \epsilon\log(1/\epsilon)\cdot\kappa 2^{12\overline{k}}.
\end{equation}

Moreover, if $\kappa\leq 2^{-8\max(k,k_1,k_2)-\D_1}$ then
\begin{equation}\label{Alx64.3}
\begin{split}
&\sup_{\xi}\int_{\mathbb{R}^{2}}\mathbf{1}_{E''_{k,k_1,k_2;\eps,\kappa}}(\xi,\eta)\varphi_{\leq q}(\nabla_\eta\Phi(\xi,\eta))\,d\eta\lesssim \kappa 2^{q}2^{8\overline{k}},\\
&\sup_{\eta}\int_{\mathbb{R}^{2}}\mathbf{1}_{E''_{k,k_1,k_2;\eps,\kappa}}(\xi,\eta)\varphi_{\leq p}(\nabla_\xi\Phi(\xi,\eta))\,d\xi\lesssim \kappa 2^{p}2^{8\overline{k}}.
\end{split}
\end{equation}
\end{proposition}

\begin{proof} The condition \eqref{Alx64.5} is natural due to \eqref{try5.5}, otherwise $|\Phi(\xi,\eta)|\approx 2^{\min(k,k_1,k_2)/2}$ 
in $\mathcal{D}_{k,k_1,k_2}$. Compare also with the condition $\epsilon\leq 2^{\min(k,k_1,k_2,0)/2-\D_1}$ in (ii) and (iii).

(i) By symmetry, it suffices to prove the inequality in the first line of \eqref{cas4}. We may assume that $k_2\leq k_1$, so, using \eqref{Alx64.5},
\begin{equation}\label{Alx64.55}
k_1,\max(k,k_2)\in[\overline{k}-10,\overline{k}],\qquad k,k_2\geq -\overline{k}-100.
\end{equation}
Assume that $\xi=(s,0),\eta=(r\cos\theta,r\sin\theta)$, so
\begin{equation}\label{cas7.1}
-\Phi(\xi,\eta)=-\sigma\lambda(s)+\nu\lambda(r)+\mu\lambda((s^2+r^2-2sr\cos\theta)^{1/2})=:Z(r,\theta).
\end{equation}
We may assume that $\epsilon\leq 2^{\min(k,k_2)}2^{\overline{k}/2-\D_1}$. Notice that
\begin{equation}\label{car7.11}
\Big|\frac{d}{d\theta}Z(r,\theta)\Big|=\Big|\lambda'((s^2+r^2-2sr\cos\theta)^{1/2})\frac{sr\sin\theta}{(s^2+r^2-2sr\cos\theta)^{1/2}}\Big|.
\end{equation}

Assume that $|s-r|\geq 2^{\overline{k}-100}$, $s\in[2^{k-4},2^{k+4}]$, $r\in[2^{k_2-4},2^{k_2+4}]$. Then, for $r,s$ fixed,
\begin{equation}\label{car7.12}
\big|\{\theta\in[0,2\pi]:\,|Z(r,\theta)|\leq\epsilon\}\big|\lesssim\sum_{b\in\{0,1\}}\frac{\epsilon}{\sqrt{2^{\overline{k}/2}2^{\min(k,k_2)}(\epsilon+Z(r,b\pi))}}.
\end{equation}
Indeed, this follows from \eqref{car7.11} since in this case $\big|\partial_\theta Z(r,\theta)\big|\approx 2^{\min(k,k_2)}2^{\overline{k}/2}|\sin\theta|$ for all $\theta\in[0,2\pi]$. Next, we observe that
\begin{equation}\label{car7.13}
\big|\{r\in[2^{k_2-4},2^{k_2+4}]:\,|s-r|\geq 2^{\overline{k}-100}\,\,\text{ and }\,\,|Z(r,b\pi)|\leq \kappa 2^{\min(k,k_2)}2^{\overline{k}/2}\}\big|\lesssim \kappa 2^{k_2},
\end{equation}
provided that $\overline{k}\geq 200$ and $b\in\{0,1\}$. Indeed, in proving \eqref{car7.13} we may assume that $\kappa\leq 2^{-\D_1}$. Then we notice that the set in the left-hand side of \eqref{car7.13} is nontrivial only if 
\begin{equation*}
\begin{split}
\text{ either}\,\,&\pm Z(r,b\pi)=\lambda(s)-\lambda(s\pm r)\pm\lambda(r)\,\,\text{ and }\,\,s\in[2^{\overline{k}-10},2^{\overline{k}+10}],\,r\in[2^{-\overline{k}-10},2^{-\overline{k}+10}],\\
\text{ or}\,\,&\pm Z(r,b\pi)=\lambda(r)-\lambda(r\pm s)\pm\lambda(s)\,\,\text{ and }\,\,r\in[2^{\overline{k}-10},2^{\overline{k}+10}],\,s\in[2^{-\overline{k}-10},2^{-\overline{k}+10}].
\end{split}
\end{equation*}  
In all cases, the desired conclusion \eqref{car7.13} follows easily, since $|\partial_rZ(r,b\pi)|$ is suitably bounded away from $0$. Using also \eqref{car7.12} it follows that
\begin{equation}\label{car7.14}
\big|\{\eta:\,|\eta|\in[2^{k_2-4},2^{k_2+4}],\,\big||\xi|-|\eta|\big|\geq 2^{\overline{k}-100}\text{ and }\,|\Phi(\xi,\eta)|\leq\eps\}\big|\lesssim \eps 2^{-\overline{k}/2}2^{4k_2^+}
\end{equation}
provided that $|\xi|\in[2^{k-4},2^{k+4}]$, $\overline{k}\geq 200$, and \eqref{Alx64.55} holds. 

The case $\overline{k}\leq 200$ is easier. In this case we have $2^k,2^{k_1},2^{k_2}\approx 1$, due to \eqref{Alx64.55}. In view of Proposition \ref{spaceres11} (iii), if $|Z(r,b\pi)|\leq \kappa\leq 2^{-2\D_1}$ and $\big|\partial_rZ(r,b\pi)\big|\leq 2^{-2\D_1}$
then $s$ is close to $\gamma_1$, $r$ is close to $\gamma_1/2$, $b=0$. As a consequence $\big|\partial^2_rZ(r,b\pi)\big|\gtrsim 1$. It follows from Lemma \ref{lemma00} that
\begin{equation*}
\big|\{r\in[2^{k_2-4},2^{k_2+4}]:\,|s-r|\geq 2^{\overline{k}-100}\,\,\text{ and }\,\,|Z(r,b\pi)|\leq \kappa \}\big|\lesssim \kappa^{1/2},
\end{equation*}
provided that $\overline{k}\leq 200$ and $\kappa\in\mathbb{R}$. Using \eqref{car7.12} again it follows that 
\begin{equation}\label{car7.15}
\big|\{\eta:\,|\eta|\in[2^{k_2-4},2^{k_2+4}],\,\big||\xi|-|\eta|\big|\geq 2^{\overline{k}-100}\text{ and }\,|\Phi(\xi,\eta)|\leq\eps\}\big|\lesssim \eps \log(2+1/\eps)
\end{equation}
provided that $|\xi|\in[2^{k-4},2^{k+4}]$ and $\overline{k}\leq 200$.

Finally, we estimate the contribution of the set where $\big||\xi|-|\eta|\big|\leq 2^{\overline{k}-100}$. In this case we may assume that $k,k_1,k_2\geq \overline{k}-20$. We replace \eqref{car7.12} by
\begin{equation}\label{car7.18}
\big|\{\theta\in[2^{-\D_1},2\pi-2^{-\D_1}]:\,|Z(r,\theta)|\leq\epsilon\}\big|\lesssim\frac{\epsilon}{\sqrt{2^{3\overline{k}/2}(\epsilon+Z(r,\pi))}},
\end{equation}
which follows from \eqref{car7.11} (since $\big|\partial_\theta Z(r,\theta)\big|\approx 2^{3\overline{k}/2}|\sin\theta|$ for all $\theta\in[2^{-\D_1},2\pi-2^{-\D_1}]$). The proof proceeds as before, by analyzing the vanishing of the function $r\to Z(r,\pi)$ (it is in fact slightly easier since $|Z(r,\pi)|\gtrsim 2^{3\overline{k}/2}$ if $\overline{k}\geq 200$). It follows that
\begin{equation*}
\big|\{\eta:\,|\eta|\in[2^{k_2-4},2^{k_2+4}],\,\big||\xi|-|\eta|\big|\leq 2^{\overline{k}-100}\text{ and }\,|\Phi(\xi,\eta)|\leq\eps\}\big|\lesssim \eps \log(2+1/\eps)2^{\overline{k}/2}.
\end{equation*}
The desired bound in the first line of \eqref{cas4} follows using also \eqref{car7.14}--\eqref{car7.15}.

(ii) We may assume that $\min(k,k_2)\geq -2\D_1$ and that $\eps'\leq 2^{-\D_1^2}$.  Define
\begin{equation}\label{cas50}
\begin{split}
&E'_{1}:=\{(\xi,\eta)\in E'_{k,k_1,k_2;\eps,\eps'}:\,|\nabla_{\eta}\Phi(\xi,\eta)|\geq 2^{-20\D_1}\},\\
&E'_{2}:=\{(\xi,\eta)\in E'_{k,k_1,k_2;\eps,\eps'}:\,|\nabla_{\xi}\Phi(\xi,\eta)|\geq 2^{-20\D_1}\}.
\end{split}
\end{equation}
It is easy to see that $E'_{k,k_1,k_2;\eps,\eps'}=E'_1\cup E'_2$, using Proposition \ref{spaceres11} (ii). By symmetry, it suffices to prove \eqref{cas5.5} for the first term in the left-hand side. Let $\xi=(s,0)$, $\eta=(r\cos\theta,r\sin\theta)$, and 
\begin{equation}\label{cas51}
\begin{split}
E'_{1,\xi,1}:&=\{\eta:(\xi,\eta)\in E'_{1},\,|\sin\theta|\leq(\eps')^{1/2}2^{-2k_2}\},\\
E'_{1,\xi,2}:&=\{\eta:(\xi,\eta)\in E'_{1},\,|\sin\theta|\geq(\eps')^{1/2}2^{-2k_2}\}.
\end{split}
\end{equation}
It follows from Lemma \ref{lemma00} that $\big|E'_{1,\xi,1}\big|\lesssim \epsilon\cdot(\epsilon')^{1/2}$. Indeed, since $|\nabla_{\eta}\Phi(\xi,\eta)|\geq 2^{-20\D_1}$ and $|\sin\theta|\leq(\eps')^{1/2}2^{-2k_2}$, it follows from formula \eqref{cas7.1} that $|\partial_r[\Phi(\xi,\eta)]|\geq 2^{-21\D_1}$ in $E'_{1,\xi,1}$. The desired conclusion follows by applying Lemma \ref{lemma00} for every suitable angle $\theta$.

To estimate $\big|E'_{1,\xi,2}\big|$ we use the formula \eqref{cas7.1}. It follows from definitions that 
\begin{equation*}
E'_{1,\xi,2}\subseteq\{\eta:r\in[2^{k_2-4},2^{k_2+4}],\,\lambda(r)\in K_{s,r_0},\,|\sin\theta|\geq (\eps')^{1/2} 2^{-2k_2},\,|\Phi(\xi,\eta)|\leq \eps\},
\end{equation*}
where $K_{s,r_0}$ is an interval of length $\lesssim \eps'$ and $k_2\geq -2\D_1$. Therefore, using the formula \eqref{cas7.1} as before, $\big|E'_{1,\xi,2}\big|\lesssim 2^{2k_2}\eps(\eps')^{1/2}$, as desired.

(iii) For \eqref{Alx64.1} it suffices to prove the inequality in the first line. We may also assume that \eqref{Alx64.5} holds, and that $\kappa\leq 2^{q-2\max(k,k_1,k_2)-\D_1}$. Assume, as before, that $\xi=(s,0)$, $\eta=(r\cos\theta,r\sin\theta)$. Since $$|(\Omega_\eta\Phi)(\xi,\eta)|=\frac{\lambda'(|\xi-\eta|)}{|\xi-\eta|}|(\xi\cdot\eta^\perp)|,$$ the condition $|(\Omega_\eta\Phi)(\xi,\eta)|\leq\kappa$ gives
\begin{equation}\label{Alx64.6}
|\sin\theta|\lesssim \kappa 2^{k_1-k-k_2-|k_1|/2},
\end{equation}
in the support of the integral. The formula \eqref{cas7.1} shows that
\begin{equation*}
r^{-1}|\partial_\theta\Phi(\xi,\eta)|=\frac{\lambda'(|\xi-\eta|)}{|\xi-\eta|}s|\sin\theta|\lesssim \kappa 2^{-k_2}
\end{equation*}
in the support of the integral. Therefore $|\partial_r\Phi(\xi,\eta)|\geq 2^{q-4}$ in the support of the integral. 

We assume now that $\theta$ is fixed satisfying \eqref{Alx64.6}. If $||k_2|-|k_1||\geq 100$ then $|\partial_r\Phi(\xi,\eta)|\gtrsim 2^{|k_1|/2}+2^{|k_2|/2}$ for all $(\xi,\eta)\in\mathcal{D}_{k,k_1,k_2}$, and the desired bound follows from \eqref{scale1}, with $l=1$ and $n=1$. If $||k_2|-|k_1||\leq 100$ then we use still use \eqref{scale1} to conclude that the integral is dominated by
\begin{equation*}
C\eps 2^{-2q}2^{5|k_1|/2}\cdot \kappa 2^{k_1-k-|k_1|/2}\lesssim \eps\kappa 2^{-2q} 2^{4|k_1|}.
\end{equation*}
This suffices to prove \eqref{Alx64.1} if $2^q\geq 2^{-6\max(k,k_1,k_2)-\D_1}$. Finally, if
\begin{equation*}
||k_2|-|k_1||\leq 100,\quad 2^q\leq 2^{-6\max(k,k_1,k_2)-\D_1},\quad \kappa\leq 2^{q-2\max(k,k_1,k_2)-\D_1},
\end{equation*} 
then we would like to apply \eqref{scale2}. For this it suffices to verify that for any $\theta$ fixed satisfying \eqref{Alx64.6} the number of intervals (in the variable $r$) where $|\partial_r\Phi(\xi,\eta)|\leq 2^{q-4}$ is uniformly bounded. In view of Proposition \ref{spaceres11} (iii) these intervals are present only when $k,k_1,k_2\in[-10,10]$, $|s-\gamma_1|\ll 1$, $|r-\gamma_1/2|\ll 1$, and $\Phi(\xi,\eta)=\pm[\lambda(s)-\lambda(r)-\lambda( (s^2+r^2-2sr\cos\theta)^{1/2})]$. In this case, however $|\partial_r^2 \Phi(\xi,\eta)|\gtrsim 1$. As a consequence, for any $s$ and $\theta$ there is at most one interval in $r$ where $|\partial_r\Phi(\xi,\eta)|\leq 2^{q-4}$, and the desired bound follows from \eqref{scale2}. 

The decomposition \eqref{Alx64.2} follows from \eqref{Alx64.1} and Proposition \ref{spaceres11} (iii), by setting $2^p=2^q=2^{-2\D_1}2^{-2\max(k,k_1,k_2)}$. 

To prove the first inequality in \eqref{Alx64.3}, we may assume that $q\leq -5\max(k,k_1,k_2)-\D_1$ (due to \eqref{Alx64.6}). In view of 
Proposition \eqref{spaceres11} (iii) we may assume that $k,k_1,k_2\in[-10,10]$, $|s-\gamma_1|\ll 1$, $|r-\gamma_1/2|\ll 1$ and 
$\Phi(\xi,\eta)=\pm[\lambda(s)-\lambda(r)-\lambda( (s^2+r^2-2sr\cos\theta)^{1/2})]$. As before, $|\partial_r^2 \Phi(\xi,\eta)|\gtrsim 1$ in this case. 
As a consequence, for any $s$ and $\theta$ fixed, the measure of the set of numbers $r$ for which $|\partial_r\Phi(\xi,\eta)|\lesssim 2^q$ is 
bounded by $C2^q$, and the desired bound follows.
\end{proof}

We will also need a variant of Schur's lemma for suitably localized kernels.

\begin{lemma}\label{Shur2Lem}
Assume that $n,p\leq -\D/10$, $k,k_1,k_2\in\mathbb{Z}$, $l\leq\min(k,k_1,k_2,0)/2-\D/10$, $\rho_1,\rho_2\in\{\gamma_0,\gamma_1\}$. Then, with $\mathcal{D}_{k,k_1,k_2}$ as in \eqref{Rset}, and assuming that $\big\Vert \sup_{\omega\in\mathbb{S}^1}|\widehat{f}(r\omega)|\,\big\Vert_{L^2(rdr)}\leq 1$,
\begin{equation}\label{Shur2Lem1}
\Big\Vert \int_{\mathbb{R}^2}\mathbf{1}_{\mathcal{D}_{k,k_1,k_2}}(\xi,\eta)\varphi_l(\Phi(\xi,\eta))\varphi_n(|\xi-\eta|-\rho_1)\widehat{f}(\xi-\eta)\widehat{g}(\eta)\,d\eta\Big\Vert_{L^2_\xi}\lesssim 2^{(l+n)/2}\Vert g\Vert_{L^2},
\end{equation}
\begin{equation}\label{Shur2Lem3}
\begin{split}
\Big\Vert \int_{\mathbb{R}^2}\mathbf{1}_{\mathcal{D}_{k,k_1,k_2}}(\xi,\eta)\varphi_l(\Phi(\xi,\eta))&\varphi_n(|\xi-\eta|-\rho_1)\varphi_p(|\eta|-\rho_2)\widehat{f}(\xi-\eta)\widehat{g}(\eta)\,d\eta\Big\Vert_{L^2_\xi}\\
&\lesssim \min\{2^{l/2},2^{p/2}\}2^{(l+n)/2}\Vert g\Vert_{L^2},
\end{split}
\end{equation}
and
\begin{equation}\label{Shur2Lem2}
\Big\Vert \int_{\mathbb{R}^2}\mathbf{1}_{\mathcal{D}_{k,k_1,k_2}}(\xi,\eta)\varphi_l(\Phi(\xi,\eta))\widehat{f}(\xi-\eta)\widehat{g}(\eta)d\eta\Big\Vert_{L^2_\xi}\lesssim 2^{5|k_1|}2^{3l/4}(1+|l|)\Vert g\Vert_{L^2}.
\end{equation} 
\end{lemma}

\begin{proof} 
In view of \eqref{try5.5}, we may assume that $\min(k,k_1,k_2)+\overline{k}\geq -100$, where $\overline{k}=\max(k,k_1,k_2)$.

We start with \eqref{Shur2Lem1}. We may assume that $\min(k,k_1,k_2)\geq -200$. By Schur's test, it suffices to show that
\begin{equation}\label{Shur2suff}
\begin{split}
\sup_{\xi}\int_{\mathbb{R}^2}\mathbf{1}_{\mathcal{D}_{k,k_1,k_2}}(\xi,\eta)\varphi_l(\Phi(\xi,\eta))\varphi_n(|\xi-\eta|-\rho_1)|\widehat{f}(\xi-\eta)|\,d\eta&\lesssim 2^{(l+n)/2},\\
\sup_{\eta}\int_{\mathbb{R}^2}\mathbf{1}_{\mathcal{D}_{k,k_1,k_2}}(\xi,\eta)\varphi_l(\Phi(\xi,\eta))\varphi_n(|\xi-\eta|-\rho_1)|\widehat{f}(\xi-\eta)|\,d\xi &\lesssim 2^{(l+n)/2}.
\end{split}
\end{equation}
We focus on the first inequality. Fix $\xi\in\mathbb{R}^2$ and introduce polar coordinates, $\eta=\xi-r\omega$, $r\in(0,\infty)$, $\omega\in\mathbb{S}^1$.  The left-hand side is dominated by
\begin{equation*}
C\int_{\omega\in\mathbb{S}^1}\int_{2^{k_1-4}}^{2^{k_1+4}} \mathbf{1}_{\mathcal{D}_{k,k_1,k_2}}(\xi,\xi-r\omega)\varphi_l(\Phi(\xi,\xi-r\omega))\varphi_{n}(r-\rho_1)\vert \widehat{f}(r\omega)\vert rdrd\omega,
\end{equation*}
for a constant $C$ sufficiently large. Therefore it suffices to show that
\begin{equation}\label{Alx34}
\sup_{r,\xi}\int_{\omega\in\mathbb{S}^1}\mathbf{1}_{\mathcal{D}_{k,k_1,k_2}}(\xi,\xi-r\omega)\varphi_l(\Phi(\xi,\xi-r\omega))\,d\omega\lesssim 2^{l/2}2^{|k_1|/2},
\end{equation}
which is easily verified as in Proposition \ref{volume}, using the identity \eqref{cas7.1}. Indeed for $\xi$ and $r$ fixed, and letting $\omega=(\cos\theta,\sin\theta)$, the absolute value of the $d/d\theta$ derivative of the function $\theta\to \Phi(\xi,\xi-r(\cos\theta,\sin\theta))$ is bounded from below by $c|\sin\theta|2^{k+k_1-k_2}2^{|k_2|/2}\gtrsim |\sin\theta|2^{-|k_1|/2}$. The bound \eqref{Alx34} follows using also \eqref{scale2}. The second inequality in \eqref{Shur2suff} follows similarly.

We prove now \eqref{Shur2Lem3}. We may assume that $k,k_1,k_2\in[-80,80]$ and it suffices to show that
\begin{equation*}
\begin{split}
\sup_{\xi}\int_{\mathbb{R}^2}\mathbf{1}_{\mathcal{D}_{k,k_1,k_2}}(\xi,\eta)\varphi_l(\Phi(\xi,\eta))\varphi_n(|\xi-\eta|-\rho_1)\varphi_p(|\eta|-\rho_2)|\widehat{f}(\xi-\eta)|\,d\eta&\lesssim 2^{n/2}\min(2^l,2^p),\\
\sup_{\eta}\int_{\mathbb{R}^2}\mathbf{1}_{\mathcal{D}_{k,k_1,k_2}}(\xi,\eta)\varphi_l(\Phi(\xi,\eta))\varphi_n(|\xi-\eta|-\rho_1)\varphi_p(|\eta|-\rho_2)|\widehat{f}(\xi-\eta)|\,d\xi
&\lesssim 2^{l+n/2}.
\end{split}
\end{equation*}

We proceed as for \eqref{Shur2suff} but replace \eqref{Alx34} by
\begin{equation}\label{Alx34.1}
\begin{split}
\sup_{|\xi|\approx 1}\sup_{r}\int_{\omega\in\mathbb{S}^1}\varphi_l(\Phi(\xi,\xi-r\omega))\varphi_{n}(r-\rho_1)\varphi_p(|\xi-r\omega|-\rho_2)\,d\omega&\lesssim \min\{2^l,2^p\},\\
\sup_{\eta}\sup_{r}\int_{\omega\in\mathbb{S}^1}\varphi_l(\Phi(\eta+r\omega,\eta))\varphi_{n}(r-\rho_1)\varphi_p(|\eta|-\rho_2)\varphi_{\geq-90}(\eta+r\omega)\,d\omega&\lesssim 2^l.
\end{split}
\end{equation}
The bounds \eqref{Alx34.1} follow easily, using also the formula \eqref{cas7.1} to prove the $2^l$ bounds, once we notice that $|\sin\theta|\gtrsim 1$ in the support of the integrals. For this we only need to verify that the points $\xi$ and $\eta$ cannot be almost aligned; more precisely, we need to verify that if $\xi$ and $\eta$ are aligned then $|\Phi(\xi,\xi-\eta)|+\big||\xi-\eta|-\rho_2\big|+||\eta|-\rho_1|\gtrsim 1$. For this it suffices to notice that
\begin{equation*}
\big|\pm\lambda(|\xi|)\pm\lambda(\rho_1)\pm\lambda(\rho_2)\big|\gtrsim 1\qquad\text{ if }\qquad |\xi|\gtrsim 1\text{ and }\pm|\xi|\pm\rho_1\pm\rho_2=0.
\end{equation*}
Recalling that $\rho_1,\rho_2\in\{\gamma_0,\gamma_1\}$, it suffices to verify that $\lambda(2\gamma_0)-2\lambda(\gamma_0)\neq 0$, $\lambda(2\gamma_1)-2\lambda(\gamma_1)\neq 0$, $\lambda(\gamma_0+\gamma_1)-\lambda(\gamma_0)-\lambda(\gamma_1)\neq 0$, $\lambda(-\gamma_0+\gamma_1)+\lambda(\gamma_0)-\lambda(\gamma_1)\neq 0$. These claims follow from Lemma \ref{LambdaBasic} (iv), since the numbers $\gamma_0^2, \gamma_1^2$, $\gamma_0\gamma_1$, and $\gamma_0(\gamma_1-\gamma_0)$ are not in the interval $[4/9,1/2]$.

We now turn to \eqref{Shur2Lem2}. By Schur's lemma it suffices to show that
\begin{equation}\label{Alx30}
\begin{split}
\sup_{\xi}\int_{\mathbb{R}^2}\varphi_l(\Phi(\xi,\eta))\mathbf{1}_{\mathcal{D}_{k,k_1,k_2}}(\xi,\eta)\vert \widehat{f}(\xi-\eta)\vert \,d\eta & \lesssim 2^{5|k_1|}2^{3l/4}(1+|l|),\\
\sup_\eta\int_{\mathbb{R}^2}\varphi_l(\Phi(\xi,\eta))\mathbf{1}_{\mathcal{D}_{k,k_1,k_2}}(\xi,\eta)\vert \widehat{f}(\xi-\eta)\vert \,d\xi& \lesssim 2^{5|k_1|}2^{3l/4}(1+|l|).
\end{split}
\end{equation}
We show the first inequality. Introducing polar coordinates, as before, we estimate
\begin{equation*}
\begin{split}
&\int_{\mathbb{R}^2} \varphi_l(\Phi(\xi,\xi-r\omega))\mathbf{1}_{\mathcal{D}_{k,k_1,k_2}}(\xi,\xi-r\omega)\vert \widehat{f}(r\omega)\vert \,rdr d\omega\\
&\lesssim \big\Vert\sup_\omega |\widehat{f}(r\omega)|\,\big\Vert_{L^2(rdr)}\Big\Vert \int_{\mathbb{S}^1}\varphi_l(\Phi(\xi,\xi-r\omega))\mathbf{1}_{\mathcal{D}_{k,k_1,k_2}}(\xi,\xi-r\omega)\,d\omega\Big\Vert_{L^2(rdr)}\\
&\lesssim \Vert \varphi_{\leq l+2}(\Phi(\xi,\xi-\eta))\mathbf{1}_{\mathcal{D}_{k,k_1,k_2}}(\xi,\xi-\eta)\Vert_{L^2_\eta} \big\|\varphi_{\leq l+2}(\Phi(\xi,\xi-r\omega))\mathbf{1}_{\mathcal{D}_{k,k_1,k_2}}(\xi,\xi-r\omega)\big\|_{L^\infty_rL^2_\omega}\\
&\lesssim 2^{5|k_1|}2^{3l/4}(1+|l|),
\end{split}
\end{equation*}
using Proposition \ref{volume} (i) and \eqref{Alx34}. The second inequality in \eqref{Alx30} follows similarly.
\end{proof}

\subsection{Iterated resonances}\label{cubph}

In this subsection we prove a lemma concerning some properties of the cubic phases
\begin{equation}\label{ub1}
\widetilde{\Phi}(\xi,\eta,\sigma)=\widetilde{\Phi}_{+\mu\beta\gamma}(\xi,\eta,\sigma)=\Lambda(\xi)-\Lambda_\mu(\xi-\eta)-\Lambda_\beta(\eta-\sigma)-\Lambda_\gamma(\sigma).
\end{equation}
These properties are used only in the proof of Lemma \ref{lLargeBBound} and Lemma \ref{lsmallBound}.

\begin{lemma}\label{cubicphase} 
(i) Assume that $\xi,\eta,\sigma\in\mathbb{R}^2$ satisfy
\begin{equation}\label{phaseassump1}
\max(||\xi-\eta|-\gamma_{0}|,||\eta-\sigma|-\gamma_{0}|,||\sigma|-\gamma_{0}|)\leq 2^{-\D_1/2},
\end{equation} 
and
\begin{equation}\label{phaseassump2}
|\nabla_{\eta,\sigma}\widetilde{\Phi}(\xi,\eta,\sigma)|\leq \kappa_{1}\leq 2^{-4\D_1}.
\end{equation}
Then, for $\nu\in\{+,-\}$,
\begin{equation}\label{ub2}
\Lambda(\xi)-\Lambda_\mu(\xi-\eta)-\Lambda_\nu(\eta)\gtrsim |\eta|.
\end{equation}
Moreover,
\begin{equation}\label{phaseresult}
\text{ if }\,\,\,\,|\nabla_{\xi}\widetilde{\Phi}(\xi,\eta,\sigma)|\geq \kappa_{2}\geq 2^{\D_1}\kappa_1\,\,\,\,\text{ then }\,\,\,\,|\widetilde{\Phi}(\xi,\eta,\sigma)|\gtrsim \kappa_{2}^{3/2}.
\end{equation}

(ii) Assume that $\xi,\eta,\sigma\in\mathbb{R}^2$ satisfy $|\xi-\eta|,\,|\eta-\sigma|,\,|\sigma|\in [2^{-10},2^{10}]$ and
\begin{equation}\label{kro2}
\begin{split}
&|\Phi_{+\mu\nu}(\xi,\eta)|=|\Lambda(\xi)-\Lambda_\mu(\xi-\eta)-\Lambda_\nu(\eta)|\leq 2^{-2\D_1},\\
&|\Phi_{\nu\beta\gamma}(\eta,\sigma)|=|\Lambda_\nu(\eta)-\Lambda_\beta(\eta-\sigma)-\Lambda_\gamma(\sigma)|\leq 2^{-2\D_1}.
\end{split}
\end{equation}
If 
\begin{equation}\label{kro3}
|\nabla_{\eta,\sigma}\widetilde{\Phi}(\xi,\eta,\sigma)|\leq \kappa\leq 2^{-4\D_1}
\end{equation}
then
\begin{equation}\label{kro4}
\mu=-,\,\nu=\beta=\gamma=+,\qquad|\eta-2\sigma|+|\xi-\sigma|\lesssim \kappa,\qquad|\nabla_{\xi}\widetilde{\Phi}(\xi,\eta,\sigma)|\lesssim \kappa.
\end{equation}
\end{lemma}

\begin{proof} (i) If \eqref{phaseassump1} and \eqref{phaseassump2} hold then the vectors $\xi-\eta,\eta-\sigma,\sigma$ are almost aligned. Thus either $|\eta|\leq 2^{-\D_1/2+10}$ or $||\eta|-2\gamma_{0}|\leq 2^{-\D_1/2+10}$. We will assume that we are in the second case, $||\eta|-2\gamma_{0}|\leq 2^{-\D_1/2+10}$ (the other case is similar, in fact slightly easier because the inequality \eqref{ub2} is a direct consequence of \eqref{try5.5}). Therefore either $||\xi|-3\gamma_0|\leq 2^{-\D_1/2+20}$, and in this case the desired conclusions are trivial, or $||\xi|-\gamma_0|\leq 2^{-\D_1/2+20}$. In this case \eqref{ub2} follows since $|\lambda(\gamma_0)\pm\lambda(\gamma_0)\pm\lambda(2\gamma_0)|\gtrsim 1$; it remains to prove \eqref{phaseresult} in the case $\mu=-,\beta=\gamma=+$,
\begin{equation}\label{ub5}
\begin{split}
&\widetilde{\Phi}(\xi,\eta,\sigma)=\Lambda(\xi)+\Lambda(\xi-\eta)-\Lambda(\eta-\sigma)-\Lambda(\sigma),\\
&||\eta|-2\gamma_{0}|\leq 2^{-\D_1/2+20},\qquad ||\xi|-\gamma_0|\leq 2^{-\D_1/2+20}.
\end{split}
\end{equation} 

In view of \eqref{phaseassump2}, the angle between any two of the vectors $\{\xi-\eta,\eta-\sigma,\sigma\}$ is either $O(\kappa_{1})$ or $\pi+O(\kappa_{1})$. Given $\sigma= ze$ for some $e\in\mathbb{S}^{1}$, we write $\eta= y e+\eta'$, $\xi= x e +\xi'$, with $e\cdot \eta'=e\cdot\xi'=0$ and $|\eta'|+|\xi'|\lesssim \kappa_1$. Notice that $|\widetilde{\Phi}(\xi,\eta,\sigma)-\widetilde{\Phi}(xe,ye,ze)|\lesssim \kappa_1^2$. Therefore, we may assume that 
\begin{equation}\label{phaseassump4}
\begin{split}
&|x-\gamma_0|+|y-2\gamma_0|+|z-\gamma_0|\leq 2^{-\D_1/2+30},\\
&|\lambda'(y-z)-\lambda'(z)|\leq 2\kappa_{1},\quad|\lambda'(y-x)-\lambda'(y-z)|\leq 2\kappa_{1},\quad|\lambda'(x)-\lambda'(y-x)|\geq\kappa_{2}/2,
\end{split}
\end{equation}
and it remains to prove that
\begin{equation}\label{ub6}
|\widetilde{\Phi}(xe,ye,ze)|=|\lambda(x)+\lambda(y-x)-\lambda(y-z)-\lambda(z)|\gtrsim \kappa_2^{3/2}.
\end{equation}

Let $z'\neq z$ denote the unique solution to the equation $\lambda'(z')=\lambda'(z)$, and let $d:=|z-\gamma_{0}|$. Then $|z'-\gamma_{0}|\approx d$, in view of \eqref{Bas1.7}. Moreover $d\geq \sqrt{\kappa_1}$; otherwise $|y-z-\gamma_0|\lesssim\sqrt{\kappa_1}$, $|y-x-\gamma_0|\lesssim\sqrt{\kappa_1}$, so $|x-\gamma_0|\lesssim\sqrt{\kappa_1}$, in contradiction with the assumption  $|\lambda'(x)-\lambda'(y-x)|\geq\kappa_{2}/2$. Moreover,  \begin{equation}\label{phaseassump5}
\text{ there are }\sigma_{1},\sigma_{2}\in\{z,z'\}\text{ such that }|y-z-\sigma_{1}|+|y-x-\sigma_{2}|\lesssim \kappa_{1}/d.
\end{equation} 
In fact, we may assume $d\geq 2^{-\D_1/4}\kappa_{2}^{1/2}$, since otherwise $|x-\gamma_{0}|+|y-x-\gamma_{0}|\lesssim d$, and hence $|\lambda'(x)-\lambda'(y-x)|\lesssim d^{2}$, which contradicts \eqref{phaseassump2}.

Now we must have $\sigma_{1}=z$; in fact, if $\sigma_{1}=z'$, then $x=z+z'-\sigma_{2}+O(\kappa_{1}/d)$, thus \[|\lambda'(x)-\lambda'(\sigma_{2})|\lesssim \kappa_{1},\] 
which again contradicts \eqref{phaseassump4}. Similarly $\sigma_{2}=z'$. Therefore
\begin{equation}\label{ub7}
y=2z+O(\kappa_{1}/d),\quad x=2z-z'+O(\kappa_{1}/d),\quad y-x=z'+O(\kappa_{1}/d).
\end{equation}

We expand the function $\lambda$ at $\gamma_{0}$ in its Taylor series
\[\lambda(v)=\lambda(\gamma_{0})+c_{1}(v-\gamma_{0})+c_{3}(v-\gamma_{0})^{3}+O(v-\gamma_{0})^{4},\] 
where $c_{1},c_{3}\neq 0$. 
Using \eqref{ub7} we have
\begin{equation*}
\begin{split}
\widetilde{\Phi}(xe,ye,ze)&=c_{3}[(x-\gamma_{0})^{3}+(y-x-\gamma_{0})^{3}-(z-\gamma_{0})^{3}-(y-z-\gamma_{0})^{3}]+O(d^{4})\\
&=c_{3}[(2z-z'-\gamma_{0})^{3}+(z'-\gamma_{0})^{3}-2(z-\gamma_{0})^{3}]+O(d^{4}+\kappa_{1}d).
\end{split}
\end{equation*}
In view of \eqref{Bas1.7}, $z+z'-2\gamma_0=O(d^2)$. Therefore $\widetilde{\Phi}(xe,ye,ze)=24(z-\gamma_{0})^{3}+O(d^{4}+\kappa_{1}d)$ which shows that $|\widetilde{\Phi}(xe,ye,ze)|\gtrsim d^{3}$. The desired conclusion \eqref{ub6} follows.

(ii) The conditions $|\Phi_{\nu\beta\gamma}(\eta,\sigma)|\leq 2^{-2\D_1}$ and $|(\nabla_\sigma\Phi_{\nu\beta\gamma})(\eta,\sigma)|\leq \kappa$ show that $\eta$ corresponds to a space-time resonance output. It follows from Lemma \ref{spaceres11} (iii) that
\begin{equation}\label{kro5}
|\eta-ye|+|\sigma-ye/2|\lesssim \kappa,\quad |y-\gamma_1|\lesssim 2^{-2\D_1},\quad \nu=\beta=\gamma,
\end{equation}
 for some $e\in\mathbb{S}^1$. Let $b\approx 0.207$ denote the unique nonnegative number $b\neq\gamma_1/2$ with the property that $\lambda'(b)=\lambda'(\gamma_1/2)$. The condition $|\nabla_{\eta}\widetilde{\Phi}(\xi,\eta,\sigma)|\leq \kappa$ shows that $\xi-\eta$ is close to one of the vectors $(\gamma_1/2)e,-(\gamma_1/2)e,be,-be$. However, $\lambda(b)\approx 0.465$, $\lambda(\gamma_1+b)\approx 2.462$, $\lambda(\gamma_1-b)\approx 1.722$, $\lambda(\gamma_1)\approx 2.060$. Therefore, the condition $|\Phi_{+\mu\nu}(\xi,\eta)|\leq 2^{-2\D_1}$ prevents $\xi-\eta$ from being close to one of the vectors $be$ or $-be$. Similarly $\xi-\eta$ cannot be close to the vector $ (\gamma_1/2)e$, since $\lambda(\gamma_1/2)\approx 1.030$, $\lambda(3\gamma_1/2)\approx 3.416$. It follows that $|(\xi-\eta)+(\gamma_1/2)e|\lesssim 2^{-2\D_1}$, $||\xi|-\gamma_1/2|\lesssim 2^{-2\D_1}$, $\mu=-$, $\nu=+$. The condition $|\nabla_{\eta}\widetilde{\Phi}(\xi,\eta,\sigma)|\leq \kappa$ then gives $|(\eta-\xi)-(\eta-\sigma)|\lesssim \kappa$, and remaining bounds in \eqref{kro4} follow using also \eqref{kro5}. 
\end{proof}

\section{The functions $\Upsilon$}\label{upsilon}

The analysis in the proofs of the crucial $L^2$ lemmas in section \ref{L2proof} depends on understanding the properties of the functions
$\Upsilon: \mathbb{R}^2\times \mathbb{R}^2 \to \mathbb{R}$,
\begin{equation}\label{ph6}
\Upsilon(\xi,\eta):=(\nabla^2_{\xi,\eta}\Phi)(\xi,\eta)\big[(\nabla^\perp_\xi\Phi)(\xi,\eta),(\nabla^\perp_\eta\Phi)(\xi,\eta)\big].
\end{equation}

We calculate
\begin{equation}\label{ph4}
\begin{split}
&(\nabla_\eta\Phi)(\xi,\eta)=-\lambda'_\nu(|\eta|)\frac{\eta}{|\eta|}+\lambda'_\mu(|\xi-\eta|)\frac{\xi-\eta}{|\xi-\eta|},\\
&(\nabla_\xi\Phi)(\xi,\eta)=\lambda'_\sigma(|\xi|)\frac{\xi}{|\xi|}-\lambda'_\mu(|\xi-\eta|)\frac{\xi-\eta}{|\xi-\eta|},
\end{split}
\end{equation}
and
\begin{equation}\label{ph5}
\begin{split}
(\nabla^2_{\xi,\eta}\Phi)(\xi,\eta)[\partial_i,\partial_j]&=\lambda''_\mu(|\xi-\eta|)\frac{(\xi_i-\eta_i)(\xi_j-\eta_j)}{|\xi-\eta|^2}\\
&+\lambda'_\mu(|\xi-\eta|)\frac{\delta_{ij}|\xi-\eta|^2-(\xi_i-\eta_i)(\xi_j-\eta_j)}{|\xi-\eta|^3}.
\end{split}
\end{equation}

Using these formulas and the identity $(v\cdot w^\perp)^2+(v\cdot w)^2=|v|^2|w|^2$ we calculate
\begin{equation}\label{ph5.1}
\begin{split}
-\Upsilon(\xi,\eta)&=\frac{\lambda''_\mu(|z|)}{|z|^2}\frac{\lambda'_\sigma(|\xi|)}{|\xi|}\frac{\lambda'_\nu(|\eta|)}{|\eta|}(\eta\cdot\xi^{\perp})^2\\
&+\frac{\lambda'_\mu(|z|)}{|z|^3}\Big\{\lambda'_\mu(|z|)|z|-\frac{\lambda'_\sigma(|\xi|)}{|\xi|}\xi\cdot z\Big\}\Big\{\lambda'_\mu(|z|)|z|-\frac{\lambda'_\nu(|\eta|)}{|\eta|}\eta\cdot z\Big\},
\end{split}
\end{equation}
where $z:=\xi-\eta$. We define also the normalized function $\hat{\Upsilon}$,
\begin{equation}\label{hatUps}
\hat{\Upsilon}(\xi,\eta):=\frac{\Upsilon(\xi,\eta)}{|(\nabla_\xi\Phi)(\xi,\eta)|\cdot |(\nabla_\eta\Phi)(\xi,\eta)|}.
\end{equation}

We consider first the case of large frequencies:

\begin{lemma}\label{Geomgamma0}

Assume that $\sigma=\nu=+$, $k\ge \D_1$, and $p-k/2\le -\D_1$.

(i) Assume that
\begin{equation}\label{AssHM}
\vert \Phi(\xi,\eta)\vert\le 2^p,\qquad \vert \xi\vert,\vert\eta\vert\in[2^{k-2},2^{k+2}],\qquad 2^{-20}\le  \vert \xi-\eta\vert\le 2^{20}.
\end{equation}
Let $z:=\xi-\eta$. Then, with $p^+=\max(p,0)$,
\begin{equation}\label{Orthoxx-y}
\frac{|\xi\cdot\eta^\perp|}{|\xi||\eta|}\approx 2^{-k},\qquad \frac{|\xi\cdot z|}{|\xi||z|}+\frac{|\eta\cdot z|}{|\eta||z|}\lesssim 2^{p^+-k/2}.
\end{equation}
Moreover, we can write
\begin{equation}\label{ModUp}
\begin{split}
&-\mu \Upsilon(\xi,\eta)=\lambda''(\vert z\vert)A(\xi,\eta)+B(\xi,z)B(\eta,z),\\
&\vert A(\xi,\eta)\vert\gtrsim 2^{k},\quad \| D^\alpha A\|_{L^\infty}\lesssim_\alpha 2^k,\quad \|B\|_{L^\infty}\lesssim 2^{p^+},\quad \| D^\alpha B\|_{L^\infty}\lesssim_\alpha 2^{k/2}.
\end{split}
\end{equation}

(ii) Assume that $z=(\rho\cos\theta,\rho\sin\theta)$, $|\rho|\in[2^{-20},2^{20}]$. There exists functions $\theta^1=\theta^1_{|\xi|,\mu}$ and $\theta^2=\theta^2_{|\eta|,\mu}$ such that,
\begin{equation}\label{Deftheta1}
\begin{split}
&\text{if}\quad 2^{k-2}\le \vert \xi\vert\le 2^{k+2}\,\,\text{ and }\,\,\vert \Phi(\xi,\xi-z)\vert\le 2^p\,\,\text{ then }\,\,\min_{\mp}\vert\theta-\arg(\xi)\mp\theta^1(\rho)\vert\lesssim 2^{p-k/2},\\
&\text{if}\quad 2^{k-2}\le \vert \eta\vert\le 2^{k+2}\,\,\text{ and }\,\,\vert \Phi(\eta+z,\eta)\vert\le 2^p\,\,\text{ then }\,\,\min_{\mp}\vert\theta-\arg(\eta)\mp\theta^2(\rho)\vert\lesssim 2^{p-k/2}.
\end{split}
\end{equation}
Moreover
\begin{equation}\label{Deftheta2}
|\theta^1(\rho)-\pi/2|+|\theta^2(\rho)-\pi/2|\lesssim 2^{-k/2},\qquad\vert\partial_\rho\theta^1\vert+\vert\partial_\rho\theta^2\vert\lesssim 2^{-k/2}.
\end{equation}

(iii) Assume that $|\xi|,|\eta|\in[2^{k-2},2^{k+2}]$. For $0<\kappa\leq 2^{-\D_1}$ and integers $r$, $q$ such that $q\leq -\D_1$, $|\kappa r|\in[1/4,4]$, define
\begin{equation}\label{DefRect1}
\begin{split}
\mathcal{S}_{p,q,r}^{1,\mp}(\xi):=&\{z:\,\,|z|=\rho\in[2^{-15},2^{15}],\,\, \vert \Phi(\xi,\xi-z)\vert\le 2^p,\\
&\vert\arg(z)-\arg(\xi)\mp\theta^1(\rho)\vert\leq 2^{-\D_1/2},\,\,\vert\hat{\Upsilon}(\xi,\xi-z)-\kappa r 2^q\vert\le 10\kappa 2^q\},
\end{split}
\end{equation}
and
\begin{equation}\label{DefRect2}
\begin{split}
\mathcal{S}_{p,q,r}^{2,\mp}(\eta):=&\{z:\,\,|z|=\rho\in[2^{-15},2^{15}],\,\, \vert \Phi(\eta+z,\eta)\vert\le 2^p,\\
&\vert\arg(z)-\arg(\eta)\mp\theta^1(\rho)\vert\leq 2^{-\D_1/2},\,\,\vert\hat{\Upsilon}(\eta+z,\eta)-\kappa r 2^q\vert\le 10\kappa 2^q\}.
\end{split}
\end{equation}
Then, for any $\iota\in\{+,-\}$,
\begin{equation}\label{VolumeSpq}
\begin{split}
\vert\mathcal{S}^{1,\iota}_{p,q,r}(\xi)\vert+\vert\mathcal{S}^{2,\iota}_{p,q,r}(\eta)\vert\lesssim 2^{q+p-k/2},\quad\hbox{diam}(\mathcal{S}^{1,\iota}_{p,q,r}(\xi))+\hbox{diam}(\mathcal{S}^{2,\iota}_{p,q,r}(\eta))\lesssim 2^{p-k/2}+\kappa 2^q.
\end{split}
\end{equation}
Moreover, if $2^{p-k/2}\ll \kappa 2^{q}$ then there exist intervals $I^1_{p,q,r}$ and $I^2_{p,q,r}$ such that
\begin{equation}\label{VolumeSpq2}
\begin{split}
&\mathcal{S}^{1,\mp}_{p,q,r}(\xi)\subseteq\{(\rho\cos\theta,\rho\sin\theta):\,\, \rho\in I^1_{p,q,r},\,\,\vert\theta-\arg(\xi)\mp\theta^1(\rho)\vert\lesssim 2^{p-k/2}\},\quad\vert I^1_{p,q,r}\vert\lesssim\kappa 2^q,\\
&\mathcal{S}^{2,\mp}_{p,q,r}(\eta)\subseteq\{(\rho\cos\theta,\rho\sin\theta):\,\, \rho\in I^2_{p,q,r},\,\,\vert\theta-\arg(\eta)\mp\theta^2(\rho)\vert\lesssim 2^{p-k/2}\},\quad\vert I^2_{p,q,r}\vert\lesssim\kappa 2^q.
\end{split}
\end{equation}
\end{lemma}

\begin{proof} (i) Notice that if $|\xi|=s$, $|\eta|=r$, and $z=\xi-\eta=(\rho\cos\theta,\rho\sin\theta)$ then
\begin{equation}\label{ph5.3}
\begin{split}
&2\xi\cdot\eta=r^2+s^2-\rho^2,\qquad 2z\cdot \xi=\rho^2+s^2-r^2,\qquad 2z\cdot \eta=s^2-r^2-\rho^2,\\
&(2\eta\cdot \xi^\perp)^2=4r^2s^2-(r^2+s^2-\rho^2)^2.
\end{split}
\end{equation}
Under the assumptions \eqref{AssHM}, we see that $|\lambda(r)-\lambda(s)|\lesssim 2^{p^+}$, therefore $|r-s|\lesssim 2^{-k/2}2^{p^+}$. The bounds \eqref{Orthoxx-y} follow using also \eqref{ph5.3}. The decomposition \eqref{ModUp} follows from \eqref{ph5.1}, with
\begin{equation*}
\begin{split}
A(x,y)&:=\frac{\lambda^\prime(\vert x\vert)}{\vert x\vert}\frac{\lambda^\prime(\vert y\vert)}{\vert y\vert}\frac{(x\cdot y^\perp)^2}{\vert x-y\vert^2},\qquad B(w,z):=\frac{\sqrt{\lambda'(\vert z\vert)}}{\vert z\vert^{3/2}}\Big\{\vert z\vert\lambda'(\vert z\vert)-\frac{\lambda^\prime(\vert w\vert)}{\vert w\vert}(w\cdot z)\Big\}.
\end{split}
\end{equation*}
The bounds in the second line of \eqref{ModUp} follow from this definition and \eqref{Orthoxx-y}.

(ii) We will show the estimates for fixed $\xi$, since the estimates for fixed $\eta$ are similar. We may assume that $\xi=(s,0)$, so
\begin{equation}\label{tri1}
\Phi(\xi,\xi-z)=\lambda(s)-\lambda_\mu(\rho)-\lambda\big(\sqrt{s^2+\rho^2-2s\rho\cos\theta}\big).
\end{equation}
Let $f(\theta):=-\lambda(s)+\lambda_\mu(\rho)+\lambda\big(\sqrt{s^2+\rho^2-2s\rho\cos\theta}\big)$. We notice that $-f(0)\gtrsim 2^{k/2}$, $f(\pi)\gtrsim 2^{k/2}$, and $f'(\theta)\approx 2^{k/2}\sin\theta$ for $\theta\in[0,\pi]$. Therefore $f$ is increasing on the interval $[0,\pi]$ and vanishes at a unique point $\theta^1(\rho)=\theta^1_{s,\mu}(\rho)$. Moreover, it is easy to see that $|\cos(\theta^1(\rho))|\lesssim 2^{-k/2}$, therefore $|\theta^1(\rho)-\pi/2|\lesssim 2^{-k/2}$. The remaining conclusions in \eqref{Deftheta1}--\eqref{Deftheta2} follow easily.

(iii) We will only prove the estimates for the sets $\mathcal{S}_{p,q,r}^{1,-}(\xi)$, since the others are similar. With $z=(\rho\cos\theta,\rho\sin\theta)$ 
and $\xi=(s,0)$, we define $F(\rho,\theta):=\Phi(\xi,\xi-z)$ and $G(\rho,\theta):=\hat{\Upsilon}(\xi,\xi-z)$. 
The condition $|\hat{\Upsilon}(\xi,\xi-z)|\lesssim  2^{-\D_1}$ shows that $|\Upsilon(\xi,\xi-z)|\lesssim 2^{k-\D_1}$, thus 
$|\rho-\gamma_0|\leq 2^{-\D_1/2}$ (see \eqref{ModUp}). Moreover, $|\theta-\pi/2|\lesssim 2^{-\D_1/2}$ in view of \eqref{Deftheta1}--\eqref{Deftheta2}.
Using \eqref{tri1},
\begin{equation*}
|\partial_\theta F(\rho,\theta)|\approx 2^{k/2},\qquad |\partial_\rho F(\rho,\theta)|\lesssim 2^{k/2-\D_1/2}
\end{equation*}
in the set $\{(\rho,\theta):|\rho-\gamma_0|\leq 2^{-\D_1/2},\,\,|\theta-\pi/2|\lesssim 2^{-\D_1/2}\}$. In addition, using \eqref{ModUp} we have
\begin{equation*}
\begin{split}
&-\mu\partial_\rho G(\rho,\theta)=\lambda'''(\rho)\frac{A(\xi,\xi-z)}{|\Lambda'(\xi)||\Lambda'(\xi-z)|}+O(2^{-\D_1/2}),\qquad |\partial_\theta G(\rho,\theta)|=O(2^{-\D_1/2}).
\end{split}
\end{equation*}
Therefore, the mapping $(\rho,\theta)\mapsto [2^{-k/2}F(\rho,\theta),G(\rho,\theta)]$ is a regular change of variables for $\rho,\theta$ satisfying $|\rho-\gamma_0|\leq 2^{-\D_1/2},\,|\theta-\pi/2|\lesssim 2^{-\D_1/2}$. The desired conclusions follow.
\end{proof}

It follows from \eqref{ph5.1} and \eqref{ph5.3} that if $|\xi|=s,\,|\eta|=r,\,|\xi-\eta|=\rho$ then
\begin{equation}\label{ph6.5}
-4\Upsilon(\xi,\eta)\frac{\rho^3}{\lambda'_\mu(\rho)}\frac{s}{\lambda'_\sigma(s)}\frac{r}{\lambda'_\nu(r)}=G(s,r,\rho),
\end{equation}
where
\begin{equation}\label{ph102}
\begin{split}
G(s,r,\rho):&=\frac{\rho\lambda''(\rho)}{\lambda'(\rho)}\big[4r^2s^2-(r^2+s^2-\rho^2)^2\big]\\
&+\Big[2\rho s\frac{\lambda'(\rho)}{\lambda'(s)}-\rho^2-s^2+r^2\Big]\Big[2\rho r\frac{\lambda'(\rho)}{\lambda'(r)}+\rho^2+r^2-s^2\Big].
\end{split}
\end{equation}

We assume now that $|\xi-\eta|$ is close to $\gamma_0$ and consider the case of bounded frequencies.

\begin{lemma}\label{lemmaD1} If $|\xi|=s,\,|\eta|=r,\,|\xi-\eta|=\rho$, $\big|\rho-\gamma_0\big|\leq 2^{-8\D_1}$, and $2^{-200}\leq r,s\leq 2^{2\D_1}$ then
\begin{equation}\label{ph9}
|\Phi(\xi,\eta)|+|\Upsilon(\xi,\eta)|\gtrsim 1.
\end{equation}
\end{lemma}

\begin{proof} {\bf{Case 1: $(\sigma,\mu,\nu)=(+,+,+)$}}. Notice first that the function $f(r):=\lambda(r)+\lambda(\gamma_0)-\lambda(r+\gamma_0)$ is concave down for $r\in[0,\gamma_0]$ (in view of \eqref{ph3}) and satisfies $f(0)=0$, $f(\gamma_0)\geq 0.1$. Therefore $f(r)\gtrsim 1$ if $r\in[2^{-200},\gamma_0]$, so
\begin{equation}\label{ph9.1}
|\Phi(\xi,\eta)|\gtrsim 1\qquad\text{ if }\qquad r\leq \gamma_0\,\,\text{ or }\,\,s\leq 2\gamma_0.
\end{equation}

Assume, for contradiction, that \eqref{ph9} fails. In view of \eqref{ph6.5}, $|\Phi(\xi,\eta)|\ll 1$ and
\begin{equation}\label{ph10.5}
\Big|\Big[2\rho r\frac{\lambda'(\rho)}{\lambda'(r)}+(\rho^2+r^2-s^2)\Big]\Big[2\rho s\frac{\lambda'(\rho)}{\lambda'(s)}-(\rho^2+s^2-r^2)\Big]\Big|\ll 1+s+r.
\end{equation}
It is easy to see that if $|\Phi(\xi,\eta)|=|\lambda(s)-\lambda(\rho)-\lambda(r)|\ll 1$, $r\geq 100$, and $|\rho-\gamma_0|\leq 2^{-8\D_1}$ then
\begin{equation*}
r\leq s-\frac{\lambda(\rho)-0.1}{\lambda'(s)}\qquad\text {and }\qquad s\geq r+\frac{\lambda(\rho)-0.1}{\lambda'(r)}.
\end{equation*}
Therefore, using \eqref{ph2}--\eqref{ph3.1}, if $r\geq 100$ then
\begin{equation*}
\begin{split}
&-2\rho s\frac{\lambda'(\rho)}{\lambda'(s)}+\rho^2+s^2-r^2\geq \frac{2s}{\lambda'(s)}\big(\lambda(\rho)-0.1-\rho\lambda'(\rho)\big)\gtrsim \sqrt{s}\\
&-2\rho r\frac{\lambda'(\rho)}{\lambda'(r)}-\rho^2-r^2+s^2\geq \frac{2r}{\lambda'(r)}\big(\lambda(\rho)-0.1-\rho\lambda'(\rho)\big)-\rho^2\gtrsim\sqrt{r}.
\end{split}
\end{equation*}
In particular, \eqref{ph10.5} cannot hold if $r\geq 100$.

For $y\in[0,\infty)$, the equation $\lambda(x)=y$ admits a unique solution $x\in[0,\infty)$,
\begin{equation}\label{Sol}
x=-\frac{1}{Y(y)}+\frac{Y(y)}{3},\qquad Y(y):=\Big(\frac{27y^2+\sqrt{27}\sqrt{27y^4+4}}{2}\Big)^{1/3}.
\end{equation}

Assuming $|\rho-\gamma_0|\leq 2^{-8\D_1}$, $2\gamma_0\leq s\leq 110$, and $|\lambda(s)-\lambda(r)-\lambda(\rho)|\ll 1$, we show now that $G(s,r,\rho)\gtrsim 1$, where $G$ is as in \eqref{ph102}. Indeed, we solve the equation $\lambda(r(s))=\lambda(s)-\lambda(\gamma_0)$ according to \eqref{Sol} and define the function $G_0(s):=G(s,r(s),\gamma_0)$. A simple {\it{Mathematica}} program shows that $G_0(s)\gtrsim 1$ if $2\gamma_0\leq s\leq 110$. This completes the proof of \eqref{ph9} when $(\sigma,\mu,\nu)=(+,+,+)$.

{\bf{Case 2: the other triplets.}} Notice that if $(\sigma,\mu,\nu)=(+,-,+)$ then
\begin{equation}\label{ph11.1}
\Phi_{+-+}(\xi,\eta)=-\Phi_{+++}(\eta,\xi),\qquad \Upsilon_{+-+}(\xi,\eta)=-\Upsilon_{+++}(\eta,\xi).
\end{equation}
The desired bound in this case follows from the case $(\sigma,\mu,\nu)=(+,+,+)$ analyzed earlier.

On the other hand, if $(\sigma,\mu,\nu)=(+,-,-)$ then $\Phi(\xi,\eta)=\lambda(s)+\lambda(r)+\lambda(\rho)\gtrsim 1$, so \eqref{ph9} is clearly verified. Finally, if $(\sigma,\mu,\nu)=(+,+,-)$ then $\Phi(\xi,\eta)=\lambda(s)+\lambda(r)-\lambda(\rho)$ and we estimate, assuming $2^{-200}\leq r\leq \rho/2$,
\begin{equation*}
\lambda(s)+\lambda(r)-\lambda(\rho)\geq \lambda(r)+\lambda (\rho-r)-\lambda(\rho)=\int_0^r[\lambda'(x)-\lambda'(x+\rho-r)]\,dx\gtrsim 1.
\end{equation*}
A similar estimate holds if $2^{-200}\leq s\leq \rho/2$ or if $s,r\geq \rho/2$. Therefore $\Phi(\xi,\eta)\gtrsim 1$ in this case.

The cases corresponding to $\sigma=-$ are similar by replacing $\Phi$ with $-\Phi$ and $\Upsilon$ with $-\Upsilon$. This completes the proof of the lemma.
\end{proof}

Finally, we consider the case when $|\xi-\eta|$ is close to $\gamma_1$.

\begin{lemma}\label{lemmaD2} If $|\xi|=s,\,|\eta|=r,\,|\xi-\eta|=\rho$, $|\rho-\gamma_1|\leq 2^{-\D_1}$, and $2^{-200}\leq r,s$ then
\begin{equation}\label{kn2}
\begin{split}
&|\Phi(\xi,\eta)|+\frac{|\Upsilon(\xi,\eta)|}{|\xi|+|\eta|}+\frac{|(\nabla_\eta\Upsilon)(\xi,\eta)\cdot (\nabla_\eta^\perp\Phi)(\xi,\eta)|}{(|\xi|+|\eta|)^6}\gtrsim 1,\\
&|\Phi(\xi,\eta)|+\frac{|\Upsilon(\xi,\eta)|}{|\xi|+|\eta|}+\frac{|(\nabla_\xi\Upsilon)(\xi,\eta)\cdot (\nabla_\xi^\perp\Phi)(\xi,\eta)|}{(|\xi|+|\eta|)^6}\gtrsim 1,
\end{split}
\end{equation}
and
\begin{equation}\label{kn3}
\begin{split}
&|\Phi(\xi,\eta)|+\frac{|\Upsilon(\xi,\eta)|}{|\xi|+|\eta|}+\frac{|(\xi-\eta)\cdot (\nabla_\eta^\perp\Phi)(\xi,\eta)|}{(|\xi|+|\eta|)^6}\gtrsim 1,\\
&|\Phi(\xi,\eta)|+\frac{|\Upsilon(\xi,\eta)|}{|\xi|+|\eta|}+\frac{|(\xi-\eta)\cdot (\nabla_\xi^\perp\Phi)(\xi,\eta)|}{(|\xi|+|\eta|)^6}\gtrsim 1.
\end{split}
\end{equation}
\end{lemma}

\begin{proof}
{\bf{Case 1: $(\sigma,\mu,\nu)=(+,+,+)$}}. Notice first that the function $f(r):=\lambda(r)+\lambda(\gamma_1)-\lambda(r+\gamma_1)$ is concave down for $r\in[0,0.3]$ (in view of \eqref{ph3}) and satisfies $f(0)=0$, $f(0.3)\geq 0.02$. Therefore $f(r)\gtrsim 1$ if $r\in[2^{-200},0.3]$, so
\begin{equation}\label{kn6.5}
|\Phi(\xi,\eta)|\gtrsim 1\qquad\text{ if }\qquad r\leq 0.3\,\,\text{ or }\,\,s\leq\gamma_1+0.3.
\end{equation}

On the other hand, if $|\Phi(\xi,\eta)|\ll 1$, $r\geq 1000$, and $|\rho-\gamma_1|\leq 2^{-\D_1}$ then
\begin{equation*}
s\leq r+\frac{\lambda(\rho)+0.2}{\lambda'(r)}\qquad\text{ and }\qquad r\geq s-\frac{\lambda(\rho)+0.2}{\lambda'(s)}.
\end{equation*}
Therefore, using also \eqref{kn1}, if $r\geq 1000$ then
\begin{equation*}
\begin{split}
&2\rho r\frac{\lambda'(\rho)}{\lambda'(r)}+\rho^2+r^2-s^2\geq \frac{2r}{\lambda'(r)}(\rho\lambda'(\rho)-\lambda(\rho)-0.2)\gtrsim \sqrt{r},\\
&2\rho s\frac{\lambda'(\rho)}{\lambda'(s)}-\rho^2-s^2+r^2\geq \frac{2s}{\lambda'(s)}(\rho\lambda'(\rho)-\lambda(\rho)-0.2)-\rho^2\gtrsim \sqrt{s},\\
&\frac{\rho\lambda''(\rho)}{\lambda'(\rho)}\big[4r^2s^2-(r^2+s^2-\rho^2)^2\big]\gtrsim r^2.
\end{split}
\end{equation*}
Using the formula \eqref{ph6.5} and assuming $|\rho-\gamma_1|\leq 2^{-\D_1}$, it follows that
\begin{equation}\label{kn6.6}
\text{ if }\,\,|\Phi(\xi,\eta)|\ll 1\,\,\text{ and }\,\,r\geq 1000\,\,\text{ then }\,\,-\Upsilon(\xi,\eta)\gtrsim r.
\end{equation}
Therefore both \eqref{kn2} and \eqref{kn3} follow if $r\geq 1000$.

It remains to consider the case $\gamma_1+0.3\leq s\leq 1010$. We show first that
\begin{equation}\label{ph101}
\text{ if }\,\,3\leq s\leq 1010\,\,\text{ and }\,\,|\lambda(s)-\lambda(r)-\lambda(\rho)|\ll 1\,\,\text{ then }-\Upsilon(\xi,\eta)\gtrsim 1.
\end{equation}
Indeed, we solve the equation $\lambda(r(s))=\lambda(s)-\lambda(\gamma_1)$ according to \eqref{Sol}, and define the function $G_1(s):=G(s,r(s),\gamma_1)$, see \eqref{ph6.5}--\eqref{ph102}. A simple {\it{Mathematica}} program shows that $G_1(s)\gtrsim 1$ if $3\leq s\leq 1010$. The bound \eqref{ph101} follows, so both \eqref{kn2} and \eqref{kn3} follow if $3\leq s\leq 1010$.

On the other hand the function $G_1(s)$ does vanish for some $s\in[\gamma_1+0.3,3]$ (more precisely at $s\approx 1.94$). In this range we can only prove the weaker estimates in the lemma. Notice that
\begin{equation*}
\begin{split}
\Upsilon(\xi,\eta)=\widetilde{\Upsilon}(|\xi|,|\eta|,|\xi-\eta|),\qquad\widetilde{\Upsilon}(s,r,\rho):=-\frac{1}{4}G(s,r,\rho)\frac{\lambda'(\rho)}{\rho^3}\frac{\lambda'(s)}{s}\frac{\lambda'(r)}{r}.
\end{split}
\end{equation*}
Then, using also \eqref{ph4}, we have
\begin{equation}\label{Sol2}
\begin{split}
(\nabla_\eta\Upsilon)(\xi,\eta)\cdot (\nabla_\eta^\perp\Phi)(\xi,\eta)&=(r\rho)^{-1}(\eta\cdot\xi^\perp)\big[(\partial_r\widetilde{\Upsilon})(s,r,\rho)\lambda'(\rho)-(\partial_\rho\widetilde{\Upsilon})(s,r,\rho)\lambda'(r)\big],\\
(\nabla_\xi\Upsilon)(\xi,\eta)\cdot (\nabla_\xi^\perp\Phi)(\xi,\eta)&=(s\rho)^{-1}(\xi\cdot\eta^\perp)\big[(\partial_s\widetilde{\Upsilon})(s,r,\rho)\lambda'(\rho)+(\partial_\rho\widetilde{\Upsilon})(s,r,\rho)\lambda'(s)\big].
\end{split}
\end{equation}
It is easy to see, using the formulas \eqref{ph5.3} and \eqref{ph6.5}, that
\begin{equation}\label{Sol3}
|\Phi(\xi,\eta)|+|\Upsilon(\xi,\eta)|+|\xi\cdot\eta^\perp|\gtrsim 1
\end{equation}
if $s\in[\gamma_1+0.3,3]$. Moreover, let
\begin{equation*}
\begin{split}
G_{11}(s)&:=(\partial_r\widetilde{\Upsilon})(s,r(s),\gamma_1)\lambda'(\gamma_1)-(\partial_\rho\widetilde{\Upsilon})(s,r(s),\gamma_1)\lambda'(r(s)),\\
G_{12}(s)&:=(\partial_s\widetilde{\Upsilon})(s,r(s),\gamma_1)\lambda'(\gamma_1)+(\partial_\rho\widetilde{\Upsilon})(s,r(s),\gamma_1)\lambda'(s),
\end{split}
\end{equation*}
where, as before, $r(s)$ is the unique solution of the equation $\lambda(r(s))=\lambda(s)-\lambda(\gamma_1)$, according to \eqref{Sol}. A simple {\it{Mathematica}} program shows that $G_1(s)+G_{11}(s)\gtrsim 1$ and $G_1(s)+G_{12}(s)\gtrsim 1$ if $s\in[\gamma_1+0.3,3]$. Using also \eqref{Sol2} and \eqref{Sol3} it follows that
\begin{equation}\label{Sol4}
\begin{split}
&|\Upsilon(\xi,\eta)|+|(\nabla_\eta\Upsilon)(\xi,\eta)\cdot (\nabla_\eta^\perp\Phi)(\xi,\eta)|\gtrsim 1,\\
&|\Upsilon(\xi,\eta)|+|(\nabla_\xi\Upsilon)(\xi,\eta)\cdot (\nabla_\xi^\perp\Phi)(\xi,\eta)|\gtrsim 1,
\end{split}
\end{equation}
if $s\in[\gamma_1+0.3,3]$, $|\Phi(\xi,\eta)|\ll 1$, and $|\rho-\gamma_0|\leq 2^{-\D_1}$. The bounds \eqref{kn2} follow from \eqref{kn6.5}--\eqref{ph101} and \eqref{Sol4}. The bounds \eqref{kn3} follow from \eqref{kn6.5}--\eqref{ph101}, and \eqref{Sol3}.

{\bf{Case 2: the other triplets.}} The desired bounds in the case $(\sigma,\mu,\nu)=(+,-,+)$ follow from the corresponding bounds the case $(\sigma,\mu,\nu)=(+,+,+)$ and \eqref{ph11.1}. Moreover, if $(\sigma,\mu,\nu)=(+,-,-)$ then $\Phi(\xi,\eta)=\lambda(s)+\lambda(r)+\lambda(\rho)\gtrsim 1$, so \eqref{kn2}--\eqref{kn3} are clearly verified.

Finally, if $(\sigma,\mu,\nu)=(+,+,-)$ then $\Phi(\xi,\eta)=\lambda(s)+\lambda(r)-\lambda(\rho)$. We may assume that $s,r\in[2^{-20},\gamma_1]$. In this case we prove the stronger bound
\begin{equation}\label{Sol6}
|\Phi(\xi,\eta)|+|\Upsilon(\xi,\eta)|\gtrsim 1.
\end{equation}
Indeed, for this is suffices to notice that the function $x\to \lambda(x)+\lambda(\gamma_1-x)-\lambda(\gamma_1)$ is nonnegative for $x\in[0,\gamma_1]$ and vanishes only when $x\in\{0,\gamma_1/2,\gamma_1\}$. Moreover $\Upsilon((\gamma_1/2)e,-(\gamma_1/2)e)\neq 0$ if $|e|=1$ (using \eqref{ph5.1}), and the lower bound \eqref{Sol6} follows.

The cases corresponding to $\sigma=-$ are similar by replacing $\Phi$ with $-\Phi$ and $\Upsilon$ with $-\Upsilon$. This completes the proof of the lemma.
\end{proof}

\appendix
\section{Paradifferential calculus}\label{SecParaOp}

The paradifferential calculus allows us to understand the high frequency structure of our system. In this section we record the definitions, and state and prove several useful lemmas. 

\subsection{Operators bounds}\label{BdPara}

In this subsection we define our main objects, and prove several basic nonlinear bounds. 

\subsubsection{Fourier multipliers}\label{secmultipliers}
We will mostly work with bilinear and trilinear multipliers. Many of the simpler estimates follow from the following basic result (see \cite[Lemma 5.2]{IoPu2} for the proof).

\begin{lemma}\label{touse}
(i) Assume $l\geq 2$, $f_1,\ldots,f_l,f_{l+1}\in L^2(\mathbb{R}^2)$, and $m:(\mathbb{R}^2)^l\to\mathbb{C}$ is a continuous compactly supported function. Then
\begin{equation}\label{ener62}
\begin{split}
\Big|\int_{(\mathbb{R}^2)^l}m(\xi_1,\ldots,\xi_l)\widehat{f_1}(\xi_1)\cdot\ldots\cdot\widehat{f_l}(\xi_l)\cdot\widehat{f_{l+1}}(-\xi_1-\ldots-\xi_l)\,d\xi_1\ldots d\xi_l\Big|\\
\lesssim \big\|\mathcal{F}^{-1}(m)\big\|_{L^1}\|f_1\|_{L^{p_1}}\cdot\ldots\cdot\|f_{l+1}\|_{L^{p_{l+1}}},
\end{split}
\end{equation}
for any exponents $p_1,\ldots p_{l+1}\in[1,\infty]$ satisfying $\frac{1}{p_1}+\ldots+\frac{1}{p_{l+1}}=1$. 

(ii) Assume $l\geq 2$ and $L_m$ is the multilinear operator defined by
\begin{equation*}
\mathcal{F}\{L_m[f_1,\ldots,f_l]\}(\xi)=\int_{(\mathbb{R}^2)^{l-1}}m(\xi,\eta_2,\ldots,\eta_l)\widehat{f_1}(\xi-\eta_2)\cdot\ldots\cdot \widehat{f_{l-1}}(\eta_{l-1}-\eta_l)\widehat{f_l}(\eta_l)\,d\eta_2\ldots d\eta_l.
\end{equation*}
Then, for any exponents $p,q_1,\ldots q_l\in[1,\infty]$ satisfying $\frac{1}{q_1}+\ldots+\frac{1}{q_{l}}=\frac{1}{p}$, we have
\begin{equation}\label{mk6}
\big\|L_m[f_1,\ldots,f_l]\big\|_{L^p}\lesssim \big\|\mathcal{F}^{-1}(m)\big\|_{L^1}\|f_1\|_{L^{q_1}}\cdot\ldots\cdot\|f_{l}\|_{L^{q_{l}}}.
\end{equation}
\end{lemma}

Given a multiplier $m:(\mathbb{R}^2)^2\to\mathbb{C}$, we define the bilinear operator $M$ by the formula
\begin{equation}\label{symnot0}
\mathcal{F}[M[f,g])](\xi)=\frac{1}{4\pi^2}\int_{\mathbb{R}^2}m(\xi,\eta)\widehat{f}(\xi-\eta)\widehat{g}(\eta)\,d\eta.
\end{equation}
With $\Omega=x_1\partial_2-x_2\partial_1$, we notice the formula
\begin{equation}\label{symnot1}
\Omega M[f,g]= M[\Omega f,g]+M[f,\Omega g]+\widetilde{M}[f,g],
\end{equation}
where $\widetilde{M}$ is the bilinear operator defined by the multiplier $\widetilde{m}(\xi,\eta)=(\Omega_\xi+\Omega_\eta)m(\xi,\eta)$. 

For simplicity of notation, we define the following classes of bilinear multipliers:
\begin{equation}\label{defclassA}
\begin{split}
&S^\infty:= \{m: (\R^2)^n \to \mathbb{C} : \,m \text{ continuous and } 
  {\| m \|}_{S^\infty} :=\|\mathcal{F}^{-1} m \|_{L^1} < \infty \},\\
&S^\infty_{\Omega} := \{m: (\R^2)^2 \to \mathbb{C} : \,m \text{ continuous and } 
  {\| m \|}_{S^\infty_\Omega} := \sup_{l\le N_1}\|(\Omega_\xi+\Omega_\eta)^l m \|_{S^\infty} < \infty \}.
	\end{split}
\end{equation}

We will often need to analyze bilinear operators more carefully, by localizing in the frequency space. We therefore define, for any symbol $m$,
\begin{align}
\label{mloc}
m^{k, k_1, k_2}(\xi,\eta) := \varphi_k(\xi)\varphi_{k_1}(\xi-\eta)\varphi_{k_2}(\eta) m(\xi,\eta).
\end{align}

For any $t\in[0,T]$, $p\geq -N_3$, and $m\geq 1$ let $\langle t\rangle=1+t$ and let $\mathcal{O}_{m,p}=\mathcal{O}_{m,p}(t)$ denote the Banach spaces of functions $f\in L^2$ defined by the norms
\begin{align}\label{fw1}
\|f\|_{\mathcal{O}_{m,p}}:=\langle t\rangle ^{(m-1)(5/6-20\delta^2)-\delta^2}\big[\| f\|_{H^{N_0+p}}+\|f\|_{H_\Omega^{N_1,N_3+p}}+ \langle t\rangle^{5/6-2\delta^2}\|f\|_{\widetilde{W}_{\Omega}^{N_1/2,N_2+p}}\big].
\end{align}
This is similar to the definition of the spaces $O_{m,p}$ in Definition \ref{DefOHierarchy}, except for the supremum over $t\in[0,T]$. We show first that these spaces are compatible with $S^\infty_{\Omega}$ multipliers.

\begin{lemma}\label{lemmaaux1}
Assume $M$ is a bilinear operator with symbol $m$ satisfying $\|m^{k,k_1,k_2}\|_{S^\infty_{\Omega}}\leq 1$, for any $k,k_1,k_2\in\mathbb{Z}$. Then, if $p\in[-N_3,10]$, $t\in[0,T]$, and $m,n\geq 1$,
\begin{align}
\label{aux1}
\langle t\rangle^{12\delta^2}\|M[f,g]\|_{\mathcal{O}_{m+n,p}}\lesssim \|f\|_{\mathcal{O}_{m,p}}\|g\|_{\mathcal{O}_{n,p}}.
\end{align}
\end{lemma}

\begin{proof} In view of the definition we may assume that $m=n=1$ and $\|f\|_{\mathcal{O}_{m,p}}=\|g\|_{\mathcal{O}_{n,p}}=1$. Therefore, we may assume that
\begin{equation}\label{Ali1}
\|h\|_{H^{N_0+p}}+\sup_{j\leq N_1}\|\Omega^jh\|_{H^{N_3+p}}\leq \langle t\rangle^{\delta^2},\qquad \sup_{j \leq N_1/2} {\| \Omega^j h \|}_{\widetilde{W}^{N_2+p}}\leq \langle t\rangle^{3\delta^2-5/6},
\end{equation}
where $h\in\{f(t),g(t)\}$. With $F:=M[f(t),g(t)]$, it suffices to prove that
\begin{equation}\label{Ali2}
\begin{split}
\|F\|_{H^{N_0+p}}+\sup_{j\leq N_1}\|\Omega^jF\|_{H^{N_3+p}}&\lesssim \langle t\rangle^{6\delta^2-5/6},\\
\sup_{j \leq N_1/2} {\| \Omega^j P_kF\|}_{\widetilde{W}^{N_2+p}}&\lesssim \langle t\rangle^{8\delta^2-5/3}.
\end{split}
\end{equation}

For $k,k_1,k_2\in\mathbb{Z}$ let
\begin{equation*}
F_k:=P_k M[f(t),g(t)],\qquad F_{k,k_1,k_2}:=P_k M[P_{k_1}f(t),P_{k_2}g(t)]. 
\end{equation*}
For $k\in\mathbb{Z}$ let
\begin{equation*}
\begin{split}
&\mathcal{X}_k^1:=\{(k_1,k_2)\in\mathbb{Z}\times\mathbb{Z}: k_1\leq k-8,\, |k_2-k|\leq 4\},\\
&\mathcal{X}_k^2:=\{(k_1,k_2)\in\mathbb{Z}\times\mathbb{Z}: k_2\leq k-8,\, |k_1-k|\leq 4\},\\
&\mathcal{X}_k^3:=\{(k_1,k_2)\in\mathbb{Z}\times\mathbb{Z}: \min(k_1,k_2)\geq k-7,\, |k_1-k_2|\leq 20\}, 
\end{split}
\end{equation*}
and let $\mathcal{X}_k:=\mathcal{X}_k^1\cup\mathcal{X}_k^2\cup\mathcal{X}_k^3$. Let
\begin{equation}\label{Ali4}
\begin{split}
&a_k:=\|P_kh\|_{H^{N_0+p}},\quad b_k:=\sup_{0\leq j\leq N_1}\|\Omega^jP_kh\|_{H^{N_3+p}},\quad c_k:=\sup_{0 \leq j \leq N_1/2} \| \Omega^j P_kh \|_{\widetilde{W}^{N_2+p}},\\
&\widetilde{a}_k:=\sum_{m\in\mathbb{Z}}a_{k+m}2^{-|m|/100},\quad \widetilde{b}_k:=\sum_{m\in\mathbb{Z}}b_{k+m}2^{-|m|/100},\quad \widetilde{c}_k:=\sum_{m\in\mathbb{Z}}c_{k+m}2^{-|m|/100}.
\end{split}
\end{equation}

We can prove now \eqref{Ali2}. Assuming $k\in\mathbb{Z}$ fixed we estimate, using Lemma \ref{touse} (ii),
\begin{equation}\label{Ali5}
\begin{split}
&\|F_{k,k_1,k_2}\|_{H^{N_0+p}}\lesssim a_{k_1}(2^{-4\max(k_2,0)}c_{k_2})\qquad \text{ if }(k_1,k_2)\in\mathcal{X}_k^2,\\
&\|F_{k,k_1,k_2}\|_{H^{N_0+p}}\lesssim a_{k_2}(2^{-4\max(k_1,0)}c_{k_1})\qquad \text{ if }(k_1,k_2)\in\mathcal{X}_k^1\cup\mathcal{X}_k^3.
\end{split}
\end{equation}
Since $\sum_{l}c_l\leq \langle t\rangle^{3\delta^2-5/6}$, it follows that
\begin{equation}\label{Ali5.5}
\sum_{(k_1,k_2)\in\mathcal{X}_k}\|F_{k,k_1,k_2}\|_{H^{N_0+p}}\lesssim \langle t\rangle^{3\delta^2-5/6}\big[\widetilde{a}_k+\sum_{l\geq k}\widetilde{a}_l2^{-4l_+}\big].
\end{equation}
Therefore, since $\sum_{k\in\mathbb{Z}}\widetilde{a}_k^2\lesssim \langle t\rangle^{2\delta^2}$, it follows that
\begin{equation}\label{Ali6}
\Big[\sum_{2^k\geq (1+t)^{-10}}\|F_{k}\|_{H^{N_0+p}}^2\Big]^{1/2}\lesssim \langle t\rangle^{6\delta^2-5/6}.
\end{equation}
To bound the contribution of small frequencies, $2^k\leq \langle t\rangle^{-10}$, we also use the bound
\begin{equation}\label{Ali6.5}
\|F_{k,k_1,k_2}\|_{L^2}\lesssim 2^k\|F_{k,k_1,k_2}\|_{L^1}\lesssim 2^k a_{k_1}a_{k_2}.
\end{equation}
when $(k_1,k_2)\in \mathcal{X}_k^3$, in addition to the bounds \eqref{Ali5}. Therefore
\begin{equation}\label{Ali7}
\sum_{(k_1,k_2)\in\mathcal{X}_k}\|F_{k,k_1,k_2}\|_{H^{N_0+p}}\lesssim \langle t\rangle^{3\delta^2-5/6}\widetilde{a}_k+2^k\sum_{l\in\mathbb{Z}}a_l^2,
\end{equation}
if $2^k\leq \langle t\rangle^{-10}$. It follows that
\begin{equation}\label{Ali8}
\Big[\sum_{2^k\leq \langle t\rangle^{-10}}\|F_{k}\|_{H^{N_0+p}}^2\Big]^{1/2}\lesssim \langle t\rangle^{6\delta^2-5/6},
\end{equation}
and the desired bound $\|F\|_{H^{N_0+p}}\lesssim (1+t)^{6\delta^2-5/6}$ in \eqref{Ali2} follows.

The proof of the second bound in \eqref{Ali2} is similar. We start by estimating, as in \eqref{Ali5},
\begin{equation*}
\|\Omega^jF_{k,k_1,k_2}\|_{H^{N_3+p}}\lesssim 2^{(N_3+p)k^+}\big[b_{k_1}2^{-(N_3+p)k_1^{+}}c_{k_2}2^{-(N_2+p)k_2^{+}}+b_{k_2}2^{-(N_3+p)k_2^{+}}c_{k_1}2^{-(N_2+p)k_1^{+}}\big]
\end{equation*}
for any $j\in[0,N_1]$. We remark that this is weaker than \eqref{Ali5} since the $\Omega$ derivatives can distribute on either $P_{k_1}f(t)$ or $P_{k_2}(t)$, and we are forced to estimate the factor with more than $N_1/2$ $\Omega$ derivatives in $L^2$. To bound the contributions of small frequencies we also estimate
\begin{equation*}
\|\Omega^jF_{k,k_1,k_2}\|_{H^{N_3+p}}\lesssim 2^{\min(k,k_1,k_2)}b_{k_1}b_{k_2},
\end{equation*}
as in \eqref{Ali6.5}. Recall that $N_2-N_3\geq 5$. We combine these two bounds to estimate
\begin{equation*}
\sum_{(k_1,k_2)\in\mathcal{X}_k}\|\Omega^jF_{k,k_1,k_2}\|_{H^{N_3+p}}\lesssim \langle t\rangle^{3\delta^2-5/6}\big[\widetilde{b}_k+\sum_{l\geq k}\widetilde{b}_l2^{-4l^+}\big]+\langle t\rangle^{2\delta^2}2^{-(N_2-N_3)k^+}\widetilde{c}_k.
\end{equation*}
When $2^k\leq (1+t)^{-10}$ this does not suffice; we have instead the bound
\begin{equation*}
\sum_{(k_1,k_2)\in\mathcal{X}_k}\|\Omega^jF_{k,k_1,k_2}\|_{H^{N_3+p}}\lesssim \langle t\rangle^{3\delta^2-5/6}\widetilde{b}_k+2^k\sum_{l\in\mathbb{Z}}b_l^2+\langle t\rangle^{2\delta^2}2^{-(N_2-N_3)k^+}\widetilde{c}_k.
\end{equation*}
The desired estimate $\|\Omega^j F\|_{H^{N_3+p}}\lesssim \langle t\rangle^{6\delta^2-5/6}$ in \eqref{Ali2} follows.

For the last bound in \eqref{Ali2}, we estimate as before for any $j\in[0,N_1/2]$,
\begin{equation*}
\|\Omega^jF_{k,k_1,k_2}\|_{\widetilde{W}^{N_2+p}}\lesssim 2^{(N_2+p)k^+}c_{k_1}2^{-(N_2+p)k_1^{+}}c_{k_2}2^{-(N_2+p)k_2^{+}},\quad \|\Omega^jF_{k,k_1,k_2}\|_{\widetilde{W}^{N_2+p}}\lesssim 2^{2k}b_{k_1}b_{k_2},
\end{equation*}
where the last estimate holds only for $k\leq 0$. The desired bound follows as before.
\end{proof}

\subsubsection{Paradifferential operators}

We recall first the definition of paradifferential operators (see \eqref{Tsigmaf}: given a symbol $a=a(x,\zeta):\mathbb{R}^2\times\mathbb{R}^2\to\mathbb{C}$, we define the operator $T_a$ by
\begin{equation}\label{Tsigmaf2}
\begin{split}
\mathcal{F}\left\{T_{a}f\right\}(\xi)=\frac{1}{4\pi^2}\int_{\mathbb{R}^2}\chi\Big(\frac{\vert\xi-\eta\vert}{\vert\xi+\eta\vert}\Big)\widetilde{a}(\xi-\eta,(\xi+\eta)/2)\widehat{f}(\eta)d\eta,
\end{split}
\end{equation}
where $\widetilde{a}$ denotes the partial Fourier transform of $a$ in the first coordinate and $\chi=\varphi_{-20}$. We define the Poisson bracket between two symbols $a$ and $b$ by
\begin{align}
\label{Poisson2}
\{ a,b \} := \nabla_x a \cdot\nabla_\zeta b - \nabla_\zeta a \cdot\nabla_x b.
\end{align}

We will use several norms to estimate symbols of degree $0$. For $q\in\{2,\infty\}$, $r\in\mathbb{Z}_+$, let
\begin{equation}\label{SymNorm1}
\Vert a\Vert_{\mathcal{M}_{r,q}}:=\sup_{\zeta}\|\,|a|_{r}(.,\zeta) \|_{L^q_x},\quad\text{ where }\quad |a|_{r}(x,\zeta):=\sum_{\vert\alpha\vert+\vert\beta\vert\le r}\vert\zeta\vert^{\vert\beta\vert}|\partial_\zeta^\beta\partial_x^\alpha a(x,\zeta)|.
\end{equation}
At later stages we will use more complicated norms, which also keep track of multiplicity and degree. For now we record a few simple properties, which follow directly from definitions:
\begin{equation}\label{PropSymb1}
\begin{split}
\Vert ab\Vert_{\mathcal{M}_{r,q}}+\Vert \,|\zeta|\{a,b\}\Vert_{\mathcal{M}_{r-2,q}}&\lesssim \Vert a\Vert_{\mathcal{M}_{r,q_1}}\Vert b\Vert_{\mathcal{M}_{r,q_2}},\qquad \{\infty,q\}=\{q_1,q_2\},\\ 
\Vert P_ka\Vert_{\mathcal{M}_{r,q}}&\lesssim 2^{-sk}\Vert P_ka\Vert_{\mathcal{M}_{r+s,q}},\,\,\,\,\qquad q\in\{2,\infty\},\,\,k\in\mathbb{Z},\,\,s\in\mathbb{Z}_+.
\end{split}
\end{equation}

We start with some simple properties.

\begin{lemma}\label{PropSym} (i) Let $a$ be a symbol and $1\le q\le\infty$, then
\begin{equation}\label{LqBdTa}
\begin{split}
\Vert P_kT_af\Vert_{L^q}&\lesssim \Vert a\Vert_{\mathcal{M}_{8,\infty}}\Vert P_{[k-2,k+2]}f\Vert_{L^q}
\end{split}
\end{equation}
and
\begin{equation}
\label{L2BdTa}
\begin{split}
\Vert P_kT_af\Vert_{L^2}&\lesssim \Vert a\Vert_{\mathcal{M}_{8,2}}\Vert P_{[k-2,k+2]}f\Vert_{L^\infty}.
\end{split}
\end{equation}

(ii) If $a\in \mathcal{M}_{8,\infty}$ is real-valued then $T_a$ is a bounded self-adjoint operator on $L^2$.

(iii) We have
\begin{equation}\label{Alu2}
\overline{T_af}=T_{a'}\overline{f},\quad\text{ where }\quad a'(y,\zeta):=\overline{a(y,-\zeta)}
\end{equation} 
and
\begin{equation}\label{Alu3}
\Omega (T_af)=T_a(\Omega f)+T_{a''}f\quad \text{ where }\quad a''(y,\zeta)=(\Omega_ya)(y,\zeta)+(\Omega_{\zeta}a)(y,\zeta).
\end{equation} 
\end{lemma}

\begin{proof}
(i) Inspecting the Fourier transform, we directly see that $P_kT_af=P_kT_aP_{[k-2,k+2]}f$. By rescaling, we may assume that $k=0$ and write 
\begin{equation*}
\begin{split}
\langle P_0T_ah,g\rangle&=C\int_{\mathbb{R}^4} \overline{g}(x)h(y)I(x,y)dxdy,\\
I(x,y)&=\int_{\mathbb{R}^6} a(z,(\xi+\eta)/2)e^{i\xi\cdot (x-z)}e^{i\eta\cdot (z-y)}\chi\Big(\frac{\vert \xi-\eta\vert }{\vert\xi+\eta\vert}\Big)\varphi_0(\xi) \,d\eta d\xi dz\\
&=\int_{\mathbb{R}^6} a(z,\xi+\theta/2)e^{i\theta\cdot(z-y)}e^{i\xi\cdot (x-y)}\chi\Big(\frac{\vert \theta\vert }{\vert 2\xi+\theta\vert}\Big)\varphi_0(\xi) \,d\xi d\theta dz.
\end{split}
\end{equation*}
We observe that
\begin{equation*}
\begin{split}
(1+\vert x-y\vert^2)^2&I(x,y)=\int_{\mathbb{R}^6} \frac{a(z,\xi+\theta/2)}{(1+\vert z-y\vert^2)^2}\chi\Big(\frac{\vert \theta\vert }{\vert 2\xi+\theta\vert}\Big)\varphi_0(\xi)\\
&\times \left[(1-\Delta_\theta)^2(1-\Delta_\xi)^2\{e^{i\theta\cdot(z-y)}e^{i\xi\cdot (x-y)}\}\right] \,d\xi d\theta dz.
\end{split}
\end{equation*}
By integration by parts in $\xi$ and $\theta$ it follows that
\begin{equation}\label{Alu1}
(1+\vert x-y\vert^2)^2|I(x,y)|\lesssim \int_{\mathbb{R}^6} \frac{|a|_{8}(z,\xi+\theta/2)}{(1+\vert z-y\vert^2)^2}\varphi_{[-4,4]}(\xi)\varphi_{\leq -10}(\theta)\,d\xi d\theta dz,
\end{equation}
where $|a|_{8}$ is defined as in \eqref{SymNorm1}.

The bounds \eqref{LqBdTa} and \eqref{L2BdTa} now follow easily. Indeed, it follows from \eqref{Alu1} that $$(1+\vert x-y\vert^2)^2|I(x,y)|\lesssim \Vert a\Vert_{\mathcal{M}_{8,\infty}}.$$
Therefore $|\langle P_0T_ah,g\rangle|\lesssim \Vert a\Vert_{\mathcal{M}_{8,\infty}}\|h\|_{L^q}\|g\|_{L^{q'}}$. This gives \eqref{LqBdTa}, and \eqref{L2BdTa} follows similarly.

Part (ii) and \eqref{Alu2} follow directly from definitions. To prove \eqref{Alu3} we start from the formula
\begin{equation*}
\mathcal{F}\left\{\Omega T_{a}f\right\}(\xi)=\frac{1}{4\pi^2}\int_{\mathbb{R}^2}(\Omega_{\xi}+\Omega_{\eta})\Big[\chi\Big(\frac{\vert\xi-\eta\vert}{\vert\xi+\eta\vert}\Big)\widetilde{a}(\xi-\eta,(\xi+\eta)/2)\widehat{f}(\eta)\Big]d\eta,
\end{equation*} 
and notice that $(\Omega_{\xi}+\Omega_{\eta})\Big[\chi\Big(\frac{\vert\xi-\eta\vert}{\vert\xi+\eta\vert}\Big)\Big]\equiv 0$. The formula \eqref{Alu3} follows.
\end{proof}

The paradifferential calculus is useful to linearize products and compositions. More precisely:

\begin{lemma}\label{PropHHSym}
(i) If $f,g\in L^2$ then
\begin{equation*}
fg=T_fg+T_gf+\mathcal{H}(f,g)
\end{equation*}
where $\mathcal{H}$ is smoothing in the sense that
\begin{equation*}
\Vert P_k\mathcal{H}(f,g)\Vert_{L^q}\lesssim \sum_{k',k''\geq k-40,\,|k'-k''|\leq 40}\min\big(\Vert P_{k'}f\Vert_{L^q}\Vert P_{k''}g\Vert_{L^\infty},\Vert P_{k'}f\Vert_{L^\infty}\Vert P_{k''}g\Vert_{L^q}\big).
\end{equation*}
As a consequence, if $f\in \mathcal{O}_{m,-5}$ and $g\in \mathcal{O}_{n,-5}$ then
\begin{equation}\label{aux1.1}
\langle t\rangle^{12\delta^2}\|\mathcal{H}(f,g)\|_{\mathcal{O}_{m+n,5}}\lesssim \|f\|_{\mathcal{O}_{m,-5}}\|g\|_{\mathcal{O}_{n,-5}}.
\end{equation}

(ii) Assume that $F(z)=z+h(z)$, where $h$ is analytic for $\vert z\vert <1/2$ and satisfies $\vert h(z)\vert\lesssim \vert z\vert^3$. If $\Vert u\Vert_{L^\infty}\le 1/100 $ and $N\geq 10$ then
\begin{equation}\label{Paracomp}
\begin{split}
F(u) & = T_{F^\prime(u)}u + E(u),
\\
\langle t\rangle^{12\delta^2}\Vert E(u)\Vert_{\mathcal{O}_{3,5}}&\lesssim \Vert u\Vert_{\mathcal{O}_{1,-5}}^3\qquad\text{ if }\qquad\Vert u\Vert_{\mathcal{O}_{1,-5}}\leq 1.
\end{split}
\end{equation}
\end{lemma}

\begin{proof}
(i) This follows easily by defining $\mathcal{H}(f,g)=fg-T_fg-T_gf$ and observing that 
\begin{equation*}
P_k\mathcal{H}(P_{k'}f,P_{k''}g)\equiv 0\quad\text{ unless }\quad k',k''\geq k-40,\,|k'-k''|\leq 40. 
\end{equation*}
The bound \eqref{aux1.1} follows as in the proof of Lemma \ref{lemmaaux1} (the remaining bilinear interactions correspond essentially to the set $\mathcal{X}_k^3$)

(ii) Since $F$ is analytic, it suffices to show this for $F(x)=x^n$, $n\ge 3$. This follows, however, as in part (i), using the Littlewood--Paley decomposition for $u$.
\end{proof}

We show now that compositions of paradifferential operators can be approximated well by paradifferential operators with suitable symbols. More precisely: 

\begin{proposition}\label{PropCompSym}
Let $1\le q\le\infty$. Given symbols $a$ and $b$, we may decompose
\begin{equation}
\label{PropCompSym1}
T_aT_b=T_{ab}+\frac{i}{2}T_{\{a,b\}}+E(a,b).
\end{equation}
The error $E$ obeys the following bounds: assuming $k\geq -100$,
\begin{equation}
\label{RemRemBounds}
\Vert P_kE(a,b)f\Vert_{L^q}\lesssim 2^{-2k}\Vert a\Vert_{\mathcal{M}_{16,\infty}}\Vert b\Vert_{\mathcal{M}_{16,\infty}}\Vert P_{[k-5,k+5]}f\Vert_{L^q},\quad\text{ for }q\in\{2,\infty\},
\end{equation}
\begin{equation}
\label{RemRemBounds'}
\begin{split}
\Vert P_kE(a,b)f\Vert_{L^2}\lesssim 2^{-2k}\Vert a\Vert_{\mathcal{M}_{16,2}}\Vert b\Vert_{\mathcal{M}_{16,\infty}}\Vert P_{[k-5,k+5]}f\Vert_{L^\infty},\\
\Vert P_kE(a,b)f\Vert_{L^2}\lesssim 2^{-2k}\Vert a\Vert_{\mathcal{M}_{16,\infty}}\Vert b\Vert_{\mathcal{M}_{16,2}}\Vert P_{[k-5,k+5]}f\Vert_{L^\infty}.
\end{split}
\end{equation}
Moreover $E(a,b)=0$ if both $a$ and $b$ are independent of $x$.
\end{proposition}

\begin{proof}
We may assume that $a=P_{\leq k-100}a$ and $b=P_{\leq k-100}$, since the other contributions can also be estimated using Lemma \ref{PropSym} (i) and \eqref{PropSymb1}. In this case we write
\begin{equation*}
\begin{split}
(16\pi^4)&\mathcal{F}\left\{P_k(T_aT_b-T_{ab})f\right\}(\xi)= \varphi_k(\xi)\int_{\mathbb{R}^4}\widehat{f}(\eta)\varphi_{\leq k-100}(\xi-\theta)\varphi_{\leq k-100}(\theta-\eta)\\
&\times\Big[\widetilde{a}(\xi-\theta,\frac{\xi+\theta}{2})\widetilde{b}(\theta-\eta,\frac{\eta+\theta}{2})-\widetilde{a}(\xi-\theta,\frac{\xi+\eta}{2})\widetilde{b}(\theta-\eta,\frac{\xi+\eta}{2}) \Big]\,d\eta d\theta.
\end{split}
\end{equation*}
Moreover, using the definition,
\begin{equation*}
\begin{split}
&(16\pi^4)\mathcal{F}\left\{P_k(i/2)T_{\{a,b\}}f\right\}(\xi)= \varphi_k(\xi)\int_{\mathbb{R}^4}\widehat{f}(\eta)\varphi_{\leq k-100}(\xi-\theta)\varphi_{\leq k-100}(\theta-\eta)\\
&\times\Big[\frac{\theta-\eta}{2}(\nabla_\zeta\widetilde{a})(\xi-\theta,\frac{\xi+\eta}{2})\widetilde{b}(\theta-\eta,\frac{\xi+\eta}{2})-\widetilde{a}(\xi-\theta,\frac{\xi+\eta}{2})\frac{\xi-\theta}{2}(\nabla_\zeta\widetilde{b})(\theta-\eta,\frac{\xi+\eta}{2}) \Big]\,d\eta d\theta.
\end{split}
\end{equation*}
Therefore
\begin{equation}\label{Alu10}
\begin{split}
&(16\pi^4)P_kE(a,b)f= U^1f+U^2f+U^3f,\\
&\mathcal{F}(U^jf)(\xi)=\varphi_k(\xi)\int_{\mathbb{R}^4}\widehat{f}(\eta)\varphi_{\leq k-100}(\xi-\theta)\varphi_{\leq k-100}(\theta-\eta)m^j(\xi,\eta,\theta)\,d\eta d\theta,
\end{split}
\end{equation}
where
\begin{equation}\label{Alu11}
\begin{split}
m^1(\xi,\eta,\theta)&:=\widetilde{a}(\xi-\theta,\frac{\xi+\eta}{2})\widetilde{b}(\theta-\eta,\frac{\eta+\theta}{2})-\widetilde{a}(\xi-\theta,\frac{\xi+\eta}{2})\widetilde{b}(\theta-\eta,\frac{\xi+\eta}{2})\\
&-\widetilde{a}(\xi-\theta,\frac{\xi+\eta}{2})\frac{\theta-\xi}{2}(\nabla_\zeta\widetilde{b})(\theta-\eta,\frac{\xi+\eta}{2}),
\end{split}
\end{equation}
\begin{equation}\label{Alu12}
\begin{split}
m^2(\xi,\eta,\theta)&:=\widetilde{a}(\xi-\theta,\frac{\xi+\theta}{2})\widetilde{b}(\theta-\eta,\frac{\eta+\theta}{2})-\widetilde{a}(\xi-\theta,\frac{\xi+\eta}{2})\widetilde{b}(\theta-\eta,\frac{\eta+\theta}{2})\\
&-\frac{\theta-\eta}{2}(\nabla_\zeta\widetilde{a})(\xi-\theta,\frac{\xi+\eta}{2})\widetilde{b}(\theta-\eta,\frac{\eta+\theta}{2}),
\end{split}
\end{equation}
and
\begin{equation}\label{Alu13}
\begin{split}
m^3(\xi,\eta,\theta)&:=\frac{\theta-\eta}{2}(\nabla_\zeta\widetilde{a})(\xi-\theta,\frac{\xi+\eta}{2})\Big[\widetilde{b}(\theta-\eta,\frac{\eta+\theta}{2})-\widetilde{b}(\theta-\eta,\frac{\xi+\eta}{2})\Big].
\end{split}
\end{equation}

It remains to prove the bounds \eqref{RemRemBounds} and \eqref{RemRemBounds'} for the operators $U^j$, $j\in\{1,2,3\}$. The operators $U^j$ are similar, so we will only provide the details for the operator $U^1$. We rewrite
\begin{equation}\label{Alu13.5}
m^1(\xi,\eta,\theta)=\int_0^1\widetilde{a}(\xi-\theta,\frac{\xi+\eta}{2})\frac{(\theta-\xi)_j(\theta-\xi)_k}{4}(\partial_{\zeta_j}\partial_{\zeta_k}\widetilde{b})(\theta-\eta,\frac{\xi+\eta}{2}+s\frac{\theta-\xi}{2})(1-s)\, d s.
\end{equation}
Therefore
\begin{equation}\label{Alu14}
U^1f(x)=\int_{\mathbb{R}^2}f(y)K^1(x,y)\,dy
\end{equation}
where
\begin{equation*}
K^1(x,y):=C\int_{\mathbb{R}^6}e^{-iy\cdot \eta}e^{ix\cdot\xi}\varphi_k(\xi)\varphi_{\leq k-100}(\xi-\theta)\varphi_{\leq k-100}(\theta-\eta)m^1(\xi,\eta,\theta)\,d\eta d\theta d\xi.
\end{equation*}

We use the formula \eqref{Alu13.5} and make changes to variables to rewrite
\begin{equation*}
\begin{split}
K^1(x,y)&=C\int_0^1\,ds(1-s)\int_{\mathbb{R}^{10}}e^{-iy\cdot (\xi+\mu+\nu)}e^{ix\cdot\xi}e^{iz\cdot \mu}e^{iw\cdot \nu}\varphi_k(\xi)\varphi_{\leq k-100}(\mu)\varphi_{\leq k-100}(\nu)\\
&\times (\partial_{x_j}\partial_{x_k}a)(z,\xi+\mu/2+\nu/2)(\partial_{\zeta_j}\partial_{\zeta_k}b)(w,\xi+\mu/2+\nu/2+s\mu/2)\,d\mu d\nu d\xi dz dw.
\end{split}
\end{equation*}
We integrate by parts in $\xi,\mu,\nu$, using the operators $(2^{-2k}-\Delta_\xi)^2$, $(2^{-2k}-\Delta_\mu)^2$ and $(2^{-2k}-\Delta_\nu)^2$. It follows that
\begin{equation}\label{Alu15}
|K^1(x,y)|\lesssim \int_{\mathbb{R}^{10}}\frac{2^{-2k}}{(2^{-2k}+|x-y|^2)^2}\frac{2^{-2k}}{(2^{-2k}+|z-y|^2)^2}\frac{2^{-2k}}{(2^{-2k}+|w-y|^2)^2}F_{a,b}(z,w)\,dzdw,
\end{equation}
where, with $\underline{\varphi}(X,Y,Z):=\varphi_0(X)\varphi_{\leq -100}(Y)\varphi_{\leq -100}(Z)$,
\begin{equation*}
\begin{split}
F_{a,b}(z,w):=&2^{6k}\int_0^1 ds\int_{\mathbb{R}^6}
\Big|\big[(2^{-2k}-\Delta_\xi)^2(2^{-2k}-\Delta_\mu)^2(2^{-2k}-\Delta_\nu)^2\big]\big\{\underline{\varphi}(2^{-k}\xi,2^{-k}\mu,2^{-k}\nu)\\
&\times(\partial_{x_j}\partial_{x_k}a)(z,\xi+\mu/2+\nu/2)(\partial_{\zeta_j}\partial_{\zeta_k}b)(w,\xi+\mu/2+\nu/2+s\mu/2)\big\}\Big|\,d\xi d\mu d\nu.
\end{split}
\end{equation*}

With $|a|_{16}$ and $|b|_{16}$ defined as in \eqref{SymNorm1}, it follows that 
\begin{equation*}
\begin{split}
|F_{a,b}(z,w)|\lesssim 2^{-2k}&\int_0^1 ds\int_{\mathbb{R}^6}
|a|_{16}(z,\xi+\mu/2+\nu/2)|b|_{16}(w,\xi+\mu/2+\nu/2+s\mu/2)\\
&\times \varphi_{[-4,4]}(2^{-k}\xi)\varphi_{\leq -10}(2^{-k}\mu)\varphi_{\leq -10}(2^{-k}\nu)\,\frac{d\xi d\mu d\nu}{2^{6k}}.
\end{split}
\end{equation*}
The desired bounds \eqref{RemRemBounds} and \eqref{RemRemBounds'} for $U^1$ follow using also \eqref{Alu14} and \eqref{Alu15}.
\end{proof}

We also make the following observation: if $a=a(\zeta)$ is a Fourier multiplier, $b$ is a symbol, and $f$ is a function, then 
\begin{equation}\label{ExEab}
\begin{split}
\widehat{E(a,b)f}(\xi) & =\frac{1}{4\pi^2}\int_{\mathbb{R}^2} \chi(\frac{\vert\xi-\eta\vert}{\vert\xi+\eta\vert})
  \Big( a(\xi)-a(\frac{\xi+\eta}{2})-\frac{\xi-\eta}{2}\cdot\nabla a(\frac{\xi+\eta}{2}) \Big)
  \widetilde{b}(\xi-\eta,\frac{\xi+\eta}{2})\widehat{f}(\eta)d\eta,
\\
\widehat{E(b,a)f}(\xi)& = \frac{1}{4\pi^2}\int_{\mathbb{R}^2} \chi(\frac{\vert\xi-\eta\vert}{\vert\xi+\eta\vert})
  \Big( a(\eta)-a(\frac{\xi+\eta}{2})-\frac{\eta-\xi}{2}\cdot\nabla a(\frac{\xi+\eta}{2}) \Big)
  \widetilde{b}(\xi-\eta,\frac{\xi+\eta}{2})\widehat{f}(\eta)d\eta.
\end{split}
\end{equation}

\subsection{Decorated norms and estimates}

In the previous subsection we proved bounds on paraproduct operators. In our study of the water wave problem, we need to keep track of several parameters, such as order, decay, and vector-fields. It is convenient to use two compatible hierarchies of bounds, one for functions and one for symbols of operators.

\subsubsection{Decorated norms}

Recall the spaces $\mathcal{O}_{m,p}$ defined in \eqref{fw1}. We define now the norms we will use to measure symbols.

\begin{definition}\label{DefSym}
For $l\in[-10,10]$, $r \in\mathbb{Z}_+$, $m \in \{1,2,3,4\}$, $t\in[0,T]$, and $q\in\{2,\infty\}$,
we define classes of symbols 
$\mathcal{M}^{l,m}_{r,q}=\mathcal{M}^{l,m}_{r,q}(t)\subseteq C(\mathbb{R}^2\times\mathbb{R}^2:\mathbb{C})$
by the norms
\begin{equation}\label{Alu20}
\Vert a\Vert_{\mathcal{M}^{l,m}_{r,\infty}}
:=\sup_{j \leq N_1/2} \sup_{\vert \alpha\vert+\vert\beta\vert\le r}\sup_{\zeta\in\mathbb{R}^2}\,
 \langle t\rangle^{m(5/6-20\delta^2)+16\delta^2} \langle \zeta\rangle^{-l}\|\, 
   \vert\zeta\vert^{\vert\beta\vert} \partial^\beta_\zeta \partial_x^\alpha \Omega_{x,\zeta}^ja \Vert_{L^\infty_x},
\end{equation}
\begin{equation}\label{Alu21}
\Vert a\Vert_{\mathcal{M}^{l,m}_{r,2}}:=\sup_{j \leq N_1} \sup_{\vert \alpha\vert+\vert\beta\vert\le r}\sup_{\zeta\in\mathbb{R}^2}
\,\langle t\rangle^{(m-1)(5/6-20\delta^2)-2\delta^2}\langle \zeta\rangle^{-l}\Vert  \,  
   \vert\zeta\vert^{\vert\beta\vert} \partial^\beta_\zeta \partial_x^\alpha \Omega_{x,\zeta}^j a \Vert_{L^2_x}.
\end{equation}
Here
\begin{equation*}
\Omega_{x,\zeta}a:=\Omega_xa+\Omega_\zeta a=(x_1\partial_{x_2}-x_2\partial_{x_1}+\zeta_1\partial_{\zeta_2}-\zeta_2\partial_{\zeta_1})a,
\end{equation*}
see \eqref{Alu3}. We also define
\begin{align}
\Vert a\Vert_{\mathcal{M}^{l,m}_r}:=
  \Vert a\Vert_{\mathcal{M}^{l,m}_{r,\infty}}+\Vert a\Vert_{\mathcal{M}^{l,m}_{r,2}},\quad m\ge 1.
\end{align}
\end{definition}

Note that this hierarchy is naturally related to the hierarchy in terms of $\mathcal{O}_{m,p}$. In this definition the parameters $m$ (the ``multiplicity'' of $a$, related to the decay rate) and $l$ (the ``order'') will play an important role. Observe that for a function $f=f(x)$, and $m\in [1,4]$,
\begin{align}\label{OimpliesS}
\|f\|_{\M^{0,m}_{N_3+p}}\lesssim \|f\|_{\mathcal{O}_{m,p}}.
\end{align}
Note also that we have the simple linear rule
\begin{equation}\label{SymbDer}
\begin{split}
\Vert P_ka\Vert_{\mathcal{M}^{l,m}_{r,q}}&\lesssim 2^{-sk}\Vert P_ka\Vert_{\mathcal{M}^{l,m}_{r+s,q}},\qquad k\in\mathbb{Z},\,s\geq 0,\,q\in\{2,\infty\},
\end{split}
\end{equation}
and the basic multiplication rules
\begin{equation}\label{SymbMult}
\langle t\rangle^{2\delta^2}\big[\Vert ab\Vert_{\mathcal{M}^{l_1+l_2,m_1+m_2}_{r}}+\Vert \zeta\{a,b\}\Vert_{\mathcal{M}^{l_1+l_2,m_1+m_2}_{r-2}}\big]\lesssim \Vert a\Vert_{\mathcal{M}^{l_1,m_1}_{r}}\Vert b\Vert_{\mathcal{M}^{l_2,m_2}_{r}}.
\end{equation}

\subsubsection{Bounds on operators}
We may now pass the bounds proved in subsection \ref{BdPara} to decorated norms. We consider the action of paradifferential operators on the classes $\mathcal{O}_{k,p}$. 
We will often use the following simple facts:
paradifferential operators preserve frequency localizations,
\begin{align}
\label{CommFreqP}
P_k T_{a} f & = P_k T_{a} P_{[k-4,k+4]}f = P_k T_{a(x,\zeta)\varphi_{\leq k+4}(\zeta)} f ;
\end{align}
the rotation vector-field $\Omega$ acts nicely on such operators, see \eqref{Alu3},
\begin{align}
\label{CommFreqPR}
\Omega (T_{a} f)&=T_{\Omega_{x,\zeta} a}f + T_{a} (\Omega f);
\end{align}
the following relations between basic and decorated norms for symbols hold:
\begin{align}
\label{CommFreqP'}
\begin{split}
{\| \Omega_{x,\zeta}^j a(x,\zeta) \varphi_{\leq k}(\zeta) \|}_{\M_{r,\infty}} 
  & 
  \lesssim 2^{lk^+} {\| a\|}_{\M^{l,m}_{r,\infty}}\langle t\rangle^{-m(5/6-20\delta^2)-16\delta^2}, \quad 0 \leq j \leq N_1/2,
\\
{\| \Omega_{x,\zeta}^j a(x,\zeta) \varphi_{\leq k}(\zeta) \|}_{\M_{r,2}} & 
  \lesssim 2^{lk^+}{\| a\|}_{\M^{l,m}_{r,2}}\langle t\rangle^{-(m-1)(5/6-20\delta^2) + 2\delta^2}, \quad 0\leq j \leq N_1.
\end{split}
\end{align}
A simple application of the above remarks and Lemma \ref{PropSym} (i) gives the bound
\begin{equation}\label{ParaSob}
\Vert T_{\sigma}f\Vert_{H^s}\lesssim \langle t\rangle^{-m(5/6-20\delta^2)-16\delta^2}\Vert \sigma\Vert_{\mathcal{M}^{l,m}_{8}}\Vert f\Vert_{H^{s+l}}.
\end{equation}

We prove now two useful lemmas:

\begin{lemma}\label{lemmaaux2}
If $q,q-l\in[-N_3,10]$ and $m,m_1\geq 1$ then
\begin{align}\label{aux21}
\langle t\rangle^{12\delta^2}T_{a} \mathcal{O}_{m,q}\subseteq \mathcal{O}_{m+m_1,q-l}, \qquad \mbox{for} \quad a \in \M^{l,m_1}_{10},
\end{align}
In particular, using also \eqref{OimpliesS},
\begin{align}
\label{aux25}
\begin{split}
\langle t\rangle^{12\delta^2}T_{\mathcal{O}_{m_1,-10}} \mathcal{O}_{m,q}\subseteq \mathcal{O}_{m+m_1,q}.
\end{split}
\end{align}
\end{lemma}

\begin{proof}
The estimate \eqref{aux21} follows using the definitions and the linear estimates \eqref{LqBdTa} and \eqref{L2BdTa} in Lemma \ref{PropSym}. We may assume $m=m_1=1$. Using \eqref{LqBdTa} and \eqref{CommFreqP'} we estimate
\begin{align*}
\begin{split}
2^{(N_0+q-l)k^+} {\| P_k T_{a} f\|}_{L^2} &\lesssim  
  {\| a \|}_{\M_{8,\infty}} 2^{(N_0+q-l)k^+} {\| P_{[k-2,k+2]} f\|}_{L^2}\\
	&\lesssim \langle t\rangle^{-5/6+4\delta^2}{\| a \|}_{\M_{8,\infty}^{l,1}}2^{(N_0+q)k^+} {\| P_{[k-2,k+2]} f\|}_{L^2},
	\end{split}
\end{align*}
for any $f\in \mathcal{O}_{1,q}$. By orthogonality we deduce the desired bound on the $H^{N_0}$ norm. 

To estimate the weighted norm we use \eqref{LqBdTa}, \eqref{L2BdTa}, and \eqref{CommFreqP'} to estimate
\begin{align*}
&2^{(N_3+q-l)k^+} {\| \Omega^j P_k T_{a} f\|}_{L^2} \lesssim  
  \sum_{n\leq j/2} 2^{(N_3+q-l)k^+} \big[{\| P_k T_{\Omega_{x,\zeta}^n a} \Omega^{j-n} f\|}_{L^2} 
  + {\| P_k T_{\Omega_{x,\zeta}^{j-n} a} \Omega^{n} f\|}_{L^2}\big]
\\
& \lesssim  
  \sum_{n\leq j/2} 2^{(N_3+q-l)k^+} \big[\|\Omega^{n}_{x,\zeta}a\|_{\mathcal{M}_{8,\infty}}{\| P_{[k-2,k+2]}\Omega^{j-n} f\|}_{L^2} +\|\Omega^{j-n}_{x,\zeta}a\|_{\mathcal{M}_{8,2}}{\| P_{[k-2,k+2]}\Omega^{n} f\|}_{L^\infty}\big]
\\
& \lesssim  
  \sum_{n\leq j/2} 2^{(N_3+q)k^+} \|a\|_{\mathcal{M}^{l,1}_{8}}\big[\langle t\rangle^{-5/6+4\delta^2}{\| P_{[k-2,k+2]}\Omega^{j-n} f\|}_{L^2} +\langle t\rangle^{2\delta^2}{\| P_{[k-2,k+2]}\Omega^{n} f\|}_{L^\infty}\big],
\end{align*}
for every $j\in[0,N_1]$. The desired weighted $L^2$ bound follows since
\begin{equation*}
\begin{split}
\Big[\sum_{k\in\mathbb{Z}}2^{2(N_3+q)k^+}{\| P_{[k-2,k+2]}\Omega^{j-n} f\|}_{L^2}^2\Big]^{1/2}+\langle t\rangle^{5/6-2\delta^2}\Big[\sum_{k\in\mathbb{Z}}2^{2(N_3+q)k^+}{\| P_{[k-2,k+2]}\Omega^{n} f\|}_{L^\infty}^2\Big]^{1/2}\\
\lesssim \langle t\rangle^{2\delta^2}\|f\|_{\mathcal{O}_{1,q}}.
\end{split}
\end{equation*}

Finally, for the $L^\infty$ bound we use \eqref{LqBdTa} to estimate
\begin{align*}
2^{(N_2+q-l)k^+} {\| \Omega^j P_k T_{a}f\|}_{L^\infty} & \lesssim  
  \sum_{j_1,j_2 \leq N_1/2} 2^{(N_2+q-l)k^+} {\|\Omega_{x,\zeta}^{j_1} a\|}_{\mathcal{M}_{8,\infty}} {\| P_{[k-2,k+2]}\Omega^{j_2} f\|}_{L^\infty}\\
& \lesssim  \langle t\rangle^{-5/6+4\delta^2} {\| a\|}_{\M^{l,1}_{8,\infty}}
  \sum_{j_2\leq N_1/2} 2^{(N_2+q)k^+} {\| P_{[k-2,k+2]} \Omega^{j_2} f\|}_{L^\infty},
\end{align*}
for any $j\in[0,N_1/2]$. The desired bound follows by summation over $k$.
\end{proof}

\begin{lemma}\label{lemmaaux3}
Let $E$ be defined as in Proposition \ref{PropCompSym}. Assume that $m,m_1,m_2\geq 1$, $q,q-l_1,q-l_2,q-l_1-l_2\in[-N_3,10]$ and consider $a \in \M^{l_1,m_1}_{20}$, $b \in \M^{l_2,m_2}_{20}$. Then
\begin{align}
\begin{split}
\label{auxEab}
&\langle t\rangle^{12\delta^2}P_{\geq -100}E(a,b) \mathcal{O}_{m,q}\subseteq \mathcal{O}_{m+m_1+m_2,q -l_1-l_2+2},\\
&\langle t\rangle^{12\delta^2}P_{\geq -100}(T_aT_b+T_bT_a-2T_{ab})\mathcal{O}_{m,q}\subseteq \mathcal{O}_{m+m_1+m_2,q -l_1-l_2+2}.
\end{split}
\end{align}
In addition,
\begin{align}
\label{auxcomm}
\begin{split}
\langle t\rangle^{12\delta^2}[T_a, T_b] \mathcal{O}_{m,q} & \subseteq \mathcal{O}_{m+m_1+m_2,q -l_1-l_2+1},
\\
\langle t\rangle^{12\delta^2}(T_a T_b - T_{ab}) \mathcal{O}_{m,q} & \subseteq \mathcal{O}_{m+m_1+m_2,q -l_1-l_2+1}.
\end{split}
\end{align}
Moreover, if $a \in \M^{0,m_1}_{20}$, $b \in \M^{0,m_2}_{20}$ are functions then
\begin{align}\label{auxcomm2}
\langle t\rangle^{12\delta^2}(T_a T_b - T_{ab}) \mathcal{O}_{m,-5} \subseteq \mathcal{O}_{m+m_1+m_2,5}.
\end{align}
\end{lemma}

\begin{proof}
We record the formulas
\begin{equation}\label{Alu30}
\Omega_{x,\zeta}(ab)=(\Omega_{x,\zeta}a)b+a(\Omega_{x,\zeta}b),\qquad \Omega_{x,\zeta}(\{a,b\})=\{\Omega_{x,\zeta}a,b\}+\{a,\Omega_{x,\zeta}b\}.
\end{equation}
Therefore, letting $U(a,b):=T_aT_b-T_{ab}$, we have
\begin{equation}\label{Alu32}
\begin{split}
&[T_a,T_b]=U(a,b)-U(b,a),\qquad E(a,b)=U(a,b)-(i/2)T_{\{a,b\}},\\
&T_aT_b+T_bT_a-2T_{ab}=E(a,b)+E(b,a),
\end{split}
\end{equation}
and
\begin{equation}\label{Alu31}
\begin{split}
\Omega(U(a,b)f)&=U(\Omega_{x,\zeta}a,b)f+U(a,\Omega_{x,\zeta}b)f+U(a,b)\Omega f,\\
\Omega(T_{\{a,b\}}f)&=T_{\{\Omega_{x,\zeta}a,b\}}f+T_{\{a,\Omega_{x,\zeta}b\}}f+T_{\{a,b\}}\Omega f,\\
\Omega(E(a,b)f)&=E(\Omega_{x,\zeta}a,b)f+E(a,\Omega_{x,\zeta}b)f+E(a,b)\Omega f.
\end{split}
\end{equation}

The bound \eqref{auxcomm2} follows as in the proof of Lemma \ref{lemmaaux1}, once we notice that
\begin{equation*}
P_k[(T_a T_b - T_{ab}) f]=\sum_{\max(k_1,k_2)\geq k-40}P_k[(T_{P_{k_1}a} T_{P_{k_2}b} - T_{P_{k_1}aP_{k_2}b}) f].
\end{equation*}

The bounds \eqref{auxEab} follow from \eqref{RemRemBounds}--\eqref{RemRemBounds'} and \eqref{CommFreqP'}, in the same way the bound \eqref{aux21} in Lemma \ref{lemmaaux2} follows from \eqref{LqBdTa}--\eqref{L2BdTa}. Moreover, using \eqref{SymbMult},
\begin{equation*}
\langle t\rangle^{12\delta^2}\|\{a,b\}(x,\zeta)\varphi_{\geq -200}(\zeta)\|_{\mathcal{M}^{l_1+l_2-1,m_1+m_2}_{18}}\lesssim \|a\|_{\M^{l_1,m_1}_{20}}\|b\|_{\M^{l_2,m_2}_{20}}.
\end{equation*}
Therefore, using \eqref{aux21} and frequency localization,
\begin{equation}\label{Alu33}
\langle t\rangle^{12\delta^2}P_{\geq -100}T_{\{a,b\}}\mathcal{O}_{m,q}\subseteq \mathcal{O}_{m+m_1+m_2,q -l_1-l_2+1}.
\end{equation}
Therefore, using \eqref{Alu32} and \eqref{auxEab}, 
\begin{equation*}
\langle t\rangle^{12\delta^2}P_{\geq -100}U(a,b)\mathcal{O}_{m,q}\subseteq \mathcal{O}_{m+m_1+m_2,q -l_1-l_2+1}.
\end{equation*}

For \eqref{auxcomm} it remains to prove that
\begin{equation}\label{Alu35}
\langle t\rangle^{12\delta^2}P_{\leq 0}U(a,b)\mathcal{O}_{m,q}\subseteq \mathcal{O}_{m+m_1+m_2,q -l_1-l_2+1}.
\end{equation}
However, using \eqref{aux21} and \eqref{SymbMult},
\begin{equation*}
\langle t\rangle^{12\delta^2}T_aT_b\mathcal{O}_{m,q}\subseteq \mathcal{O}_{m+m_1+m_2,q -l_1-l_2},\qquad \langle t\rangle^{12\delta^2}T_{ab}\mathcal{O}_{m,q}\subseteq \mathcal{O}_{m+m_1+m_2,q -l_1-l_2},
\end{equation*}
and \eqref{Alu35} follows. This completes the proof of \eqref{auxcomm}. 
\end{proof}

\section{The Dirichlet-Neumann operator}\label{DNOpe}

Assume $(h,\phi)$ are as in Proposition \ref{MainBootstrapEn} and let $\Omega:=\{(x,z)\in\mathbb{R}^3:\,z\le h(x)\}$. Let $\Phi$ denote the (unique in a suitable space, see Lemma \ref{LemFirstLinearization}) harmonic function in $\Omega$ satisfying $\Phi(x,h(x))=\phi(x)$. We define the Dirichlet-Neumann\footnote{To be precise this is $\sqrt{1+\vert\nabla h\vert^2}$ times the standard Dirichlet-Neumann operator,
but we will slightly abuse notation and call $G(h)\phi$ the Dirichlet-Neumann operator.} map as
\begin{equation}\label{DefDN}
G(h)\phi=\sqrt{1+\vert\nabla h\vert^2}(\nu\cdot\nabla\Phi)
\end{equation}
where $\nu$ denotes the outward pointing unit normal to the domain $\Omega$.
The main result of this section is the following paralinearization of the Dirichlet-Neumann map.

\begin{proposition}
\label{DNmainpro}
Assume that $t\in[0,T]$ is fixed and $(h,\phi)$ satisfy 
\begin{equation}\label{Asshu}
\|\langle \nabla\rangle h\|_{\mathcal{O}_{1,0}}+\|\,|\nabla|^{1/2}\phi\|_{\mathcal{O}_{1,0}}\lesssim \varep_1.
\end{equation}
Define
\begin{align}\label{DefBV}
B := \frac{G(h)\phi+\nabla_xh\cdot\nabla_x\phi}{1+\vert\nabla h\vert^2}, 
  \qquad V := \nabla_x \phi - B\nabla_x h, \qquad \omega := \phi - T_B h.
\end{align}
Then we can paralinearize the Dirichlet-Neumann operator as
\begin{equation}
\label{DNmainformula}
G(h)\phi = T_{\lambda_{DN}}\omega- \div (T_V h) + G_2 + \varep_1^3\mathcal{O}_{3,3/2},
\end{equation}
recall the definition \eqref{fw1}, where
\begin{align}
\label{deflambda}
\begin{split}
\lambda_{DN} &:= \lambda^{(1)}+\lambda^{(0)}, 
\\
\lambda^{(1)}(x,\zeta) &:= \sqrt{(1+\vert\nabla h\vert^2)\vert\zeta\vert^2-(\zeta\cdot\nabla h)^2},
\\
\lambda^{(0)}(x,\zeta) &:= \Big( \frac{(1+|\nabla h|^2)^2}{2\lambda^{(1)}} 
  \Big\{ \frac{\lambda^{(1)}}{1+|\nabla h|^2}, \frac{\zeta \cdot \nabla h}{1+|\nabla h|^2} \Big\} + \frac{1}{2}\Delta h \Big) \varphi_{\geq 0}(\zeta).
\end{split}
\end{align}
The quadratic terms are given by
\begin{equation}
 \label{DNquadratic}
 \begin{split}
G_2 = G_2(h,\vert\nabla\vert^{1/2}\omega)\in\varep_1^2\mathcal{O}_{2,5/2},\qquad \widehat{G_2}(\xi)=\frac{1}{4\pi^2}\int_{\mathbb{R}^2} g_2(\xi,\eta) \widehat{h}(\xi-\eta) \vert\eta\vert^{1/2}\widehat{\omega}(\eta)\,d\eta,
\end{split}
\end{equation}
where $g_2$ is a symbol satisfying (see the definition of the class $S^\infty_\Omega$ in \eqref{defclassA})
\begin{align}
\begin{split}
 \label{DNquadraticsym}
{\| g_2^{k, k_1, k_2}(\xi,\eta) \|}_{S^\infty_\Omega} & \lesssim 2^k2^{\min\{k_1,k_2\}/2}\Big(\frac{1+2^{\min\{k_1,k_2\}}}{1+2^{\max\{k_1,k_2\}}}\Big)^{7/2}.
\end{split}
\end{align}
.
\end{proposition}

\begin{remark}\label{RmkLambda}
Using \eqref{deflambda}, Definition \ref{DefSym}, and \eqref{OimpliesS}--\eqref{SymbMult} we see that, for any $t\in[0,T]$,
\begin{align}
\lambda^{(1)}=|\zeta|(1+\varep_1^2\M^{0,2}_{N_3-1}), \qquad \lambda^{(0)} \in \varep_1\M^{0,1}_{N_3-2}.
\end{align}
For later use we further decompose $\lambda^{(0)}$ into its linear and higher order parts:
\begin{align}
\label{lambda0dec}
\begin{split}
& \lambda^{(0)} = \lambda^{(0)}_1 + \lambda^{(0)}_2,\qquad\lambda^{(0)}_1 := \Big[\frac{1}{2} \Delta h - \frac{1}{2} \frac{\zeta_j\zeta_k\partial_j\partial_kh}{|\zeta|^2}\Big] 
  \varphi_{\geq 0}(\zeta),
  \qquad \lambda^{(0)}_2 \in \varep_1^3\M^{0,3}_{N_3-2} .
\end{split}
\end{align}
According to the formulas in \eqref{deflambda} and \eqref{lambda0dec} we have:
\begin{align}
\label{estlambda}
\begin{split}
& \lambda_{DN} - |\zeta| - \lambda^{(0)}_1 \in \varep_1^2\M^{1,2}_{N_3-2}, \qquad  \lambda_{DN} - \lambda^{(1)} - \lambda^{(0)}_1 \in \varep_1^3\M^{0,3}_{N_3-2}.
\end{split}
\end{align}
\end{remark}

The proof of Proposition \ref{DNmainpro} relies on several results and is given at the end of the section.

\subsection{Linearization}\label{DNLinear}
We start with a result that identifies the linear and quadratic part of the Dirichlet-Neumann operator. 

We first use a change of variable to flatten the surface. We thus define
\begin{equation}\label{udef}
\begin{split}
u(x,y) &:= \Phi(x,h(x)+y),\qquad(x,y)\in\mathbb{R}^2\times(-\infty,0],
\\
\Phi(x,z) &= u(x,z-h(x)).
\end{split}
\end{equation}
In particular $u_{|y=0} = \phi$, $\partial_y u_{|y=0} = B$,
and the Dirichlet-Neumann operator is given by
\begin{equation}\label{DNIny}
G(h)\phi=(1+\vert\nabla h\vert^2)\partial_yu_{|y=0}-\nabla_xh\cdot\nabla_xu_{|y=0}.
\end{equation}
A simple computation yields
\begin{equation}\label{EllEq}
\begin{split}
0&=\Delta_{x,z}\Phi=(1+\vert\nabla_xh\vert^2)\partial_y^2u+\Delta_xu-2\partial_y\nabla_xu\cdot\nabla_xh-\partial_yu\Delta_xh.
\end{split}
\end{equation}
Since we will also need to study the linearized operator, it is convenient to also allow for error terms and consider
the equation
\begin{equation}\label{EllEq2}
\begin{split}
(1+\vert\nabla_xh\vert^2)\partial_y^2u+\Delta_xu-2\partial_y\nabla_xu\cdot\nabla_xh-\partial_yu\Delta_xh
  =\partial_y\mathfrak{e}_a+\vert\nabla\vert \mathfrak{e}_b.
\end{split}
\end{equation}
With $\mathcal{R}:=|\nabla|^{-1}\nabla$ (the Riesz transform), this can be rewritten in the form
\begin{equation}\label{fma1}
\begin{split}
& (\partial_y^2-\vert\nabla\vert^2)u = \partial_yQ_a+\vert\nabla\vert Q_b,
\\
& Q_a:=\nabla u\cdot\nabla h-\vert\nabla h\vert^2\partial_yu+\mathfrak{e}_a, \qquad Q_b:=\mathcal{R}(\partial_yu\nabla h)+\mathfrak{e}_b.
\end{split}
\end{equation}

To study the solution $u$ we will need an additional class of Banach spaces, to measure functions that depend on $y\in(-\infty,0]$ and $x\in\mathbb{R}^2$. These spaces are only used in this section. 

\begin{definition}\label{Lspaces}
For $t\in[0,T]$, $p\geq -10$, and $m\geq 1$ let $\mathcal{L}_{m,p}=\mathcal{L}_{m,p}(t)$ denote the Banach space of functions $g\in C((-\infty,0]:\dot{H}^{1/2,1/2})$ defined by the norm
\begin{align}\label{fw2}
\|g\|_{\mathcal{L}_{m,p}}:=\||\nabla| g\|_{L^2_y\mathcal{O}_{m,p}}+\|\partial_yg\|_{L^2_y\mathcal{O}_{m,p}}+\||\nabla|^{1/2} g\|_{L^\infty_y\mathcal{O}_{m,p}}.
\end{align}
\end{definition}

The point of these spaces is to estimate solutions of equations of the form $(\partial_y-|\nabla|)u=\mathcal{N}$, in terms of the initial data $u(0)=\psi$. It is easy to see that if $|\nabla|^{1/2}\psi\in {\mathcal{O}_{m,p}}$ then 
\begin{equation}\label{fw3}
\|e^{y|\nabla|}\psi\|_{\mathcal{L}_{m,p}}\lesssim \||\nabla|^{1/2}\psi\|_{\mathcal{O}_{m,p}}.
\end{equation} 
To see this estimate for the $L^2_y\widetilde{W}_{\Omega}^{N_1/2,N_2+p}$ component we use the bound $\|c\|_{L^2_y\ell^1_k}\lesssim \|c\|_{\ell^1_kL^2_y}$ for any $c:\mathbb{Z}\times(-\infty,0]\to\mathbb{C}$. Moreover, if $Q\in L^2_y\mathcal{O}_{m,p}$ then
\begin{equation}\label{fma7}
\Big\||\nabla|^{1/2}\int_{-\infty}^0e^{-\vert y-s\vert\vert\nabla\vert}\mathbf{1}_{\pm}(y-s)Q(s)\,ds\Big\|_{L^\infty_y\mathcal{O}_{m,p}}\lesssim 
\langle t\rangle^{\delta^2/2}\|Q\|_{L^2_y\mathcal{O}_{m,p}}
\end{equation}
and
\begin{equation}\label{fma8}
\Big\||\nabla|\int_{-\infty}^0e^{-\vert y-s\vert\vert\nabla\vert}\mathbf{1}_{\pm}(y-s)Q(s)\,ds\Big\|_{L^2_y\mathcal{O}_{m,p}}\lesssim 
\langle t\rangle^{\delta^2/2}\|Q\|_{L^2_y\mathcal{O}_{m,p}}.
\end{equation}
Indeed, these bounds follow directly from the definitions for the $L^2$-based components of the space $\mathcal{O}_{m,p}$, which are $H^{N_0+p}$ and 
$H_\Omega^{N_1,N_3+p}$. For the remaining component one can control uniformly the $\widetilde{W}_{\Omega}^{N_1/2,N_2+p}$ norm of the function localized 
at every single dyadic frequency, without the factor of $\langle t\rangle^{\delta^2/2}$ in the right-hand side. The full bounds follow once we notice that 
only the frequencies satisfying $2^k\in[\langle t\rangle^{-8},\langle t\rangle^{8}]$ are relevant in the $\widetilde{W}_{\Omega}^{N_1/2,N_2+p}$ component of the space $\mathcal{O}_{1,p}$; the other frequencies are already accounted by the stronger Sobolev norms.

Our first result is the following:

\begin{lemma}\label{LemFirstLinearization}
(i) Assume that $t\in[0,T]$ is fixed, $\|\langle \nabla\rangle h\|_{\mathcal{O}_{1,0}}\lesssim\varep_1$, as in \eqref{Asshu}, and 
\begin{equation}\label{fw4}
\||\nabla|^{1/2}\psi\|_{\mathcal{O}_{1,p}}\leq A<\infty,\quad \|\mathfrak{e}_a\|_{L^2_y\mathcal{O}_{1,p}}+\|\mathfrak{e}_b\|_{L^2_y\mathcal{O}_{1,p}}\leq A\varepsilon_1\langle t\rangle^{-12\delta^2},
\end{equation}
for some $p\in[-10,0]$. Then there is a unique solution $u\in \mathcal{L}_{1,p}$ of the equation
\begin{equation}\label{fma2}
\begin{split}
u(y) & = e^{y\vert\nabla\vert}\Big(\psi-\frac{1}{2}\int_{-\infty}^0e^{s\vert\nabla\vert}(Q_a(s)-Q_b(s))ds\Big)
\\
&+\frac{1}{2}\int_{-\infty}^0e^{-\vert y-s\vert\vert\nabla\vert}(\mathrm{sgn}(y-s)Q_a(s)-Q_b(s))ds,
\end{split}
\end{equation}
where $Q_a$ and $Q_b$ are as in \eqref{fma1}. Moreover, $u$ is a solution of the equation $(\partial_y^2-\vert\nabla\vert^2)u = \partial_yQ_a+\vert\nabla\vert Q_b$ in \eqref{fma1} (and therefore a solution of \eqref{EllEq2} in $\mathbb{R}^2\times(-\infty,0]$), and
\begin{equation}\label{fma6}
\|u\|_{\mathcal{L}_{1,p}}=\||\nabla| u\|_{L^2_y\mathcal{O}_{1,p}}+\|\partial_yu\|_{L^2_y\mathcal{O}_{1,p}}+\||\nabla|^{1/2} u\|_{L^\infty_y\mathcal{O}_{1,p}}\lesssim A.
\end{equation}

(ii) Assume that we make the stronger assumptions, compare with \eqref{fw4},
\begin{equation}\label{fma6.1}
\||\nabla|^{1/2}\psi\|_{\mathcal{O}_{1,p}}\leq A<0,\quad \|\partial_y^j\mathfrak{e}\|_{L^2_y\mathcal{O}_{2,p-j}}+\|\partial_y^j\mathfrak{e}\|_{L^\infty_y\mathcal{O}_{2,p-1/2-j}}\leq A\varepsilon_1\langle t\rangle^{-12\delta^2},
\end{equation}
for $\mathfrak{e}\in\{\mathfrak{e}_a,\mathfrak{e}_b\}$ and $j\in\{0,1,2\}$. Then
\begin{equation}\label{fma6.2}
\|\partial_y^j(\partial_y u-|\nabla|u)\|_{L^2_y\mathcal{O}_{2,p-j}}+\|\partial_y^j(\partial_y u-|\nabla|u)\|_{L^\infty_y\mathcal{O}_{2,p-1/2-j}}\lesssim A\varepsilon_1.
\end{equation}
\end{lemma}

\begin{proof}
(i) We use a fixed point argument in a ball of radius $\approx A$ in $\mathcal{L}_{1,p}$, for the functional 
\begin{equation}\label{fma3}
\begin{split}
\Phi(u):&=e^{y\vert\nabla\vert}\Big[\psi-\frac{1}{2}\int_{-\infty}^0e^{s\vert\nabla\vert}(Q_a(s)-Q_b(s))ds\Big]\\
&+\frac{1}{2}\int_{-\infty}^0e^{-\vert y-s\vert\vert\nabla\vert}(\mathrm{sign}(y-s)Q_a(s)-Q_b(s))ds.
\end{split}
\end{equation}
Notice that, using Lemma \ref{lemmaaux1} and \eqref{fw4}, if $\|u\|_{\mathcal{L}_{1,p}}\lesssim 1$ then
\begin{equation*}
\|Q_a\|_{L^2_y\mathcal{O}_{1,p}}+\|Q_b\|_{L^2_y\mathcal{O}_{1,p}}\lesssim A\varepsilon_1\langle t\rangle^{-12\delta^2}.
\end{equation*}
Therefore, using \eqref{fw3}--\eqref{fma8}, $\|\Phi(u)-e^{y|\nabla|}\psi\|_{\mathcal{L}_{1,p}}\lesssim A\varepsilon_1$. Similarly, one can show that $\|\Phi(u)-\Phi(v)\|_{\mathcal{L}_{1,p}}\lesssim \varepsilon_1\|u-v\|_{\mathcal{L}_{1,p}}$, and the desired conclusion follows.

(ii) The identity \eqref{fma2} shows that
\begin{equation}\label{QuadraticTermDN}
\partial_yu(y)-\vert\nabla\vert u(y)=Q_a(y) + \int_{-\infty}^y \vert\nabla\vert e^{-\vert s-y\vert \vert\nabla\vert}(Q_b(s)-Q_a(s))ds.
\end{equation}
Given \eqref{fma6}, the definition \eqref{fma1}, and the stronger assumptions in \eqref{fma6.1}, we have
\begin{equation}\label{fma8.4}
\|Q\|_{L^2_y\mathcal{O}_{2,p}}+\|Q\|_{L^\infty_y\mathcal{O}_{2,p-1/2}}\lesssim A\varepsilon_1\langle t\rangle^{-12\delta^2},
\end{equation}
for $Q\in\{Q_a,Q_b\}$. Using estimates similar to \eqref{fma7} and \eqref{fma8} it follows that
\begin{equation}\label{fma9.5}
\|\partial_yu-\vert\nabla\vert u\|_{L^2_y\mathcal{O}_{2,p}}+\|\partial_yu-\vert\nabla\vert u\|_{L^\infty_y\mathcal{O}_{2,p-1/2}}\lesssim A\varepsilon_1.
\end{equation}

To prove \eqref{fma6.2} for $j\in\{1,2\}$, we observe that, as a consequence of \eqref{EllEq2},
\begin{equation}\label{fma9}
\partial_y^2u-|\nabla|^2u= (1+\vert\nabla_xh\vert^2)^{-1}
  (-|\nabla|^2u|\nabla_xh|^2+2\partial_y\nabla_xu\cdot\nabla_xh+\partial_yu\Delta_xh+\partial_y\mathfrak{e}_a+\vert\nabla\vert \mathfrak{e}_b).
\end{equation}
Using \eqref{fma6} and \eqref{fma9.5}, together with Lemma \ref{lemmaaux1}, it follows that 
\begin{equation*}
\|\partial_y^2u-|\nabla|^2u\|_{L^2_y\mathcal{O}_{2,p-1}}+\|\partial_y^2 u-|\nabla|^2u\|_{L^\infty_y\mathcal{O}_{2,p-3/2}}\lesssim A\varepsilon_1.
\end{equation*}
The desired bound \eqref{fma6.2} for $j=1$ follows using also \eqref{fma9.5}. The bound for $j=2$ then follows by differentiating \eqref{fma9} with respect to $y$. This completes the proof of the lemma.
\end{proof}

\subsection{Paralinearization}

The previous analysis allows us to isolate the linear (and the higher order) components of the Dirichlet-Neumann operator.
However, this is insufficient for our purpose because we also need to avoid losses of derivatives in the equation. To deal with this we follow the approach of Alazard-Metivier \cite{AlMet1}, 
Alazard-Burq-Zuily \cite{ABZ1,ABZ2} and Alazard-Delort \cite{ADa} using paradifferential calculus.
Our choice is to work with the (somewhat unusual) Weyl quantization, instead of the standard one used by the cited authors.
We refer to Appendix \ref{SecParaOp} for a review of the paradifferential calculus using the Weyl quantization.

For simplicity of notation, we set $\alpha=\vert\nabla h\vert^2$ and let 
\begin{equation}
 \label{defomega0}
\omega := u - T_{\partial_yu} h.
\end{equation}
Notice that $\omega$ is naturally extended to the fluid domain, compare with the definition \eqref{DefBV}. We will also assume \eqref{Asshu} and use Lemma \ref{LemFirstLinearization}. Using \eqref{aux25} in Lemma \ref{lemmaaux2} and \eqref{fma6.2}, we see that
\begin{equation}
\label{relomegau}
\|\omega - u\|_{L^2_y\mathcal{O}_{2,1}\cap L^\infty_y\mathcal{O}_{2,1}}\lesssim \varep_1^2.
\end{equation}
Using Lemma \ref{PropHHSym} to paralinearize products, we may rewrite the equation \eqref{EllEq} as
\begin{equation}\label{fma25}
\begin{split}
T_{1+\alpha}\partial_y^2\omega+\Delta \omega-2T_{\nabla h}\nabla\partial_y\omega-T_{\Delta h}\partial_y\omega=\mathcal{Q}+\mathcal{C}
\end{split}
\end{equation}
where
\begin{equation}\label{QuadraticCubic1}
\begin{split}
-\mathcal{Q}&=-2\mathcal{H}(\nabla h,\nabla \partial_yu)-\mathcal{H}(\Delta h,\partial_yu),
\\
-\mathcal{C}&=\partial_y(T_{1+\alpha}T_{\partial_y^2u}+T_{\Delta u}-2T_{\nabla h}T_{\nabla\partial_yu}-T_{\Delta h}T_{\partial_yu})h+2(T_{\partial_y^2u}T_{\nabla h}-T_{\nabla h}T_{\partial_y^2u})\nabla h
\\
&+T_{\partial_y^2u}\mathcal{H}(\nabla h,\nabla h)+\mathcal{H}(\alpha,\partial_y^2u).
\end{split}
\end{equation}
Notice that the error terms are quadratic and cubic strongly semilinear. More precisely, 
using Lemma \ref{PropHHSym}, Lemma \ref{lemmaaux3}, and the equation \eqref{EllEq}, we see that
\begin{equation}\label{CO3}
\mathcal{Q}\in\varep_1^2[L^2_y\mathcal{O}_{2,4}\cap L^\infty_y\mathcal{O}_{2,4}],\qquad \mathcal{C}\in\varep_1^3\langle t\rangle^{-11\delta^2}[L^2_y\mathcal{O}_{3,4}\cap L^\infty_y\mathcal{O}_{3,4}].
\end{equation}

We now look for a factorization of the main elliptic equation into
\begin{equation*}
\begin{split}
& T_{1+\alpha}\partial_y^2+\Delta -2T_{\nabla h}\nabla\partial_y-T_{\Delta h}\partial_y
\\
& = (T_{\sqrt{1+\alpha}}\partial_y-A+B)(T_{\sqrt{1+\alpha}}\partial_y-A-B)+\mathcal{E}
\\
& = T_{\sqrt{1+\alpha}}^2\partial_y^2-\big\{(AT_{\sqrt{1+\alpha}}+T_{\sqrt{1+\alpha}}A)+[T_{\sqrt{1+\alpha}},B]\big\}\partial_y+A^2-B^2+[A,B]+\mathcal{E}
\end{split}
\end{equation*}
where the error term is acceptable (in a suitable sense to be made precise later), and $[A,\partial_y]=0,[B,\partial_y]=0$. Identifying the terms, this leads to the system
\begin{equation*}
\begin{split}
T_{\sqrt{1+\alpha}}A+AT_{\sqrt{1+\alpha}}+[T_{\sqrt{1+\alpha}},B] & = 2T_{i\zeta\cdot\nabla h}+\mathcal{E},
\\
A^2-B^2+[A,B]&=\Delta+\mathcal{E}.
\end{split}
\end{equation*}
We may now look for paraproduct solutions in the form
\begin{equation*}
A=iT_{a},\quad a=a^{(1)}+a^{(0)},\qquad B=T_b,\quad b=b^{(1)}+b^{(0)}
\end{equation*}
where both $a$ and $b$ are real and are a sum of a two symbols of order $1$ and $0$.
Therefore $A$ corresponds to the skew-symmetric part of the system,
while $B$ corresponds to the symmetric part. Using Proposition \ref{PropCompSym}, and formally identifying the symbols, we obtain the system
\begin{equation*}
\begin{split}
2ia\sqrt{1+\alpha}+i\{\sqrt{1+\alpha},b\} & =2i\zeta\cdot\nabla h +\varep_1^2\mathcal{M}^{-1,2}_{N_3-2},
\\
a^2 + b^2 + \{a,b\} & = \vert\zeta\vert^2+\varep_1^2\mathcal{M}^{0,2}_{N_3-2}.
\end{split}
\end{equation*}
We can solve this by letting
\begin{align*}
& a^{(1)} := \frac{\zeta\cdot\nabla h}{\sqrt{1+\alpha}}, \qquad a^{(0)} := -\frac{1}{2\sqrt{1+\alpha}}\{\sqrt{1+\alpha},b^{(1)}\}\varphi_{\geq 0}(\zeta),
\\
& b^{(1)} = \sqrt{|\zeta|^2 - (a^{(1)})^2}, 
  \qquad b^{(0)} = \frac{1}{2b^{(1)}} \big( -2a^{(1)}a^{(0)} - \{a^{(1)}, b^{(1)}\}\varphi_{\geq 0}(\zeta) \big).
\end{align*}
This gives us the following formulas:
\begin{align}
\label{Symbol1}
a^{(1)}&=\frac{1}{\sqrt{1+\vert\nabla h\vert^2}}\left(\zeta\cdot\nabla h\right)
  & = &\,\,(\zeta\cdot\nabla h)(1+\varep_1^2\mathcal{M}^{0,2}_{N_3}),
\\
b^{(1)}&=\sqrt{\frac{(1+\vert\nabla h\vert^2)\vert\zeta\vert^2-(\zeta\cdot\nabla h)^2}{1+\vert\nabla h\vert^2}}
  & = &\,\,\vert\zeta\vert(1+\varep_1^2\mathcal{M}^{0,2}_{N_3}),
\label{Symbol2}
\\
a^{(0)}&=-\frac{\big\{\sqrt{1+\vert\nabla h\vert^2},b^{(1)} \big\}}{2\sqrt{1+\vert\nabla h\vert^2}} \,\varphi_{\geq 0}(\zeta)
  & = &\,\,\varphi_{\geq 0}(\zeta)\varep_1^2\mathcal{M}^{0,2}_{N_3-1},
\label{Symbol3}
\\
b^{(0)}&=-\frac{\sqrt{1+\vert\nabla h\vert^2}}{2b^{(1)}}
  \left\{\frac{\zeta\cdot\nabla h}{1+\vert\nabla h\vert^2},b^{(1)}\right\}\,\varphi_{\geq 0}(\zeta)&=&\,\,\varphi_{\geq 0}(\zeta)\Big[-\frac{\zeta_j\zeta_k\partial_{j}\partial_kh}
  {2\vert\zeta\vert^2}+\varep_1^3\mathcal{M}^{0,3}_{N_3-1}\Big]
\label{Symbol4}.
\end{align}

We now verify that
\begin{equation}\label{fma26}
\begin{split}
& (T_{\sqrt{1+\alpha}}\partial_y-iT_a+T_b)(T_{\sqrt{1+\alpha}}\partial_y-iT_a-T_b)
\\
&= T_{1+\alpha}\partial_y^2-\big(2T_{a\sqrt{1+\alpha}}+T_{\{\sqrt{1+\alpha},b^{(1)}\}}\big)i\partial_y-T_{a^2}-T_{b^2}-T_{\{a^{(1)},b^{(1)}\}\varphi_{\geq 0}(\zeta)}+\mathcal{E},
\end{split}
\end{equation}
where
\begin{equation*}
\begin{split}
\mathcal{E}&:=(T_{\sqrt{1+\alpha}}T_{\sqrt{1+\alpha}}-T_{1+\alpha})\partial_y^2-\big(T_aT_{\sqrt{1+\alpha}}+T_{\sqrt{1+\alpha}}T_a-2T_{a\sqrt{1+\alpha}}\big)i\partial_y-[T_{\sqrt{1+\alpha}},T_{b^{(0)}}]\partial_y
\\
&-\big([T_{\sqrt{1+\alpha}},T_{b^{(1)}}]-iT_{\{\sqrt{1+\alpha},b^{(1)}\}}\big)\partial_y
+(T_{a^2}- T_a^2) + (T_{b^2}-T_b^2) + i[T_{a},T_{b}]+T_{\{a^{(1)},b^{(1)}\}\varphi_{\geq 0}(\zeta)}.
\end{split}
\end{equation*}
We also verify that
\begin{equation*}
\begin{split}
&2a\sqrt{1+\alpha}+\{\sqrt{1+\alpha},b^{(1)}\}=2\zeta\cdot\nabla h+\{\sqrt{1+\alpha},b^{(1)}\}\varphi_{\leq -1}(\zeta),\\
&a^2 + b^2 + \{a^{(1)},b^{(1)}\}\varphi_{\geq 0}(\zeta) = \vert\zeta\vert^2+(a^{(0)})^2+(b^{(0)})^2.
\end{split}
\end{equation*}

\begin{lemma}\label{lemmafma}
With the definitions above, we have 
\begin{equation}\label{factor1}
(T_{\sqrt{1+\alpha}}\partial_y-iT_a+T_b)(T_{\sqrt{1+\alpha}}\partial_y-iT_a-T_b)\omega= Q_0 + \widetilde{\mathcal{C}},
\end{equation}
where
\begin{equation}\label{fma27}
\begin{split}
&\widetilde{\mathcal{C}}\in \varep_1^3\langle t\rangle^{-11\delta^2}[L^\infty_y\mathcal{O}_{3,1/2}\cap L^2_y\mathcal{O}_{3,1}],\qquad Q_0\in \varep_1^2[L^\infty_y\mathcal{O}_{2,3/2}\cap L^2_y\mathcal{O}_{2,2}],\\
&\widehat{Q_0}(\xi,y)=\frac{1}{4\pi^2}\int_{\mathbb{R}^2}q_0(\xi,\eta)\widehat{h}(\xi-\eta)\widehat{u}(\eta,y)\,d\eta,
\end{split}
\end{equation}
and
\begin{equation}\label{fma28}
\begin{split}
q_0(\xi,\eta):&=\chi\big(\frac{|\xi-\eta|}{|\xi+\eta|}\big)\frac{(|\xi|-|\eta|)^2(|\xi|+|\eta|)}{2}\Big[\frac{2\xi\cdot\eta-2|\xi||\eta|}{|\xi+\eta|^2}\varphi_{\geq 0}\big(\frac{\xi+\eta}{2}\big)+\varphi_{\leq -1}\big(\frac{\xi+\eta}{2}\big)\Big]\\
&+\Big[1-\chi\big(\frac{|\xi-\eta|}{|\xi+\eta|}\big)-\chi\big(\frac{|\eta|}{|2\xi-\eta|}\big)\Big](|\eta|^2-|\xi|^2)|\eta|.
\end{split}
\end{equation}
Notice that (see \eqref{mloc} for the definition),
\begin{equation}\label{Quad1symbound}
\|q_0^{k,k_1,k_2}\|_{S^\infty_\Omega}\lesssim 2^{k_2}2^{2k_1}\big[2^{-(2k_2-2k_1)}\mathbf{1}_{[-40,\infty)}(k_2-k_1)+\mathbf{1}_{(-\infty,4]}(k_2)\big],\qquad (\Omega_\xi+\Omega_\eta)q_0=0.
\end{equation}
\end{lemma}

\begin{proof} Using \eqref{fma25} and \eqref{fma26} we have
\begin{equation*}
\begin{split}
(T_{\sqrt{1+\alpha}}&\partial_y-iT_a+T_b)(T_{\sqrt{1+\alpha}}\partial_y-iT_a-T_b)\omega\\
&=\mathcal{Q}+\mathcal{C}+\mathcal{E}\omega-T_{(a^{(0)})^2+(b^{(0)})^2}\omega-T_{\{\sqrt{1+\alpha},b^{(1)}\}\varphi_{\leq -1}(\zeta)}i\partial_y\omega.
\end{split}
\end{equation*}
The terms $\mathcal{C}$, $T_{(a^{0})^2+(b^{(0)})^2}\omega$ and $T_{\{\sqrt{1+\alpha},b^{(1)}\}\varphi_{\leq -1}(\zeta)}i\partial_y\omega$ are in 
$\varep_1^3(1+t)^{-11\delta^2}[L^\infty_y\mathcal{O}_{3,1/2}\cap L^2_y\mathcal{O}_{3,1}]$. Moreover, using  Lemma \ref{LemFirstLinearization}, Lemma \ref{lemmaaux3}, and \eqref{Symbol1}--\eqref{Symbol4}, we can verify that
\begin{equation*}
\mathcal{E}\omega-\big[T_{2|\zeta|b^{(0)}_1}-T_{|\zeta|}T_{b^{(0)}_1}-T_{b^{(0)}_1}T_{|\zeta|}\big]\omega-\big[i[T_{\zeta\cdot\nabla h},T_{|\zeta|}]+T_{\{\zeta\cdot\nabla h,\vert\zeta\vert\}\varphi_{\geq 0}(\zeta)}\big]\omega
\end{equation*}
is an acceptable cubic error, where $b^{(0)}_1:=-\varphi_{\geq 0}(\zeta)\frac{\zeta_j\zeta_k\partial_j\partial_kh}{2|\zeta|^2}$. Indeed, most of the 
terms in $\mathcal{E}$ are already acceptable cubic errors; the last three terms become acceptable cubic errors after removing the quadratic components 
corresponding to the symbols $\zeta\cdot\nabla h$ in $a^{(1)}$, $|\zeta|$ in $b^{(1)}$, and $b^{(0)}_1$ in $b^{(0)}$. As a consequence, 
$\mathcal{E}\omega-Q'_0\in \varep_1^3\langle t\rangle^{-11\delta^2}[L^\infty_y\mathcal{O}_{3,1/2}\cap L^2_y\mathcal{O}_{3,1}]$, where
\begin{equation*}
\begin{split}
&\widehat{Q'_0}(\xi,y):=\frac{1}{4\pi^2}\int_{\mathbb{R}^2}\chi\Big(\frac{|\xi-\eta|}{|\xi+\eta|}\Big)q'_0(\xi,\eta)\widehat{h}(\xi-\eta)\widehat{\omega}(\eta,y)\,d\eta,\\
&q'_0(\xi,\eta):=\frac{(|\xi|-|\eta|)^2(|\xi|+|\eta|)(\xi\cdot\eta-|\xi||\eta|)}{|\xi+\eta|^2}\varphi_{\geq 0}\big(\frac{\xi+\eta}{2}\big)+\frac{(|\xi|-|\eta|)^2(|\xi|+|\eta|)}{2}\varphi_{\leq -1}\big(\frac{\xi+\eta}{2}\big).
\end{split}
\end{equation*}
The desired conclusions follow, using also the formula $\mathcal{Q}=2\mathcal{H}(\nabla h,\nabla \partial_yu)+\mathcal{H}(\Delta h,\partial_yu)$ in \eqref{QuadraticCubic1}, and the approximations $\partial_yu\approx |\nabla|u$, $\omega\approx u$, up to suitable quadratic errors. 
\end{proof}

In order to continue we want to invert the first operator in \eqref{factor1} which is elliptic in the domain under consideration.

\begin{lemma}\label{LemInt}
Let $U:=(T_{\sqrt{1+\alpha}}\partial_y-iT_a-T_b)\omega\in \varep_1[L^\infty_y\mathcal{O}_{1,-1/2}\cap L^2_y\mathcal{O}_{1,0}]$, so
\begin{equation}\label{fma39}
(T_{\sqrt{1+\alpha}}\partial_y-iT_a+T_b)U = Q_0 + \widetilde{\mathcal{C}}.
\end{equation}
Define
\begin{equation}\label{QuadrTrans}
\begin{split}
\widehat{M_0[f,g]}(\xi)&=\frac{1}{4\pi^2}\int_{\mathbb{R}^2}m_0(\xi,\eta)\widehat{f}(\xi-\eta)\widehat{g}(\eta)d\eta,\qquad m_0(\xi,\eta):=\frac{q_0(\xi,\eta)}{\vert\xi\vert+\vert\eta\vert}.
\end{split}
\end{equation}
Then, recalling the notation \eqref{fw1}, and letting $U_0:=U_{|y=0}$, $u_0:=u_{|y=0}=\phi$, we have
\begin{equation}\label{fma39.5}
P_{\ge -10} \big(U_0-M_0[h,u_0]\big)\in\varep_1^3\langle t\rangle^{-\delta^2}\mathcal{O}_{3,3/2}.
\end{equation}
\end{lemma}

\begin{proof}
Set
\begin{equation}\label{fma40}
\widetilde{U}:=T_{(1+\alpha)^{1/4}}U\in \varep_1[L^\infty_y\mathcal{O}_{1,-1/2}\cap L^2_y\mathcal{O}_{1,0}],\qquad \sigma:=\frac{b-ia}{\sqrt{1+\alpha}}=\vert\zeta\vert(1+\varep_1\mathcal{M}^{0,1}_{N_3-1}).
\end{equation}
Using \eqref{fma39} and Lemma \ref{lemmaaux3}, and letting $f:=(1+\alpha)^{1/4}-1\in\varep_1^2\mathcal{O}_{2,0}$, we calculate
\begin{equation*}
\begin{split}
& T_{(1+\alpha)^{1/4}}(\partial_y+T_{\sigma})\widetilde{U}=Q_0+\mathcal{C}_1,
\\
& \mathcal{C}_1:=\widetilde{\mathcal{C}}+[T^2_{f}-T_{f^2}]\partial_y U+\big[T_{f+1}T_{\sigma}T_{f+1}-T_{(f+1)^2\sigma}\big]U\in \varep_1^3\langle t\rangle^{-11\delta^2}\big[L^\infty_y\mathcal{O}_{3,1/2}\cap L^2_y\mathcal{O}_{3,1}\big].
\end{split}
\end{equation*}
Let $g=(1+f)^{-1}-1\in \varep_1^2\mathcal{O}_{2,0}$ and apply the operator $T_{1+g}$ to the identity above. Using Lemma \ref{lemmaaux3}, it follows that
\begin{equation}\label{NewParEq}
(\partial_y+T_{\sigma})\widetilde{U}=Q_0+\mathcal{C}_2,\qquad \mathcal{C}_2\in \varep_1^3\langle t\rangle^{-11\delta^2}\big[L^\infty_y\mathcal{O}_{3,1/2}\cap L^2_y\mathcal{O}_{3,1}\big].
\end{equation}

Notice that, using Lemma \ref{LemFirstLinearization}, \eqref{Quad1symbound}, \eqref{QuadrTrans} and Lemma \ref{lemmaaux1},
\begin{equation}\label{RegM0M1}
M_0[h,u]\in\varep_1^2[L^\infty_y\mathcal{O}_{2,5/2}\cap L^2_y\mathcal{O}_{2,3}],\qquad M_0[h,\partial_yu]\in\varep_1^2[L^\infty_y\mathcal{O}_{2,3/2}\cap L^2_y\mathcal{O}_{2,2}].
\end{equation}
We define $V:=\widetilde{U}-M_0[h,u]$. Since
\begin{equation*}
V=T_{(1+\alpha)^{1/4}}U-M_0[h,u]=T_{(1+\alpha)^{1/4}}\big(U-M_0[h,u]\big)+\mathcal{C'},\qquad\mathcal{C}'\in\varep_1^3\langle t\rangle^{-11\delta^2}L^\infty_y\mathcal{O}_{3,3/2},
\end{equation*}
for \eqref{fma39.5} it suffices to prove that
\begin{equation}\label{fma39.6}
P_{\ge -20} V(y)\in\varep_1^3\langle t\rangle^{-\delta^2}\mathcal{O}_{3,3/2}\quad\text{ for any }\quad y\in(-\infty,0].
\end{equation}

Using also \eqref{fma6.2} we verify that
\begin{equation}\label{fma74}
\begin{split}
(\partial_y+T_{\sigma})V&=(\partial_y+T_{\sigma})\widetilde{U}-(\partial_y+\vert\nabla\vert)M_0[h,u]-T_{(\sigma-\vert\zeta\vert)}M_0[h,u]\\
&=\mathcal{C}_2+M_0[h,\vert\nabla\vert u-\partial_yu]-T_{(\sigma-\vert\zeta\vert)}M_0[h,u]\\
&=\mathcal{C}_3\in \varep_1^3\langle t\rangle^{-11\delta^2}[L^\infty_y\mathcal{O}_{3,1/2}\cap L^2_y\mathcal{O}_{3,1}].
\end{split}
\end{equation}
Letting $\sigma':=\sigma-|\zeta|$ and $V_k:=P_k V$, $k\in\mathbb{Z}$, we calculate
\begin{equation*}
(\partial_y+T_{|\zeta|})V_k=P_k\mathcal{C}_3-P_{k}T_{\sigma'}V.
\end{equation*}
We can rewrite this equation in integral form,
\begin{equation}\label{fma71}
V_k(y)=\int_{-\infty}^ye^{(s-y)|\nabla|}[P_k\mathcal{C}_3(s)-P_{k}T_{\sigma'}V(s)]\,ds.
\end{equation}

To prove the desired bound for the high Sobolev norm, let, for $k\in\mathbb{Z}$,
\begin{equation*}
X_k:=\sup_{y\leq 0}2^{(N_0+3/2)k}\|V_k(y)\|_{L^2}.
\end{equation*}
Since $\sigma'/|\zeta|\in \varep_1\mathcal{M}^{0,1}_{N_3-1}$, it follows from Lemma \ref{lemmaaux2} that, for any $y\leq 0$,
\begin{equation*}
\begin{split}
2^{(N_0+3/2)k}\int_{-\infty}^y&\|e^{(s-y)|\nabla|}P_{k}T_{\sigma'}V(s)\|_{L^2}\,ds\\
&\lesssim 2^{(N_0+3/2)k}\varep_1\sum_{|k'-k|\leq 4}\int_{-\infty}^ye^{(s-y)2^{k-4}}2^k\|P_{k'}V(s)\|_{L^2}\,ds\lesssim \varep_1\sum_{|k'-k|\leq 4}X_{k'}.
\end{split}
\end{equation*}
It follows from \eqref{fma71} that, for any $k\in\mathbb{Z}$
\begin{equation*}
\begin{split}
X_k&\lesssim \varep_1\sum_{|k'-k|\leq 4}X_{k'}+\sup_{y\leq 0}2^{(N_0+3/2)k}\int_{-\infty}^ye^{(s-y)2^{k-4}}\|P_k\mathcal{C}_3(s)\|_{L^2}\,ds\\
&\lesssim \varep_1\sum_{|k'-k|\leq 4}X_{k'}+2^{(N_0+1)k}\Big[\int_{-\infty}^0\Vert P_k\mathcal{C}_3(s)\Vert_{L^2}^2\,ds\Big]^{1/2}.
\end{split}
\end{equation*}
We take $l^2$ summation in $k$, and absorb the first term in the right-hand side\footnote{To make this step rigorous, one can modify the definition of $X_k$ to $X'_k:=\sup_{y\leq 0}2^{(N_0+3/2)\min(k,K)}\|V_k(y)\|_{L^2}$, in order to make sure that $\sum_{k}(X'_k)^2<\infty$, and then prove uniform estimates in $K$ and finally let $K\to\infty$.} into the left-hand side, to conclude that
\begin{equation}\label{fma75}
\big(\sum_{k\in\mathbb{Z}}X_k^2\big)^{1/2}\lesssim \Big[\sum_{k\in\mathbb{Z}}2^{2(N_0+1)k}\int_{-\infty}^0 \Vert P_k\mathcal{C}_3(s)\Vert_{L^2}^2\,ds\Big]^{1/2}
\lesssim\varep_1^3\langle t\rangle^{-11\delta^2}\langle t\rangle^{-2(5/6-20\delta^2)+\delta^2},
\end{equation}
where the last inequality in this estimate is a consequence of $\mathcal{C}_3\in \varep_1^3\langle t\rangle^{-11\delta^2}L^2_y\mathcal{O}_{3,1}$. 
The desired bound $\|P_{\geq -20} V(y)\|_{H^{N_0+3/2}}\lesssim \varep_1^3\langle t\rangle^{-11\delta^2}\langle t\rangle^{-2(5/6-20\delta^2)+\delta^2}$ in \eqref{fma39.6} follows.

The proof of the bound for the weighted norms is similar. For $k\in\mathbb{Z}$ let
\begin{equation*}
Y_k:=\sup_{y\leq 0}2^{(N_3+3/2)k}\sum_{j\leq N_1}\|\Omega^jV_k(y)\|_{L^2}.
\end{equation*}
As before, we have the bounds,
\begin{equation*}
2^{(N_3+3/2)k}\int_{-\infty}^y\|e^{(s-y)|\nabla|}\Omega^jP_{k}T_{\sigma'}V(s)\|_{L^2}\,ds\lesssim \varep_1\sum_{|k'-k|\leq 4}[Y_{k'}+\langle t\rangle^{6\delta^2}X_{k'}],
\end{equation*}
for any $y\in(-\infty,0]$ and $j\leq N_1$, and therefore, using \eqref{fma71},
\begin{equation*}
Y_k\lesssim \varep_1\sum_{|k'-k|\leq 4}Y_{k'}+\varep_1\langle t\rangle^{6\delta^2}\sum_{|k'-k|\leq 4}X_{k'}+\sum_{j\leq N_1}2^{(N_3+1)k}\Big[\int_{-\infty}^0\Vert \Omega^jP_k\mathcal{C}_3(s)\Vert_{L^2}^2\,ds\Big]^{1/2}.
\end{equation*}
As before, we take the $l^2$ sum in $k$ and use \eqref{fma75} and the hypothesis $\mathcal{C}_3\in \varep_1^3\langle t\rangle^{-11\delta^2}L^2_y\mathcal{O}_{3,1}$. 
The desired bound $\|P_{\geq -20} V(y)\|_{H_\Omega^{N_1,N_3+3/2}}\lesssim \varep_1^3\langle t\rangle^{-4\delta^2}\langle t\rangle^{-2(5/6-20\delta^2)+\delta^2}$ in \eqref{fma39.6} follows.

Finally, for the $L^\infty$ bound, we let, for $k\in\mathbb{Z}$,
\begin{equation*}
Z_k:=\sup_{y\leq 0}2^{(N_2+3/2)k}\sum_{j\leq N_1/2}\|\Omega^jV_k(y)\|_{L^\infty}.
\end{equation*}
As before, using \eqref{fma71} it follows that
\begin{equation*}
Z_k\lesssim \varep_1\sum_{|k'-k|\leq 4}Z_{k'}+\sum_{j\leq N_1/2}2^{(N_2+1)k}\Big[\int_{-\infty}^0\Vert \Omega^jP_k\mathcal{C}_3(s)\Vert_{L^\infty}^2\,ds\Big]^{1/2}.
\end{equation*}
After taking $l^2$ summation in $k$ it follows that
\begin{equation*}
\begin{split}
\big(\sum_{k\in\mathbb{Z}}Z_k^2\big)^{1/2}&\lesssim \sum_{j\leq N_1/2}\Big[\sum_{k\in\mathbb{Z}}2^{2(N_2+1)k}\int_{-\infty}^0 \Vert \Omega^jP_k\mathcal{C}_3(s)\Vert_{L^\infty}^2\,ds\Big]^{1/2}\lesssim\varep_1^3\langle t\rangle^{-11\delta^2}\langle t\rangle^{-5/2+45\delta^2},
\end{split}
\end{equation*}
where the last inequality is a consequence of $\mathcal{C}_3\in \varep_1^3\langle t\rangle^{-11\delta^2}L^2_y\mathcal{O}_{3,1}$. The desired bound on $\|P_{\geq -20} V(y)\|_{\widetilde{W}_\Omega^{N_1/2,N_2+3/2}}$ in \eqref{fma39.6} follows, once we recall that only the sum over $2^{|k|}\leq \langle t\rangle^8$ is relevant when estimating the $\widetilde{W}_\Omega^{N_1/2,N_2+3/2}$ norm; the remaining frequencies are already accounted for by the stronger Sobolev norms. 
\end{proof}

We are now ready to obtain the paralinearization of the Dirichlet-Neumann operator.

\begin{proof}[Proof of Proposition \ref{DNmainpro}]
Recall that $G(h)\phi=(1+\vert\nabla h\vert^2)\partial_yu_{|y=0}-\nabla h \cdot \nabla u_{|y=0}$, see \eqref{DNIny}, and $B=\partial_yu_{|y=0}$. All the calculations below are done on the interface, at $y=0$. We observe that, using Corollary \ref{CorFixedPoint}, 
\begin{equation}\label{DNSmallFreq}
\begin{split}
P_{\le 6}&\big((1+\vert\nabla h\vert^2)\partial_yu - \nabla h \cdot \nabla u\big)=P_{\le 6}\left(\partial_yu- \nabla h \cdot \nabla u\right)+\varep_1^3\mathcal{O}_{3,3/2}\\
&=P_{\le 6}\left(\vert\nabla\vert\omega-\hbox{div}(T_Vh)\right)+P_{\le 6}\left(\hbox{div}(T_Vh)+\vert\nabla\vert T_{\vert\nabla\vert \omega}h+N_2[h,\omega]\right)+\varep_1^3\mathcal{O}_{3,3/2}.
\end{split}
\end{equation}
Thus low frequencies give acceptable contributions. To estimate high frequencies we compute
\begin{equation*}
\begin{split}
& (1+\vert\nabla h\vert^2)\partial_yu - \nabla h \cdot \nabla u
\\ 
& = T_{1+\alpha}\partial_yu - T_{\nabla h}\nabla u
  - T_{\nabla u}\nabla h+T_{\partial_yu}\alpha+\mathcal{H}(\alpha,\partial_yu)-\mathcal{H}(\nabla h,\nabla u)
\\
& = T_{1+\alpha}\partial_y\omega-T_{\nabla h}\nabla \omega-T_{\nabla u}\nabla h + T_{\nabla h}T_{\partial_yu}\nabla h
\\
& +(T_{\partial_yu}\alpha-2T_{\nabla h}T_{\partial_yu}\nabla h)+T_{1+\alpha}T_{\partial_y^2u}h-T_{\nabla h}T_{\nabla \partial_yu}h
  +\mathcal{H}(\alpha,\partial_yu)-\mathcal{H}(\nabla h,\nabla u).
\end{split}
\end{equation*}
Using Lemma \ref{LemInt} with $U= (T_{\sqrt{1+\alpha}}\partial_y-iT_a-T_b)\omega$ and \eqref{RegM0M1}, Lemma \ref{lemmaaux2}, and Lemma \ref{lemmaaux3}, we find that
\begin{equation*}
\begin{split}
T_{1+\alpha}\partial_y\omega & = T_{\sqrt{1+\alpha}}
  \big(iT_a\omega+T_b\omega+M_0[h,u]+\mathcal{C}'\big)+(T_{1+\alpha}-T_{\sqrt{1+\alpha}}^2)\partial_y\omega
\\
& =T_{\sqrt{1+\alpha}}(T_b+iT_a)\omega+M_0[h,u]+\mathcal{C}'',
\end{split}
\end{equation*}
where $\mathcal{C}''$ satisties $P_{\geq -6}\mathcal{C}''\in\varep_1^3\mathcal{O}_{3,3/2}$. Therefore, with $V=\nabla u-\partial_yu\nabla h$,
\begin{equation}
 \label{DNmainpro11}
\begin{split}
(1+\vert\nabla h\vert^2)\partial_yu-\nabla h\cdot\nabla u&=T_{\sqrt{1+\alpha}}(T_b+iT_{a})\omega+M_0[h,u]+\mathcal{C}''\\
&-T_{\nabla h}\nabla\omega-\hbox{div}(T_Vh)+\mathcal{C}_1+\mathcal{C}_2 -\mathcal{H}(\nabla h,\nabla u),
\end{split}
\end{equation}
with cubic terms $\mathcal{C}_1,\mathcal{C}_2$ given explicitly by
\begin{equation*}
\begin{split}
\mathcal{C}_1&=(T_{\partial_yu}\alpha-2T_{\nabla h}T_{\partial_yu}\nabla h)+\mathcal{H}(\alpha,\partial_yu),
\\
\mathcal{C}_2&=(T_{\mathrm{div}\,V}+T_{1+\alpha}T_{\partial_y^2u}-T_{\nabla h}T_{\nabla \partial_yu})h+(T_{\nabla h}T_{\partial_yu}
  -T_{\partial_yu\nabla h})\nabla h.
\end{split}
\end{equation*}
Notice that $\mathrm{div}\,V+(1+\alpha)\partial_y^2u-\nabla h\nabla \partial_yu=0$, as a consequence of \eqref{EllEq}. Using also Lemma \ref{lemmaaux3} it follows that $\mathcal{C}_1, \mathcal{C}_2\in \varep_1^3\mathcal{O}_{3,3/2}$. 

Moreover, using the formulas \eqref{Symbol2}, \eqref{Symbol4}, Lemma \ref{PropCompSym}, and Lemma \ref{lemmaaux3}, we see that
\begin{align*}
T_{\sqrt{1+\alpha}}T_b\omega &= T_{b\sqrt{1+\alpha}}\omega + \frac{i}{2}T_{\{\sqrt{1+\alpha},b\}}\omega + E(\sqrt{1+\alpha}-1,b)\omega
  \\
& = T_{\lambda^{(1)}}\omega + T_{b^{(0)}\sqrt{1+\alpha}}\omega + \frac{i}{2}T_{\{\sqrt{1+\alpha},b^{(1)}\}}\omega + \varep_1^3\mathcal{O}_{3,3/2}
\end{align*}
where $\lambda^{(1)}$ is the principal symbol in \eqref{deflambda}. Similarly, using \eqref{Symbol1}, \eqref{Symbol3},
\begin{align*}
iT_{\sqrt{1+\alpha}}T_{a}\omega - T_{\nabla h}\nabla \omega
  & = T_{i \zeta \cdot \nabla h}\omega - T_{\nabla h}\nabla \omega
  + iT_{ a^{(0)} \sqrt{1+\alpha}}\omega - \frac{1}{2}T_{\{\sqrt{1+\alpha},a\}}\omega + iE(\sqrt{1+\alpha}-1,a)\omega
\\
& = \frac{1}{2}T_{\Delta h}\omega + iT_{ a^{(0)} \sqrt{1+\alpha}}\omega - \frac{1}{2}T_{\{\sqrt{1+\alpha},a^{(1)} \}}\omega
  + \varep_1^3\mathcal{O}_{3,3/2}.
\end{align*}
Summing these last two identities and using \eqref{Symbol1}-\eqref{Symbol4} we see that
\begin{align}
\label{DNmainpro12}
T_{\sqrt{1+\alpha}} T_b\omega &+ iT_{\sqrt{1+\alpha}}T_{a}\omega - T_{\nabla h}\nabla \omega = T_{\lambda^{(1)}}\omega 
  + T_{m}\omega + \varep_1^3\mathcal{O}_{3,3/2}
\end{align}
where
\begin{align}
\label{DNmainpro13}
\begin{split}
m := & \, b^{(0)}\sqrt{1+\alpha} - \frac{1}{2} \{\sqrt{1+\alpha},a^{(1)} \} + \frac{1}{2}\Delta h
\\ 
= & \frac{(1+\alpha)^{3/2}}{2\lambda^{(1)}}\Big\{\frac{\lambda^{(1)}}{\sqrt{1+\alpha}},\frac{\zeta\cdot\nabla h}{1+\alpha}\Big\}\varphi_{\geq 0}(\zeta)-\frac{1}{2}\Big\{\sqrt{1+\alpha},\frac{\zeta\cdot\nabla h}{\sqrt{1+\alpha}}\Big\} 
  + \frac{1}{2}\Delta h\\
	=& \lambda^{(0)}-\frac{1}{2}\Big\{\sqrt{1+\alpha},\frac{\zeta\cdot\nabla h}{\sqrt{1+\alpha}}\Big\}\varphi_{\leq -1}(\zeta)+\frac{\Delta h}{2}\varphi_{\leq -1}(\zeta).
\end{split}
\end{align}
We conclude from \eqref{DNmainpro11} and \eqref{DNmainpro12} that
\begin{equation*}
\begin{split}
P_{\ge 7}\big( (1+\vert\nabla h\vert^2)\partial_yu-\nabla h\nabla u  \big)&=P_{\ge 7}\big(T_{\lambda_{DN}}\omega-\hbox{div}(T_Vh)+M_0[h,u]-\mathcal{H}(\nabla h,\nabla u)+\varep_1^3\widetilde{\mathcal{O}}_{3,3/2}\big).
\end{split}
\end{equation*}
Moreover, the symbol of the bilinear operator $M_0[h,u]-\mathcal{H}(\nabla h,\nabla u)$ is
\begin{equation*}
\frac{q_0(\xi,\eta)}{|\xi|+|\eta|}+\Big[1-\chi\big(\frac{|\xi-\eta|}{|\xi+\eta|}\big)-\chi\big(\frac{|\eta|}{|2\xi-\eta|}\big)\Big](\xi-\eta)\cdot\eta,
\end{equation*}
where $q_0$ is defined in \eqref{fma28}. The symbol bounds \eqref{DNquadraticsym} follow. Combining this with \eqref{DNSmallFreq}, we finish the proof.
\end{proof}

\section{Taylor expansion of the Dirichlet--Neumann operator}\label{AppC}

\subsection{A simple expansion} We start a simple expansion the Dirichlet-Neumann operator, using only the $O_{m,p}$ hierarchy, which suffices in many cases.  

\begin{corollary}\label{CorFixedPoint}
(i) Assume that $\|\langle\nabla\rangle h\|_{\mathcal{O}_{1,0}}+\||\nabla|^{1/2}\psi\|_{\mathcal{O}_{1,0}}\lesssim \varepsilon_1$ and $\mathfrak{e}_a=0$, $\mathfrak{e}_b=0$, and define $u$ as in Lemma \ref{LemFirstLinearization}. Then we have an expansion
\begin{equation}\label{TaylorExp}
\partial_yu=\vert\nabla\vert u+\nabla h\cdot\nabla u+N_2[h,u]+\mathcal{E}^{(3)},\qquad \|\mathcal{E}^{(3)}\|_{L^2_y\mathcal{O}_{3,0}\cap L^\infty_y\mathcal{O}_{3,-1/2}}\lesssim \varepsilon_1^3\langle t\rangle^{-11\delta^2},
\end{equation}
where 
\begin{equation}\label{cvb8.8}
\mathcal{F}\{N_{2}[h,\phi]\}(\xi)=\frac{1}{4\pi^2}\int_{\mathbb{R}^2}n_2(\xi,\eta)\widehat{h}(\xi-\eta)\widehat{\phi}(\eta)\,d\eta,\qquad n_2(\xi,\eta):=\xi\cdot\eta-|\xi||\eta|,
\end{equation}
In particular,
\begin{equation}\label{TaylorExp2}
\big\|G(h)\psi-|\nabla|\psi-N_2[h,\psi]\big\|_{\mathcal{O}_{3,-1/2}}\lesssim\varep_1^3\langle t\rangle^{-11\delta^2}.
\end{equation}
Moreover
\begin{equation}\label{Qsmooth}
\Vert n_2^{k,k_1,k_2}\Vert_{S^\infty_{\Omega}}\lesssim 2^{\min\{k,k_1\}}2^{k_2},\qquad (\Omega_\xi + \Omega_\eta)n_2\equiv 0.
\end{equation}

(ii) As in Proposition \ref{MainBootstrapEn}, assume that $(h,\phi)\in C([0,T]:H^{N_0+1}\times \dot{H}^{N_0+1/2,1/2})$
is a solution of the system \eqref{WW0} with $g=1$ and $\sigma=1$, $t\in[0,T]$ is fixed, and \eqref{Asshu} holds. Then
\begin{equation}\label{dtg}
\big\|\partial_t(G(h)\phi)-\vert\nabla\vert\partial_t\phi\big\|_{\mathcal{O}_{2,-2}}\lesssim\varep_1^2.
\end{equation}
\end{corollary}

\begin{proof} (i) Let $u^{(1)}:=e^{y|\nabla|}\psi$ and $Q_a^{(1)}:=\nabla u^{(1)}\cdot\nabla h$, $Q_b^{(1)}:=\mathcal{R}(\partial_yu^{(1)}\nabla h)$. It follows from \eqref{fma7}--\eqref{fma8} and Lemma \ref{LemFirstLinearization} (more precisely, from \eqref{fma6}, \eqref{fma6.2}, and \eqref{fma8.4}) that
\begin{equation}\label{fma10}
\begin{split}
\||\nabla|^{1/2}(u-u^{(1)})&\|_{L^\infty_y\mathcal{O}_{2,0}}+\||\nabla|(u-u^{(1)})\|_{L^2_y\mathcal{O}_{2,0}}\\
&+\|\partial_y(u-u^{(1)})\|_{L^\infty_y\mathcal{O}_{2,-1/2}}+\|\partial_y(u-u^{(1)})\|_{L^2_y\mathcal{O}_{2,0}}\lesssim \varep_1^2.
\end{split}
\end{equation}
Therefore, using Lemma \ref{lemmaaux1}, for $d\in\{a,b\}$,
\begin{equation}\label{fma11}
\|Q_d-Q_d^{(1)}\|_{L^\infty_y\mathcal{O}_{3,-1/2}}+\|Q_d-Q_d^{(1)}\|_{L^2_y\mathcal{O}_{3,0}}\lesssim \varep_1^3\langle t\rangle^{-12\delta^2}.
\end{equation}
Therefore, using \eqref{fma7}--\eqref{fma8} and \eqref{QuadraticTermDN},
\begin{equation*}
\begin{split}
\Big\|\partial_yu-\vert\nabla\vert u-\nabla h\cdot\nabla u-\int_{-\infty}^y \vert\nabla\vert e^{-\vert s-y\vert \vert\nabla\vert}&(Q_b^{(1)}(s)-Q_a^{(1)}(s))ds\Big\|_{L^2_y\mathcal{O}_{3,0}\cap L^\infty_y\mathcal{O}_{3,-1/2}}\lesssim \varep_1^3\langle t\rangle^{-11\delta^2}.
\end{split}
\end{equation*}
Since
\begin{equation*}
\mathcal{F}\big\{Q_b^{(1)}(s)-Q_a^{(1)}(s)\big\}(\xi)=\frac{1}{4\pi^2}\int_{\mathbb{R}^2}\Big[\eta\cdot (\xi-\eta)-\frac{\xi\cdot(\xi-\eta)}{|\xi|}|\eta|\Big]\widehat{h}(\xi-\eta)e^{s|\eta|}\widehat{\psi}(\eta)\,d\eta,
\end{equation*}
we have
\begin{equation*}
\begin{split}
&\mathcal{F}\Big\{\int_{-\infty}^y \vert\nabla\vert e^{-\vert s-y\vert \vert\nabla\vert}(Q_b^{(1)}(s)-Q_a^{(1)}(s))ds\Big\}(\xi)\\
&=\frac{1}{4\pi^2}\int_{\mathbb{R}^2}\Big[\eta\cdot (\xi-\eta)-\frac{\xi\cdot(\xi-\eta)}{|\xi|}|\eta|\Big]\frac{|\xi|}{|\xi|+|\eta|}\widehat{h}(\xi-\eta)e^{y|\eta|}\widehat{\psi}(\eta)\,d\eta\\
&=\mathcal{F}\{N_2[h,u^{(1)}]\}(\xi).
\end{split}
\end{equation*}
Moreover, using the assumption $\|\langle \nabla\rangle h\|_{\mathcal{O}_{1,0}}\lesssim \varepsilon_1$ and the bounds \eqref{fma10}, we have
\begin{equation*}
\|N_2[h,u-u^{(1)}]\|_{L^2_y\mathcal{O}_{3,0}\cap L^\infty_y\mathcal{O}_{3,-1/2}}\lesssim \varep_1^3\langle t\rangle^{-11\delta^2},
\end{equation*}
as a consequence of Lemma \ref{lemmaaux1}. The desired identity \eqref{TaylorExp} follows. 
The bound \eqref{TaylorExp2} follows using also the identity \eqref{DNIny}.

(ii) We define $u=u(x,y,t)$ as in \eqref{udef}, let $v=\partial_tu$, differentiate \eqref{EllEq} with respect to $t$, and find that $v$ satisfies \eqref{EllEq2} with
\begin{equation*}
\mathfrak{e}_a=\nabla_xu\cdot\nabla_x\partial_th-2\partial_yu\nabla_x h\cdot\nabla_x\partial_th,\qquad\mathfrak{e}_b
  =\mathcal{R}(\partial_yu\nabla_x\partial_th).
\end{equation*}
In view of \eqref{TaylorExp2},
\begin{equation*}
\|\partial_th\|_{\mathcal{O}_{1,-1/2}}+\|\partial_t\phi\|_{\mathcal{O}_{1,-1}}\lesssim\varep_1.
\end{equation*}
Therefore the triplet $(\partial_t\phi,\mathfrak{e}_a,\mathfrak{e}_b)$ satisfies \eqref{fma6.1} with $p=-3/2$. Therefore, using \eqref{fma6.2}, $$\|\partial_yv-|\nabla|v\|_{L^\infty_y\mathcal{O}_{2,-2}}\lesssim\varep_1^2,$$
and the desired bound \eqref{dtg} follows using also \eqref{DNIny}.
\end{proof}

\subsection{Proof of Proposition \ref{MainBootstrapDisp}}\label{ProofYu}

We show now that Proposition \ref{MainBootstrapDisp} follows from Proposition \ref{bootstrap}. The starting point is the 
system \eqref{WW0}. We need to verify that it can be rewritten in the form stated in Proposition \ref{bootstrap}. For this we need to expand the 
Dirichlet--Neumann operator
\begin{equation*}
 G(h)\phi=|\nabla|\phi+N_2[h,\phi]+N_3[h,h,\phi]+\text{Quartic\,Remainder},
\end{equation*}
and then prove the required claims. To justify this rigorously and estimate the remainder, the main issue is to prove space localization. We prefer not to work with the $Z$ norm itself, 
which is too complicated, but define instead certain auxiliary spaces which are used only in this section.

{\bf{Step 1.}} We assume that the bootstrap assumption (\ref{Ama32}) holds. Notice first that
\begin{align}
&\label{AssWeight2}\sup_{2a+|\alpha|\leq N_1+N_{4},\,a\leq N_1/2+20}\sum_{(k,j)\in\mathcal{J}}2^{\theta j}2^{-\theta|k|/2}\Vert Q_{jk} D^{\alpha}\Omega^a\mathcal{U}(t)\Vert_{L^2}\lesssim \varepsilon_1(1+t)^{\theta+6\delta^2},\\
&\label{AssWeight3}\sup_{2a+|\alpha|\leq N_1+N_{4},\,a\leq N_1/2+20}\sum_{(k,j)\in\mathcal{J}}2^{\theta j}2^{-\theta|k|/2}\Vert Q_{jk}D^{\alpha}\Omega^a\mathcal{U}(t)\Vert_{L^\infty}\lesssim \varepsilon_1(1+t)^{-5/6+\theta+6\delta^2},
\end{align}
for $\theta\in[0,1/3]$, where the operators $Q_{jk}$ are defined as in \eqref{qjk}. Indeed, let $f=e^{it\Lambda}\Omega^{a}D^{\alpha}\mathcal{U}(t)$ and assume that $t\in[2^m-1,2^{m+1}]$, $m\geq 0$. We have
\begin{equation}\label{extra01}
\|f\|_{H^{N_{0}'}\cap H_{\Omega}^{N'_{1}}}+\|f\|_{Z_1}\lesssim\varepsilon_{1}2^{\delta^2m},
\end{equation} 
as a consequence \eqref{Ama32}, where, as in \eqref{vd1}, $N'_1:=(N_1-N_4)/2=1/(2\delta)$ and $N'_0:=(N_0-N_3)/2-N_4=1/\delta$.
 To prove (\ref{AssWeight2}) we need to show that
\begin{equation}\label{extra02}
\sum_{(k,j)\in\mathcal{J}}2^{\theta j}2^{-\theta|k|/2}\|Q_{jk}e^{-it\Lambda}f\|_{L^{2}}\lesssim \varepsilon_{1}2^{\theta m+6\delta^2m}.
\end{equation}
The sum over $j\leq m+\delta^2m+|k|/2$ or over $j\leq |k|+\D$ is easy to control. On the other hand, if $j\geq \max(m+\delta^2m+|k|/2,|k|+\D)$ then we decompose $f=\sum_{(k',j')\in\mathcal{J}}f_{j',k'}$ as in \eqref{Alx100}. We may assume that $|k'-k|\leq 10$; the contribution of $j'\leq j-\delta^2j$ is negligible, using integration by parts, while for $j'\geq j-\delta^2j-10$ we have
\begin{equation*}
\|Q_{jk}e^{-it\Lambda}f_{j',k'}\|_{L^2}\lesssim \varep_12^{\delta^2m}\min(2^{-2j'/5},2^{-N'_0k^+}).
\end{equation*}
The desired bound \eqref{extra02} follows, which completes the proof of \eqref{AssWeight2}. The proof of \eqref{AssWeight3} is similar, using also the decay bound \eqref{LinftyBd3.5}. As a consequence, it follows that
\begin{equation}\label{AssWeight}
\begin{split}
&\sum_{(k,j)\in\mathcal{J}}2^{\theta j}2^{-\theta|k|/2}\Vert Q_{jk} g(t)\Vert_{L^2}\lesssim \varepsilon_12^{\theta m+6\delta^2m},\\
&\sum_{(k,j)\in\mathcal{J}}2^{\theta j}2^{-\theta|k|/2}\Vert Q_{jk} g(t)\Vert_{L^\infty}\lesssim \varepsilon_12^{-5m/6+\theta m+6\delta^2m}.
\end{split}
\end{equation}
for $g\in\{D^\alpha\Omega^a\langle\nabla\rangle h,D^\alpha\Omega^a|\nabla|^{1/2}\phi:\,2a+|\alpha| \leq N_1+N_{4},\,a\leq N_1/2+20\}$ and $\theta\in[0,1/3]$.
 
{\bf{Step 2.}} We need to define now certain norms that allow us to extend our estimates to the region $\{y\leq 0\}$, compare with the analysis in subsection \ref{DNLinear}.

\begin{lemma}\label{WeightedNormsLem}
For $q\geq 0$ and $\theta\in[0,1]$, $p,r\in[1,\infty]$, define the norms
\begin{equation*}
\begin{split}
\Vert f\Vert_{Y^p_{\theta,q}(\mathbb{R}^2)}&:=\sum_{(k,j)\in\mathcal{J}}2^{\theta j}2^{qk^+}\Vert Q_{jk}f\Vert_{L^p},\quad\Vert f\Vert_{L^r_yY_{\theta,q}^p(\mathbb{R}^2\times(-\infty,0])}:=\sum_{(k,j)\in\mathcal{J}}2^{\theta j}2^{qk^+}\Vert Q_{jk}f\Vert_{L^r_yL^p_x}.
\end{split}
\end{equation*}
(i) Then, for any $p\in[2,\infty]$ and $\theta\in[0,1]$,
\begin{equation}\label{cvb1}
\Vert e^{y\vert\nabla\vert}f\Vert_{L^\infty_yY_{\theta,q}^p}+\Vert |\nabla|^{1/2}e^{y\vert\nabla\vert}f\Vert_{L^2_yY_{\theta,q}^p}\lesssim \Vert f\Vert_{Y_{\theta,q}^p}
\end{equation}
and
\begin{equation}\label{cvb2}
\begin{split}
\Big\Vert \int_{-\infty}^0\vert\nabla\vert^{1/2} &e^{-\vert s-y\vert\vert\nabla\vert} \mathbf{1}_{\pm}(y-s)f(s)ds\Big\Vert_{L^\infty_yY^2_{\theta,q}}\\
&+\Big\Vert \int_{-\infty}^0\vert\nabla\vert e^{-\vert s-y\vert\vert\nabla\vert} \mathbf{1}_{\pm}(y-s)f(s)ds\Big\Vert_{L^2_yY^2_{\theta,q}}\lesssim \Vert f\Vert_{L^2_yY_{\theta,q}^2}.
\end{split}
\end{equation}

(ii) If $p_1,p_2,p,r_1,r_2,r\in\{2,\infty\}$, $1/p=1/p_1+1/p_2$, $1/r=1/r_1+1/r_2$ then 
\begin{equation}\label{cvb3}
\|(fg)\|_{L^r_yY^p_{\theta_1+\theta_2-\delta^2,q-\delta^2}}\lesssim \|f\|_{L^{r_1}_yY^{p_1}_{\theta_1,q}}\|g\|_{L^{r_2}_yY^{p_2}_{\theta_2,q}}
\end{equation}
provided that $\theta_1,\theta_2\in[0,1]$, $\theta_1+\theta_2\in[\delta^2,1]$, $q\geq \delta^2$. Moreover
\begin{equation}\label{cvb3.1}
\|(fg)\|_{L^2_yY^2_{\theta_1-\delta^2,q-\delta^2}}\lesssim \|f\|_{L^{\infty}_yY^{\infty}_{\theta_1,q}}\|g\|_{L^{2}_yH^q_x}.
\end{equation}
\end{lemma}

\begin{proof} The linear bounds in part (i) follow by parabolic estimates, once we notice that the kernel of the operator $e^{y|\nabla|}P_k$ is essentially localized in a ball of radius $\lesssim 2^{-k}$ and is bounded by $C2^{2k}(1+2^k|y|)^{-4}$. 

The bilinear estimates in part (ii) follow by unfolding the definitions. The implicit factors $2^{-\delta^2j}2^{-\delta^2k^+}$ in the left-hand side allow one to prove the estimate for $(k,j)$ fixed. Then one can decompose $f=\sum f_{j_1,k_1}$, $g= \sum g_{j_2,k_2}$ as in \eqref{Alx100} and estimate $\|Q_{jk}(f_{j_1,k_1}\cdot g_{j_2,k_2})\|_{L^r_yL^p_x}$ using simple product estimates. The case $j=-k\gg \min(j_1,j_2)$ requires some additional attention; in this case one can use first Sobolev imbedding and the hypothesis $\theta_1+\theta_2\leq 1$.
\end{proof}
 
{\bf{Step 3.}} Recall now the formula \eqref{fma2}
\begin{equation*}
\begin{split}
&u = e^{y\vert\nabla\vert}\phi+L(u),\\
&L(u):=-\frac{1}{2}e^{y\vert\nabla\vert}\int_{-\infty}^0e^{s\vert\nabla\vert}(Q_a(s)-Q_b(s))ds+\frac{1}{2}\int_{-\infty}^0e^{-\vert y-s\vert\vert\nabla\vert}(\mathrm{sgn}(y-s)Q_a(s)-Q_b(s))ds,
\end{split}
\end{equation*}
where $Q_a[u]=\nabla u\cdot\nabla h-|\nabla h|^2\partial_yu$ and $Q_b[u]=\mathcal{R}(\partial_yu\nabla h)$. Let, as in Corollary \ref{CorFixedPoint},
\begin{equation}\label{cvb3.5}
u^{(1)} = e^{y\vert\nabla\vert}\phi,\qquad u^{(n+1)}= e^{y\vert\nabla\vert}\phi+L(u^{(n)}),\ n\geq 1.
\end{equation}

We can prove now a precise asymptotic expansion on the Dirichlet--Neumann operator.

\begin{lemma}\label{cvb100} We have
\begin{equation}\label{cvb8.7}
G(h)\phi=|\nabla|\phi+N_2[h,\phi]+N_3[h,\phi]+|\nabla|^{1/2}N_4[h,\phi],
\end{equation}
where $N_2$ is as in \eqref{cvb8.8}, 
\begin{equation}\label{cvb8.9}
\begin{split}
&\mathcal{F}\{N_{3}[h,\phi]\}(\xi)=\frac{1}{(4\pi^2)^2}\int_{(\mathbb{R}^2)^2}n_3(\xi,\eta,\sigma)\widehat{h}(\xi-\eta)\widehat{h}(\eta-\sigma)\widehat{\phi}(\sigma)\,d\eta d\sigma,\\
&n_3(\xi,\eta,\sigma):=\frac{|\xi||\sigma|}{|\xi|+|\sigma|}\big[(|\xi|-|\eta|)(|\eta|-|\sigma|)-(\xi-\eta)\cdot(\eta-\sigma)\big],
\end{split}
\end{equation}
and, for $\theta\in [\delta^2,1/3]$ and $V\in\{D^\alpha\Omega^a:a \leq N_1/2+20,\,2a+|\alpha|\leq N_1+N_{4}-2\}$,
\begin{equation}\label{cvb8.10}
\|VN_4[h,\phi]\|_{Y^2_{3\theta-3\delta^2,1-3\delta ^2}}\lesssim \varep_1^4 2^{3\theta m-5m/2+24\delta^2m}.
\end{equation}
\end{lemma}

\begin{proof} Recall that $h$ is constant in $y$. In view of \eqref{AssWeight} we have, for $t\in[2^m-1,2^{m+1}]$,
\begin{equation}\label{cvb9}
\||\nabla|^{1/6}\langle\nabla\rangle^{5/6}Vh(t)\|_{L^\infty_yY^2_{\theta,1}}\lesssim \varep_1 2^{\theta m+6\delta^2m},\qquad\theta\in[0,1/3],
\end{equation}
and
\begin{equation}\label{cvb9.5}
\||\nabla|^{1/6}\langle\nabla\rangle^{5/6}Vh(t)\|_{L^\infty_yY^\infty_{\theta,1}}\lesssim \varep_1 2^{\theta m-5m/6+6\delta^2m},\qquad\theta\in[0,1/3],
\end{equation}
for $V\in\{D^\alpha\Omega^a:a \leq N_1/2+20,\,2a+|\alpha|\leq N_1+N_{4}-2\}$. Moreover, using also \eqref{fma6}
\begin{equation}\label{cvb9.6}
\||\nabla|Vu(t)\|_{L^2_yH^1_x}+\|(\partial_y Vu)(t)\|_{L^2_yH^1_x}\lesssim \varep_1 2^{6\delta^2m},
\end{equation}
for operators $V$ as before. Therefore, using \eqref{cvb3.1},
\begin{equation*}
\|V[Q[u]]\|_{L^2_yY^2_{\theta-\delta^2,1-\delta^2}}\lesssim \varep_1^2 2^{\theta m-5m/6+12\delta^2m},
\end{equation*}
for $Q\in \{Q_a,Q_b\}$ and $\theta\in[\delta^2,1/3]$. Therefore
\begin{equation}\label{cvb11}
\||\nabla|VL(u)\|_{L^2_yY^2_{\theta-\delta^2,1-\delta^2}}+\|\partial_yVL(u)\|_{L^2_yY^2_{\theta-\delta^2,1-\delta^2}}\lesssim \varep_1^2 2^{\theta m-5m/6+12\delta^2m},
\end{equation}
using \eqref{cvb1}--\eqref{cvb2}. Therefore, using the definition,
\begin{equation}\label{cvb11.1}
\||\nabla|V[u-u^{(1)}]\|_{L^2_yY^2_{\theta-\delta^2,1-\delta^2}}+\|\partial_yV[u-u^{(1)}]\|_{L^2_yY^2_{\theta-\delta^2,1-\delta^2}}\lesssim \varep_1^2 2^{\theta m-5m/6+12\delta^2m}.
\end{equation}

Since $u-u^{(2)}=L(u-u^{(1)})$, we can repeat this argument to prove that for $\theta\in [\delta^2,1/3]$ and $V\in\{D^\alpha\Omega^a:a \leq N_1/2+20,\,2a+|\alpha|\leq N_1+N_{4}-2\}$,
\begin{equation}\label{cvb8.6}
\||\nabla|V[u-u^{(2)}]\|_{L^2_yY^2_{2\theta-2\delta^2,1-2\delta^2}}+\|\partial_yV[u-u^{(2)}]\|_{L^2_yY^2_{2\theta-2\delta^2,1-2\delta^2}}\lesssim \varep_1^3 2^{2\theta m-5m/3+18\delta^2m}.
\end{equation}

To prove the decomposition \eqref{cvb8.7} we start from the identities \eqref{QuadraticTermDN} and \eqref{DNIny}, which gives $G(h)\phi=\partial_yu-Q_a$. Letting $Q_a^{(n)}=Q_a[u^{(n)}]$,  $Q_b^{(n)}=Q_b[u^{(n)}]$, $n\in\{1,2\}$, it follows that
\begin{equation}\label{cvb15}
\begin{split}
&G(h)\phi=|\nabla|\phi+\int_{-\infty}^0|\nabla|e^{-|s||\nabla|}(Q_b^{(2)}(s)-Q_a^{(2)}(s))\,ds+N_{4,1},\\
&N_{4,1}:=\int_{-\infty}^0|\nabla|e^{-|s||\nabla|}[(Q_b-Q_b^{(2)})(s)-(Q_a-Q_a^{(2)})(s)]\,ds.
\end{split}
\end{equation}
In view of \eqref{cvb8.6}, \eqref{cvb9.5}, and the algebra rule \eqref{cvb3.1}, we have
\begin{equation*}
\|V(Q-Q^{(2)})\|_{L^2_yY^2_{3\theta-3\delta^2,1-3\delta^2}}\lesssim \varep_1^4 2^{3\theta m-5m/2+24\delta^2m},
\end{equation*}
for $Q\in\{Q_a,Q_b\}$. Therefore, using \eqref{cvb2}, $|\nabla|^{-1/2}N_{4,1}$ satisfies the desired bound \eqref{cvb8.10}.

It remains to calculate the integral in the first line of \eqref{cvb15}. Letting $\alpha=|\nabla h|^2$ we have
\begin{equation}\label{cvb17}
\begin{split}
&\mathcal{F}\{u^{(1)}\}(\xi,y)=e^{y|\xi|}\widehat{\phi}(\xi),\\
&\mathcal{F}\{Q_a^{(1)}\}(\xi,y)=-\frac{1}{4\pi^2}\int_{\mathbb{R}^2}(\xi-\eta)\cdot\eta e^{y|\eta|}\widehat{h}(\xi-\eta)\widehat{\phi}(\eta)\,d\eta-\frac{1}{4\pi^2}\int_{\mathbb{R}^2} |\eta|e^{y|\eta|}\widehat{\alpha}(\xi-\eta)\widehat{\phi}(\eta)\,d\eta,\\
&\mathcal{F}\{Q_b^{(1)}\}(\xi,y)=-\frac{1}{4\pi^2}\int_{\mathbb{R}^2}\frac{(\xi-\eta)\cdot\xi}{|\xi|}|\eta| e^{y|\eta|}\widehat{h}(\xi-\eta)\widehat{\phi}(\eta)\,d\eta.
\end{split}
\end{equation}
Therefore 
\begin{equation*}
\begin{split}
\mathcal{F}\{L(u^{(1)})\}(\xi,y)&=\frac{1}{8\pi^2}\int_{\mathbb{R}^2}(e^{y|\xi|}-e^{y|\eta|})\Big[\frac{(\xi-\eta)\cdot\eta}{|\xi|+|\eta|}-\frac{|\eta|(\xi-\eta)\cdot\xi}{|\xi|(|\xi|+|\eta|)}\Big]\widehat{h}(\xi-\eta)\widehat{\phi}(\eta)\,d\eta\\
&+\frac{1}{8\pi^2}\int_{\mathbb{R}^2}(e^{y|\xi|}-e^{y|\eta|})\Big[\frac{(\xi-\eta)\cdot\eta}{-|\xi|+|\eta|}+\frac{|\eta|(\xi-\eta)\cdot\xi}{|\xi|(-|\xi|+|\eta|)}\Big]\widehat{h}(\xi-\eta)\widehat{\phi}(\eta)\,d\eta\\
&+\widehat{E_1}(\xi,y),
\end{split}
\end{equation*}
where
\begin{equation*}
\||\nabla|VE_1\|_{L^2_yY^2_{2\theta-2\delta^2,1-2\delta^2}}+\|\partial_yVE_1\|_{L^2_yY^2_{2\theta-2\delta^2,1-2\delta^2}}\lesssim \varep_1^3 2^{2\theta m-5m/3+18\delta^2m}.
\end{equation*}
After algebraic simplifications, this gives
\begin{equation*}
\mathcal{F}\{L(u^{(1)})\}(\xi,y)=-\frac{1}{4\pi^2}\int_{\mathbb{R}^2}(e^{y|\xi|}-e^{y|\eta|})|\eta|\widehat{h}(\xi-\eta)\widehat{\phi}(\eta)\,d\eta+\widehat{E_1}(\xi,y).
\end{equation*}

Since $u^{(2)}-u^{(1)}=L(u^{(1)})$ we calculate
\begin{equation}\label{cvb20}
\begin{split}
\mathcal{F}&\{Q_a^{(2)}-Q_a^{(1)}\}(\xi,y)\\
&=\frac{1}{16\pi^4}\int_{(\mathbb{R}^2)^2}|\sigma|(\xi-\eta)\cdot\eta (e^{y|\eta|}-e^{y|\sigma|})\widehat{h}(\xi-\eta)\widehat{h}(\eta-\sigma)\widehat{\phi}(\sigma)\,d\eta d\sigma+\widehat{E_2}(\xi,y)
\end{split}
\end{equation}
and
\begin{equation}\label{cvb21}
\begin{split}
\mathcal{F}&\{Q_b^{(2)}-Q_b^{(1)}\}(\xi,y)\\
&=\frac{1}{16\pi^4}\int_{(\mathbb{R}^2)^2}|\sigma|\frac{(\xi-\eta)\cdot\xi}{|\xi|} (|\eta|e^{y|\eta|}-|\sigma|e^{y|\sigma|})\widehat{h}(\xi-\eta)\widehat{h}(\eta-\sigma)\widehat{\phi}(\sigma)\,d\eta d\sigma+\widehat{E_3}(\xi,y)
\end{split}
\end{equation}
where
\begin{equation*}
\|VE_2\|_{L^2_yY^2_{3\theta-3\delta^2,1-3\delta^2}}+\|VE_3\|_{L^2_yY^2_{3\theta-3\delta^2,1-3\delta^2}}\lesssim \varep_1^4 2^{3\theta m-5m/2+24\delta^2m}.
\end{equation*}

We examine now the formula in the first line of \eqref{cvb15}. The contributions of $E_2$ and $E_3$ can be estimated as part of the quartic error term, using also \eqref{cvb2}. The main contributions can be divided into quadratic terms (coming from $Q_a^{(1)}$ and $Q_b^{(1)}$ in \eqref{cvb17}), and cubic terms coming from \eqref{cvb20}--\eqref{cvb21} and the cubic term in $Q_a^{(1)}$. The conclusion of the lemma follows.
\end{proof}

{\bf{Step 4.}} Finally, we can prove the desired expansion of the water-wave system.

\begin{lemma}\label{TaylorNew0}
Assume that $(h,\phi)$ satisfy \eqref{WW0} and \eqref{Ama32}. Then 
\begin{equation}\label{cvb30}
(\partial_t+i\Lambda)\mathcal{U}=\mathcal{N}_2+\mathcal{N}_3+\mathcal{N}_{\geq 4},
\end{equation}
where $\mathcal{U}=\langle\nabla\rangle h+i|\nabla|^{1/2}\phi$ and $\mathcal{N}_2$, $\mathcal{N}_3$, $\mathcal{N}_{\geq 4}$ are as in subsection \ref{Duhamel}.
\end{lemma}

\begin{proof} We rewrite \eqref{WW0} in the form
\begin{equation}\label{cvb30.5}
\partial_t \mathcal{U}= \langle\nabla\rangle G(h) \phi+i|\nabla|^{1/2}\Big[-h  +\div \Big[ \dfrac{\nabla h}{ (1+|\nabla h|^2)^{1/2} } \Big]
  - \dfrac{1}{2} {|\nabla \phi|}^2 + \dfrac{( G(h)\phi + \nabla h \cdot \nabla \phi)^2}{2(1+{|\nabla h|}^2)}\Big].
\end{equation}
We use now the formula \eqref{cvb8.7} to extract the linear, the quadratic, and the cubic terms in the right-hand side of this formula. More precisely, we set
\begin{equation}\label{cvb31}
\begin{split}
&\mathcal{N}_1:=\langle\nabla\rangle|\nabla|\phi+i|\nabla|^{1/2}(-h+\Delta h)=-i\Lambda\mathcal{U},\\
&\mathcal{N}_2:=\langle\nabla\rangle N_2[h,\phi]+i|\nabla|^{1/2}\big[-\dfrac{1}{2} {|\nabla \phi|}^2 + \dfrac{1}{2}(|\nabla|\phi)^2\big],\\
&\mathcal{N}_3:=\langle\nabla\rangle N_3[h,h,\phi]+i|\nabla|^{1/2}\big[-\frac{1}{2}\div(\nabla h|\nabla h|^2)+|\nabla|\phi\cdot(N_2[h,\phi]+\nabla h\cdot\nabla\phi)\big].
\end{split}
\end{equation}

Then we substitute $h=\langle\nabla\rangle^{-1}(\mathcal{U}+\overline{\mathcal{U}})/2$ and $|\nabla|^{1/2}\phi=(\mathcal{U}-\overline{\mathcal{U}})/(2i)$. The symbols that define the quadratic component $\mathcal{N}_2$ are linear combinations of the symbols
\begin{equation*}
n_{2,1}(\xi,\eta)=\sqrt{1+\vert\xi\vert^2}\frac{\xi\cdot\eta-\vert\xi\vert\vert\eta\vert}{\vert\eta\vert^{1/2}\sqrt{1+\vert\xi-\eta\vert^2}},\qquad
n_{2,2}(\xi,\eta)=\vert\xi\vert^{1/2}\frac{(\xi-\eta)\cdot\eta+\vert\xi-\eta\vert\vert\eta\vert}{\vert\xi-\eta\vert^{1/2}\vert\eta\vert^{1/2}}.
\end{equation*}
It is easy to see that these symbols verify the properties \eqref{Assumptions2}. A slightly nontrivial argument is needed for $n_{2,1}$ in the case $k_1=\min(k,k_1,k_2)\ll k$.

The cubic terms in $\mathcal{N}_3$ in \eqref{cvb31} are defined by finite linear combinations of the symbols
\begin{equation*}
\begin{split}
&n_{3,1}(\xi,\eta,\sigma)=\sqrt{\frac{1+\vert\xi\vert^2}{(1+\vert\xi-\eta\vert^2)(1+\vert\eta-\sigma\vert^2)}}\frac{|\xi||\sigma|^{1/2}}{|\xi|+|\sigma|}\big[(|\xi|-|\eta|)(|\eta|-|\sigma|)-(\xi-\eta)\cdot(\eta-\sigma)\big],\\
&n_{3,2}(\xi,\eta,\sigma)=\vert\xi\vert^{1/2}\frac{(\xi\cdot(\xi-\eta))((\eta-\sigma)\cdot\sigma)}{\sqrt{(1+\vert\xi-\eta\vert^2)(1+\vert\eta-\sigma\vert^2)(1+\vert\sigma\vert^2)}},\\
&n_{3,3}(\xi,\eta,\sigma)=\vert\xi\vert^{1/2}\vert\xi-\eta\vert^{1/2}\vert\sigma\vert^{1/2}\frac{\vert\sigma-\eta\vert}{\sqrt{1+\vert\eta-\sigma\vert^2}}.
\end{split}
\end{equation*}
It is easy to verify the properties \eqref{Assumptions3} for these explicit symbols.

The higher order remainder in the right-hand of \eqref{cvb30.5} can be written in the form
\begin{equation}\label{cvb40}
\mathcal{N}_{\geq 4}=|\nabla|^{1/2}N'_4,\qquad \sup_{a \leq N_1/2+20,\,2a+|\alpha|\leq N_1+N_4-4}\|D^\alpha\Omega^aN'_4\|_{Y^2_{1-\delta,1-\delta}}\lesssim \varep_1^4 2^{-3m/2+\delta m},
\end{equation}
using \eqref{cvb8.10}, \eqref{AssWeight}, and the algebra property \eqref{cvb3}. Moreover, using only the $\mathcal{O}$ hierarchy as in the proof of Corollary \ref{CorFixedPoint}, we have $\|\mathcal{N}_{\geq 4}\|_{\mathcal{O}_{4,-4}}\lesssim\varep_1^4$, i.e.
\begin{equation}\label{cvb41}
\|\mathcal{N}_{\geq 4}\|_{H^{N_0-4}}+\|\mathcal{N}_{\geq 4}\|_{H^{N_1,N_3-4}}\lesssim\varep_1^4 2^{-5m/2+\delta m}.
\end{equation}
These two bounds suffice to prove the desired claims on $\mathcal{N}_{\geq 4}$ in \eqref{bootstrap2}. Indeed, the $L^2$ bound follows directly from \eqref{cvb41}. For the $Z$ norm bound it suffices to prove that, for any $(k,j)\in\mathcal{J}$,
\begin{equation}\label{cvb42}
\sup_{a \leq N_1/2+20,\,2a+|\alpha|\leq N_1+N_4}2^{j(1-50\delta)}\|Q_{jk}e^{it\Lambda}D^\alpha\Omega^a\mathcal{N}_{\geq 4}\|_{L^2}\lesssim \varep_1^42^{-m-\delta m}.
\end{equation}
This follows easily from \eqref{cvb41} and \eqref{cvb40}, unless
\begin{equation*}
j\geq 3m/2+(N_0/4)k^++\D\quad\text{ and }\quad j\geq 3m/2-k/2+\D.
\end{equation*}
On the other hand, if these inequalities hold then let $f=D^\alpha\Omega^a\mathcal{N}_{\geq 4}$, $a \leq N_1/2+20$, $2a+|\alpha|\leq N_1+N_4$, and decompose $f=\sum_{(k',j')\in\mathcal{J}} f_{j',k'}$ as in \eqref{Alx100}. 
The bound \eqref{cvb40} shows that
\begin{equation}\label{cvb43}
\sum_{(k',j')\in\mathcal{J}}2^{-4\max(k',0)}2^{j'(1-\delta)}\|f_{j',k'}\|_{L^2}\lesssim\varep_1^4 2^{-3m/2+\delta m}.
\end{equation}
The desired bound \eqref{cvb41} follows by the usual approximate finite speed of propagation argument: we may assume $|k'-k|\leq 4$ and consider the cases $j'\leq j-\delta j$ (which gives negligible contributions) and $j'\geq j-\delta j$ (in which case \eqref{cvb43} suffices). This completes the proof. 
\end{proof}


\end{document}